\documentclass[a4paper,11pt]{amsart}

\usepackage[T1]{fontenc}
\usepackage{amssymb,amsfonts,amsmath,amsthm}
\usepackage{url}
\usepackage{color}
\usepackage{setspace}
\usepackage{enumerate}
\usepackage{epsfig,graphicx,pinlabel}
\usepackage{mathrsfs}
\usepackage{tikz}

\input{xy}
\xyoption{all}
\objectmargin={3mm}

\newcounter{notes}
\newcommand{\ignore}[1]{}

\newtheorem{theorem}{Theorem}
\newtheorem{proposition}[theorem]{Proposition}
\newtheorem{corollary}[theorem]{Corollary}
\newtheorem{lemma}[theorem]{Lemma}
\newtheorem{claim}[theorem]{Claim}

\newtheorem{conjecture}[theorem]{Conjecture}
\newtheorem{fact}[theorem]{Fact}

\theoremstyle{definition}
\newtheorem{definition}[theorem]{Definition}
\newtheorem{remark}[theorem]{Remark}
\newtheorem{remarks}[theorem]{Remarks}

\newtheorem{notation}[theorem]{Notation}
\newtheorem{example}[theorem]{Example}

\newtheoremstyle{theoremwithref}{}{}{\itshape}{}{\bfseries}{.}{.5em}{#1 #2 #3}
\theoremstyle{theoremwithref}

\newcommand{\ie}{i.e.\ }
\newcommand{\eg}{e.g.\ }
\newcommand{\resp}{\text{resp.\ }}
\newcommand{\C}{\mathbb{C}}
\newcommand{\R}{\mathbb{R}}
\newcommand{\Z}{\mathbb{Z}}
\newcommand{\N}{\mathbb{N}}
\newcommand{\SL}{\mathrm{SL}}
\newcommand{\PSL}{\mathrm{PSL}}
\newcommand{\PGL}{\mathrm{PGL}}
\newcommand{\OO}{\mathrm{O}}
\newcommand{\PO}{\mathrm{PO}}
\newcommand{\SO}{\mathrm{SO}}

\newcommand{\PSU}{\mathrm{PSU}}

\newcommand{\Hom}{\mathrm{Hom}}

\newcommand{\Diag}{\mathrm{Diag}}
\newcommand{\HH}{\mathbb{H}}
\newcommand{\Isom}{\mathrm{Isom}}

\newcommand{\Lip}{\mathrm{Lip}}

\newcommand{\arccosh}{\mathrm{arccosh}}
\newcommand{\arcsinh}{\mathrm{arcsinh}}
\newcommand{\A}{\mathcal{A}}
\newcommand{\D}{\mathcal{D}}
\newcommand{\T}{\mathcal{T}}
\newcommand{\F}{\mathcal{F}}
\newcommand{\LL}{\mathscr{L}}
\newcommand{\wl}{\mathtt{wl}}
\newcommand{\Conv}[1]{\mathrm{Conv}(#1)}
\newcommand*\circled[1]{\tikz[baseline=(char.base)]{\node[shape=circle,draw,inner sep=1pt] (char) {#1};}}

\title[Maximally stretched laminations]{Maximally stretched laminations on geometrically finite hyperbolic manifolds}

\author{Fran\c{c}ois Gu\'eritaud}
\author{Fanny Kassel}
\address{CNRS and Universit\'e Lille 1, Laboratoire Paul Painlev\'e, 59655 Villeneuve d'Ascq Cedex, France}
\email{\newline francois.gueritaud@math.univ-lille1.fr \newline fanny.kassel@math.univ-lille1.fr}

\thanks{The authors were partially supported by the Agence Nationale de la Recherche under the grants ANR-11-BS01-013: DiscGroup (``Facettes des groupes discrets'') and ANR-09-BLAN-0116-01: ETTT (``Extension des th\'eories de Teichm\"uller--Thurston''), and through the Labex CEMPI (ANR-11-LABX-0007-01)}

\begin{document}

\numberwithin{theorem}{section}
\numberwithin{equation}{section}

\begin{abstract}
Let $\Gamma_0$ be a discrete group.
For a pair $(j,\rho)$ of representations of $\Gamma_0$ into $\PO(n,1)=\Isom(\HH^n)$ with $j$ geometrically finite, we study the set of $(j,\rho)$-equivariant Lipschitz maps from the real hyperbolic space~$\HH^n$ to itself that have minimal Lipschitz constant.
Our main result is the existence of a geodesic lamination that is ``maximally stretched'' by all such maps when the minimal constant is at least~$1$.
As an application, we generalize two-dimensional results and constructions of Thurston and extend his asymmetric metric on Teichm\"uller space to a geometrically finite setting and to higher dimension.
Another application is to actions of discrete subgroups $\Gamma$ of $\PO(n,1)\times\PO(n,1)$ on $\PO(n,1)$ by right and left multiplication: we give a double properness criterion for such actions, and prove that for a large class of groups~$\Gamma$ the action remains properly discontinuous after any small deformation of $\Gamma$ inside $\PO(n,1)\times\PO(n,1)$.
\end{abstract}

\maketitle

\tableofcontents


\section{Introduction}\label{sec:intro}

For $n\geq 2$, let $G$ be the group $\PO(n,1)=\OO(n,1)/\{ \pm 1\}$ of isometries of the real hyperbolic space~$\HH^n$.
In this paper we consider pairs $(j,\rho)$~of~represen\-tations of a discrete group $\Gamma_0$ into~$G$ with $j$ injective, discrete, and~$j(\Gamma_0)\backslash\HH^n$ geometrically finite, and we investigate the set of $(j,\rho)$-equivariant Lipschitz maps $\HH^n\rightarrow\HH^n$ with minimal Lipschitz constant.
We develop applications, both to properly discontinuous actions on~$G$ and to the geometry of some generalized Teichm\"uller spaces (via a generalization of Thurston's asymmetric metric).
Some of our main results,~in~parti\-cular Theorems \ref{thm:adm}~and \ref{thm:deform}, Corollary~\ref{cor:CC'}, and Theorem~\ref{thm:sharp}, were initially obtained in \cite[Chap.\,5]{kasPhD} in the case $n=2$ with $j$ convex cocompact.

\subsection{Equivariant maps of~$\HH^n$ with minimal Lipschitz constant}

Let $\Gamma_0$ be a discrete group.
We say that a representation $j\in\Hom(\Gamma_0,G)$ of $\Gamma_0$ in $G=\PO(n,1)$ is \emph{convex cocompact} (\resp \emph{geometrically finite}) if it is injective with a discrete image $j(\Gamma_0)\subset G$ and if the convex core of the hyperbolic orbifold $M:=j(\Gamma_0)\backslash\HH^n$ is compact (\resp has finite $m$-volume, where $m\leq n$ is its dimension).
In this case, the group~$\Gamma_0$ identifies with the (orbifold) fundamental group of~$M$.
Parabolic elements in $j(\Gamma_0)$ correspond to \emph{cusps} in~$M$; they do not exist if $j$ is convex cocompact.
We refer to Section~\ref{subsec:geo-finiteness} for full definitions.

Let $j\in\Hom(\Gamma_0,G)$ be geometrically finite and let $\rho\in\Hom(\Gamma_0,G)$ be another representation, not necessarily injective or discrete.
In this paper we examine \emph{$(j,\rho)$-equivariant} Lipschitz maps~of~$\HH^n$, \ie Lipschitz maps $f :\nolinebreak\HH^n\rightarrow\HH^n$ such that
$$f(j(\gamma)\cdot x) = \rho(\gamma)\cdot f(x)$$
for all $\gamma\in\Gamma_0$ and $x\in\HH^n$.
A constant that naturally appears is the infimum of all possible Lipschitz constants of such maps:
\begin{equation}\label{eqn:defC}
C(j,\rho) := \inf\big\{ \Lip(f)~|~f : \HH^n\rightarrow\HH^n\text{ $(j,\rho)$-equivariant}\big\}.
\end{equation}
A basic fact (Section~\ref{subsec:C<infty}) is that $C(j,\rho)<+\infty$ unless there is an obvious obstruction, namely an element $\gamma\in\Gamma_0$ with $j(\gamma)$ parabolic and $\rho(\gamma)$ hyperbolic.
Here we use the usual terminology: a nontrivial element of~$G$ is \emph{elliptic} if it fixes a point in~$\HH^n$, \emph{parabolic} if it fixes exactly one point on the boundary at infinity of~$\HH^n$, and \emph{hyperbolic} otherwise (in which case it preserves a unique geodesic line in~$\HH^n$).
To make the statements of our theorems simpler, we include the identity element of~$G$ among the elliptic elements.

We shall always assume $C(j,\rho)<+\infty$.
Then there exists a $(j,\rho)$-equiva\-riant map $f : \HH^n\rightarrow\HH^n$ with minimal constant $C(j,\rho)$, except possibly if the group $\rho(\Gamma_0)$ has a unique fixed point on the boundary at infinity $\partial_{\infty}\HH^n$ of~$\HH^n$ (see Section~\ref{subsec:Fnonempty}, as well as Sections \ref{ex:nonreductive1} and~\ref{ex:nonreductive2} for examples).

We fix once and for all a geometrically finite representation $j_0\in\Hom(\Gamma_0,G)$.
Dealing with cusps is a substantial aspect of the paper; we make the following definitions, which are relevant only when $j$ is \emph{not} convex cocompact.

\begin{definition}\label{def:typedet}
We say that $j\in\Hom(\Gamma_0,G)$ has the \emph{cusp type} of $j_0$ if for any $\gamma\in\Gamma_0$, the element $j(\gamma)$ is parabolic if and only if $j_0(\gamma)$ is parabolic.
We say that $\rho\in\Hom(\Gamma_0,G)$ is \emph{cusp-deteriorating} with respect to $j$ (or that the pair $(j,\rho)$ is cusp-deteriorating) if for any $\gamma\in\Gamma_0$ with $j(\gamma)$ parabolic, the element $\rho(\gamma)$ is elliptic.
\end{definition}

In the sequel, we will always assume that $j$ has the cusp type of the fixed representation~$j_0$.
Therefore, we will often just use the phrase ``$\rho$ cusp-deteriorating'', leaving $j$ implied.
Of course, this is an empty condition when $j$ is convex cocompact.

\subsection{The stretch locus}\label{subsec:introstretchlocus}

The main point of the paper is to initiate a systematic study of the \emph{stretch locus} of equivariant maps of~$\HH^n$ with minimal Lipschitz constant.

\begin{definition}\label{def:stretchlocus}
Let $f : \HH^n\rightarrow\HH^n$ be a $(j,\rho)$-equivariant map realizing the minimal Lipschitz constant $C(j,\rho)$.
The \emph{stretch locus} $E_f$ of~$f$ is the ($j(\Gamma_0)$-invariant) set of points $x\in\HH^n$ such that the restriction of~$f$ to any neighborhood of~$x$ in~$\HH^n$ has Lipschitz constant $C(j,\rho)$ (and no smaller).
\end{definition}

It follows from our study that the geometry of the stretch locus depends on the value of $C(j,\rho)$.
We prove the following.

\begin{theorem}\label{thm:lamin}
Let $(j,\rho)\in\Hom(\Gamma_0,G)^2$ be a pair of representations with $j$ geometrically finite and $C(j,\rho)<+\infty$.
Assume that there exists a $(j,\rho)$-equivariant map $f : \HH^n\rightarrow\HH^n$ with minimal Lipschitz constant $C(j,\rho)$, and let $E(j,\rho)$ be the intersection of the stretch loci of all such maps.
Then
\begin{itemize}
  \item $E(j,\rho)$ is nonempty, except possibly if $C(j,\rho)=\nolinebreak 1$ and $\rho$ is \emph{not} cusp-deteriorating (see Section~\ref{ex:nondeteriorating} for an example);
  \item there exists an ``optimal'' $(j,\rho)$-equivariant, $C(j,\rho)$-Lipschitz map\linebreak $f_0 : \HH^n\rightarrow\HH^n$ whose stretch locus is exactly $E(j,\rho)$;
  \item if $C(j,\rho)>1$ (\resp if $C(j,\rho)=1$ and $\rho$ is cusp-deteriorating), then $E(j,\rho)$ is a geodesic lamination (\resp contains a $k$-dimensional geodesic lamination for some $k\geq 1$) with the following properties:
  \begin{itemize}
    \item the lamination is ``maximally stretched'' by any $(j,\rho)$-equivariant map $f : \HH^n\rightarrow\HH^n$ with minimal Lipschitz constant $C(j,\rho)$, in the sense that $f$ multiplies all distances by $C(j,\rho)$ on every leaf of the lamination;
    \item the projection to $j(\Gamma_0)\backslash\HH^n$ of the lamination is compact and contained in the convex core.
  \end{itemize}
\end{itemize}
\end{theorem}

By a geodesic lamination (\resp a $k$-dimensional geodesic lamination) of~$\HH^n$ we mean a nonempty disjoint union $\LL$ of geodesic lines (\resp $k$-planes) of~$\HH^n$, called \emph{leaves}, that is closed in the space of geodesic lines (\resp $k$-planes) of~$\HH^n$.
The image in $j(\Gamma_0)\backslash\HH^n$ of a $j(\Gamma_0)$-invariant geodesic lamination of~$\HH^n$ is a geodesic lamination in the usual sense.

We note that an ``optimal'' map~$f_0$ is usually not unique since it can be slightly perturbed outside of the stretch locus $E(j,\rho)$.

In Section~\ref{subsec:Thurstonchainrec} we explain how, in the case that $n=2$ and that $j$ and~$\rho$ are both injective and discrete with finite covolume, Theorem~\ref{thm:lamin} follows from Thurston's theory \cite{thu86} of the asymmetric metric on Teichm\"uller space.

More precise results in the case $C(j,\rho)=1$ are given (for arbitrary~$n$) in Section~\ref{sec:Kirszbraunopt}, leading to a reasonable understanding of the stretch locus when $C(j,\rho)\geq 1$.
On the other hand, for $C(j,\rho)<1$ the stretch locus is more mysterious; we make the following conjecture.

\begin{conjecture} \label{conj:gramination}
For $n=2$, let $(j,\rho)\in\Hom(\Gamma_0,G)^2$ be a pair of representations with $j$ geometrically finite and let $E(j,\rho)$ be the intersection of the stretch loci of all $(j,\rho)$-equivariant maps with minimal Lipschitz constant $C(j,\rho)\in (0,1)$.
Then $E(j,\rho)$ is the lift to~$\HH^2$ of a \emph{gramination} (contraction of ``graph'' and ``lamination'') of $M:=j(\Gamma_0)\backslash\HH^2$, by which we mean the union of a finite set~$F$ and of a lamination in $M\smallsetminus F$ with finitely many leaves terminating on~$F$.
\end{conjecture}

We discuss this conjecture and provide evidence for it in Section~\ref{subsec:EforC<1}.

We also examine the behavior of the minimal Lipschitz constant $C(j,\rho)$ and of the stretch locus $E(j,\rho)$ under small deformations of $j$ and~$\rho$.
We prove that the constant $C(j,\rho)$ behaves well for convex cocompact~$j$.

\begin{proposition}\label{prop:contCcc}
The map $(j,\rho)\mapsto C(j,\rho)$ is continuous on the set of pairs $(j,\rho)\in\Hom(\Gamma_0,G)^2$ with $j$ convex cocompact.
\end{proposition}

Here $\Hom(\Gamma_0,G)$ is endowed with the natural topology (see Section~\ref{sec:lipcont}).

For $j$ geometrically finite but \emph{not} convex cocompact, the constant $C(j,\rho)$ behaves in a more chaotic way.
For $n\leq 3$, we prove that continuity holds when $C(j,\rho)>1$ and that the condition $C(j,\rho)<1$ is open on the set of pairs $(j,\rho)$ with $j$ geometrically finite of fixed cusp type and $\rho$ cusp-deteriorating (Proposition~\ref{prop:contpropertiesC}).
However, semicontinuity (both upper and lower) fails when $C(j,\rho)\leq 1$ (see Sections \ref{ex:discontinu} and~\ref{ex:discontinu2}).
For $n\geq 4$, the condition $C(j,\rho)<1$ is not open and upper semicontinuity fails for any value of $C(j,\rho)$ (see Sections \ref{ex:dim4upper} and~\ref{ex:dim4C<1}).
This is related to the fact that geometrical finiteness itself is not an open condition when $n\geq 4$, even under fixed cusp type.
We refer to Section~\ref{sec:lipcont} for a more thorough discussion.

It is natural to hope that when the function $(j,\rho)\mapsto C(j,\rho)$ is continuous the map $(j,\rho)\mapsto E(j,\rho)$ should be at least upper semicontinuous with respect to the Hausdorff topology.
We prove this semicontinuity in dimension $n=2$ when $C(j,\rho)>1$ and $\rho(\Gamma_0)$ does not have a unique fixed point at infinity (Proposition~\ref{prop:semicontE}), generalizing a result of Thurston \cite{thu86}.

\subsection{Extension of Lipschitz maps in~$\HH^n$}

In order to prove Theorem~\ref{thm:lamin}, following the approach of \cite{kasPhD}, we develop the extension theory of Lipschitz maps in~$\HH^n$ and, more precisely, refine an old theorem of Kirszbraun \cite{kir34} and Valentine \cite{val44}, which states that any Lipschitz map from a compact subset of~$\HH^n$ to~$\HH^n$ with Lipschitz constant $\geq 1$ can be extended to a map from $\HH^n$ to itself with the same Lipschitz constant.
We prove the following.

\begin{theorem}\label{thm:Kirszbraunequiv}
Let $\Gamma_0$ be a discrete group and $(j,\rho)\in\Hom(\Gamma_0,G)^2$ a pair of representations of $\Gamma_0$ in~$G$ with $j$ geometrically finite.
\begin{enumerate}
  \item For any $j(\Gamma_0)$-invariant subset $K\neq\emptyset$ of~$\HH^n$ and any $(j,\rho)$-equivariant map $\varphi : K\rightarrow\HH^n$ with Lipschitz constant $C_0\geq 1$, there exists a $(j,\rho)$-equivariant extension $f : \HH^n\rightarrow\HH^n$ of~$\varphi$ with Lipschitz constant~$C_0$.
  \item For any $j(\Gamma_0)$-invariant subset $K\neq\emptyset$ of~$\HH^n$ whose image in $j(\Gamma_0)\backslash\HH^n$ is bounded and for any $(j,\rho)$-equivariant map $\varphi : K\rightarrow\HH^n$ with Lipschitz constant $C_0<1$, there exists a $(j,\rho)$-equivariant extension $f : \HH^n\rightarrow\HH^n$ of~$\varphi$ with Lipschitz constant $<1$.
\end{enumerate}
\end{theorem}

The point of Theorem~\ref{thm:Kirszbraunequiv} is that we can extend~$\varphi$ \emph{in an equivariant way}, without increasing the Lipschitz constant~$C_0$ if it is $\geq 1$, and still keeping it $<1$ if it was originally $<1$.
Moreover, we control the \emph{local} Lipschitz constant when $C_0\geq 1$ (Theorem~\ref{thm:Kirszbraunopt}).
Intuitively (at least when $C_0>1$), the idea is that one should be able to choose an $f$ whose stretch locus consists of \emph{stretch segments} with endpoints in~$K$, moved apart by a factor $C_0$ under~$\varphi$.

We believe (see Appendix~\ref{sec:questionlip}) that in Theorem~\ref{thm:Kirszbraunequiv}.(2) the best Lipschitz constant of an equivariant extension~$f$ could be bounded away from $1$ in terms of~$C_0$ alone.
This would allow to remove the assumption that $K$ has a bounded image in $j(\Gamma_0)\backslash\HH^n$, using the Arzel\`a--Ascoli theorem (see Section~\ref{subsec:coroElamin}).

In Theorem~\ref{thm:Kirszbraunopt} below we refine Theorem~\ref{thm:Kirszbraunequiv} and actually allow $K$ to be the empty set.
(In this case we define $C_0$ to be the supremum of ratios $\lambda(\rho(\gamma))/\lambda(j(\gamma))$ for $\gamma\in\Gamma_0$ with $j(\gamma)$ hyperbolic, where
\begin{equation}\label{eqn:deflambda}
\lambda(g) := \inf_{x\in\HH^n} d(x,g\cdot x)
\end{equation}
is the translation length of $g$ in~$\HH^n$ if $g\in G$ is hyperbolic, and $0$ if $g$ is~para\-bolic or elliptic.)
Theorem~\ref{thm:lamin} is equivalent to the case $K=\emptyset$ in Theorem~\ref{thm:Kirszbraunopt}.

Theorem~\ref{thm:Kirszbraunequiv} and its refinements such as Theorem~\ref{thm:Kirszbraunopt} should be compared to a number of recent results in the theory of extension of Lipschitz maps: see Lang--Schr\"oder~\cite{ls97}, Lang--Pavlovi\'c--Schr\"oder~\cite{lps00}, Buyalo--Schr\"oder~\cite{bs01}, Lee--Naor~\cite{ln05}, etc.
We also point to \cite{dgk12} for an infinitesimal version.

\subsection{An application to the study of complete manifolds locally modeled on $G=\PO(n,1)$}\label{subsec:introquotientsG}

One important motivation for examining equivariant Lipschitz maps of minimal Lipschitz constant is the link with certain manifolds locally modeled on~$G$, namely quotients of~$G$ by discrete subgroups of $G\times G$ acting properly discontinuously and freely on~$G$ by right and left multiplication: $(g_1,g_2)\cdot g=g_2gg_1^{-1}$.
This link was first noticed in \cite{sal00}, then developed in \cite{kasPhD}.

For $n=2$, the manifolds locally modeled on $\PO(2,1)_0\cong\PSL_2(\R)$ are the \emph{anti-de Sitter} $3$-manifolds, or Lorentzian $3$-manifolds of constant negative curvature, which are Lorentzian analogues of the hyperbolic $3$-manifolds.
For $n=3$, the manifolds locally modeled on $\PO(3,1)_0\cong\PSL_2(\C)$ are the $3$-dimensional complex \emph{holomorphic-Riemannian} manifolds of constant nonzero curvature, which can be considered as complex analogues of the hyperbolic $3$-manifolds (see \cite{dz09} for details).
For $n=2$, all \emph{compact} manifolds locally modeled on~$G$ are quotients of~$G$ by discrete subgroups of $G\times G$, up to a finite covering \cite{kli96,kr85}; for $n=3$, a similar property has been conjectured in \cite{dz09} (see Section~\ref{subsec:(G,X)struct}).

Recall that the quotient of $G$ by a discrete group~$\Gamma$ is Hausdorff (\resp is a manifold) if and only if the action of $\Gamma$ on~$G$ is properly discontinuous (\resp properly discontinuous and free).
Let $\Gamma$ be a discrete subgroup of $G\times G$ acting properly discontinuously on~$G$ by right and left multiplication.
The key point here is that if $\Gamma$ is torsion-free, then it is a graph of the form
\begin{equation}\label{eqn:Gamma(j,rho)}
\Gamma_0^{j,\rho} = \{ (j(\gamma),\rho(\gamma))~|~\gamma\in\Gamma_0\}
\end{equation}
where $\Gamma_0$ is a discrete group and $j,\rho\in\Hom(\Gamma_0,G)$ are representations with $j$ injective and discrete (up to switching the two factors): this was proved in \cite{kr85} for $n=2$, and in \cite{kas08} (strengthening partial results of \cite{kob93}) for general rank-one groups~$G$.
The group~$\Gamma$ is thus isomorphic to the fundamental group of the hyperbolic $n$-manifold $M:=j(\Gamma_0)\backslash\HH^n$, and the quotient of $G$ by $\Gamma=\Gamma_0^{j,\rho}$ is compact if and only if $M$ is compact (by a classical cohomological argument, see Section~\ref{subsec:proofdeform}).
In general, if $\Gamma$ is finitely generated, the Selberg lemma \cite[Lem.\,8]{sel60} ensures the existence of a finite-index subgroup of~$\Gamma$ that is torsion-free, hence of the form $\Gamma_0^{j,\rho}$ or $\Gamma_0^{\rho,j}$ as above.

As before, we set $\lambda(g):=\inf_{x\in\HH^n} d(x,g\cdot x)$ for any $g\in G$.
The following terminology is partly adopted from Salein \cite{sal00}.

\begin{definition}\label{def:admissible}
A pair $(j,\rho)\in\Hom(\Gamma_0,G)^2$ is called \emph{admissible} if the action of $\Gamma_0^{j,\rho}$ on~$G$ by right and left multiplication is properly discontinuous.
It is called \emph{left} (\resp \emph{right}) \emph{admissible} if, in addition, there exists $\gamma\in\Gamma_0$ such that $\lambda(j(\gamma))>\lambda(\rho(\gamma))$ (\resp $\lambda(j(\gamma))<\lambda(\rho(\gamma))$).
\end{definition}

By \cite{sal00} (for $n=2$) and \cite{kas08} (for general~$n$), an admissible pair $(j,\rho)$ is either left admissible or right admissible; it cannot be both.
Without loss of generality, we may restrict to left admissible pairs.

For a pair $(j,\rho)\in\Hom(\Gamma_0,G)^2$ with $j$ injective and discrete, we set
\begin{equation}\label{eqn:defClambda}
C'(j,\rho) := \sup_{\gamma\in\Gamma_0\text{ with $j(\gamma)$ hyperbolic}}\ \frac{\lambda(\rho(\gamma))}{\lambda(j(\gamma))}
\end{equation}
if the group $j(\Gamma_0)$ contains hyperbolic elements, and $C'(j,\rho):=C(j,\rho)$ otherwise (case of an elementary group fixing a point in~$\HH^n$ or a unique point in $\partial_{\infty}\HH^n$).
With~this notation, a consequence of Theorem~\ref{thm:lamin} is the following (double) left admissibility criterion, which was first established in \cite[Chap.\,5]{kasPhD} for $n=2$ and convex cocompact~$j$.

\begin{theorem}\label{thm:adm}
Let $\Gamma_0$ be a discrete group.
A pair $(j,\rho)\in\Hom(\Gamma_0,G)^2$ with $j$ geometrically finite is left admissible if and only if
\begin{enumerate}
  \item the infimum $C(j,\rho)$ of Lipschitz constants of $(j,\rho)$-equivariant Lipschitz maps $f : \HH^n\rightarrow\HH^n$ is $<1$.
\end{enumerate}
This is equivalent to the condition that
\begin{enumerate}
  \item[(2)] the supremum $C'(j,\rho)$ of ratios of translation lengths $\lambda(\rho(\gamma))/\lambda(j(\gamma))$ for $\gamma\in\Gamma_0$ with $j(\gamma)$ hyperbolic is $<1$,
\end{enumerate}
except possibly in the degenerate case where $\rho(\Gamma_0)$ has a unique fixed point in $\partial_{\infty}\HH^n$ and $\rho$ is not cusp-deteriorating.
In particular, left admissibility is always equivalent to (1) and to~(2) if $j$ is convex cocompact.
\end{theorem}

In other words, Theorem~\ref{thm:adm} states that $(j,\rho)$ is left admissible if and only if ``$\rho$ is uniformly contracting with respect to~$j$''; this uniform contraction can be expressed in two equivalent ways: in terms of Lipschitz maps (condition~(1)) and in terms of ratios of lengths (condition~(2)).

Note that the inequality $C'(j,\rho)\leq C(j,\rho)$ is always true (see \eqref{eqn:ClambdaCLip}).
It can occur quite generically that $C'(j,\rho)<C(j,\rho)$ below~$1$, even when $j$ and~$\rho$ are both convex cocompact (see Sections \ref{ex:C'neqC} and~\ref{ex:C'neqCnoncompact}).
In the degenerate case where $\rho(\Gamma_0)$ has a unique fixed point in $\partial_{\infty}\HH^n$ and $\rho$ is not cusp-deteriorating, it can also happen that $C'(j,\rho)<C(j,\rho)=1$ (see Section~\ref{ex:CC'nonreductive}).
However, when we are not in this degenerate case, it follows from Theorem~\ref{thm:lamin} that $C(j,\rho)\geq\nolinebreak 1$ implies $C'(j,\rho)=C(j,\rho)$ (Corollary~\ref{cor:CC'}); in particular, $C'(j,\rho)<1$ implies $C(j,\rho)<1$.

In Theorem~\ref{thm:adm}, the fact that if $C(j,\rho)<1$ then $(j,\rho)$ is left admissible easily follows from the general \emph{properness criterion} of Benoist \cite{ben96} and Kobayashi \cite{kob96} (see Section~\ref{subsec:Cartanproj}); this was first observed by Salein \cite{sal00}.
Conversely, suppose that $(j,\rho)$ is left admissible.
Then $C'(j,\rho)\leq 1$ (because $(j,\rho)$ cannot be simultaneously left and right admissible, as mentioned above); the point is to prove that $C'(j,\rho)=1$ is impossible.
This is done in Section~\ref{subsec:proofadmred}, following the strategy of \cite{kasPhD}: we use Theorem~\ref{thm:lamin} to establish that $C'(j,\rho)=1$ implies, not only that $C(j,\rho)=1$ (Corollary~\ref{cor:CC'}), but also that the stretch locus $E(j,\rho)$ contains a geodesic line of~$\HH^n$; it is then easy to find a sequence of elements of~$\Gamma_0$ contradicting proper discontinuity by following this geodesic line.

We note that in Theorem~\ref{thm:adm} it is necessary for $\Gamma_0$ to be finitely generated: indeed, for infinitely generated~$\Gamma_0$ there exist left admissible pairs $(j,\rho)\in\Hom(\Gamma_0,G)^2$ of injective and discrete representations that satisfy $C(j,\rho)=C'(j,\rho)=1$ (see Section~\ref{ex:infinite}).
It would be interesting to know whether Theorem~\ref{thm:adm} still holds for finitely generated but geometrically infinite~$j$ (Appendix~\ref{sec:geom-infinite}).

Here is a consequence of Proposition~\ref{prop:contCcc} and Theorem~\ref{thm:adm}.

\begin{theorem}\label{thm:deformcompact}
Let $G=\PO(n,1)$ and let $\Gamma$ be a discrete subgroup of $G\times G$ acting properly discontinuously, freely, and cocompactly on~$G$ by right and left multiplication.
There is a neighborhood $\mathcal{U}\subset\Hom(\Gamma,G\times G)$ of the natural inclusion such that for any $\varphi\in\mathcal{U}$, the group $\varphi(\Gamma)$ is discrete in $G\times G$ and acts properly discontinuously, freely, and cocompactly on~$G$.
\end{theorem}

A particular case of Theorem~\ref{thm:deformcompact} was proved by Kobayashi \cite{kob98}, namely the so-called ``special standard'' case (terminology of \cite{zeg98}) where $\Gamma$ is contained in $G\times\{ 1\}$; for $n=3$, this was initially proved by Ghys \cite{ghy95}.
The general case for $n=2$ follows from the completeness of compact anti-de Sitter manifolds, due to Klingler \cite{kli96}, and from the Ehresmann--Thurston principle on the deformation of holonomies of $(\mathbf{G},\mathbf{X})$-structures on compact manifolds.
An interpretation of Theorem~\ref{thm:deformcompact} in terms of $(\mathbf{G},\mathbf{X})$-structures is given in Section~\ref{subsec:(G,X)struct}.

We extend Theorem~\ref{thm:deformcompact} to proper actions on~$G$ that are not necessarily cocompact, using the following terminology.

\begin{definition}\label{def:cococo}
We say that a quotient of~$G$ by a discrete subgroup $\Gamma$ of $G\times\nolinebreak G$ is \emph{convex cocompact} (\resp \emph{geometrically finite}) if, up to switching the two factors and to passing to a finite-index subgroup, $\Gamma$ is of the form $\Gamma_0^{j,\rho}$ as in \eqref{eqn:Gamma(j,rho)} with $j$ convex cocompact (\resp geometrically finite) and $(j,\rho)$ left admissible.
\end{definition}

This terminology is justified by the fact that convex cocompact (\resp geo\-metrically finite) quotients of 
$\PO(n,1)$ fiber, with compact fiber $\OO(n)$, over convex cocompact (\resp geometrically finite) hyperbolic manifolds, up to a finite covering (see Proposition~\ref{prop:quotients} or \cite[Th.\,1.2]{dgk12}).

We prove the following extension of Theorem~\ref{thm:deformcompact}
(see also \cite{kas09} for a $p$-adic analogue).

\begin{theorem}\label{thm:deform}
Let $G=\PO(n,1)$ and let $\Gamma$ be a discrete subgroup of $G\times\nolinebreak G$ acting properly discontinuously on~$G$, with a convex cocompact quotient.
There is a neighborhood $\mathcal{U}\subset\Hom(\Gamma,G\times G)$ of the natural inclusion such that for any $\varphi\in\mathcal{U}$, the group $\varphi(\Gamma)$ is discrete in $G\times G$ and acts properly discontinuously on~$G$, with a convex cocompact quotient.
Moreover, if the quotient of $G$ by $\Gamma$ is compact, then so is the quotient of $G$ by $\varphi(\Gamma)$ for $\varphi\in\mathcal{U}$.
If the quotient of $G$ by $\Gamma$ is a manifold, then so is the quotient of $G$ by $\varphi(\Gamma)$ for $\varphi\in\mathcal{U}$ close enough to the natural inclusion.
\end{theorem}

Note that Theorem~\ref{thm:deform} is not true if we replace ``convex cocompact'' with ``geometrically finite'': for a given $j$ with cusps, the constant representation $\rho=\nolinebreak 1$ may have small deformations~$\rho'$ that are not cusp-deteriorating, hence for which $(j,\rho')$ cannot be admissible.
However, we prove that Theorem~\ref{thm:deform} is true in dimension $n=2$ or~$3$ if we restrict to deformations of $\Gamma$ of the form $\Gamma_0^{j,\rho}$ with geometrically finite~$j$ and \emph{cusp-deteriorating}~$\rho$ (Theorem~\ref{thm:deformwithcusps}); a similar statement is \emph{not} true for $n>3$ (see Section~\ref{ex:dim4C<1}).

Theorem~\ref{thm:adm} implies that any geometrically finite quotient of~$G$ is \emph{sharp} in the sense of \cite{kk12}; moreover, by Theorem~\ref{thm:deform}, if the quotient is convex cocompact, then it remains sharp after any small deformation of the discrete group~$\Gamma$ inside $G\times G$ (see Section~\ref{subsec:proofdeform}).
This implies the existence of an infinite discrete spectrum for the (pseudo-Riemannian) Laplacian on any geometrically finite quotient of~$G$: see \cite{kk12}.

\subsection{A generalization of Thurston's asymmetric metric on Teichm\"uller space}\label{subsec:Thurstonmetric}

Let $S$ be a hyperbolic surface of finite volume.
The Teichm\"uller space $\T(S)$ of~$S$ can be defined as the space of conjugacy classes of geometrically finite representations of $\Gamma_0:=\pi_1(S)$ into $\PO(2,1)\cong\PGL_2(\R)$ corresponding to finite-volume hyperbolic surfaces homeomorphic to~$S$.\linebreak
Thurston \cite{thu86} proved that $C(j,\rho)=C'(j,\rho)\geq 1$ for all $j,\rho\in\T(S)$; the function
$$d_{\mathrm{Th}} := \log C = \log C' : \T(S)\times\T(S) \longrightarrow \R_+$$
is the \emph{Thurston metric} on $\T(S)$, which was introduced and extensively studied in \cite{thu86}.
It is an ``asymmetric metric'', in the sense that $d_{\mathrm{Th}}(j,\rho)\geq 0$ for all $j,\rho\in\T(S)$, that $d_{\mathrm{Th}}(j,\rho)=0$ if and only if $j=\rho$ in $\T(S)$, that $d_{\mathrm{Th}}(j_1,j_3)\leq d_{\mathrm{Th}}(j_1,j_2)+d_{\mathrm{Th}}(j_2,j_3)$ for all $j_i\in\T(S)$, but that in general $d_{\mathrm{Th}}(j,\rho)\neq d_{\mathrm{Th}}(\rho,j)$.

We generalize Thurston's result that $C(j,\rho)=C'(j,\rho)$ to any dimension $n\geq 2$, to geometrically finite representations $j$ that are not necessarily of finite covolume, and to representations~$\rho$ that are not necessarily injective or discrete: as a consequence of Theorem~\ref{thm:lamin}, we obtain the following.

\begin{corollary}\label{cor:CC'}
For $G=\PO(n,1)$, let $(j,\rho)\in\Hom(\Gamma_0,G)^2$ be a pair of representations with $j$ geometrically finite.
If $C(j,\rho)\geq 1$, then
\begin{equation}\label{eqn:CC'}
C(j,\rho) = C'(j,\rho),
\end{equation}
except possibly in the degenerate case where $C(j,\rho)=1$, where $\rho(\Gamma_0)$ has a unique fixed point in $\partial_{\infty}\HH^n$, and where $\rho$ is not cusp-deteriorating.
\end{corollary}

In particular, $C(j,\rho)\geq 1$ always implies \eqref{eqn:CC'} if $j$ is convex cocompact, and $C'(j,\rho)\geq 1$ always implies \eqref{eqn:CC'} without any assumption on~$j$.

In order to generalize the Thurston metric, we consider a nonelementary hyperbolic manifold $M$ of any dimension $n\geq 2$ and let $\T(M)$ be the set of conjugacy classes of geometrically finite representations of $\Gamma_0:=\pi_1(M)$ into $G=\PO(n,1)$ with the homeomorphism type and cusp type of~$M$.
We allow $M$ to have infinite volume, otherwise $\T(M)$ is trivial for $n>2$ by Mostow rigidity.
We set
\begin{equation} \label{eqn:thurston-distance-corrected}
d_{\mathrm{Th}}(j,\rho) := \log\bigg(C(j,\rho) \, \frac{\delta(\rho)}{\delta(j)}\bigg)
\end{equation}
for all $j,\rho\in\T(M)$, where $\delta : \T(M)\rightarrow (0,n-1]$ is the \emph{critical exponent} (see \eqref{eqn:defcritexp}) giving the exponential growth rate of orbits in~$\HH^n$ or, equivalently in this setting, the Hausdorff dimension of the limit set \cite{sul79,sul84}.
It easily follows from the definition of~$\delta$ that $d_{\mathrm{Th}}(j,\rho)\geq 0$ for all $j,\rho \in \T(M)$~(Lemma~\ref{lem:dThgeq0}).
Let $\T(M)_{\mathrm{Zs}}$ be the subset of $\T(M)$ consisting of elements $j$ such that the Zariski closure of $j(\Gamma_0)$ in~$G$ is simple (for instance equal to $G$ itself).
We prove the following.

\begin{proposition}\label{prop:generalThurston}
The function $d_{\mathrm{Th}}$ defines an asymmetric metric on\linebreak $\T(M)_{\mathrm{Zs}}$.
\end{proposition}

The point of Proposition~\ref{prop:generalThurston} is that $d_{\mathrm{Th}}(j,\rho)=0$ implies $j=\rho$ on $\T(M)_{\mathrm{Zs}}$.

On the other hand, for convex cocompact~$M$, it follows from work of Burger \cite[Th.\,1]{bur93} (see also \cite{bcls15}) that
$$d'_{\mathrm{Th}}(j,\rho) := \log\!\bigg(\!C'(j,\rho) \, \frac{\delta(\rho)}{\delta(j)}\bigg)$$
defines an asymmetric metric on~$\T(M)_{\mathrm{Zs}}$.
Kim \cite[Cor.\,3]{kim01} also proved that if $\log C'(j,\rho)=0$ and $\delta(j)=\delta(\rho)$, then $j=\rho$ in $\T(M)_{\mathrm{Zs}}$.
By Corollary~\ref{cor:CC'}, the asymmetric metrics $d_{\mathrm{Th}}$ and $d'_{\mathrm{Th}}$ are equal on the set
$$\big\{ (j,\rho)\in\T(M)_{\mathrm{Zs}}^2 ~|~ C(j,\rho)\geq 1\big\} = \big\{ (j,\rho)\in\T(M)_{\mathrm{Zs}}^2 ~|~ C'(j,\rho)\geq 1\big\}.$$
However, they differ in general (see Sections \ref{ex:C'neqC} and~\ref{ex:C'neqCnoncompact}).
It would be interesting to compare them.

In dimension $n\leq 3$ the asymmetric metric $d_{\mathrm{Th}}$ is always continuous, and in dimension $n\geq 4$ it is continuous when $M$ is convex cocompact (Lemma~\ref{lem:dThcont}).

\subsection{Organization of the paper}

Section~\ref{sec:prelim} contains reminders and basic facts on geometrical finiteness, Lipschitz maps, and convex interpolation in the real hyperbolic space~$\HH^n$.
In Section~\ref{sec:Kirszbraun} we recall the classical Kirszbraun--Valentine theorem and establish an equivariant version of it for amenable groups.
We then derive general properties of the stretch locus in Section~\ref{sec:stretchlocus}.
In Section~\ref{sec:Kirszbraunopt} we prove an optimized, equivariant Kirszbraun--Valentine theorem for geometrically finite representations of discrete groups; this yields in particular Theorems \ref{thm:lamin} and~\ref{thm:Kirszbraunequiv}, as well as Corollary~\ref{cor:CC'}.
In Section~\ref{sec:lipcont} we examine the continuity properties of the minimal Lipschitz constant $C(j,\rho)$; in particular, we prove Proposition~\ref{prop:contCcc}.
In Section~\ref{sec:properactions} we apply the theory to properly discontinuous actions on $G=\PO(n,1)$ (proving Theorems \ref{thm:adm}, \ref{thm:deformcompact}, and~\ref{thm:deform}), and in Section~\ref{sec:Thurston} we generalize Thurston's asymmetric metric on Teichm\"uller space (proving Proposition~\ref{prop:generalThurston}).
In Section~\ref{sec:dim2} we focus on the case $n=2$: we recover and extend results of Thurston for $C(j,\rho)>1$, and discuss the nature of the stretch locus for $C(j,\rho)<1$.
Finally, in Section~\ref{sec:ex} we give a number of examples and counterexamples designed to make the theory more concrete while pointing out some subtleties.
We collect useful formulas in Appendix~\ref{sec:appendix}, technical facts on geometrically finite representations in Appendix~\ref{sec:conv-fundamental}, and open questions in Appendix~\ref{sec:questions}.

\medskip
\noindent
\textit{Note.}
We have tried, inside each section, to clearly separate the arguments needed for the convex cocompact case from the ones specific to the cusps.
Skipping the latter should decrease the length of the paper substantially.

\subsection*{Acknowledgements}

We are grateful to Maxime Wolff for his comments on a preliminary version of this paper, to Jeff Danciger for numerous discussions on related topics, and to Samuel Tapie for his indications on the Bowen--Margulis--Sullivan measure.
We are indebted to Marc Burger for the idea of introducing the correcting factor $\delta(\rho)/\delta(j)$ in the definition of $d_{\mathrm{Th}}$ (in an earlier version of this paper we had to restrict to the case $\delta(j)=\delta(\rho)$).
We would like to thank him for interesting discussions on the two possible generalizations $d_{\mathrm{Th}}$ and $d'_{\mathrm{Th}}$ of the Thurston metric.
We are grateful to an anonymous referee for carefully reading the manuscript and making many valuable suggestions.
Finally, we thank the University of Chicago for its support and the Institut CNRS-Pauli (UMI 2842) in Vienna for its hospitality.

\section{Preliminary results}\label{sec:prelim}

In this section we recall a few well-known facts and definitions on geometrically finite hyperbolic orbifolds, on Lipschitz constants, and on barycenters in the real hyperbolic space~$\HH^n$.
We also expand on the notion of cusp-deterioration introduced in Definition~\ref{def:typedet}.
In the whole paper, $G$ is the full group $\PO(n,1)=\OO(n,1)/\{ \pm 1\}$ of isometries of~$\HH^n$.
If $n$ is even, then $G$ identifies with $\SO(n,1)$.

\subsection{Geometrical finiteness} \label{subsec:geo-finiteness}

Let $j\in\Hom(\Gamma_0,G)$ be an injective representation of a discrete group~$\Gamma_0$, with $j(\Gamma_0)$ discrete.
The quotient $M:=j(\Gamma_0)\backslash\HH^n$ is a smooth, $n$-dimensional orbifold; it is a manifold if and only if $\Gamma_0$ is torsion-free.
The \emph{convex core} of~$M$ is the smallest closed convex subset of~$M$ containing all closed geodesics; its preimage in~$\HH^n$ is the convex hull of the limit set $\Lambda_{j(\Gamma_0)}\subset\partial_{\infty}\HH^n$ of~$j(\Gamma_0)$.
(The convex hull is empty only in the degenerate case where the group $j(\Gamma_0)$ has a fixed point in~$\HH^n$ or a unique fixed point in $\partial_{\infty}\HH^n$; we do not exclude this case.)
Following \cite{bow93}, we will say that the injective and discrete representation~$j$ is \emph{geometrically finite} if $\Gamma_0$ is finitely generated and if for any $\varepsilon>0$, the $\varepsilon$-neighborhood of the convex core of~$M$ has finite volume.
In dimension $n=2$, any injective and discrete representation in~$G$ of a finitely generated group is geometrically finite.
In general, $j$ is geometrically finite if and only if the convex core of~$M$ is contained in the union of a compact set and of finitely many disjoint \emph{cusps}, whose boundaries have compact intersection with the convex core.
We now explain what we mean by cusp, following \cite{bow93}.

Let $B$ be a horoball of~$\HH^n$, centered at a point $\xi\in\partial_{\infty}\HH^n$, and let $S\subset\Gamma_0$ be the stabilizer of $B$ under~$j$.
The group $j(S)$ is discrete (possibly trivial) and consists of nonhyperbolic elements.
It preserves the horosphere $\partial B\simeq\R^{n-1}$ and acts on it by affine Euclidean isometries.
By the first Bieberbach theorem (see \cite[Th.\,2.2.5]{bow93}), there is a finite-index normal subgroup $S'$ of~$S$ that is isomorphic to~$\Z^m$ for some $0\leq m<n$, and whose index in~$S$ is bounded by some $\nu(n)\in\N$ depending only on the dimension~$n$; we have $m\geq 1$ if and only if $S$ contains a parabolic element.
The group $j(S)$ preserves and acts cocompactly on some $m$-dimensional affine subspace $\mathcal{V}$ of $\partial B\simeq\R^{n-1}$, unique up to translation; the subgroup $j(S')$ acts on~$\mathcal{V}$ by translation.
Let $\HH_{\mathcal{V}}$ be the closed $(m+1)$-dimensional hyperbolic subspace of~$\HH^n$ containing $\xi$ in its boundary such that $\HH_{\mathcal{V}}\cap\partial B=\mathcal{V}$, and let $\pi : \HH^n\rightarrow\HH_{\mathcal{V}}$ be the closest-point projection (see Figure~\ref{fig:A}).
The group $j(S)$ preserves the convex set $\mathfrak{C}:=\pi^{-1}(\HH_{\mathcal{V}}\cap B)\subset\HH^n$.
Following \cite{bow93}, we say that the image of $\mathfrak{C}$ in~$M$ is a \emph{cusp} if $m\geq 1$ and $\mathfrak{C}\cap j(\gamma)\cdot\mathfrak{C}=\emptyset$ for all $\gamma\in\Gamma_0\smallsetminus S$.
The cusp is then isometric to $j(S)\backslash\mathfrak{C}$; its intersection with the convex core of~$M$ is contained in $j(S)\backslash B'$ for some horoball $B'\supset B$.
The integer~$m$ is called the \emph{rank} of the cusp.

\begin{figure}[h!]
\begin{center}
\labellist
\small\hair 2pt
 \pinlabel{$(\xi=\infty)$} at 275 300
 \pinlabel{$\HH_{\mathcal{V}}$} at 178 150
 \pinlabel{$j(S')$} at 225 217
 \pinlabel{$\mathcal{V}$} at 300 170
 \pinlabel{$\mathfrak{C}$} at 30 48
 \pinlabel{$B$} at 470 220
 \pinlabel{$\partial B$} at 490 171
 \pinlabel{$\partial_{\infty}\HH^3$} at 390 8
 \endlabellist
\includegraphics[width=10cm]{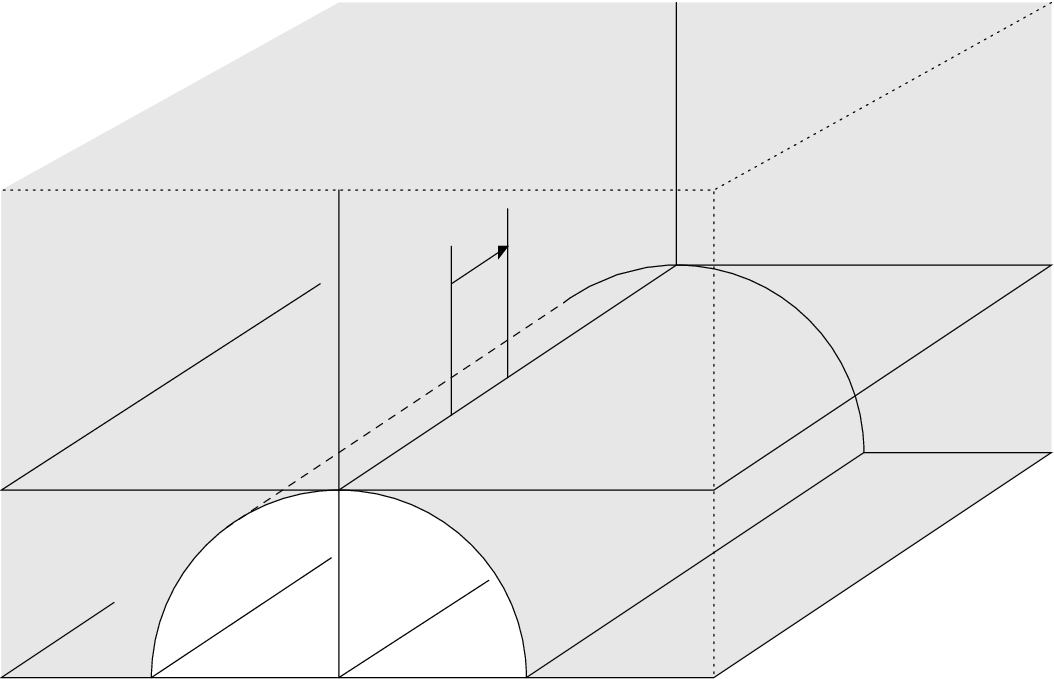}
\caption{A rank-one cusp centered at $\xi=\infty$ in the upper half-space model of~$\HH^3$. The limit set is contained in $\{\xi\}\cup\partial_\infty(\HH^3\smallsetminus\mathfrak{C})$.}
\label{fig:A}
\end{center}
\end{figure}
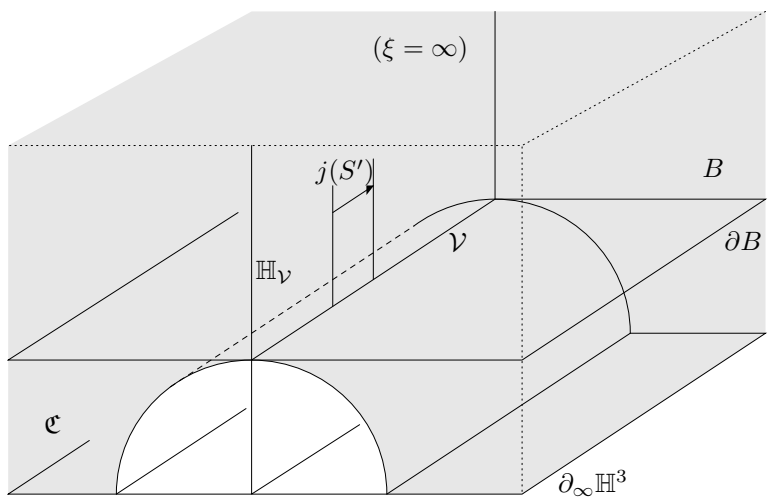

When the convex core of~$M$ is nonempty, we may assume that it contains the image of~$\mathcal{V}$, after possibly replacing $B$ by some smaller horoball and $\mathcal{V}$ by some translate.

We shall use the following description.

\begin{fact}\label{fact:geomfinite}
If $j$ is geometrically finite, then $M=j(\Gamma_0)\backslash\HH^n$ is the union of a closed subset $M'$ and of finitely many disjoint quotients $j(S_i)\backslash B_i$, where $B_i$ is a horoball of~$\HH^n$ and $j(S_i)$ a discrete group of isometries of~$B_i$ containing a parabolic element, such that
\begin{itemize}
  \item the intersection of~$M'$ with the convex core of~$M$ is compact;
  \item for any~$i$ we have $M'\cap (j(S_i)\backslash B_i)=j(S_i)\backslash\partial B_i$; in particular, the intersection of $j(S_i)\backslash\partial B_i$ with the convex core of~$M$ is compact;
  \item for any~$i$ the intersection in~$\HH^n$ of~$B_i$ with the preimage $N$ of the convex core of~$M$ is the convex hull of $\partial B_i\cap N$, and $j(S_i) \backslash (\partial B_i\cap N)$ is compact.
\end{itemize}
\end{fact}

\begin{definition}\label{def:standardcusp}
We shall call the intersections of the sets $j(S_i)\backslash B_i$ with the convex core of~$M$ \emph{standard cusp regions}.
\end{definition}

If $j$ is geometrically finite, then the complement of the convex core of~$M$ has finitely many connected components, called the \emph{funnels} of~$M$.
By definition, $j$ is \emph{convex cocompact} if it is geometrically finite with no cusp; when $\Gamma_0$ is infinite, this is equivalent to the convex core being nonempty and compact.
The set of convex cocompact representations is open in $\Hom(\Gamma_0,G)$ (see \cite[Prop.\,4.1]{bow98} or Proposition~\ref{prop:rem-ccopen}).

In Sections \ref{sec:stretchlocus} and~\ref{sec:Kirszbraunopt}, we shall consider a $j(\Gamma_0)$-invariant subset $K$ of~$\HH^n$ whose image in~$M$ is compact.
We shall then use the following notation.

\begin{notation}\label{not:ConvK}
In the rest of the paper, $\Conv{K}\subset\HH^n$ denotes:
\begin{itemize}
  \item the convex hull of~$K$ if $K$ is nonempty,
  \item the preimage in~$\HH^n$ of the convex core of $M=j(\Gamma_0)\backslash\HH^n$ if $K$ is empty and the convex core is nonempty (leaving $j$ implicit),
  \item any \emph{nonempty} $j(\Gamma_0)$-invariant convex subset of~$\HH^n$ if $K$ and the convex core of $M=j(\Gamma_0)\backslash\HH^n$ are both empty (case when $j(\Gamma_0)$ is an elementary group fixing a point in~$\HH^n$ or a unique point in $\partial_{\infty}\HH^n$).
\end{itemize}
\end{notation}

\noindent
In all three cases the set $\Conv{K}$ is nonempty and contains the preimage in~$\HH^n$ of the convex core of~$M$.
In Fact~\ref{fact:geomfinite}, we can take $M'$ and the $B_i$ with the following properties:
\begin{itemize}
  \item $M'$ contains the (compact) image of $K$ in~$M$;
  \item the intersection of~$M'$ with the image of $\Conv{K}$ in~$M$ is compact;
  \item for any~$i$ the set $B_i\cap\Conv{K}$ is the convex hull of $\partial B_i\cap\Conv{K}$, and $j(S_i)\backslash (\partial B_i\cap\Conv{K})$ is compact.
\end{itemize}

\subsection{Cusp deterioration}\label{subsec:cuspdet}

Let $j\in\Hom(\Gamma_0,G)$ be a geometrically finite~re\-presentation and let $B_1,\dots,B_c$ be horoballs of~$\HH^n$ whose projections $j(S_i)\backslash B_i$ to $M=j(\Gamma_0)\backslash\HH^n$ are disjoint and intersect the convex core in standard cusp regions representing all the cusps, as in Section~\ref{subsec:geo-finiteness}.
Consider $\rho\in\nolinebreak\Hom(\Gamma_0,G)$.

\begin{definition}\label{def:detinacusp}
For $1\leq i\leq c$, we say that $\rho$ is \emph{deteriorating in~$B_i$} if $\rho(S_i)$ contains only elliptic elements.
\end{definition}

Thus $\rho$ is cusp-deteriorating in the sense of Definition~\ref{def:typedet} if and only if it is deteriorating in~$B_i$ for all $1\leq i\leq c$.

Depending on whether $\rho$ is deteriorating in~$B_i$ or not, we shall use the following classical fact with $\Gamma'=\rho(S_i)$.

\begin{fact}[see {\cite[Th.\,III.3.1]{parPhD}}]\label{fact:fixedpoint}
Let $\Gamma'$ be a finitely generated subgroup of~$G$.
\begin{enumerate}
  \item If all elements of~$\Gamma'$ are elliptic, then $\Gamma'$ has a fixed point in~$\HH^n$.
  \item If all elements of~$\Gamma'$ are elliptic or parabolic and if $\Gamma'$ contains at least one parabolic element, then $\Gamma'$ has a unique fixed point in the boundary at infinity $\partial_{\infty}\HH^n$ of~$\HH^n$.
\end{enumerate}
\end{fact}

\begin{lemma}\label{lem:disthorosphere}
Let $\Gamma'$ be as in Fact~\ref{fact:fixedpoint}.(2) and let $\wl : \Gamma'\rightarrow\N$ be the word length function with respect to some fixed finite generating subset $F'$ of~$\Gamma'$.
Fix $p\in\HH^n$.
\begin{itemize}
  \item There exists $R>0$ such that for all $\gamma'\in\Gamma'$,
  $$d(p,\gamma'\cdot p) \leq 2 \log\big(1+\wl(\gamma')\big) + R.$$
  \item If $\Gamma'$ is discrete in~$G$, then there exists $R'>0$ such that for all $\gamma'\in\Gamma'$,
  $$d(p,\gamma'\cdot p) \geq 2 \log\big(1+\wl(\gamma')\big) - R'.$$
\end{itemize}
\end{lemma}

\begin{proof}
Let $\xi\in\partial_{\infty}\HH^n$ be the fixed point of~$\Gamma'$ and let $\partial B$ be the horosphere through~$p$ centered at~$\xi$.
For any $q,q'\in\partial B$, let $d_{\partial B}(q,q')$ be the length of the shortest path from $q$ to~$q'$ that is contained in $\partial B$.
Then $d_{\partial B}$ is a Euclidean metric on $\partial B\simeq\R^{n-1}$ and
\begin{equation}\label{eqn:disthorosph}
d(q,q') = 2\,\arcsinh\left(\frac{d_{\partial B}(q,q')}{2}\right)
\end{equation}
for all $q,q'\in\partial B$ (see \eqref{eqn:horomu}).
In particular, $|d-2\log(1+d_{\partial B})|$ is bounded on $\partial B\times\partial B$.
By the triangle inequality,
$$d_{\partial B}(p,\gamma'\cdot p) \leq \Big(\max_{f'\in F'}\, d_{\partial B}(p,f'\cdot p)\Big) \cdot \wl(\gamma')$$
for all $\gamma'\in\Gamma'$, which implies the first statement of the lemma.

If $\Gamma'$ is discrete in~$G$, then it acts properly discontinuously on $\partial B$ and has a finite-index subgroup isomorphic to $\Z^m$ (for some $0<m<n$), acting as a lattice of translations on some $m$-dimensional affine subspace $\mathcal{V}$ of the Euclidean space $\partial B\simeq\R^{n-1}$ (see Section~\ref{subsec:geo-finiteness}). 
In a Euclidean lattice, the norm of a vector is estimated, up to a bounded multiplicative factor, by its word length in any given finite generating set: therefore there exist $c,Q>0$ such that 
$$d_{\partial B}(p,\gamma' \cdot p) \geq c\,\wl(\gamma') - Q$$
for all $\gamma'\in\Gamma'$.
The second statement of the lemma follows by using \eqref{eqn:disthorosph} and the properness of the function $\wl$ on~$\Gamma'$.
\end{proof}

Here is a consequence of Lemma~\ref{lem:disthorosphere}, explaining why the notion of cusp-deterioration naturally appears in our setting.

\begin{lemma}\label{lem:parabdet}
Let $\rho\in\Hom(\Gamma_0,G)$.
If there exists a $(j,\rho)$-equivariant map $f : \HH^n\rightarrow\HH^n$ with Lipschitz constant $<1$, then $\rho$ is cusp-deteriorating with respect to~$j$.
\end{lemma}

\begin{proof}
Let $f : \HH^n\rightarrow\HH^n$ be a $(j,\rho)$-equivariant map.
Suppose that $\rho$ is \emph{not} cusp-deteriorating.
Then there is an element $\gamma\in\Gamma_0$ such that $j(\gamma)$ is parabolic and $\rho(\gamma)$ is either parabolic or hyperbolic.
Fix a point $p\in\HH^n$.
By Lemma~\ref{lem:disthorosphere}, we have $d(p,j(\gamma^k)\cdot p)\sim 2\log k$ as $k\rightarrow +\infty$.
If $\rho(\gamma)$ is parabolic, then similarly $d(f(p),\rho(\gamma^k)\cdot f(p))\sim 2\log k$, and if $\rho(\gamma)$ is hyperbolic, then $|d(f(p),\rho(\gamma^k)\cdot f(p))-k\,\lambda(\rho(\gamma))|$ is uniformly bounded (for instance by twice the distance from $f(p)$ to the translation axis of $\rho(\gamma)$ in~$\HH^n$).
In both cases, we see that
$$\limsup_{k\rightarrow +\infty}\ \frac{d\big(f(p),\rho(\gamma^k)\cdot f(p)\big)}{d(p,j(\gamma^k)\cdot p)} \geq 1,$$
hence the $(j,\rho)$-equivariant map~$f$ cannot have Lipschitz constant $<1$.
\end{proof}

\subsection{Lipschitz constants}
For any subset $X$ of~$\HH^n$ and any map $f$ from $X$ to some metric space $(Z,d_Z)$ (in practice, $\HH^n$ or~$\R$), we denote by
$$\Lip(f) = \sup_{x,x'\in X,\ x\neq x'} \frac{d_Z(f(x),f(x'))}{d(x,x')}$$
the Lipschitz constant of~$f$.
For any $Y\subset X$ and any $x\in X$, we set
\begin{eqnarray*}
\Lip_Y(f) & = & \Lip(f|_Y),\\
\Lip_x(f) & = & \inf_{r>0}\ \Lip_{B_x(r)}(f),
\end{eqnarray*}
where $B_x(r)$ is the closed ball of radius~$r$ centered at~$x$ in~$\HH^n$.
We call $\Lip_x(f)$ the \emph{local Lipschitz constant} of~$f$ at~$x$.

\begin{remarks}\label{rem:pathlength}
\begin{enumerate}
  \item Let $f$ be a $C$-Lipschitz map from a geodesic segment $[x,x']$ of~$\HH^n$ to~$\HH^n$.
  If $d(f(x),f(x'))=Cd(x,x')$, then $f$ ``stretches maximally'' $[x,x']$, in the sense that $d(f(y),f(y'))=Cd(y,y')$ for all $y,y'\in [x,x']$.
  \item Let $X$ be a convex subset of~$\HH^n$, covered by a collection of open sets $\mathcal{U}_t$, $t\in T$.
  For any map $f : X\rightarrow\HH^n$,
  $$\Lip(f) = \sup_{t\in T}\ \Lip_{X\cap\mathcal{U}_t}(f).$$
  \item For any rectifiable path $\mathscr{C}$ in some subset $X$ of~$\HH^n$ and for any map $f : X\rightarrow\HH^n$,
$$\mathrm{length}(f(\mathscr{C})) \leq \sup_{x\in\mathscr{C}}\ \Lip_x(f) \cdot \mathrm{length}(\mathscr{C}).$$
\end{enumerate}
\end{remarks}

\noindent
Indeed, (1) follows from the fact that if the points $x,y,y',x'$ lie in this order, then $d(x,x')=d(x,y)+d(y,y')+d(y',x')$ while $d(f(x),f(x'))\leq d(f(x),f(y))+d(f(y),f(y'))+d(f(y'),f(x'))$ by the triangle inequality.
To prove~(2), we just need to check that the right-hand side is an upper bound for $\Lip(f)$ (it is also clearly a lower bound). Any geodesic segment $[p,q]\subset X$ can be divided into finitely many subsegments, each contained in one of the open sets~$\mathcal{U}_t$; we use again the additivity of distances at the source and the subadditivity of distances at the target.
Finally, (3) follows from the definition of the length of a path (obtained by summing up the distances between points of smaller and smaller subdivisions and taking a limit) and from the definition of the local Lipschitz constant.

\begin{lemma}\label{lem:localLip}
The local Lipschitz constant function $x\mapsto\Lip_x(f)$ is upper semicontinuous: for any converging sequence $x_k\rightarrow x$,
$$\Lip_x(f) \geq \limsup_{k\rightarrow +\infty}\, \Lip_{x_k}(f).$$
In particular, for any compact subset $K$ of~$X$, the supremum of $\Lip_x(f)$ for $x\in K$ is achieved on some nonempty closed subset of~$K$.
Moreover, if $X$ is convex, then
\begin{equation}\label{eqn:supLip}
\Lip(f) = \sup_{x\in X}\, \Lip_x(f).
\end{equation}
\end{lemma}

\begin{proof}
Upper semicontinuity follows from an easy diagonal extraction argument.
The inequality $\Lip(f)\geq\sup_{x\in X} \Lip_x(f)$ is clear.
The converse inequality for convex~$X$ follows from Remark~\ref{rem:pathlength}.(3) where $\mathscr{C}$ is any geodesic segment $[p,q]\subset X$.
\end{proof}

Note that the convexity of~$X$ is required for \eqref{eqn:supLip} to hold: for instance, an arclength-preserving map taking a horocycle~$X$ to a straight line is not even Lipschitz, although its local Lipschitz constant is everywhere~$1$.

As a consequence of Lemma~\ref{lem:localLip}, the \emph{stretch locus} of any Lipschitz map $f : X\rightarrow\HH^n$ is closed in~$X$ for the induced topology.
Here we use the following terminology, which agrees with Definition~\ref{def:stretchlocus}.

\begin{definition}\label{def:stretchlocusgen}
For any subset $X$ of~$\HH^n$ and any Lipschitz map $f : X\rightarrow\nolinebreak\HH^n$, the \emph{stretch locus} $E_f$ of~$f$ is the set of points $x\in X$ such that $\Lip_x(f)=\nolinebreak\Lip(f)$.
The \emph{enhanced stretch locus} $\widetilde{E}_f$ of~$f$ is the union of $\{(x,x)\in X\times X~|~x\in\nolinebreak E_f\}$ and
$$\big\{(x,x')\in X\!\times\!X ~|~ x\neq x' \text{ and } d(f(x),f(x'))=\Lip(f)\,d(x,x')\big\}.$$
\end{definition}

By Remark~\ref{rem:pathlength}.(1), both projections of $\widetilde{E}_f$ to~$X$ are equal to~$E_f$, but $\widetilde{E}_f$ records a little extra information, namely the positions of the maximally stretched segments between points of the stretch locus~$E_f$.

\subsection{Barycenters in~$\HH^n$}\label{subsec:bary}

For any index set $I$ equal to $\{1,2,\dots,k\}$ for some $k\geq 1$ or to~$\N^{\ast}$, and for any tuple $\underline{\alpha}=(\alpha_i)_{i\in I}$ of nonnegative reals summing up to~$1$, we set
$$(\HH^n)^I_{\underline{\alpha}} := \bigg\{ (p_i)\in(\HH^n)^I~\bigg|~\sum_{i\in I} \alpha_i\,d(p_1,p_i)^2<+\infty\bigg\} .$$
This set contains at least all bounded sequences $(p_i)\in(\HH^n)^I$, and it is just the direct product $(\HH^n)^k$ if $k<+\infty$.

The following result is classical, and actually holds in any $\mathrm{CAT}(0)$ space.

\begin{lemma}\label{lem:bary}
For any index set $I$ equal to $\{1,2,\dots,k\}$ for some $k\geq 1$ or to~$\N^{\ast}$ and for any tuple $\underline{\alpha}=(\alpha_i)_{i\in I}$ of nonnegative reals summing up to~$1$, the map
$$\boldsymbol{m}^{\underline{\alpha}} : (\HH^n)^I_{\underline{\alpha}} \longrightarrow \HH^n$$
taking $(p_i)_{i\in I}$ to the minimizer of $\sum_{i\in I} \alpha_i\,d(\,\cdot\, , p_i)^2$ is well defined and $\alpha_i$-Lipschitz in its $i$-th entry: for any $(p_i),(q_i)\in (\HH^n)^I_{\underline{\alpha}}$,
\begin{equation}\label{eqn:bary}
d\big(\boldsymbol{m}^{\underline{\alpha}}(p_1, p_2,\dots),\boldsymbol{m}^{\underline{\alpha}}(q_1, q_2,\dots) \big) \leq \sum_{i\in I} \alpha_i\,d(p_i,q_i).
\end{equation}
\end{lemma}

\begin{proof}
Fix $I$ and $\underline{\alpha}=(\alpha_i)_{i\in I}$, and consider an element $(p_i)\in (\HH^n)^I_{\underline{\alpha}}$.
For any $x\in\HH^n$,
\begin{eqnarray*}
\Phi(x) := \sum_{i\in I} \alpha_i\,d(x,p_i)^2 & \leq & \sum_{i\in I} \alpha_i\,\big(d(x,p_1) + d(p_1,p_i)\big)^2\\
& \leq & 2\, \sum_{i\in I} \alpha_i\,\big(d(x,p_1)^2 + d(p_1,p_i)^2\big) \,<\, +\infty.
\end{eqnarray*}
The function $\Phi : \HH^n\rightarrow\R$ thus defined is proper on $\HH^n$ since it is bounded from below by any proper function $\alpha_i\,d(\cdot,p_i)^2$ with $\alpha_i>0$, and it achieves its minimum on the convex hull of the~$p_i$.
Moreover, $\Phi$ is analytic: to see this on any ball $B$ of $\HH^n$, note that the unweighted summands $d(\cdot, p_i)^2$ for $p_i$ in a $1$-neighborhood of $B$ are analytic with derivatives (of any nonnegative order) bounded independently of~$i$, while the other summands can be written $\phi_i^2 + 2\phi_i d(p_1, p_i) + d(p_1, p_i)^2$, where $\phi_i:=d(\cdot, p_i)-d(p_1, p_i)$ again is analytic on~$B$, and $\phi_i$ and~$\phi_i^2$ have their derivatives (of any nonnegative order) bounded independently of~$i$.

On any unit-speed geodesic $(x_t)_{t\in\R}$ of~$\HH^n$, we have $\frac{\mathrm{d}^2}{\mathrm{d}t^2}\big|_{t=0}\,d(x_t, p_i)^2\geq 2$.
Indeed, let $\log_{x_0}:\HH^n\rightarrow T_{x_0}\HH^n$ be the inverse of the exponential map at~$x_0$.
Standard $\text{CAT}(0)$ comparison inequalities with the Euclidean metric $d_{\mathrm{Eucl}}$ yield
$$d_{\mathrm{Eucl}}(\log_{x_0}(x_t),\log_{x_0}(p_i))^2 \leq d(x_t,p_i)^2$$
for all $t\in\R$; both sides are equal at $t=0$, the first derivatives are equal at $t=0$, and the left-hand side has second derivative $\equiv 2$.
It follows that $t\mapsto \Phi(x_t)$ has second derivative at least $2\sum_I \alpha_i=2$ everywhere.
While $\boldsymbol{m}^{\underline{\alpha}}(p_1,p_2,\dots)$ is the minimizer of~$\Phi$, the point $\boldsymbol{m}^{\underline{\alpha}}(q_1,q_2,\dots)$ is the minimizer of $\Phi+\Psi$, where
$$\Psi(x) := \sum_{i\in I} \left(-d(x,p_i)^2 + d(x,q_i)^2\right) \alpha_i.$$
We claim that $\psi_i: x\mapsto -d(x,p_i)^2 + d(x,q_i)^2$ is $2d(p_i,q_i)$-Lipschitz: indeed, with $(x_t)_{t\in\R}$ as above,
\begin{eqnarray*}
\left | \frac{\mathrm{d}}{\mathrm{d}t}\Big|_{t=0} \psi_i(x_t)\right | &=& \left | 2d(x_0, p_i)\cos \widehat{p_i x_0 x_1} - 2d(x_0, q_i)\cos \widehat{q_i x_0 x_1} \right | \\
&=& 2d_{\mathrm{Eucl}}\big(\pi_{\ell}(\log_{x_0} p_i),\pi_{\ell} (\log_{x_0} q_i)\big) ~ \leq ~ 2d(p_i,q_i),
\end{eqnarray*}
where $\ell\subset T_{x_0}\HH^n$ is the tangent line to $(x_t)_{t\in\R}$ at $t=0$, and $\pi_{\ell} : T_{x_0}\HH^n\rightarrow\nolinebreak\ell$ is the closest-point projection.
Therefore, $\Psi$ is Lipschitz with constant $L:=2\sum_{i\in I} \alpha_i d(p_i,q_i)$.
Thus, for any unit-speed geodesic ray $(x_t)_{t\geq 0}$ starting from $x_0=\boldsymbol{m}^{\underline{\alpha}}(p_1,p_2,\dots)$, as soon as $t>L/2$ we have $\frac{\mathrm{d}}{\mathrm{d}t}\Phi(x_t)>L$, hence $\frac{\mathrm{d}}{\mathrm{d}t}(\Phi+\Psi)(x_t)>\nolinebreak 0$.
The minimizer of $\Phi+\Psi$ is within $L/2$ from~$x_0$, as promised.
\end{proof}

Note that the map~$\boldsymbol{m}^{\underline{\alpha}}$ is $G$-equivariant:
\begin{equation}\label{eqn:baryequiv}
\boldsymbol{m}^{\underline{\alpha}}(g\cdot p_1,g\cdot p_2,\dots) = g\cdot \boldsymbol{m}^{\underline{\alpha}}(p_1,p_2,\dots)
\end{equation}
for all $g\in G$ and $(p_i)\in(\HH^n)^I_{\underline{\alpha}}$.
It is also diagonal:
\begin{equation}\label{eqn:barydiag}
\boldsymbol{m}^{\underline{\alpha}}(p,p,\dots) = p
\end{equation}
for all $p\in\HH^n$.
If $\sigma$ is a permutation of~$I$, then
\begin{equation}\label{eqn:barypermut}
\boldsymbol{m}^{(\alpha_{\sigma(1)}, \alpha_{\sigma(2)},\dots )}(p_{\sigma(1)},p_{\sigma(2)},\dots) = \boldsymbol{m}^{(\alpha_1, \alpha_2,\dots)}(p_1,p_2,\dots)
\end{equation}
for all $(p_i)\in(\HH^n)^I_{\underline{\alpha}}$; in particular, $\boldsymbol{m}_k:=\boldsymbol{m}^{(\frac{1}{k},\dots,\frac{1}{k})}$ is symmetric in its~$k$ entries.
Unlike barycenters in vector spaces however, $\boldsymbol{m}$ has only weak associativity properties: the best one can get is associativity over equal entries, \ie if $p_1=\dots=p_k=p$ then 
$$\boldsymbol{m}^{(\alpha_1,\dots, \alpha_{k+1},\dots)}(p_1,\dots, p_{k+1},\dots )=\boldsymbol{m}^{(\alpha_1+\dots+\alpha_k,\alpha_{k+1},\dots)}(p,p_{k+1},\dots)~.$$
We will often write $\sum_{i\in I} \alpha_i\,p_i$ for $\boldsymbol{m}^{\underline{\alpha}}(p_1,p_2,\dots)$.

While \eqref{eqn:bary} controls the displacement of a barycenter under a change of points, the following lemma deals with a change of weights.

\begin{lemma}\label{lem:baryLipschitz2}
Let $I=\{1,2,\dots,k\}$ for some $k\geq 1$ or~$\N^{\ast}$ and let $\underline{\alpha}=(\alpha_i)_{i\in I}$ and $\underline{\smash{\beta}}=(\beta_i)_{i\in I}$ be two nonnegative sequences, each summing up to~$1$.
Consider points $(p_i)_{i\in I}$ of $\HH^n$, all within distance~$R$ of some $p\in \HH^n$ (in particular $(p_i)\in(\HH^n)^I_{\underline{\alpha}}\cap(\HH^n)^I_{\underline{\smash{\beta}}}$).
Then
$$d\big(\boldsymbol{m}^{\underline{\alpha}}(p_1,p_2,\dots),\boldsymbol{m}^{\underline{\smash{\beta}}}(p_1,p_2,\dots)\big) \leq R\, \sum_{i\in I} |\alpha_i-\beta_i|.$$
\end{lemma}

\begin{proof}
For any $i\in I$, we set $\delta_i:=\alpha_i-\beta_i$.
The basic observation is that if for example $\delta_1>0$, then we can transfer $\delta_1$ units of weight from $p_1$ to $p$, at the moderate cost of moving the barycenter by $\leq R\delta_1$: by Lemma~\ref{lem:bary}, the point
$$m:= \boldsymbol{m}^{(\alpha_1,\alpha_2,\alpha_3,\dots)}(p_1,p_2,p_3,\dots)=\boldsymbol{m}^{(\delta_1,\beta_1,\alpha_2,\alpha_3,\dots)}(p_1,p_1,p_2,p_3,\dots)$$
lies at distance $\leq R\delta_1$ from $\boldsymbol{m}^{(\delta_1,\beta_1,\alpha_2,\alpha_3,\dots)}(p,p_1,p_2,p_3,\dots)$.
Repeating this procedure for all indices $i\geq 1$ such that $\delta_i>0$, we find that $m$ lies at distance $\leq R\delta$ from
$$\boldsymbol{m}^{(\delta, \min\{\alpha_1,\beta_1\}, \min\{\alpha_2,\beta_2\}, \min\{\alpha_3,\beta_3\}, \dots)}(p,p_1,p_2,p_3,\dots),$$
where we set $\delta :=\sum_{\delta_i>0} \delta_i$.
This expression is symmetric in $\underline{\alpha}$ and~$\underline{\smash{\beta}}$, so $m$ lies at distance $\leq 2R\delta= R\,\sum |\alpha_i-\beta_i|$ from $\boldsymbol{m}^{(\beta_1,\beta_2,\dots)}(p_1,p_2,\dots)$.
\end{proof}

\subsection{Barycenters of Lipschitz maps and partitions of unity}

Here is an easy consequence of Lemma~\ref{lem:bary}.

\begin{lemma}\label{lem:baryLipschitz}
Let $I=\{1,2,\dots,k\}$ for some $k\geq 1$ or~$\N^{\ast}$ and let $\underline{\alpha}=(\alpha_i)_{i\in I}$ be a nonnegative sequence summing up to~$1$.
Consider $p\in X\subset\HH^n$ and a sequence of Lipschitz maps $f_i : X\rightarrow\HH^n$ with $(f_i(p))\in (\HH^n)^I_{\underline{\alpha}}$ and with $(\Lip(f_i))_{i\in I}$ bounded.
Then the map
$$f = \sum_{i\in I} \alpha_i\,f_i \,:\ x \longmapsto \boldsymbol{m}^{\underline{\alpha}}(f_1(x),f_2(x),\dots)$$
is well defined on~$X$ and satisfies
$$\Lip_{x}(f) \leq \sum_{i\in I} \alpha_i\,\Lip_{x}(f_i) \quad\mathrm{and}\quad \Lip_Y(f) \leq \sum_{i\in I} \alpha_i\,\Lip_Y(f_i)$$
for all $x\in Y\subset X$.
In particular, if $\Lip(f_i)=C=\Lip(f)$ for all $i\in I$, then the (enhanced) stretch locus of~$f$ (Definition~\ref{def:stretchlocusgen}) is contained in the intersection of the (enhanced) stretch loci of the maps~$f_i$.
\end{lemma}

\begin{proof}
We first note that $(f_i(x))\in (\HH^n)^I_{\underline{\alpha}}$ for any $x\in\HH^n$.
Indeed, using the triangle inequality and the general inequality $(a+b+c)^2\leq 3\,(a^2+b^2+c^2)$ for $a,b,c\geq 0$, we have
\begin{eqnarray*}
& & \sum_{i\in I} \alpha_i\,d\big(f_1(x),f_i(x)\big)^2\\
& \leq & 3\,\sum_{i\in I} \alpha_i\,\Big(d\big(f_1(x),f_1(p)\big)^2 + d\big(f_1(p),f_i(p)\big)^2 + d\big(f_i(p),f_i(x)\big)^2\Big),
\end{eqnarray*}
which is finite since $(f_i(p))\in (\HH^n)^I_{\underline{\alpha}}$ and $(\Lip(f_i))_{i\in I}$ is bounded.
By Lem\-ma~\ref{lem:bary}, the map~$f$ is well defined and for any $x,y\in\HH^n$,
$$d\big(f(x),f(y)\big) \leq \sum_{i\in I} \alpha_i\,d\big(f_i(x),f_i(y)\big),$$
which implies Lemma~\ref{lem:baryLipschitz}.
\end{proof}

We also consider barycenters of maps with variable coefficients.
The following result, which combines Lemmas \ref{lem:baryLipschitz2} and~\ref{lem:baryLipschitz} in an equivariant setting, is one of our main technical tools; it will be used extensively throughout Sections \ref{sec:stretchlocus} and~\ref{sec:lipcont}.

\begin{lemma}\label{lem:partofunity}
Let $\Gamma_0$ be a discrete group, $(j,\rho)\in\Hom(\Gamma_0,G)^2$ a pair of representations with $j$ injective and discrete, and $B_1,\dots,B_r$ open subsets of~$\HH^n$.
For $1\leq i\leq r$, let $f_i : j(\Gamma_0)\cdot B_i\rightarrow\HH^n$ be a $(j,\rho)$-equivariant map that is Lipschitz on~$B_i$.
For $p\in\HH^n$, let $I_p$ denote the set of indices $1\leq i \leq r$ such that $p\in j(\Gamma_0)\cdot B_i$, and define
$$R_p := \mathrm{diam}\{ f_i(p)~|~i\in I_p \} < +\infty.$$
For $1\leq i \leq r$, let also $\psi_i: \HH^n \rightarrow [0,1]$ be a $j(\Gamma_0)$-invariant Lipschitz map supported in $j(\Gamma_0)\cdot B_i$. Assume that $\psi_1,\dots, \psi_r$ induce a partition of unity on a $j(\Gamma_0)$-invariant open subset $\mathcal{B}$ of $\bigcup_{i=1}^r j(\Gamma_0)\cdot B_i$.
Then the map
$$\begin{array}{crcl}
f = \sum_{i\in I} \psi_i f_i : & \mathcal{B} & \longrightarrow & \hspace{1.1cm} \HH^n\\ 
& p &  \longmapsto & \sum_{i\in I_p} \psi_i(p) f_i(p) \end{array}$$
is $(j,\rho)$-equivariant and for any $p\in\mathcal{B}$, the following ``Leibniz rule'' holds:
\begin{equation}\label{eqn:Lipbound1}
\Lip_p(f) \leq \sum_{i\in I_p} \big(\Lip_p(\psi_i)\,R_p + \psi_i(p)\,\Lip_p(f_i)\big)\, .
\end{equation}
\end{lemma}

\begin{proof}
The map~$f$ is $(j,\rho)$-equivariant because the barycentric construction is, see \eqref{eqn:baryequiv}.
Fix $p\in\mathcal{B}$ and $\varepsilon>0$.
By definition of~$I_p$, continuity of $\psi_i$ and~$f_i$, and upper semicontinuity of the local Lipschitz constant (Lemma~\ref{lem:localLip}), there is a neighborhood $\mathcal{U}$ of $p$ in~$\mathcal{B}$ such that for all $x\in \mathcal{U}$,
\begin{itemize}
  \item $\psi_i|_{\mathcal{U}}=0$ for all $i\notin I_p$,
  \item $\psi_i(x)\leq \psi_i(p)+\varepsilon$ for all $i\in I_p$,
  \item $R_x\leq R_p+\varepsilon$,
  \item $\Lip_{\mathcal{U}}(\psi_i)\leq\Lip_p(\psi_i)+\varepsilon$ for all $i\in I_p$,
  \item $\Lip_{\mathcal{U}}(f_i)\leq\Lip_p(f_i)+\varepsilon$ for all $i\in I_p$.
\end{itemize}
By the triangle inequality, for any $x,y\in\mathcal{U}$,
\begin{eqnarray*}
d(f(x),f(y)) & = & d\Big(\sum_{i\in I_p} \psi_i(x)f_i(x),\sum_{i\in I_p} \psi_i(y)f_i(y)\Big)\\
& \leq & d\Big(\sum_{i\in I_p} \psi_i(x)f_i(x),\sum_{i\in I_p} \psi_i(y)f_i(x)\Big)\\
& & +\ d\Big(\sum_{i\in I_p} \psi_i(y)\,f_i(x),\sum_{i\in I_p} \psi_i(y)f_i(y)\Big).
\end{eqnarray*}
Using Lemma~\ref{lem:baryLipschitz2}, we see that the first term of the right-hand side is bounded by
$$\sum_{i\in I_p} \big(\Lip_{\mathcal{U}}(\psi_i)\,d(x,y)\big)\,R_x \,\leq\, d(x,y) \bigg(\sum_{i\in I_p} (\Lip_p(\psi_i)+\varepsilon)\bigg) (R_p+\varepsilon)\, ;$$
and using Lemma~\ref{lem:baryLipschitz}, that the second term is bounded by 
$$d(x,y)\sum_{i\in I_p} \psi_i(y) \Lip_{\mathcal{U}}(f_i) \,\leq\, d(x,y)\sum_{i\in I_p} (\psi_i(p) +\varepsilon) (\Lip_p(f_i)+\varepsilon)\, .$$
The bound \eqref{eqn:Lipbound1} follows by letting $\varepsilon$ go to $0$.
\end{proof}

\section{An equivariant Kirszbraun--Valentine theorem for amenable~groups}\label{sec:Kirszbraun}

One of the goals of this paper is to refine the classical Kirszbraun--Valentine theorem \cite{kir34,val44}, which states that any Lipschitz map from a compact subset of~$\HH^n$ to~$\HH^n$ with Lipschitz constant $\geq 1$ can be extended to a map from $\HH^n$ to itself with the same Lipschitz constant.
We shall in particular extend this theorem to an equivariant setting, for two actions $j,\rho\in\Hom(\Gamma_0,G)$ of a discrete group~$\Gamma_0$ on~$\HH^n$, with $j$ geometrically finite (Theorem~\ref{thm:Kirszbraunequiv}).
Before we prove Theorem~\ref{thm:Kirszbraunequiv}, we shall
\begin{itemize}
  \item reprove the classical Kirszbraun--Valentine theorem (Section~\ref{subsec:classicalKirszbraun}), both for the reader's convenience and because the main technical step (Lemma~\ref{lem:moteur}) will be useful later to control the local Lipschitz constant;
  \item examine the case when the Lipschitz constant is $<1$ (Section~\ref{subsec:KirszbraunC<1});
  \item extend the classical Kirszbraun--Valentine theorem to an equivariant setting for two actions $j,\rho\in\Hom(S,G)$ of an \emph{amenable} group~$S$ (Section~\ref{subsec:amenableKirszbraun}).
  We shall use this as a technical tool to extend maps in cusps when dealing with geometrically finite representations $j\in\nolinebreak\Hom(\Gamma_0,G)$ that are \emph{not} convex cocompact.
\end{itemize}

\subsection{The classical Kirszbraun--Valentine theorem}\label{subsec:classicalKirszbraun}

We first give a proof of the classical Kirszbraun--Valentine theorem \cite{kir34,val44}.

\begin{proposition}\label{prop:classicalKirszbraun}
Let $K\neq\emptyset$ be a compact subset of~$\HH^n$.
Any Lipschitz map $\varphi : K\rightarrow\HH^n$ with $\Lip(\varphi)\geq 1$ admits an extension $f : \HH^n\rightarrow\HH^n$ with the same Lipschitz constant.
\end{proposition}

The following is an important technical step in the proof of Proposition~\ref{prop:classicalKirszbraun}.
It will also be used in the proofs of Lemmas \ref{lem:extend-neighb-p}, \ref{lem:MaxStretchedLam}, and~\ref{lem:1StretchedLam} below.

\begin{lemma}\label{lem:moteur}
Let $\mathbf{K}\neq\emptyset$ be a compact subset of~$\HH^n$ and ${\boldsymbol\varphi} : \mathbf{K}\rightarrow\HH^n$ a nonconstant Lipschitz map.
For any $p\in\HH^n\smallsetminus \mathbf{K}$, the function
\begin{eqnarray*}
\HH^n & \longrightarrow & \hspace{1.45cm} \R^+\\
q'\ & \longmapsto & C_{q'} := \max_{k\in\mathbf{K}} \frac{d(q',{\boldsymbol\varphi}(k))}{d(p,k)}
\end{eqnarray*}
admits a minimum $C_q$ at a point $q\in\HH^n$, and $q$ belongs to the convex hull of ${\boldsymbol\varphi}(\mathbf{K}')$ where
$$\mathbf{K}' := \bigg\{ k\in\mathbf{K} ~\Big|~ \frac{d(q,{\boldsymbol\varphi}(k))}{d(p,k)} = C_q\bigg\}.$$
Moreover,
\begin{itemize}
  \item either there exist $k_1,k_2\in\mathbf{K}'$ such that $0\leq\widehat{k_1 p k_2}<\widehat{\varphi(k_1) q \varphi(k_2)}\leq\pi$,
  \item or $\widehat{k_1 p k_2}=\widehat{\varphi(k_1) q \varphi(k_2)}$ for $(\nu\times\nu)$-almost all $(k_1,k_2)\in\mathbf{K}'\times\mathbf{K}'$, where $\nu$ is some probability measure on~$\mathbf{K}'$ such that $q$ belongs to the convex hull of the support of ${\boldsymbol\varphi}_*\nu$.
\end{itemize}
\end{lemma}

Here we denote by $\widehat{abc}\in [0,\pi]$ the angle at~$b$ between three points $a,b,c\in\nolinebreak\HH^n$.

\begin{proof}[Proof of Lemma~\ref{lem:moteur}]
The function~$q'\mapsto C_{q'}$ is proper and convex, hence admits a minimum $C_q$ at some point $q\in\HH^n$.
We have $C_q>0$ since $\boldsymbol\varphi$ is nonconstant.

Suppose by contradiction that $q$ does \emph{not} belong to the convex hull of ${\boldsymbol\varphi}(\mathbf{K}')$, and let $q'$ be the projection of $q$ to this convex hull.
If $q''$ is a point of the geodesic segment $[q,q']$, close enough to~$q$, then $\frac{d(q'',{\boldsymbol\varphi}(k))}{d(p,k)}$ is uniformly $<C_q$ for $k\in\mathbf{K}$: indeed, for $k\in\mathbf{K}$ in a small neighborhood of~$\mathbf{K}'$ this follows from the inequality $d(q'',{\boldsymbol\varphi}(k))< d(q,{\boldsymbol\varphi}(k))$; for $k\in\mathbf{K}$ away from~$\mathbf{K}'$ it follows from the fact that $\frac{d(q,{\boldsymbol\varphi}(k))}{d(p,k)}$ is itself bounded away from~$C_q$ by continuity and compactness of~$\mathbf{K}$.
Thus $C_{q''}<C_q$, a contradiction.
It follows that $q$ belongs to the convex hull of ${\boldsymbol\varphi}(\mathbf{K}')$.

Let $\mathbf{K}'_{\log}\subset T_p\HH^n$ (resp.\ $\mathbf{L}_{\log}\subset T_q\HH^n$) be the (compact) set of\,vectors\,whose image by the exponential map $\exp_p : T_p\HH^n\to\HH^n$ (resp.\ $\exp_q :\nolinebreak T_q\HH^n\to\nolinebreak\HH^n$) lies in~$\mathbf{K}'$ (resp.\ in~${\boldsymbol\varphi}(\mathbf{K}')$). Let
$${\boldsymbol\varphi}_{\log} := \exp_q^{-1} \circ\, {\boldsymbol\varphi} \circ \exp_p : \mathbf{K}'_{\log} \longrightarrow \mathbf{L}_{\log}$$
be the map induced by~${\boldsymbol\varphi}$.
The fact that $q$ belongs to the convex hull of ${\boldsymbol\varphi}(\mathbf{K}')$ implies that $0$ belongs to the convex hull of $\mathbf{L}_{\log}={\boldsymbol\varphi}_{\log}(\mathbf{K}'_{\log})$.
Therefore, there exists a probability measure $\nu_{\log}$ on~$\mathbf{K}'_{\log}$ such that
$$\int_{\mathbf{K}'_{\log}} \frac{{\boldsymbol\varphi}_{\log}(x)}{\Vert{\boldsymbol\varphi}_{\log}(x)\Vert} \,\mathrm{d}\nu_{\log}(x)=0 \in T_q\HH^n.$$
(The division is legitimate since $\Vert{\boldsymbol\varphi}_{\log}(x)\Vert=d(q,{\boldsymbol\varphi}(\exp_q(x)))\geq C_q\,d(p,\mathbf{K})>\nolinebreak 0$ for all $x\in\mathbf{K}'_{\log}$.)
We set $\nu:=(\exp_p)_{\ast}\nu_{\log}$, so that $q$ belongs to the convex hull of the support of ${\boldsymbol\varphi}_*\nu$.
We then have
\begin{eqnarray*}
& & \iint_{\mathbf{K}'\times\mathbf{K}'} \left( \cos \widehat{k_1pk_2} - \cos \widehat{{\boldsymbol\varphi}(k_1)q{\boldsymbol\varphi}(k_2)} \right) \, \mathrm{d}(\nu \times \nu)(k_1,k_2)\\
& = & \iint_{\mathbf{K}'_{\log}\times \mathbf{K}'_{\log}} \left ( \big\langle {\textstyle \frac{x_1}{\Vert x_1\Vert}} \big| {\textstyle \frac{x_2}{\Vert x_2\Vert}} \big\rangle - \big\langle {\textstyle \frac{{\boldsymbol\varphi}_{\log}(x_1)}{\Vert{\boldsymbol\varphi}_{\log}(x_1)\Vert}} \big| {\textstyle\frac{{\boldsymbol\varphi}_{\log}(x_2)}{\Vert{\boldsymbol\varphi}_{\log}(x_2)\Vert}} \big\rangle \right ) \, \mathrm{d}(\nu_{\log}\times\nu_{\log})(x_1,x_2)\\
& = & \left | \int_{\mathbf{K}'_{\log}} \frac{x}{\Vert x\Vert} \,\mathrm{d}\nu_{\log}(x) \right |^2 - \left | \int_{\mathbf{K}'_{\log}} \frac{{\boldsymbol\varphi}_{\log}(x)}{\Vert{\boldsymbol\varphi}_{\log}(x)\Vert} \,\mathrm{d}\nu_{\log}(x) \right |^2 \ \geq\ 0.
\end{eqnarray*}
If the function $(k_1,k_2)\mapsto\cos \widehat{k_1pk_2} - \cos \widehat{{\boldsymbol\varphi}(k_1)q{\boldsymbol\varphi}(k_2)}$ takes a positive value at some pair $(k_1,k_2)\in\mathbf{K}'\times\mathbf{K'}$, then we have $0\leq\widehat{k_1 p k_2}<\widehat{\varphi(k_1) q \varphi(k_2)}\leq\pi$.
Otherwise, the function takes only nonpositive values, hence is zero $(\nu\times\nu)$-almost everywhere since its integral is nonnegative; in other words, $\widehat{k_1 p k_2}=\widehat{\varphi(k_1) q \varphi(k_2)}$ for $(\nu\times\nu)$-almost all $(k_1,k_2)\in\mathbf{K}'\times\mathbf{K}'$.
\end{proof}

The other main ingredient in the proof of Proposition~\ref{prop:classicalKirszbraun} is the following consequence of Toponogov's theorem, a comparison theorem expressing the divergence of geodesics in negative curvature (see \cite[Lem.\,II.1.13]{bh99}).

\begin{lemma} \label{lem:toponogov}
In the setting of Lemma~\ref{lem:moteur}, we have $C_q\leq\max(\Lip(\boldsymbol\varphi),1)$. 
\end{lemma}

\begin{proof}
We may assume $C_q\geq 1$, otherwise there is nothing to prove.
By Lemma~\ref{lem:moteur}, there exist $k_1,k_2\in\mathbf{K}'$ such that $\widehat{k_1 p k_2}\leq\widehat{{\boldsymbol\varphi}(k_1) q {\boldsymbol\varphi}(k_2)}\neq 0$.
Since
$$\frac{d(q,{\boldsymbol\varphi}(k_1))}{d(p,k_1)} = \frac{d(q,{\boldsymbol\varphi}(k_2))}{d(p,k_2)} = C_q \geq 1,$$
Toponogov's theorem implies $d({\boldsymbol\varphi}(k_1),{\boldsymbol\varphi}(k_2))\geq C_q\,d(k_1,k_2)$.
On the other hand, we have $d({\boldsymbol\varphi}(k_1),{\boldsymbol\varphi}(k_2))\leq\Lip(\boldsymbol\varphi)\,d(k_1,k_2)$ by definition of $\Lip(\boldsymbol\varphi)$, hence $C_q\leq\Lip(\boldsymbol\varphi)$.
\end{proof}

\begin{proof}[Proof of Proposition~\ref{prop:classicalKirszbraun}]
It is enough to prove that for any point $p\in\HH^n\smallsetminus K$ we can extend~$\varphi$ to $K\cup\{ p\}$ keeping the same Lipschitz constant $C_0:=\Lip(\varphi)$.
Indeed, if this is proved, then we can consider a dense sequence $(p_i)_{i\in\N}$ of points of $\HH^n\smallsetminus K$, construct by induction a $C_0$-Lipschitz extension of~$\varphi$ to $K\cup\{ p_i\,|\,i\in\N\}$, and finally extend it to~$\HH^n$ by continuity.

Let $p\in\HH^n\smallsetminus K$.
If $\varphi$ is constant, then the constant extension of $\varphi$ to $K\cup\{p\}$ still has the same Lipschitz constant.
Otherwise we apply Lemmas \ref{lem:moteur} and~\ref{lem:toponogov} with $(\mathbf{K},{\boldsymbol\varphi}):=(K,\varphi)$.
\end{proof}

\begin{remark}\label{rem:classicalKirszbraun<1}
The proof actually shows that for $C\geq 1$, any map $\varphi :\nolinebreak K\to\nolinebreak\HH^n$ with $\Lip(\varphi)\leq C$ admits an extension $f : \HH^n\to\HH^n$ with $\Lip(f)\leq C$.
\end{remark}

\begin{remark}\label{rem:EuclKirszbraun}
The same proof shows that if $K$ is a nonempty compact subset of~$\R^n$, then any Lipschitz map $\varphi : K\rightarrow\R^n$ admits an extension $f : \R^n\rightarrow\R^n$ with the same Lipschitz constant.
There is no constraint on the Lipschitz constant for~$\R^n$ since the Euclidean analogue of Toponogov's theorem holds for any $C\geq 0$.
This Euclidean extension result is the one originally proved by Kirszbraun \cite{kir34}, by a different approach based on Helly's theorem.
The hyperbolic version is due to Valentine \cite{val44}.
\end{remark}

\begin{remark}\label{rem:Knoncompact}
Proposition~\ref{prop:classicalKirszbraun} actually holds for \emph{any} subset $K$ of~$\HH^n$, not necessarily compact.
Indeed, we can always extend~$\varphi$ to the closure $\overline{K}$ of~$K$ by continuity, with the same Lipschitz constant, and view $\overline{K}$ as an increasing union of compact sets $K_i$, $i\in\N$.
Proposition~\ref{prop:classicalKirszbraun} gives extensions $f_i : \HH^n\rightarrow\nolinebreak\HH^n$ of $\varphi|_{K_i}$ with $\Lip(f_i)\leq\Lip(\varphi)$, and by the Arzel\`a--Ascoli theorem we can extract a pointwise limit $f$ from the~$f_i$, extending $\varphi$ with $\Lip(f)=\Lip(\varphi)$.
\end{remark}

\subsection{A weaker version when the Lipschitz constant is $<1$}\label{subsec:KirszbraunC<1}

Proposition~\ref{prop:classicalKirszbraun} does not hold when the Lipschitz constant is $<1$: see Example~\ref{ex:K=3points}.
However, we prove the following strengthening of Remark~\ref{rem:classicalKirszbraun<1} with $C=1$.

\begin{proposition}\label{prop:KirszbraunC<1}
Let $K\neq\emptyset$ be a compact subset of~$\HH^n$.
Any Lipschitz map $\varphi : K\rightarrow\HH^n$ with $\Lip(\varphi)<1$ admits an extension $f : \HH^n\rightarrow\HH^n$ with $\Lip(f)<1$.
\end{proposition}

It is not clear whether the analogue of Remark~\ref{rem:Knoncompact} holds for $\Lip(\varphi)<1$: see Appendix~\ref{sec:questionlip}.

Here is the main technical step in the proof of Proposition~\ref{prop:KirszbraunC<1}.

\begin{lemma} \label{lem:extend-neighb-p}
Let $K\neq\emptyset$ be a compact subset of~$\HH^n$ with convex hull $\Conv{K}$ in~$\HH^n$, and let $\varphi : K\rightarrow\HH^n$ be a Lipschitz map with $\Lip(\varphi)<1$.
For any $p\in\Conv{K}$, there is a neighborhood $\mathcal{U}_p$ of $p$ in~$\HH^n$ and a $1$-Lipschitz extension $f_p : K\cup\mathcal{U}_p\to\HH^n$ of~$\varphi$ such that $\Lip_{\mathcal{U}_p}(f_p)<1$.
\end{lemma}

\begin{proof}
We first extend $\varphi$ to a map $h : K\cup\{p\}\to\HH^n$ with $\Lip(h)<1$.
For this we may assume $p\notin K$.
By Lemma~\ref{lem:moteur} with $(\mathbf{K},{\boldsymbol\varphi}):=(K,\varphi)$, we can find points $q\in\HH^n$ and $k_1,k_2\in K$ such that $C_q:=\max_{k\in K} d(q,\varphi(k))/d(p,k)$ is minimal and such that $d(q,\varphi(k_i))=C_q\,d(p,k_i)$ for $i\in\{1,2\}$ and $\widehat{k_1pk_2}\leq \widehat{\varphi(k_1) q \varphi(k_2)} \neq 0$.
We cannot have $C_q=1$, otherwise we would have $d(\varphi(k_1),\varphi(k_2))\geq d(k_1,k_2)$ by basic trigonometry, contradicting $\Lip(\varphi)<1$.
Therefore $C_q<1$ by Lemma~\ref{lem:toponogov}.
We can then take $h : K\cup\{p\}\to\HH^n$ to be the extension of~$\varphi$ sending $p$ to~$q$.

Next, choose a small $\varepsilon>0$ such that $\Lip(h)(1+\varepsilon)^4<1$.
Let us prove that there is a ball $B$ of radius $r>0$ centered at~$p$ and an extension $h' : B\to\HH^n$ of $h|_{K\cap B}$ such that $\Lip_B(h')\leq\Lip(h)(1+\varepsilon)^4<\nolinebreak 1$.
If $p\notin K$, we just take $B$ to be disjoint from~$K$ and $h'(B)=\{q\}$. 
If $p\in K$, we remark that there is a constant $r>0$ such that for any $x\in\HH^n$, the exponential map $\exp_x : T_x\HH^n\rightarrow\HH^n$ and its inverse $\log_x : \HH^n\rightarrow T_x\HH^n$ are both $(1+\varepsilon)$-Lipschitz when restricted to the ball $B_x(r)\subset\HH^n$ of radius~$r$ centered at~$x$ and to its image $\log_xB_x(r)\subset T_x\HH^n$.
We set $B:=B_p(r)$.
Consider the map
$$\log_{q} \circ\, h \circ \exp_p : \log_p (K) \longrightarrow T_q\HH^n.$$
Its restriction to $\log_p(K\cap B)\subset T_p\HH^n$ is $\Lip(h)(1+\varepsilon)^2$-Lipschitz.
By Remark~\ref{rem:EuclKirszbraun}, this restriction admits an extension $\psi : \log_p B\rightarrow T_q\HH^n$ with the same Lipschitz constant.
Then
$$h' := \exp_q \circ\, \psi \circ \log_p \,:\ B \longrightarrow \HH^n$$
is an extension of $h|_{K\cap B}$ with $\Lip_B(h')\leq\Lip(h)(1+\varepsilon)^4<\nolinebreak 1$.

Let $B'\subset B$ be another ball centered at~$p$, of radius $r'>0$ small enough so that $\Lip(h)r+\Lip(h')r'<r-r'$.
We claim that we may take $\mathcal{U}_p:=B'$ and $f_p : K\cup\mathcal{U}_p\to\HH^n$ to be the map that coincides with $\varphi$ on~$K$ and with $h'$ on~$\mathcal{U}_p$.
Indeed, for any distinct points $(x,y)\in K\times B'$, if $x\in B$ then $\Lip_{\{x,y\}}(f_p)\leq\Lip(h')<1$, and otherwise 
\begin{eqnarray}
\frac{d(f_p(x),f_p(y))}{d(x,y)} & \leq & \frac{d(f_p(x),f_p(p)) + d(f_p(p),f_p(y))}{d(x,p) - d(p,y)} \label{eqn:cutup} \\
& \leq & \frac{\Lip(h) \, d(x,p) + \Lip(h') \, r'}{d(x,p) - r'} < 1, \notag
\end{eqnarray}
where the last inequality uses the fact that $d(x,p)\geq r$ and the monotonicity of real M\"obius maps $t\mapsto (t+a)/(t-b)$ for $a,b\geq 0$.
\end{proof}

\begin{proof}[Proof of Proposition~\ref{prop:KirszbraunC<1}]
It is sufficient to prove that any Lipschitz map $\varphi :\nolinebreak K\rightarrow\HH^n$ with $C_0:=\Lip(\varphi)<1$ admits an extension $f : \Conv{K}\rightarrow\nolinebreak\HH^n$ with $\Lip(f)<1$, because we can always precompose with the closest-point projection $\pi :\nolinebreak\HH^n\rightarrow\nolinebreak\Conv{K}$, which is $1$-Lipschitz.

By Lemma~\ref{lem:extend-neighb-p}, for any $p\in\Conv{K}$ we can find a neighborhood $\mathcal{U}_p$ of $p$ in~$\HH^n$ and a $1$-Lipschitz extension $f_p : K\cup\mathcal{U}_p\to\HH^n$ of~$\varphi$ such that $\Lip_{\mathcal{U}_p}(f_p)<1$.
By compactness of $\Conv{K}$, we can find finitely many points $p_1,\ldots,p_m\in\Conv{K}$ such that $\Conv{K}\subset\bigcup_{i=1}^m \mathcal{U}_{p_i}$.
For any $1\leq i\leq m$, using Remark~\ref{rem:classicalKirszbraun<1}, we extend $f_{p_i}$ to a $1$-Lipschitz map on $\Conv{K}\cup\mathcal{U}_{p_i}$, still denoted by~$f_{p_i}$.
By \eqref{eqn:supLip} and Lemma \ref{lem:baryLipschitz}, the symmetric barycenter
$$f \,:=\, \sum_{i=1}^m {\scriptsize\frac{1}{m}} f_{p_i}|_{\Conv{K}} :\ \Conv{K} \longrightarrow \HH^n,$$
which extends~$\varphi$, satisfies
$$\Lip(f) \leq \max_{1\leq i\leq m} \frac{\Lip_{\mathcal{U}_{p_i}}(f_{p_i})+(m-1)}{m} < 1.\qedhere$$
\end{proof}

\subsection{An equivariant Kirszbraun--Valentine theorem for amenable groups}\label{subsec:amenableKirszbraun}

We now extend Proposition~\ref{prop:classicalKirszbraun} to an equivariant setting with respect to two actions of an amenable group.
Recall that a discrete group~$S$ is said to be \emph{amenable} if there exists a sequence $(F_i)_{i\in\N}$ of finite subsets of~$S$ (called a \emph{F{\o}lner sequence}) such that for any $g\in S$,
$$\frac{\# (gF_i\,\triangle\, F_i)}{\# F_i}\ \underset{i\to +\infty}{\longrightarrow}\ 0,$$
where $\triangle$ denotes the symmetric difference.
For instance, any group which is abelian or solvable up to finite index is amenable.

The following proposition will be used throughout Section~\ref{sec:stretchlocus} to extend Lipschitz maps in horoballs of~$\HH^n$ corresponding to cusps of the geometrically finite manifold $j(\Gamma_0)\backslash\HH^n$, taking $S$ to be a cusp stabilizer.

\begin{proposition}\label{prop:amenableKirszbraun}
Let $S$ be an amenable discrete group, $(j,\rho)\subset\Hom(S,G)^2$ a pair of representations with $j$ injective and $j(S)$ discrete in~$G$, and $K\neq\emptyset$ a $j(S)$-invariant subset of~$\HH^n$ whose image in $j(S)\backslash\HH^n$ is compact.
Any $(j,\rho)$-equivariant Lipschitz map $\varphi : K\rightarrow\HH^n$ with $\Lip(\varphi)\geq 1$ admits a $(j,\rho)$-equivariant extension $f : \HH^n\rightarrow\HH^n$ with the same Lipschitz constant.
\end{proposition}

\begin{proof}
Set $C_0:=\Lip(\varphi)\geq 1$.
By Proposition~\ref{prop:classicalKirszbraun} and Remark~\ref{rem:Knoncompact}, we can find an extension $f' : \HH^n\rightarrow\HH^n$ of~$\varphi$ with $\Lip(f')= C_0$, but $f'$ is not equivariant a priori.
We shall modify it into a $(j,\rho)$-equivariant map.
For any $\gamma\in S$, the $C_0$-Lipschitz map
$$f_{\gamma} := \rho(\gamma)\circ f' \circ j(\gamma)^{-1} : \HH^n \longrightarrow \HH^n$$
extends~$\varphi$. 
For all $\gamma,\gamma'\in S$ and all $p\in\HH^n$, since $f_{\gamma}$ and~$f_{\gamma'}$ agree on~$K$, the triangle inequality gives
\begin{equation}\label{eqn:fgammaamenable}
d\big(f_{\gamma}(p),f_{\gamma'}(p)\big) \leq 2C_0\cdot d(p,K) .
\end{equation}

Fix a finite generating subset $A$ of~$S$.
Using a F{\o}lner sequence of~$S$, we see that for any $\varepsilon>0$ there is a finite subset $F$ of~$S$ such that $\# (\gamma F\,\triangle\, F) \leq \varepsilon\,\# F$ for all $\gamma\in A$.
Write $F=\{\gamma_1,\dots, \gamma_k\}$ where $k=\#F$, and set
\begin{equation}\label{eqn:fepsilonamenable}
f^{\varepsilon}(p) := \boldsymbol{m}_k(f_{\gamma_1}(p),\dots, f_{\gamma_k}(p))
\end{equation}
for all $p\in\HH^n$, where $\boldsymbol{m}_k=\boldsymbol{m}^{(\frac{1}{k},\dots,\frac{1}{k})}$ is the averaging map of Lemma~\ref{lem:bary}.
By \eqref{eqn:barydiag}, the map~$f^{\varepsilon}$ still coincides with $\varphi$ on~$K$.
Moreover, as a barycenter of $C_0$-Lipschitz maps, $f^{\varepsilon}$ is $C_0$-Lipschitz (Lemma~\ref{lem:baryLipschitz}).
Note, using \eqref{eqn:baryequiv}, that
\begin{equation}\label{eqn:amenableequiv}
\rho(\gamma) \circ f^{\varepsilon} \circ j(\gamma)^{-1} (p) = \boldsymbol{m}_k(f_{\gamma\gamma_1}(p), \dots, f_{\gamma\gamma_k}(p))
\end{equation}
for all $\gamma\in S$ and $p\in\HH^n$.
Since $\# (\gamma F\,\triangle\, F) \leq \varepsilon\,\# F$ for all $\gamma\in A$, all but $\leq\varepsilon k$ of the $k$ entries of~$\boldsymbol{m}_k$ in \eqref{eqn:amenableequiv} are the same as in \eqref{eqn:fepsilonamenable} up to order, hence
$$d\big(\rho(\gamma)\circ f^{\varepsilon}\circ j(\gamma)^{-1}(p),f^{\varepsilon}(p)\big) \leq 2C_0 \cdot d(p,K) \cdot \varepsilon$$
for all $p\in\HH^n$ by \eqref{eqn:barypermut}, Lemma~\ref{lem:bary}, and \eqref{eqn:fgammaamenable}.
We conclude by letting $\varepsilon$ go to~$0$ and extracting a pointwise limit~$f$ from the~$f^{\varepsilon}$: such a map $f : \HH^n\rightarrow\HH^n$ is $C_0$-Lipschitz, extends~$\varphi$, and is equivariant under the action of any element $\gamma$ of~$A$, hence of~$S$.
\end{proof}

\section{The relative stretch locus}\label{sec:stretchlocus}

We now fix a discrete group~$\Gamma_0$, a pair $(j,\rho)\in\Hom(\Gamma_0,G)^2$ of representations of $\Gamma_0$ in~$G$ with $j$ geometrically finite, a $j(\Gamma_0)$-invariant subset $K$ of~$\HH^n$ whose image in $M:=j(\Gamma_0)\backslash\HH^n$ is compact (possibly empty), and a $(j,\rho)$-equivariant Lipschitz map $\varphi : K\rightarrow\HH^n$.
We shall use the following terminology and notation.

\begin{definition}\label{def:CE}
\begin{itemize}
  \item The \emph{relative minimal Lipschitz constant} $C_{K,\varphi}(j,\rho)$ is the infimum of Lipschitz constants $\Lip(f)$ of $(j,\rho)$-equivariant maps $f : \HH^n\rightarrow\HH^n$ with $f|_K=\varphi$.
  \item We denote by $\F_{K,\varphi}^{j,\rho}$ the set of $(j,\rho)$-equivariant maps $f : \HH^n\rightarrow\HH^n$ with $f|_K=\varphi$ that have minimal Lipschitz constant $C_{K,\varphi}(j,\rho)$.
  \item If $\F_{K,\varphi}^{j,\rho}\neq\emptyset$, the \emph{relative stretch locus} $E_{K,\varphi}(j,\rho)\subset \HH^n$ is the intersection of the stretch loci $E_f$ (Definition~\ref{def:stretchlocusgen}) of all maps $f\in\F_{K,\varphi}^{j,\rho}$.
  \item Similarly, the \emph{enhanced relative stretch locus} $\widetilde{E}_{K,\varphi}(j,\rho)\subset (\HH^n)^2$ is the intersection of the enhanced stretch loci $\widetilde{E}_f$ (Definition~\ref{def:stretchlocusgen}) of all maps $f\in\F_{K,\varphi}^{j,\rho}$.
\end{itemize}
\end{definition}

Note that $E_{K,\varphi}(j,\rho)$ is always $j(\Gamma_0)$-invariant and closed in~$\HH^n$, because $E_f$ is for each $f\in\F_{K,\varphi}^{j,\rho}$ (Lemma~\ref{lem:localLip}).
Similarly, $\widetilde{E}_{K,\varphi}(j,\rho)$ is always $j(\Gamma_0)$-invariant (for the diagonal action of $j(\Gamma_0)$ on $(\HH^n)^2$) and closed in~$(\HH^n)^2$.

If $K$ is empty, then $C_{K,\varphi}(j,\rho)$ is the minimal Lipschitz constant $C(j,\rho)$ of \eqref{eqn:defC} and $E_{K,\varphi}(j,\rho)$ is the intersection of stretch loci $E(j,\rho)$ of Theorem~\ref{thm:lamin}, which we shall simply call the stretch locus of $(j,\rho)$.
For empty~$K$ we shall sometimes write $\F^{j,\rho}$ instead of $\F_{K,\varphi}^{j,\rho}$.

\subsection{Elementary properties of the (relative) minimal Lipschitz constant and the (relative) stretch locus}\label{subsec:elementarystretch}

We start with an easy observation for empty~$K$.

\begin{remark}\label{rem:boundedrho}
If all elements of $\rho(\Gamma_0)$ are elliptic, then $C(j,\rho)=0$ and $\F^{j,\rho}$ is the set of constant maps with image a fixed point of $\rho(\Gamma_0)$ in~$\HH^n$ (such a fixed point exists by Fact~\ref{fact:fixedpoint}); in particular, $E(j,\rho)=\HH^n$.
\end{remark}

Here are now some elementary properties of $C_{K,\varphi}(j,\rho)$ and $E_{K,\varphi}(j,\rho)$ for general~$K$.

\begin{remark}\label{rem:CLipconj}
Conjugating by elements of~$G$ leaves the relative minimal Lipschitz constant invariant and modifies the relative stretch locus (if $\F^{j,\rho}_{K,\varphi}\neq\nolinebreak\emptyset$) by a translation: for any $j,\rho\in\nolinebreak\Hom(\Gamma_0,G)$ and $g,h\in G$, we have
$$\left\{ \begin{array}{ccl}
C_{g\cdot K,\, h\circ \varphi\circ g^{-1}}(j^g, \rho^h) & = & C_{K,\varphi}(j,\rho),\\
E_{g\cdot K,\, h\circ \varphi\circ g^{-1}}(j^g, \rho^h) & = & g\cdot E_{K,\varphi}(j,\rho),
\end{array} \right.$$
where $j^g:=gj(\cdot) g^{-1}$ and $\rho^h=h\rho(\cdot) h^{-1}$.
\end{remark}

\noindent
Indeed, for any $(j,\rho)$-equivariant Lipschitz map $f : \HH^n\rightarrow\HH^n$ extending $\varphi$, the map $h\circ f\circ g^{-1} : \HH^n\rightarrow\HH^n$ extends $h\circ \varphi\circ g^{-1}$, is $(j^g,\rho^h)$-equivariant with the same Lipschitz constant, and $\Lip_{g\cdot p}(h\circ f\circ g^{-1})=\Lip_p(f)$ for all $p\in\HH^n$.

\begin{lemma}\label{lem:finiteindex}
For any finite-index subgroup $\Gamma'_0$ of~$\Gamma_0$, if we set $j':=j|_{\Gamma'_0}$ and $\rho':=\rho|_{\Gamma'_0}$, then
\begin{itemize}
  \item $C_{K,\varphi}(j,\rho)=C_{K,\varphi}(j',\rho')$;
  \item $\F_{K,\varphi}^{j,\rho}\subset\F_{K,\varphi}^{j',\rho'}$, and $\F_{K,\varphi}^{j,\rho}\neq\emptyset$ if and only if $\F_{K,\varphi}^{j',\rho'}\neq\emptyset$;
  \item in this case, $E_{K,\varphi}(j,\rho)=E_{K,\varphi}(j',\rho')$.
\end{itemize}
\end{lemma}

By Lemma~\ref{lem:finiteindex}, we may always assume that
\begin{itemize}
  \item the finitely generated group $\Gamma_0$ is torsion-free (using the Selberg lemma \cite[Lem.\,8]{sel60});
  \item $j$ and~$\rho$ take values in the group $G_0=\PO(n,1)_0\simeq\SO(n,1)_0$ of orientation-preserving isometries of~$\HH^n$.
\end{itemize}
This will sometimes be used in the proofs without further notice.

\begin{proof}[Proof of Lemma~\ref{lem:finiteindex}]
The inequality $C_{K,\varphi}(j',\rho')\leq C_{K,\varphi}(j,\rho)$ holds because any $(j,\rho)$-equivariant map is $(j',\rho')$-equivariant.
We now prove the converse inequality.
Write $\Gamma_0$ as a disjoint union of cosets $\alpha_1\Gamma'_0,\dots,\alpha_r\Gamma'_0$ where $\alpha_i\in\nolinebreak\Gamma_0$.
Let $f'$ be a $(j',\rho')$-equivariant Lipschitz extension of~$\varphi$.
For $\gamma\in \Gamma_0$, the map $f_\gamma:=\rho(\gamma)\circ f'\circ j(\gamma)^{-1}$ depends only on the coset $\gamma \Gamma'_0$.
By \eqref{eqn:barypermut}, the symmetric barycenter $f:=\sum_{i=1}^r \frac{1}{r}\,f_{\alpha_i}$ satisfies, for any $\gamma \in \Gamma_0$,
$$\rho(\gamma)\circ f \circ j(\gamma)^{-1}= \sum_{i=1}^r \,\frac{1}{r}\, f_{\gamma\alpha_i} = f, $$
because the cosets $\gamma\alpha_i\Gamma'_0$ are the $\alpha_i\Gamma'_0$ up to order.
This means that $f$ is $(j,\rho)$-equivariant.
By Lemma~\ref{lem:baryLipschitz}, we have $\Lip(f)\leq\Lip(f')$, hence $C_{K,\varphi}(j,\rho)\leq C_{K,\varphi}(j',\rho')$ by minimizing $\Lip(f')$.

Since $C_{K,\varphi}(j,\rho)=C_{K,\varphi}(j',\rho')$, it follows from the definitions that $\F_{K,\varphi}^{j,\rho}\subset\F_{K,\varphi}^{j',\rho'}$ and that, if these are nonempty, then $E_{K,\varphi}(j,\rho)\supset E_{K,\varphi}(j',\rho')$.
In fact, if $\F_{K,\varphi}^{j',\rho'}$ is nonempty, then so is $\F_{K,\varphi}^{j,\rho}$, and $E_{K,\varphi}(j,\rho)=E_{K,\varphi}(j',\rho')$.
Indeed, for any $f'\in\F_{K,\varphi}^{j',\rho'}$, the symmetric barycenter $f=\sum_{i=1}^r \frac{1}{r}\,f_{\alpha_i}$ introduced above belongs to $\F_{K,\varphi}^{j,\rho}$, and the stretch locus of~$f$ is contained in that of~$f'$ by Lemma~\ref{lem:baryLipschitz}.
\end{proof}

\begin{lemma}
The inequalities
\begin{equation}\label{eqn:ClambdaCLip}
C'(j,\rho) \,\leq\, C(j,\rho) \,\leq\, C_{K,\varphi}(j,\rho),
\end{equation}
always hold, where $C'(j,\rho)$ is given by \eqref{eqn:defClambda}.
\end{lemma}

\begin{proof}
The right-hand inequality follows from the definitions.
For the left-hand inequality, we observe that for any $\gamma\in\Gamma_0$ with $j(\gamma)$ hyperbolic and any $p\in\HH^n$ on the translation axis of~$j(\gamma)$, if $f:\HH^n\rightarrow \HH^n$ is $(j,\rho)$-equivariant and Lipschitz, then
\begin{eqnarray*}
\lambda(\rho(\gamma)) & \leq & d\big(f(p),\rho(\gamma)\cdot f(p)\big) = d\big(f(p),f(j(\gamma)\cdot p)\big)\\
& \leq & \Lip(f)\,d(p,j(\gamma)\cdot p) = \Lip(f)\,\lambda(j(\gamma)),
\end{eqnarray*}
and we conclude by letting $\Lip(f)$ tend to $C(j,\rho)$.
\end{proof}

Here is a sufficient condition for the left inequality of \eqref{eqn:ClambdaCLip} to be an equality.
We shall see in Theorem~\ref{thm:Kirszbraunopt} that for $C(j,\rho)\geq 1$ this sufficient condition is also necessary, at least when $\F_{K,\varphi}^{j,\rho}$ and $E_{K,\varphi}(j,\rho)$ are nonempty.

\begin{lemma} \label{lem:ClambdaCLipLipf}
Let $\ell$ be a geodesic ray in~$\HH^n$ whose image in $j(\Gamma_0)\backslash\HH^n$ is bounded.
If $\ell$ is maximally stretched by some $(j,\rho)$-equivariant Lipschitz map $f : \HH^n\to\HH^n$, in the sense that $f$ multiplies all distances on~$\ell$ by $\Lip(f)$, then
$$C'(j,\rho) = C(j,\rho) = \Lip(f).$$
\end{lemma}

\begin{proof}
By \eqref{eqn:ClambdaCLip} it is enough to prove that $C'(j,\rho)\geq\Lip(f)$.
Parametrize $\ell$ by arc length as $(p_t)_{t\geq 0}$.
Since the image of~$\ell$ in $j(\Gamma_0)\backslash\HH^n$ is bounded, for any $\varepsilon>0$ we can find $\gamma\in\Gamma_0$ and $0<s<t$ with $t-s\geq 1$ such that the oriented segments $j(\gamma)\cdot [p_s, p_{s+1}]$ and $[p_t,p_{t+1}]$ of $\HH^n$ are $\varepsilon$-close in the $C^1$ sense.
By the closing lemma (Lemma~\ref{lem:closinglemma}), this implies
$$\big|\lambda(j(\gamma)) - (t-s)\big| \leq 2 \varepsilon.$$
The images under~$f$ of the unit segments above are also $\varepsilon\Lip(f)$-close geodesic segments.
By the closing lemma again and the $(j,\rho)$-equivariance of~$f$,
$$\big|\lambda(\rho(\gamma)) - (t-s)\Lip(f)\big| \leq 2 \varepsilon\,\Lip(f).$$
Taking $\varepsilon$ very small, we see that $\lambda(\rho(\gamma))/\lambda(j(\gamma))$ takes values arbitrarily close to $\Lip(f)$ for $\gamma\in\Gamma_0$ with $j(\gamma)$ hyperbolic, hence $C'(j,\rho)\geq \Lip(f)$.
\end{proof}

\subsection{Finiteness of the (relative) minimal Lipschitz constant}\label{subsec:C<infty}

\begin{lemma}\label{lem:C<infty}
\begin{enumerate}
  \item If $j$ is convex cocompact, then $C_{K,\varphi}(j,\rho)<+\infty$.
  \item In general, if $j$ is geometrically finite, then $C_{K,\varphi}(j,\rho)<+\infty$ unless there exists an element $\gamma\in\Gamma_0$ such that $j(\gamma)$ is parabolic and $\rho(\gamma)$ hyperbolic.
\end{enumerate}
\end{lemma}

The following proof uses Lemma~\ref{lem:partofunity} applied to an appropriate partition of unity.
A similar proof scheme will be used again throughout Section~\ref{sec:lipcont}.

\begin{proof}[Proof of Lemma~\ref{lem:C<infty}.(1) (Convex cocompact case)]
Recall Notation~\ref{not:ConvK} for\linebreak $\Conv{K}$.
If $j$ is convex cocompact, then $\Conv{K}$ is compact modulo $j(\Gamma_0)$, hence we can find open balls $B_1,\dots,B_r$ of~$\HH^n$, projecting injectively to $j(\Gamma_0)\backslash\HH^n$, such that $\Conv{K}$ is contained in the union of the $j(\Gamma_0)\cdot B_i$.
For any~$i$, let $f_i : B_i\rightarrow\HH^n$ be a Lipschitz extension of $\varphi|_{B_i\cap K}$ (such an extension exists by Proposition~\ref{prop:classicalKirszbraun}).
We extend~$f_i$ to $j(\Gamma_0)\cdot B_i$ in a $(j,\rho)$-equivariant way (with no control on the global Lipschitz constant a priori).
The function
$$p \longmapsto R_p := \mathrm{diam}\big\{ f_i(p)~|~1\leq i\leq r\text{ and }p\in j(\Gamma_0)\cdot B_i\big\} $$
is locally bounded above and $j(\Gamma_0)$-invariant, hence uniformly bounded on $\bigcup_i j(\Gamma_0)\cdot\nolinebreak B_i$.
Let $(\psi_i)_{1\leq i\leq r}$ be a partition of unity on~$\Conv{K}$, subordinated to the covering $(j(\Gamma_0)\cdot B_i)_{1\leq i\leq r}$, with $\psi_i$ Lipschitz and $j(\Gamma_0)$-invariant for all~$i$.
Lemma~\ref{lem:partofunity} gives a $(j,\rho)$-equivariant map
$$f := \sum_{i=1}^r \psi_i f_i : \Conv{K} \longrightarrow \HH^n$$
with $\Lip_p(f)$ bounded by some constant $L$ independent of $p\in\Conv{K}$.
Then $\Lip_{\Conv{K}}(f)\leq L$ by \eqref{eqn:supLip}.
By precomposing $f$ with the closest-point projection $\pi : \HH^n\to\Conv{K}$, which is $1$-Lipschitz and $(j,j)$-equivariant, we obtain a $(j,\rho)$-equivariant Lipschitz extension of~$\varphi$ to~$\HH^n$.
\end{proof}

\begin{proof}[Proof of Lemma~\ref{lem:C<infty}.(2) (General geometrically finite case)]
Suppose that for any $\gamma\in\Gamma_0$ with $j(\gamma)$ parabolic, the element $\rho(\gamma)$ is \emph{not} hyperbolic.
The idea is the same as in the convex cocompact case, but we need to deal with the presence of cusps, which make $\Conv{K}$ noncompact modulo $j(\Gamma_0)$.
We shall apply Proposition~\ref{prop:amenableKirszbraun} (the equivariant version of Proposition~\ref{prop:classicalKirszbraun} for amenable groups) to the stabilizers of the cusps.

Let $B_1,\dots,B_c$ be open horoballs of~$\HH^n$, disjoint from~$K$, whose images in $j(\Gamma_0)\backslash\HH^n$ are disjoint and intersect the convex core in standard cusp regions (Definition~\ref{def:standardcusp}), representing all the cusps.
Let $B_{c+1},\dots,B_r$ be open balls of~$\HH^n$ that project injectively to $j(\Gamma_0)\backslash\HH^n$, such that the union of the $j(\Gamma_0)\cdot B_i$ for $1\leq i\leq r$ covers~$\Conv{K}$.
For $c+1\leq i\leq r$, we construct a $(j,\rho)$-equivariant Lipschitz map $f_i : j(\Gamma_0)\cdot B_i\rightarrow\HH^n$ as in the convex cocompact case.
For $1\leq i\leq c$, we now explain how to construct a $(j,\rho)$-equivariant Lipschitz map $f_i : j(\Gamma_0)\cdot\nolinebreak B_i\rightarrow\nolinebreak\HH^n$.

Let $S_i$ be the stabilizer of~$B_i$ in~$\Gamma_0$ for the $j$-action.
We claim that there exists a $(j|_{S_i},\rho|_{S_i})$-equivariant Lipschitz map $f_i : B_i\rightarrow\HH^n$.
Indeed, choose $p\in B_i$, not fixed by any element of $j(S_i)$, and $q\in\HH^n$. Set
$$f_i(j(\gamma)\cdot p) := \rho(\gamma)\cdot q$$
for all $\gamma\in S_i$. 
Let $\wl:S_i\rightarrow\N$ be the word length with respect to some fixed finite generating set of $S_i$.
By Lemma~\ref{lem:disthorosphere}, there exists $R'>0$ such that
$$d(p,j(\gamma)\cdot p) \geq 2 \log\big(1+\wl(\gamma)\big) - R'$$
for all $\gamma\in S_i$.
On the other hand, there exists $R>0$ such that
$$d(q,\rho(\gamma)\cdot q) \leq 2 \log\big(1+\wl(\gamma)\big) + R$$
for all $\gamma\in S_i$: if $\rho(\gamma)$ is elliptic for all $\gamma\in S_i$, this follows from the fact that the group $\rho(S_i)$ admits a fixed point in~$\HH^n$ (Fact~\ref{fact:fixedpoint}), hence $d(q,\rho(\gamma)\cdot q)$ is bounded for $\gamma\in S_i$; otherwise this follows from Lemma~\ref{lem:disthorosphere}.
Since the function $\wl$ is proper, we see that $\limsup_{\gamma\in S_i} \frac{d(q,\rho(\gamma)\cdot q)}{d(p,j(\gamma)\cdot p)}\leq 1$, hence 
$$\sup_{\gamma\in S_i\smallsetminus \{1\}} \frac{d(q,\rho(\gamma)\cdot q)}{d(p,j(\gamma)\cdot p)} < +\infty.$$
In other words, $f_i$ is Lipschitz on $j(S_i)\cdot p$.
We then use Proposition~\ref{prop:amenableKirszbraun} to extend~$f_i$ to a $(j|_{S_i},\rho|_{S_i})$-equivariant Lipschitz map $f_i : B_i\rightarrow\HH^n$.

Let us extend $f_i$ to $j(\Gamma_0)\cdot B_i$ in a $(j,\rho)$-equivariant way (with no control on the global Lipschitz constant a priori).
We claim that
$$R_p := \mathrm{diam}\{ f_i(p)~|~1\leq i\leq r\text{ and }p\in j(\Gamma_0)\cdot B_i\} $$
is uniformly bounded on~$\Conv{K}$.
Indeed, $j(\Gamma_0)\cdot B_i\, \cap \, j(\Gamma_0)\cdot B_k=\emptyset$ for all $1\leq i\neq k\leq c$ by definition of standard cusp regions.
Therefore, if $p\in\HH^n$ belongs to $j(\Gamma_0)\cdot B_i$ for more than one index $1\leq i\leq r$, then it belongs to the ``thick'' part $j(\Gamma_0)\cdot\big(\bigcup_{c<i\leq r} B_i\big)$.
But $\bigcup_{c<i\leq r} B_i$ is bounded and $p\mapsto R_p$ is locally bounded above and $j(\Gamma_0)$-invariant, hence $R_p$ is uniformly bounded on~$\Conv{K}$.

We conclude as in the convex cocompact case.
\end{proof}

The converse to Lemma~\ref{lem:C<infty} is clear: if there exists an element $\gamma\in\Gamma_0$ such that $j(\gamma)$ is parabolic and $\rho(\gamma)$ hyperbolic, then $C_{K,\varphi}(j,\rho)=+\infty$.
Indeed, for any $p,q\in\HH^n$, the distance $d(p,j(\gamma^k)\cdot p)$ grows logarithmically in~$k$ (Lemma~\ref{lem:disthorosphere}) whereas $d(q,\rho(\gamma^k)\cdot q)$ grows linearly.

\medskip

\emph{In the rest of the paper, we shall assume $C_{K,\varphi}(j,\rho)<+\infty$ whenever we discuss a fixed pair $(j,\rho)$.}
We shall allow $C(j,\rho)=+\infty$ only in Section~\ref{sec:lipcont} and Proposition~\ref{prop:semicontE}, where we discuss semicontinuity properties of the maps $(j,\rho)\mapsto C(j,\rho)$ and $(j,\rho)\mapsto E(j,\rho)$, for empty~$K$.

\subsection{Projecting onto the convex core}

In the proof of Lemma~\ref{lem:C<infty}, we used the closest-point projection $\pi : \HH^n\to\Conv{K}$, which is $1$-Lipschitz and $(j,j)$-equivariant.
This projection will be used many times in Sections \ref{sec:stretchlocus}, \ref{sec:Kirszbraunopt}, and~\ref{sec:lipcont}, with the following more precise properties.

\begin{lemma}\label{lem:proj}
Suppose $\F_{K,\varphi}^{j,\rho}\neq\emptyset$ and let $\pi : \HH^n\rightarrow\Conv{K}$ be the closest-point projection.
For any $f\in\F_{K,\varphi}^{j,\rho}$,
\begin{enumerate}
  \item $\Lip(f\circ\pi) = \Lip(f|_{\Conv{K}}) = \Lip(f) = C_{K,\varphi}(j,\rho)$ and $f\circ\pi\in\F_{K,\varphi}^{j,\rho}$;
  \item if $C_{K,\varphi}(j,\rho)>0$, then the stretch loci and enhanced stretch loci (Definition~\ref{def:stretchlocusgen}) satisfy
  $$\left\{ \begin{array}{l}
  E_{f\circ\pi} = E_{f|_{\Conv{K}}} \subset E_f \cap \Conv{K},\\
  \widetilde{E}_{f\circ\pi} = \widetilde{E}_{f|_{\Conv{K}}} \subset \widetilde{E}_f \cap \big(\Conv{K}\times\Conv{K}\big).
  \end{array}\right.$$
\end{enumerate}
In particular, if $C_{K,\varphi}(j,\rho)>0$, then the relative stretch locus $E_{K,\varphi}(j,\rho)$ is always contained in $\Conv{K}$.
\end{lemma}

(This is not true if $C_{K,\varphi}(j,\rho)=0$: see Remark~\ref{rem:boundedrho}.)

\begin{proof}
For any $f\in\F_{K,\varphi}^{j,\rho}$ we have $\Lip(f\circ\pi) = \Lip(f|_{\Conv{K}}) \leq \Lip(f)$ since $\pi : \HH^n\rightarrow\Conv{K}$ is $1$-Lipschitz.
Equality holds and $f\circ\pi\in\F_{K,\varphi}^{j,\rho}$ since $\pi$ is $(j,j)$-equivariant and $\Lip(f)=C_{K,\varphi}(j,\rho)$ is minimal.
This proves~(1).

For~(2), it is enough to consider the enhanced stretch locus $\widetilde{E}\subset\HH^n\times\HH^n$, since the stretch locus~$E$ is the projection of $\widetilde{E}$ to either of the $\HH^n$ factors.
Note that
$$\widetilde{E}_{f|_{\Conv{K}}} \subset \widetilde{E}_{f\circ\pi} \subset \Conv{K}\times\Conv{K},$$
because $\pi$ is the identity on $\Conv{K}$ and is contracting outside $\Conv{K}$.
Let us prove that $\widetilde{E}_{f\circ\pi}\subset\widetilde{E}_{f|_{\Conv{K}}}$.
Consider a pair $(x,x')\in\widetilde{E}_{f\circ\pi}$.
By definition, there are sequences $(x_k)_{k\in\N}$ converging to~$x$ and $(x'_k)_{k\in\N}$ converging to~$x'$ such that $x_k\neq x'_k$ and
$$\frac{d(f\circ\pi(x_k),f\circ\pi(x'_k))}{d(x_k,x'_k)} \underset{\scriptscriptstyle k\rightarrow +\infty}{\longrightarrow} C_{K,\varphi}(j,\rho).$$
By continuity of~$\pi$ we have $\pi(x_k)\rightarrow\pi(x)=x$ and $\pi(x'_k)\rightarrow\pi(x')=x'$.
Since
$$\frac{d(f\circ\pi(x_k),f\circ\pi(x'_k))}{d(x_k,x'_k)} \leq \frac{d(f\circ\pi(x_k),f\circ\pi(x'_k))}{d(\pi(x_k),\pi(x'_k))} \leq C_{K,\varphi}(j,\rho),$$
the middle term also tends to $C_{K,\varphi}(j,\rho)$, which shows that $(\pi(x), \pi(x'))=(x,x')$ belongs to $\widetilde{E}_{f|_{\Conv{K}}}$.
Thus $\widetilde{E}_{f\circ\pi} = \widetilde{E}_{f|_{\Conv{K}}} \subset \widetilde{E}_f$.
\end{proof}

\subsection{Equivariant extensions with minimal Lipschitz constant}\label{subsec:Fnonempty}

We shall use the following terminology.

\begin{definition}\label{def:reductive}
A representation $\rho\in\Hom(\Gamma_0,G)$ is \emph{reductive} if the Zariski closure of $\rho(\Gamma_0)$ in~$G$ is reductive, or equivalently if the number of fixed points of the group~$\rho(\Gamma_0)$ in the boundary at infinity $\partial_{\infty}\HH^n$ of~$\HH^n$ is different from~$1$.
\end{definition}

\begin{lemma}\label{lem:Fnonempty}
The set $\F_{K,\varphi}^{j,\rho}$ of Definition~\ref{def:CE} is nonempty as soon as either $K\neq\emptyset$ or $\rho$ is reductive.
\end{lemma}

When $K=\emptyset$ \emph{and} $\rho$ is nonreductive, there may or may not exist a $(j,\rho)$-equivariant map $f : \HH^n\rightarrow\HH^n$ with minimal constant $C(j,\rho)=C_{K,\varphi}(j,\rho)$: see examples in Sections \ref{ex:nonreductive1} and~\ref{ex:nonreductive2}.

\begin{proof}[Proof of Lemma~\ref{lem:Fnonempty}]
The idea is to apply the Arzel\`a--Ascoli theorem.
Set $C:=C_{K,\varphi}(j,\rho)$ and let $(f_k)_{k\in\N}$ be a sequence of $(j,\rho)$-equivariant Lipschitz maps with $f_k|_K=\varphi$ and $C+1\geq\Lip(f_k)\rightarrow C$.
The sequence $(f_k)$ is equicontinuous.
We first assume that $K\neq\emptyset$, and fix $q\in K$.
For any $k\in\N$ and any $p\in\HH^n$,
\begin{equation}\label{eqn:distf_iKnonempty}
d\big(f_k(p),\varphi(q)\big) \leq (C+1)\,d(p,q).
\end{equation}
Therefore, for any compact subset $\mathscr{C}$ of~$\HH^n$, the sets $f_k(\mathscr{C})$ for $k\in\N$ all lie in some common compact subset of~$\HH^n$.
The Arzel\`a--Ascoli theorem applies, yielding a subsequence with a $C$-Lipschitz limit; this limit necessarily belongs to $\F_{K,\varphi}^{j,\rho}$.
We now assume that $K=\emptyset$ and $\rho$ is reductive.
\begin{figure}[h!]
\begin{center}
\labellist
\small\hair 2pt
\pinlabel{$\HH^2$} at 200 320
\pinlabel{$\A_{\rho(\gamma_1)}$} at 156 260
\pinlabel{$\A_{\rho(\gamma_2)}$} at 290 200
\endlabellist
\includegraphics[width=7cm]{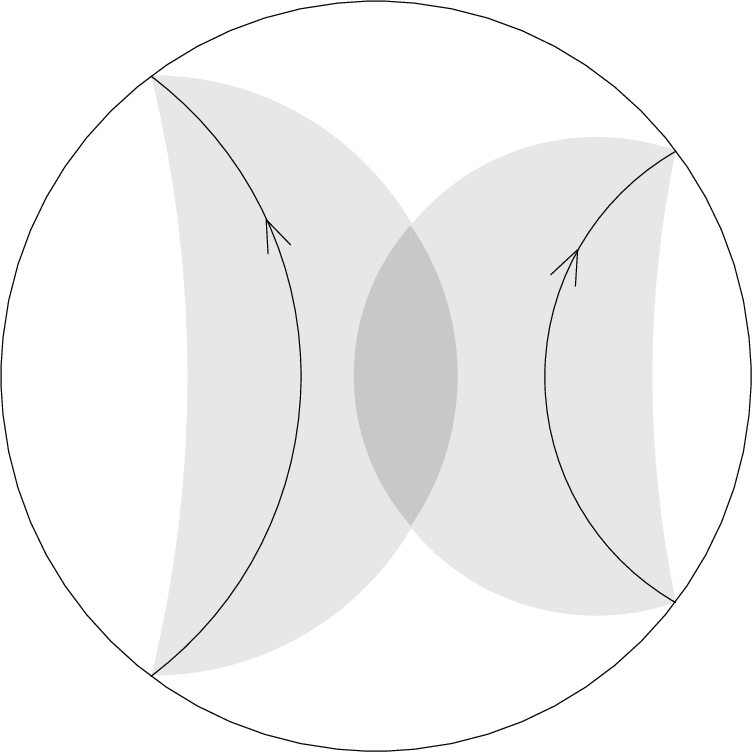}
\caption{Uniform neighborhoods of lines in $\HH^n$ with disjoint endpoints have a compact intersection.}
\label{fig:B}
\end{center}
\end{figure}
\begin{itemize}
  \item If the group $\rho(\Gamma_0)$ has no fixed point in $\HH^n$ and does not preserve any geodesic line of~$\HH^n$ (this is the generic case), then $\rho(\Gamma_0)$ contains two hyperbolic elements $\rho(\gamma_1),\rho(\gamma_2)$ whose translation axes have no common endpoint in $\partial_{\infty}\HH^n$.
  Fix a point $p\in\HH^n$.
  For any $k\in\N$ and $i\in\{ 1,2\}$,
  $$d\big(f_k(p),\rho(\gamma_i)\cdot f_k(p)\big) \leq (C+1)\,d(p,j(\gamma_i)\cdot p).$$
  Therefore, the points $f_k(p)$ for $k\in\N$ belong to some uniform neighborhood of the translation axis $\A_{\rho(\gamma_i)}$ of~$\rho(\gamma_i)$ for $i\in\{ 1,2\}$.
  Since $\mathcal{A}_{\rho(\gamma_1)}$ and~$\mathcal{A}_{\rho(\gamma_2)}$ have no common endpoint at infinity, the points $f_k(p)$ belong to some compact subset of~$\HH^n$ (see Figure~\ref{fig:B}).
  Since $\Lip(f_k)$ stays bounded, we obtain that for any compact subset $\mathscr{C}$ of~$\HH^n$, the sets $f_k(\mathscr{C})$ for $k\in\N$ all lie inside some common compact subset of~$\HH^n$, and we conclude as above using the Arzel\`a--Ascoli theorem.
  \item If the group $\rho(\Gamma_0)$ preserves a geodesic line $\A$ of~$\HH^n$, then it commutes with any hyperbolic element of~$G$ acting as a pure translation along~$\mathcal{A}$.
  For any $k\in\N$ and any such hyperbolic element~$g_k$, the map $g_k\circ f_k$ is still $(j,\rho)$-equivariant, with $\Lip(g_k\circ f_k)=\Lip(f_k)$.
  Fix $p\in\HH^n$.
  By the previous paragraph, the points $f_k(p)$ for $k\in\N$ belong to some uniform neighborhood of~$\A$.
  Therefore, after replacing $(f_k)_{k\in\N}$ by $(g_k\circ f_k)_{k\in\N}$ for some appropriate sequence $(g_k)_{k\in\N}$, we may assume that the points $f_k(p)$ for $k\in\N$ all belong to some compact subset of~$\HH^n$, and we conclude as above.
  \item If the group $\rho(\Gamma_0)$ has a fixed point in~$\HH^n$, we use Remark~\ref{rem:boundedrho}.\qedhere
\end{itemize}
\end{proof}

\subsection{The stretch locus of an equivariant extension with minimal Lipschitz constant}

\begin{lemma}\label{lem:Efnonempty}
\begin{enumerate}
  \item If $j$ is convex cocompact, then the stretch locus $E_f$ of any $f\in\F_{K,\varphi}^{j,\rho}$ is nonempty.
  \item In general, the stretch locus of any $f\in\F_{K,\varphi}^{j,\rho}$ is nonempty except possibly if $C_{K,\varphi}(j,\rho)=1$ and $\rho$ is not cusp-deteriorating.
\end{enumerate}
\end{lemma}

Recall from Definition~\ref{def:typedet} that ``$\rho$ is not cusp-deteriorating'' means there is an element $\gamma\in\Gamma_0$ such that $j(\gamma)$ and~$\rho(\gamma)$ are both parabolic.
When $C_{K,\varphi}(j,\rho)=\nolinebreak 1$, there exist examples of pairs $(j,\rho)$ with $\rho$ non-cusp-deteriora\-ting such that the stretch locus $E_f$ is empty for some maps $f\in\F_{K,\varphi}^{j,\rho}$ (see Sections \ref{ex:nondeteriorating} and~\ref{ex:CC'nonreductive}, as well as Corollary~\ref{cor:jrhononred} for elementary~$\Gamma_0$).

\begin{proof}[Proof of Lemma~\ref{lem:Efnonempty}.(1) (Convex cocompact case)]
By Lemma~\ref{lem:proj}, it is sufficient to prove that the stretch locus of $f|_{\Conv{K}}$ is nonempty.
The function $x\mapsto\Lip_x(f|_{\Conv{K}})$ is upper semicontinuous (Lemma~\ref{lem:localLip}) and $j(\Gamma_0)$-invariant.
If $j$ is convex cocompact, then $\Conv{K}$ is compact modulo~$j(\Gamma_0)$, and so $x\mapsto\Lip_x(f|_{\Conv{K}})$ achieves its maximum on~$\Conv{K}$, at a point that belongs to the stretch locus of~$f|_{\Conv{K}}$.
\end{proof}

\begin{proof}[Proof of Lemma~\ref{lem:Efnonempty}.(2) (General geometrically finite case)]
Assume either\linebreak that $C\neq 1$, or that $C=1$ and $\rho$ is cusp-deteriorating, where we set $C:=C_{K,\varphi}(j,\rho)$.
Consider $f\in\F_{K,\varphi}^{j,\rho}$.
As in the convex cocompact case, it is sufficient to prove that the stretch locus of $f|_{\Conv{K}}$ is nonempty.
Suppose by contradiction that it is empty: this means (Lemma~\ref{lem:localLip}) that $\Lip_{K'}(f)<\nolinebreak C$ for any compact subset $K'$ of $\Conv{K}$, or equivalently that the $j(\Gamma_0)$-invariant function $x\mapsto\Lip_x(f)$ only approaches~$C$ asymptotically (from below) in some cusps.
Let $B_1,\dots,B_c$ be open horoballs of~$\HH^n$, disjoint from~$K$, whose images in $j(\Gamma_0)\backslash\HH^n$ are disjoint and intersect the convex core in standard cusp regions (Definition~\ref{def:standardcusp}), representing all the cusps.
Our strategy is, for each $B_i$ on which $x\mapsto\Lip_x(f)$ approaches~$C$ asymptotically, to modify $f|_{\Conv{K}}$ on $\Conv{K}\cap j(\Gamma_0)\cdot B_i$ in a $(j,\rho)$-equivariant way so as to decrease the Lipschitz constant on $\Conv{K}\cap B_i$.
By \eqref{eqn:supLip}, this will yield a new $(j,\rho)$-equivariant extension of~$\varphi$ to~$\Conv{K}$ with a smaller Lipschitz constant than $f|_{\Conv{K}}$, which will contradict the minimality of $\Lip(f|_{\Conv{K}})=\Lip(f)$.
Let us now explain the details.

Let $B=B_i$ be an open horoball as above, on which $x\mapsto\Lip_x(f)$ approaches~$C$ asymptotically, and let $S$ be the stabilizer of $B$ in~$\Gamma_0$ for the $j$-action.
The group $j(S)$ is discrete and contains only parabolic and elliptic elements.
Since $C_{K,\varphi}(j,\rho)<+\infty$ by assumption, the group $\rho(S)$ also contains only parabolic and elliptic elements (Lemma~\ref{lem:C<infty}).

First we assume that $\rho(S)$ contains a parabolic element, \ie $\rho$ is \emph{not} deteriorating in~$B$ (Definition~\ref{def:detinacusp}).
In particular, $\rho$ is not cusp-deteriorating, hence $C\geq 1$ by Lemma~\ref{lem:parabdet} and so $C>1$ by the assumption made at the beginning of the proof.
Since $S$ is amenable, in order to decrease the Lipschitz constant on $\Conv{K}\cap B$ it is enough to prove that $\Lip_{\Conv{K}\cap\partial B}(f)<C$, because we can then apply Proposition~\ref{prop:amenableKirszbraun}.
By geometrical finiteness and the assumption that the image of $K$ in $j(\Gamma_0)\backslash\HH^n$ is compact (see Fact~\ref{fact:geomfinite} and the remarks after Notation~\ref{not:ConvK}), we can find a \emph{compact} fundamental domain $\D$ of $\Conv{K}\cap j(\Gamma_0)\cdot \partial B$ for the action of~$j(\Gamma_0)$.
Fix $p\in\D$.
By Lemma~\ref{lem:disthorosphere}, there exist $R,R'>0$ such that
\begin{equation}\label{eqn:LipEf1}
d(p,j(\gamma)\cdot p)\geq 2 \log(1+\wl(\gamma)) - R'
\end{equation}
and
\begin{equation}\label{eqn:LipEf2}
d\big(f(p),f(j(\gamma)\cdot p)\big) = d\big(f(p),\rho(\gamma)\cdot f(p)\big) \leq 2 \log(1+\wl(\gamma)) + R
\end{equation}
for all $\gamma\in S$, where $\wl : S\rightarrow\N$ denotes the word length with respect to some fixed finite generating set of~$S$.
Consider $q,q'\in\Conv{K}\cap\partial B$ with $q\in\D$; there is an element $\gamma\in\Gamma_0$ such that $d(j(\gamma)\cdot p,q')\leq\Delta$, where $\Delta>0$ is the diameter of~$\D$.
By the triangle inequality, \eqref{eqn:LipEf1}, and \eqref{eqn:LipEf2}, we have
\begin{eqnarray*}
d(q,q') & \geq & d(p,j(\gamma)\cdot p) - d(p,q) - d(j(\gamma)\cdot p,q')\\
& \geq & 2 \log(1+\wl(\gamma)) - (R'+2\Delta)
\end{eqnarray*}
and, using $\Lip(f)=C$,
\begin{eqnarray*}
d(f(q),f(q')) & \leq & d\big(f(p),f(j(\gamma)\cdot p)\big) + d(f(p),f(q)) + d(f(j(\gamma)\cdot p),f(q'))\\
& \leq & 2 \log(1+\wl(\gamma)) + (R+2C\Delta).
\end{eqnarray*}
Since $C>1$, this implies
$$\frac{d(f(q),f(q'))}{d(q,q')} \leq \frac{1+C}{2} < C$$
as soon as $\wl(\gamma)$ is large enough, or equivalently as soon as $d(q,q')$ is large enough.
However, this ratio is also bounded away from~$C$ when $d(q,q')$ is bounded, because the segment $[q,q']$ then stays in a compact part of $j(\Gamma_0)\backslash \Conv{K}$.
Therefore there is a constant $C''<C$ such that\linebreak $d(f(q),f(q'))\leq C''d(q,q')$ for all $q,q'\in\Conv{K}\cap\partial B$ with $q\in\D$, hence $\Lip_{\Conv{K}\cap\partial B}(f)\leq C''<C$ by equivariance.
By Proposition~\ref{prop:amenableKirszbraun}, we can redefine $f$ inside $\Conv{K}\cap B$ so that $\Lip_{\Conv{K}\cap B}(f)<C$.
We then extend $f$ to $\Conv{K}\cap j(\Gamma_0)\cdot B$ in a $(j,\rho)$-equivariant way.

We now assume that $\rho(S)$ consists entirely of elliptic elements, \ie $\rho$ is deteriorating in~$B$ (Definition~\ref{def:detinacusp}).
Then $\rho(S)$ admits a fixed point $q$ in~$\HH^n$ by Fact~\ref{fact:fixedpoint}.
Let $f_1 : j(\Gamma_0)\cdot B\rightarrow\HH^n$ be the $(j,\rho)$-equivariant map that is constant equal to~$q$ on~$B$, and let $\psi_1 : \HH^n\rightarrow [0,1]$ be the $j(\Gamma_0)$-invariant function supported on $j(\Gamma_0)\cdot B$ given by
$$\psi_1(p) = \varepsilon\,\psi\big(d(p,\partial B)\big)$$
for all $p\in B$, where $\psi : \R^+\rightarrow [0,1]$ is the $3$-Lipschitz function with $\psi|_{[0,1/3]}=\nolinebreak 0$ and $\psi|_{[2/3,+\infty)}=1$, and $\varepsilon>0$ is a small parameter to be adjusted later.
Let $f_2:=f$, and let $\psi_2:=1-\psi_1$.
The $(j,\rho)$-equivariant map
$$f_0 := \psi_1 f_1 + \psi_2 f_2 :\ \HH^n \longrightarrow \HH^n$$
coincides with~$f$ on $\Conv{K}\cap\partial B$.
Let us prove that if $\varepsilon$ is small enough, then $\Lip_p(f_0)$ is bounded by some uniform constant $<C$ for $p\in\Conv{K}\cap \nolinebreak B$.
Let $p\in\Conv{K}\cap B$.
Since $\Lip_p(f_1)=0$, since $f_1(p)=q$, and since $f_2=f$, Lemma~\ref{lem:partofunity} yields
$$\Lip_p(f_0) \leq \big(\Lip_p(\psi_1) + \Lip_p(\psi_2)\big)\,d(q,f(p)) + \psi_2(p)\,\Lip_p(f).$$
Let $B'$ be a horoball contained in~$B$, at distance~$1$ from~$\partial B$.
If $p\in\nolinebreak\Conv{K}\cap\nolinebreak B'$, then $\Lip_p(\psi_1)=\Lip_p(\psi_2)=0$ and $\psi_2(p)=1-\varepsilon$, hence
$$\Lip_p(f_0) \leq (1-\varepsilon)\,\Lip_p(f) \leq (1-\varepsilon)\,C.$$
If $p\in\Conv{K}\cap (B\smallsetminus B')$, then $\Lip_p(\psi_1),\Lip_p(\psi_2)\leq 3 \varepsilon$ and $\psi_2(p)\leq 1$, hence
$$\Lip_p(f_0) \leq 6\varepsilon\,d(q,f(p)) + \sup_{x\in\Conv{K}\cap (B\smallsetminus B')}\Lip_x(f)\,.$$
Note that the set $\Conv{K}\cap (B\smallsetminus B')$ is compact modulo $j(S)$, which implies, on the one hand that the $j(S)$-invariant, continuous function $p\mapsto d(q,f(p))$ is bounded on $\Conv{K}\cap (B\smallsetminus B')$, on the other hand that the $j(S)$-invariant, upper semicontinuous function $x\mapsto\Lip_x(f)$ is bounded away from~$C$ on $\Conv{K}\cap (B\smallsetminus B')$ (recall that the stretch locus of $f|_{\Conv{K}}$ is empty by assumption).
Therefore, if $\varepsilon$ is small enough, then $\Lip_p(f_0)$ is bounded by some uniform constant $<C$ for $p\in\Conv{K}\cap B$, which implies $\Lip_{\Conv{K}\cap B}(f_0)<C$ by \eqref{eqn:supLip}.
We can redefine $f$ to be~$f_0$ on $\Conv{K}\cap B$.
We then extend $f$ to $\Conv{K}\cap j(\Gamma_0)\cdot B$ in a $(j,\rho)$-equivariant way.

After redefining~$f$ as above in each cusp where the local Lipschitz constant $x\mapsto\Lip_x(f)$ approaches~$C$ asymptotically, we obtain a $(j,\rho)$-equivariant map on $\Conv{K}$ with Lipschitz constant $<C$, which contradicts the minimality of~$C$.
\end{proof}

\subsection{Optimal extensions with minimal Lipschitz constant}

\begin{definition}\label{def:relstretchlocus}
An element $f_0\in\F_{K,\varphi}^{j,\rho}$ (Definition~\ref{def:CE}) is called \emph{optimal} if its enhanced stretch locus $\widetilde{E}_{f_0}$ (Definition~\ref{def:stretchlocusgen}) is minimal, equal to
$$\widetilde{E}_{K,\varphi}(j,\rho) = \bigcap_{f\in\F_{K,\varphi}^{j,\rho}} \widetilde{E}_f.$$
\end{definition}

This means that the ordinary stretch locus $E_{f_0}$ of~$f_0$ is minimal, equal to $E_{K,\varphi}(j,\rho)=\bigcap_{f\in\F_{K,\varphi}^{j,\rho}} E_f$, and that the set of maximally stretched segments of~$f_0$ is minimal (using Remark~\ref{rem:pathlength}.(1)).
This last condition will be relevant only when $C_{K,\varphi}(j,\rho)=1$, in the proof of Lemma~\ref{lem:1StretchedLam}: indeed, when $C_{K,\varphi}(j,\rho)>1$, Theorem~\ref{thm:Kirszbraunopt} shows that $f_0\in\F_{K,\varphi}^{j,\rho}$ is optimal if and only if its ordinary stretch locus $E_{f_0}$ is minimal.

As mentioned in the introduction, in general an optimal map~$f_0$ is by no means unique, since it may be perturbed away from $E_{K,\varphi}(j,\rho)$.

\begin{lemma}\label{lem:optimalmap}
If $\F_{K,\varphi}^{j,\rho}\neq\emptyset$, then there exists an optimal element $f_0\in\F_{K,\varphi}^{j,\rho}$.
\end{lemma}

\begin{proof}
For any $f\in\F_{K,\varphi}^{j,\rho}$, the enhanced stretch locus $\widetilde{E}_f$ is closed in $\HH^n\times\HH^n$ (Lemma~\ref{lem:localLip} and Remark~\ref{rem:pathlength}.(1)) and $j(\Gamma_0)$-invariant for the diagonal action. Therefore $\widetilde{E}_{K,\varphi}(j,\rho)$ is also closed and $j(\Gamma_0)$-invariant.
By definition, for any $x=(p,q)\in(\HH^n\times\HH^n)\smallsetminus \widetilde{E}_{K,\varphi}(j,\rho)$ (possibly with $p=q$), we can find a neighborhood $\mathcal{U}_x$ of $x$ in $\HH^n\times\HH^n$, a $(j,\rho)$-equivariant map $f_x\in\F_{K,\varphi}^{j,\rho}$, and a constant $\delta_x>0$ such that
$$\sup_{(p',q')\in \mathcal{U}_x \atop p'\neq q'}\frac{d(f_x(p'),f_x(q'))}{d(p',q')} = C_{K,\varphi}(j,\rho) - \delta_x < C_{K,\varphi}(j,\rho).$$
Since $(\HH^n\times\HH^n)\smallsetminus \widetilde{E}_{K,\varphi}(j,\rho)$ is exhausted by countably many compact sets, we can write
$$(\HH^n\times\HH^n)\smallsetminus \widetilde{E}_{K,\varphi}(j,\rho) = \bigcup_{i=1}^{+\infty}\ \mathcal{U}_{x_i}$$
for some sequence $(x_i)_{i\geq 1}$ of points of $(\HH^n\times\HH^n)\smallsetminus \widetilde{E}_{K,\varphi}(j,\rho)$.
Choose a point $p\in\HH^n$ and let $\underline{\alpha}=(\alpha_i)_{i\geq 1}$ be a sequence of positive reals summing up to~$1$ and decreasing fast enough so that
$$\sum_{i=1}^{+\infty} \alpha_i\,d(f_{x_1}(p),f_{x_i}(p))^2 < +\infty\,.$$
By Lemma~\ref{lem:baryLipschitz}, the map $f_0:=\sum_{i=1}^{\infty} \alpha_i f_{x_i}$ is well defined and satisfies
$$\sup_{(p,q)\in \mathcal{U}_{x_i} \atop p\neq q}\frac{d(f_0(p ),f_0(q))}{d(p,q)} \leq C_{K,\varphi}(j,\rho) - \alpha_i\,\delta_{x_i} <C_{K,\varphi}(j,\rho)$$
for all $i$, hence $\widetilde{E}_{f_0}\cap\mathcal{U}_{x_i}=\emptyset$, which means that $\widetilde{E}_{f_0}=\widetilde{E}_{K,\varphi}(j,\rho)$.
\end{proof}

Here is an immediate consequence of Lemmas \ref{lem:Efnonempty} and~\ref{lem:optimalmap}.

\begin{corollary}\label{cor:Enonempty}
If $\F_{K,\varphi}^{j,\rho}\neq\emptyset$, then the relative stretch locus $E_{K,\varphi}(j,\rho)$
\begin{itemize}
  \item is nonempty for convex cocompact~$j$;
  \item is nonempty for geometrically finite~$j$ in general, except possibly if $C_{K,\varphi}(j,\rho)=1$ and $\rho$ is \emph{not} cusp-deteriorating.
\end{itemize}
\end{corollary}

In fact, the following holds.

\begin{lemma}\label{lem:optimallocallycst}
If $\F_{K,\varphi}^{j,\rho}\neq\emptyset$, then for any $p\in\HH^n\smallsetminus (E_{K,\varphi}(j,\rho)\cup K)$ there exists an optimal element $f_0\in\F_{K,\varphi}^{j,\rho}$ that is constant on a neighborhood of~$p$.
\end{lemma}

\begin{proof}
Assume that $\F_{K,\varphi}^{j,\rho}\neq\emptyset$ and let $f\in\F_{K,\varphi}^{j,\rho}$ be optimal (given by Lemma~\ref{lem:optimalmap}).
Fix $p\in\HH^n\smallsetminus (E_{K,\varphi}(j,\rho)\cup K)$.
Let $B\subset\HH^n$ be a closed ball centered at~$p$, with small radius $r>0$, such that $B$ does not meet $K\cup E_{K,\varphi}(j,\rho)$ and projects injectively to $j(\Gamma_0)\backslash\HH^n$.
By Lemma~\ref{lem:localLip},
$$C^{\ast} := \Lip_B(f) < C := C_{K,\varphi}(j,\rho).$$
For any small enough ball $B'\subset B$ of radius $r'$ centered at~$p$, the map
$$\iota_p : \partial B \cup B' \longrightarrow \HH^n$$
that coincides with the identity on $\partial B$ and is constant with image $\{ p\} $ on~$B'$, satisfies $1<\Lip(\iota_p)=\frac{r}{r-r'}<C/C^{\ast}$.
Proposition~\ref{prop:classicalKirszbraun} enables us to extend $\iota_p$ to a map $\iota'_p : B\rightarrow\HH^n$ fixing $\partial B$ pointwise with $\Lip(\iota'_p)<C/C^{\ast}$.
We may moreover assume $\iota'_p(B)\subset B$ up to postcomposing with the closest-point projection onto $B$.
The $(j,j)$-equivariant map $J_p:\HH^n\rightarrow \HH^n$ that coincides with $\iota'_p$ on~$B$ and with the identity on $\HH^n\smallsetminus j(\Gamma_0)\cdot B$ satisfies $\Lip_x(J_p)\leq\Lip(\iota'_p)<C/C^{\ast}$ if $x\in j(\Gamma_0)\cdot B$ and $\Lip_x(J_p)=1$ otherwise.
Thus, by \eqref{eqn:supLip}, we see that the $(j,\rho)$-equivariant map $f_0 :=f\circ J_p$ satisfies $\Lip_x(f_0)\leq C^{\ast}\,\Lip(\iota'_p)<C$ if $x\in j(\Gamma_0)\cdot B$ and $\Lip_x(f_0)=\Lip_x(f)$ otherwise.
In particular, $f_0$ is $C$-Lipschitz, constant on~$B'$, extends~$\varphi$, and its (enhanced) stretch locus is contained in that of the optimal map~$f$.
Therefore $f_0$ is optimal.
\end{proof}

\subsection{Behavior in the cusps for (almost) optimal Lipschitz maps}

In this section we consider representations~$j$ that are geometrically finite but \emph{not} convex cocompact.
We show that when $\F_{K,\varphi}^{j,\rho}$ is nonempty, we can find optimal maps $f_0\in\F_{K,\varphi}^{j,\rho}$ (in the sense of Definition~\ref{def:relstretchlocus}) that ``show no bad behavior'' in the cusps.
To express this, we consider open horoballs $B_1,\dots,B_c$ of~$\HH^n$ whose images in $M:=j(\Gamma_0)\backslash\HH^n$ are disjoint and intersect the convex core in standard cusp regions (Definition~\ref{def:standardcusp}), representing all the cusps.
We take them small enough so that $K\cap j(\Gamma_0)\cdot B_i=\emptyset$ for all~$i$.
Then the following holds.

\begin{proposition}\label{prop:goodincusps}
Consider $C^{\ast}<+\infty$ such that there exists a $C^{\ast}$-Lipschitz, $(j,\rho)$-equivariant extension $f : \HH^n\rightarrow\HH^n$ of~$\varphi$.
\begin{enumerate}
  \item If $C^{\ast}\geq 1$, then we can find a $C^{\ast}$-Lipschitz, $(j,\rho)$-equivariant extension $f_0 : \HH^n\rightarrow\HH^n$ of~$\varphi$ and horoballs $B'_i\subset B_i$ such that $\Lip_{B'_i}(f_0)=\nolinebreak 0$ for all deteriorating~$B_i$ and $\Lip_{B'_i}(f_0)=1$ for all non-deteriorating~$B_i$.
  \item If $C^{\ast}<1$, then we can find a $C^{\ast}$-Lipschitz, $(j,\rho)$-equivariant extension $f_0 : \HH^n\rightarrow\HH^n$ of~$\varphi$ that converges to a point~$p_i$ in any~$B_i$ (\ie the sets $f_0(B'_i)$ converge to $\{ p_i\}$ for smaller and smaller horoballs $B'_i\subset B_i$).
  \item If $C^{\ast}<1$, then for any $\varepsilon>0$ we can find a $(C^{\ast}+\varepsilon)$-Lipschitz, $(j,\rho)$-equivariant extension $f_0 : \HH^n\rightarrow\HH^n$ of~$\varphi$ and horoballs $B'_i\subset B_i$ such that  $\Lip_{B'_i}(f_0)=0$ for all~$i$.
\end{enumerate}
Moreover, if $C^{\ast}=C_{K,\varphi}(j,\rho)$, then in (1) and~(2) we can choose~$f_0$ such that its enhanced stretch locus is contained in that of~$f$. In particular, $f_0$ is optimal if $f$ is.
\end{proposition}

By ``$B_i$ deteriorating'' we mean that $\rho$ is deteriorating in~$B_i$ in the sense of Definition~\ref{def:detinacusp}.
Recall that all $B_i$ are deteriorating when $C^{\ast}<1$ (Lemma~\ref{lem:parabdet}).
If $B_i$ is not deteriorating, then any $(j,\rho)$-equivariant map has Lipschitz constant $\geq 1$ in~$B_i$ (see Lemma~\ref{lem:disthorosphere}), hence the property $\Lip_{B'_i}(f_0)=1$ in~(1) cannot be improved.
We believe that the condition $C^{\ast}\geq 1$ could be dropped in~(1), which would then supersede both (2) and~(3) (see Appendix \ref{app:cusps}).

Note that if $f_0$ converges to a point $p_i$ in~$B_i$, then $p_i$ must be a fixed point of the group $\rho(S_i)$, where $S_i\subset\Gamma_0$ is the stabilizer of $B_i$ under~$j$.

Here is an immediate consequence of Proposition~\ref{prop:goodincusps}.(1), of Lemma~\ref{lem:proj}, and of the fact that the complement of the cusp regions in $\Conv{K}$ is compact (Fact~\ref{fact:geomfinite}).
Recall that $\F_{K,\varphi}^{j,\rho}$ is nonempty as soon as $K\neq\emptyset$ or $\rho$ is reductive (Lemma~\ref{lem:Fnonempty}).

\begin{corollary}\label{cor:Ecompact}
Suppose $\F_{K,\varphi}^{j,\rho}\neq\emptyset$.
If
\begin{itemize}
  \item $C_{K,\varphi}(j,\rho)>1$, or
  \item $C_{K,\varphi}(j,\rho)=1$ and $\rho$ is cusp-deteriorating,
\end{itemize}
then the image of the relative stretch locus $E_{K,\varphi}(j,\rho)$ in $j(\Gamma_0)\backslash\HH^n$ is compact.
\end{corollary}

Here is another consequence of Proposition~\ref{prop:goodincusps} and Lemma~\ref{lem:proj}, in the case when the group $j(\Gamma_0)$ is virtually $\Z^m$ for some $m<n$.

\begin{corollary}\label{cor:jrhononred}
If the groups $j(\Gamma_0)$ and $\rho(\Gamma_0)$ both have a unique fixed point in $\partial_{\infty}\HH^n$, then $C(j,\rho)=1$ and $\F^{j,\rho}\neq\emptyset$ and $E(j,\rho)=\emptyset$.
\end{corollary}

\begin{proof}[Proof of Corollary~\ref{cor:jrhononred}]
If $j(\Gamma_0)$ and $\rho(\Gamma_0)$ both have a unique fixed point in $\partial_{\infty}\HH^n$, then $\rho$ is not cusp-deteriorating with respect to~$j$, and so $C(j,\rho)\geq 1$ by Lemma~\ref{lem:parabdet}.
By Proposition~\ref{prop:goodincusps}.(1) we can find a $(j,\rho)$-equivariant map $f : \HH^n\rightarrow\nolinebreak\HH^n$ and a $j(\Gamma_0)$-invariant horoball $B$ of~$\HH^n$ such that $\Lip_B(f)=1$.
If we denote by $\pi_B : \HH^n\rightarrow B$ the closest-point projection, then $f\circ\nolinebreak\pi_B : \HH^n\rightarrow\nolinebreak\HH^n$ is $(j,\rho)$-equivariant and $1$-Lipschitz.
Thus $C(j,\rho)=1$ and $f\in\nolinebreak\F^{j,\rho}$.
Lemma~\ref{lem:proj} shows that $E_f$ is contained in any $j(\Gamma_0)$-invariant horoball $B'\subset B$, hence it is empty.
In particular, $E(j,\rho)=\emptyset$.
\end{proof}


\begin{proof}[Proof of Proposition~\ref{prop:goodincusps}]
For any $1\leq i\leq c$ we explain how $f|_{\Conv{K}}$ can be modified on $j(\Gamma_0)\cdot B_i\cap\Conv{K}$ to obtain a new $(j,\rho)$-equivariant Lipschitz extension $f_0 : \Conv{K}\rightarrow\HH^n$ of~$\varphi$ such that $f_0$ (precomposed as per Lemma~\ref{lem:proj} with the closest-point projection $\pi_{\Conv{K}}$ onto $\Conv{K}$) has the desired properties, namely (A)--(B)--(C)--(D) below.
More precisely, the implications will be $(\text{A})\Rightarrow (2)$, $(\text{B})\Rightarrow (3)$, and $(\text{C})\text{--}(\text{D})\Rightarrow (1)$.
We denote by~$S_i$ the stabilizer of $B_i$ in~$\Gamma_0$ under~$j$.

\smallskip
\noindent
$\bullet$ \textbf{(A) Convergence in deteriorating cusps.}
We first consider the case where $B_i$ is deteriorating and prove that there is a $C^{\ast}$-Lipschitz, $(j,\rho)$-equiva\-riant extension $f_0 : \Conv{K}\rightarrow\HH^n$ of~$\varphi$ such that $f_0$ converges to a point on $B_i\cap\Conv{K}$, agrees with~$f$ on $\Conv{K}\smallsetminus j(\Gamma_0)\cdot B_i$, and satisfies $d(f_0(p),f_0(q))\leq d(f(p),f(q))$ for all $p,q\in\Conv{K}$.
If $C^{\ast}=C_{K,\varphi}(j,\rho)$, then this last condition implies that the enhanced stretch locus of~$f_0$ is contained in that of~$f$.

It is sufficient to prove that for any $\delta>0$ there is a $C^{\ast}$-Lipschitz, $(j,\rho)$-equivariant extension $f_{\delta} : \Conv{K}\rightarrow\HH^n$ of~$\varphi$ such that $f_{\delta}$ agrees with $f$ on $\Conv{K}\smallsetminus j(\Gamma_0)\cdot B_i$, satisfies $d(f_{\delta}(p),f_{\delta}(q))\leq d(f(p),f(q))$ for all $p,q\in B_i\cap\Conv{K}$, and for some horoball $B'_i\subset B_i$, the set $f_{\delta}(B'_i\cap\Conv{K})$ is contained in the intersection of the convex hull of $f(B'_i\cap\Conv{K})$ with a ball of radius~$\delta$.
Indeed, if this is proved, then we can apply the process to $f$ and $\delta=1$ to construct a map~$f_{(1)}$, and then inductively to $f_{(i)}$ and $\delta=1/2^i$ for any $i\geq 1$ to construct a map~$f_{(i+1)}$; extracting a pointwise limit, we obtain a map~$f_0$ satisfying the required properties.

Fix $\delta>0$ and let us construct~$f_{\delta}$ as above.
Choose a generating subset $\{s_1,\dots,s_m\}$ of~$S_i$, a compact fundamental domain $\D$ of $\partial B_i\cap\Conv{K}$ for the action of~$j(S_i)$ (use Fact~\ref{fact:geomfinite}), and a point $p\in\D$.
For $t\geq 0$, the closest-point projection $\pi_t$ from $B_i$ onto the closed horoball at distance $t$ of~$\partial B_i$ inside~$B_i$ commutes with the action of $j(S_i)$.
Set $p_t:=\pi_t(p)$; by \eqref{eqn:expdistcusps}, the number $\max_{1\leq k\leq m} d(p_t,j(s_k)\cdot p_t)$ goes to~$0$ as $t\rightarrow +\infty$.
We can also find fundamental domains $\D_t$ of $\pi_t(\partial B_i)\cap\Conv{K}$, containing~$p_t$, whose diameters go to~$0$ as $t\rightarrow +\infty$.
Since $f$ is Lipschitz and $(j,\rho)$-equivariant, the diameter of $f(\D_t)$ and the function $t\mapsto\max_{1\leq k\leq m} d(f(p_t),\rho(s_k)\cdot f(p_t))$ also tend to~$0$ as $t\rightarrow +\infty$.
Let $\mathcal{F}_i\subset\HH^n$ be the fixed set of~$\rho(S_i)$ (a single point or a copy of $\HH^d$, for some $d\leq n$).
There exists $\eta>0$ such that for any $x\in\HH^n$, if $\max_{1\leq k\leq m} d(x,\rho(s_k)\cdot x)<\eta$, then $d(x,\mathcal{F}_i)<\delta/2$.
Applying this to $x=f(p_t)$, we see that for large enough~$t$ there is a point $q_t\in\mathcal{F}_i$ such that $d(f(p_t),q_t)<\delta/2$ and the diameter of $f(\D_t)$ is $<\delta/2$, which implies that the $\rho(S_i)$-invariant set $f(\Conv{K}\cap \pi_t (\partial B_i))$ is contained in the ball $\Omega:=B_{q_t}(\delta)$ of radius~$\delta$ centered at~$q_t$.
Let $\pi_{\Omega} : \HH^n\rightarrow\Omega$ be the closest-point projection onto~$\Omega$ (see Figure~\ref{fig:C}).
The $(j,\rho)$-equivariant map $f_{\delta} : \Conv{K}\rightarrow\HH^n$ that agrees with $f$ on $\Conv{K}\smallsetminus j(\Gamma_0)\cdot\pi_t(B_i)$ and with $\pi_{\Omega}\circ f$ on $\Conv{K}\cap \pi_t(B_i)$ satisfies the required properties.

\begin{figure}[h!]
\begin{center}
\labellist
\small\hair 2pt
\pinlabel{$\HH^n$} at 203 440
\pinlabel{$B_i$} at 248 440
\pinlabel{$\mathcal{D}$} at 242 343
\pinlabel{$\pi_t(B_i)$} at 315 437
\pinlabel{$\pi_t(\mathcal{D})$} at 311 365
\pinlabel{$f$} at 160 285
\pinlabel{$f(\mathcal{D})$} at 60 88
\pinlabel{$\mathcal{F}_i=\mathrm{Fix}\,(\rho(S_i))$} at 143 50
\pinlabel{$B_{q_t}(\delta)=\Omega$} at 180 120
\pinlabel{$f(\mathcal{D})$} at 365 88
\pinlabel{$\mathcal{F}_i=\mathrm{Fix}\,(\rho(S_i))$} at 448 50
\pinlabel{$B_{q_t}(\delta)=\Omega$} at 485 120
\pinlabel{$\pi_{\Omega}$} at 275 138
\endlabellist
\includegraphics[width=12cm]{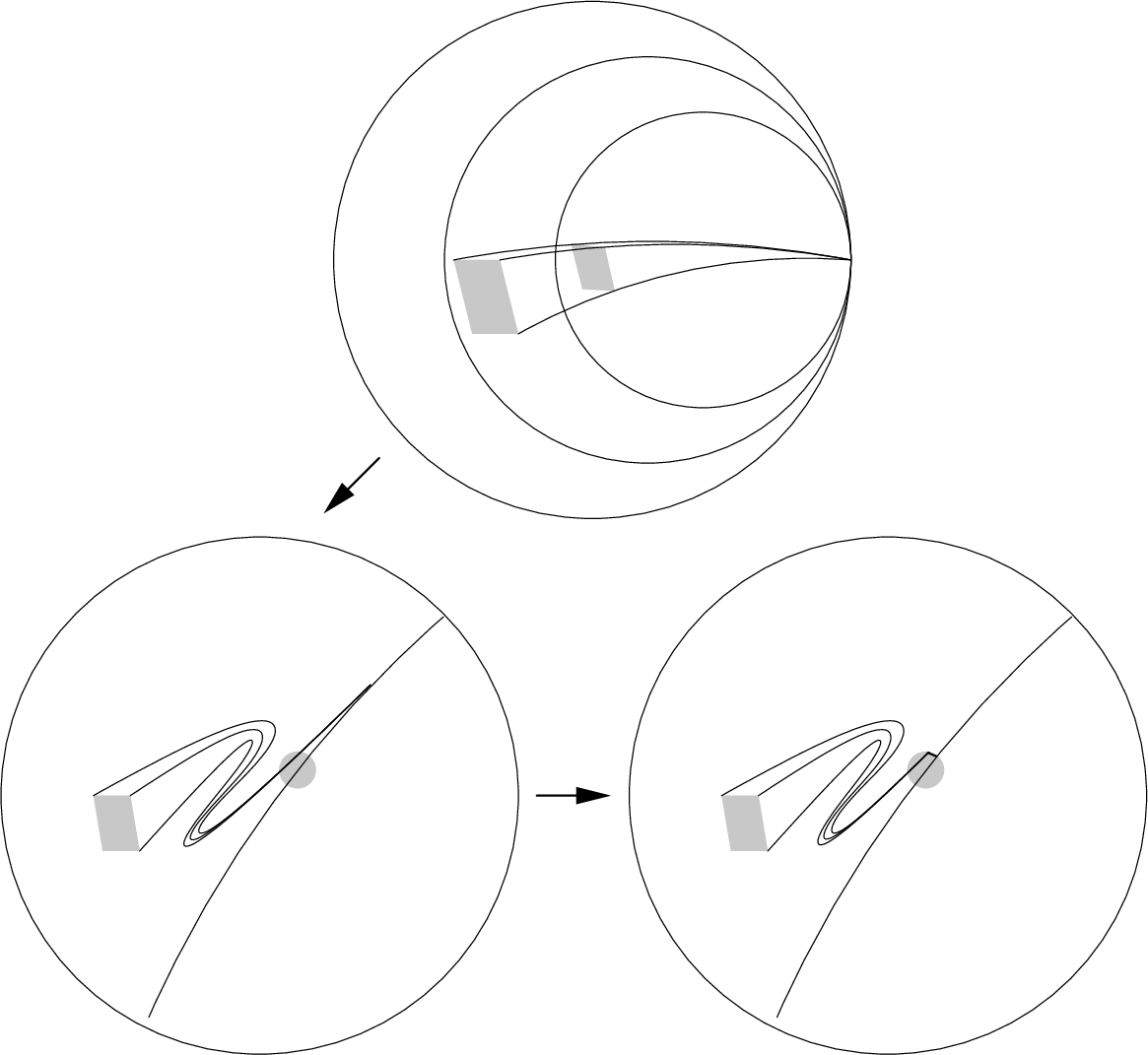}
\caption{Step~(A): Postcomposition with the closest-point projection onto the small, $\rho(S_i)$-invariant ball~$\Omega$.}
\label{fig:C}
\end{center}
\end{figure}
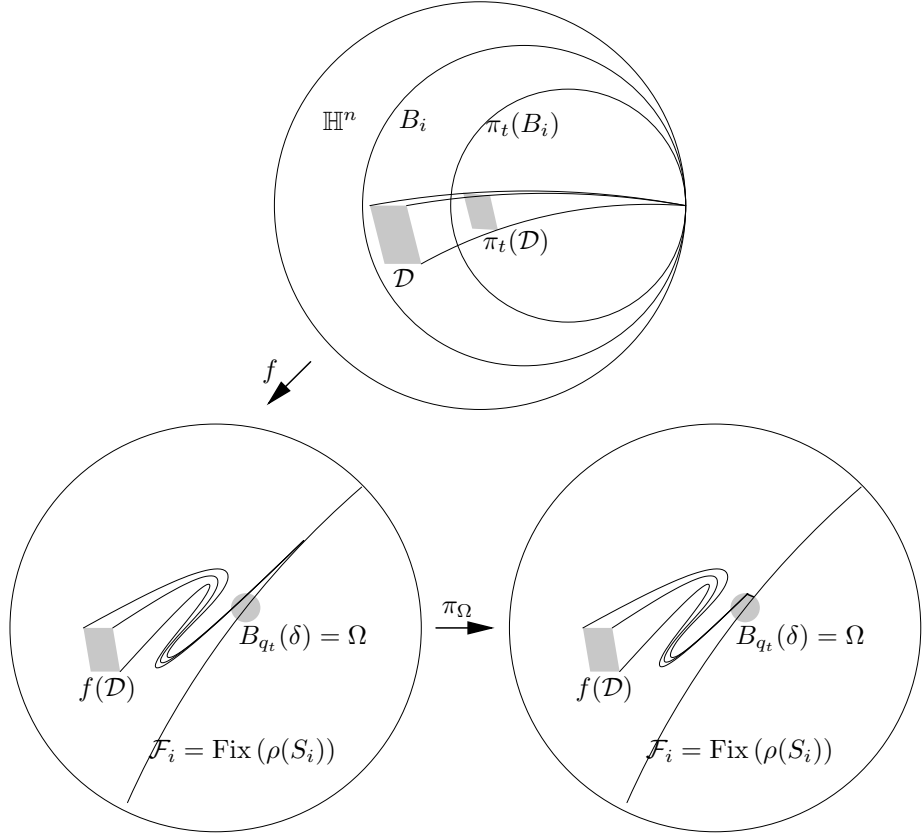

\smallskip
\noindent
$\bullet$ \textbf{(B) Constant maps with a slightly larger Lipschitz constant in deteriorating cusps.}
We still consider the case when $B_i$ is deteriorating.
For $\varepsilon>0$, we prove that there is a $(C^{\ast}+\varepsilon)$-Lipschitz, $(j,\rho)$-equivariant extension $f_0 : \HH^n\rightarrow\HH^n$ of~$\varphi$ that is constant on some horoball $B'_i\subset B_i$ and that agrees with $f$ on $\HH^n\smallsetminus j(\Gamma_0)\cdot B_i$.

Fix $\varepsilon>0$.
By~$(\text{A})$, we may assume that $f$ converges to a point~$p_i$ on~$B_i$, hence there is a horoball $B''_i\subset B_i$ such that $f(B''_i)$ is contained in the ball of diameter~$\varepsilon/2$ centered at~$p_i$.
Let $f_i:j(\Gamma_0)\cdot B''_i\rightarrow \HH^n$ be the $(j,\rho)$-equivariant map that extends the constant map~$B''_i\rightarrow \{p_i\}$, and let $\psi : \HH^n\rightarrow [0,1]$ be a $j(\Gamma_0)$-invariant, $1$-Lipschitz function equal to~$1$ on a neighborhood of $\HH^n\smallsetminus j(\Gamma_0)\cdot B''_i$ and vanishing far inside $B''_i$.
The map
$$f_0 := \psi f + (1-\psi) f_i$$
is a $(j,\rho)$-equivariant extension of~$\varphi$ that is constant on some horoball $B'_i\subset B_i$ and that agrees with $f$ on $\HH^n\smallsetminus j(\Gamma_0)\cdot B_i$.
By Lemma~\ref{lem:partofunity},
$$\Lip_p(f_0) \leq \Lip_p(f) \leq C^{\ast}$$
for all $p\in\HH^n\smallsetminus j(\Gamma_0)\cdot B''_i$, and
$$\Lip_p(f_0) \leq \Lip_p(f) + \varepsilon \leq C^{\ast} + \varepsilon$$
for all $p\in j(\Gamma_0)\cdot B''_i$, hence $f_0$ is $(C^{\ast}+\varepsilon)$-Lipschitz by \eqref{eqn:supLip}.

\smallskip
\noindent
$\bullet$ \textbf{(C) Constant maps in deteriorating cusps when $C^{\ast}\geq 1$.}
We now consider the case when $B_i$ is deteriorating and $C^{\ast}\geq 1$.
We construct a $C^{\ast}$-Lipschitz, $(j,\rho)$-equivariant extension $f_0 : \Conv{K}\rightarrow\HH^n$ of~$\varphi$ that is constant on $B'_i\cap\Conv{K}$ for some horoball $B'_i\subset B_i$ and agrees with $f$ on $\Conv{K}\smallsetminus j(\Gamma_0)\cdot B_i$.
We also prove that if $C^{\ast}=C_{K,\varphi}(j,\rho)$, then the enhanced stretch locus of~$f_0$ (hence of $f_0\circ\pi_{\Conv{K}}$ by Lemma~\ref{lem:proj}) is included in that of~$f$.

By~$(\text{A})$, we may assume that $f$ converges to a point~$p_i$ on~$B_i$.
Let $B'_i$ be a horoball strictly contained in~$B_i$.
Since the set $\partial B_i\cap\Conv{K}$ is compact modulo $j(S_i)$ (Fact~\ref{fact:geomfinite}), its image under~$f$ lies within bounded distance from~$p_i$.
Therefore, if $B'_i$ is far enough from $\partial B_i$, then the map from $(\Conv{K}\smallsetminus j(\Gamma_0)\cdot B_i)\cup (B'_i\cap\Conv{K})$ to~$\HH^n$ that agrees with $f$ on $\Conv{K}\smallsetminus j(\Gamma_0)\cdot B_i$ and that is constant equal to $p_i$ on $B'_i\cap\Conv{K}$ is $C^{\ast}$-Lipschitz.
By Proposition~\ref{prop:amenableKirszbraun}, we can extend it to a $C^{\ast}$-Lipschitz, $(j|_{S_i},\rho|_{S_i})$-equivariant map from $(\Conv{K}\smallsetminus j(\Gamma_0)\cdot B_i)\cup (B_i\cap\Conv{K})$ to~$\HH^n$.
Finally we extend this map to a $(j,\rho)$-equivariant map $f^{(1)} : \Conv{K}\rightarrow\HH^n$.
Then $f^{(1)}$ is $C^{\ast}$-Lipschitz, agrees with $f$ on $\Conv{K}\smallsetminus j(\Gamma_0)\cdot B_i$, and is constant on $B'_i\cap\Conv{K}$.

Suppose that $C^{\ast}=C_{K,\varphi}(j,\rho)$.
Then $\Lip(f^{(1)})=C^{\ast}$ (and no smaller).
The stretch locus (and maximally stretched segments) of~$f^{(1)}$ are included in those of~$f$, except possibly between $\partial B_i$ and $\partial B'_i$.
To deal with this issue, we consider two horoballs $B'''_i\subsetneq\nolinebreak B''_i$ strictly contained in~$B'_i$ and, similarly, construct a $C^{\ast}$-Lipschitz, $(j,\rho)$-equivariant map $f^{(2)} : \Conv{K}\rightarrow\HH^n$ that agrees with $f$ on $\Conv{K}\smallsetminus j(\Gamma_0)\cdot B''_i$ and is constant on $B'''_i\cap\Conv{K}$.
The $(j,\rho)$-equivariant map
$$f_0 := \frac{1}{2} f^{(1)} + \frac{1}{2} f^{(2)}$$
still agrees with $f$ on $\Conv{K}\smallsetminus j(\Gamma_0)\cdot B_i$ and is constant on $B'''_i\cap\Conv{K}$.
By Lemma~\ref{lem:baryLipschitz}, its (enhanced) stretch locus is included in that of~$f$.

\smallskip
\noindent
$\bullet$ \textbf{(D) Lipschitz constant~$1$ in non-deteriorating cusps.}
We now consider the case when $B_i$ is not deteriorating; in particular, $C^{\ast}\geq 1$ by Lemma~\ref{lem:parabdet}.
We construct a $C^{\ast}$-Lipschitz, $(j,\rho)$-equivariant extension\linebreak $f_0 :\nolinebreak \Conv{K}\rightarrow\HH^n$ of~$\varphi$ such that $\Lip_{B'_i\cap\Conv{K}}(f_0)=1$ for some horoball $B'_i\subset B_i$ and $f_0$ agrees with $f$ on $\Conv{K}\smallsetminus j(\Gamma_0)\cdot B_i$.
We also prove that if $C^{\ast}=C_{K,\varphi}(j,\rho)$ then the enhanced stretch locus of $f_0$ (hence of $f_0\circ\pi_{\Conv{K}}$) is included in that of~$f$.

We assume $C^{\ast}>1$ (otherwise we may take $f_0=f$).
It is sufficient to construct a $1$-Lipschitz, $(j|_{S_i},\rho|_{S_i})$-equivariant map $f_i : B'_i\cap\Conv{K}\rightarrow\HH^n$, for some horoball $B'_i\subset B_i$, such that the $(j|_{S_i},\rho|_{S_i})$-equivariant map
$$f^{(1)} : \big(\Conv{K}\smallsetminus j(\Gamma_0)\cdot B_i\big) \cup (B'_i\cap\Conv{K}) \longrightarrow \HH^n$$
that agrees with $f$ on $\Conv{K}\smallsetminus j(\Gamma_0)\cdot B_i$ and with $f_i$ on $B'_i\cap\Conv{K}$ satisfies $\Lip(f^{(1)})\leq C^{\ast}$.
Indeed, we can then extend~$f^{(1)}$ to a $C^{\ast}$-Lipschitz, $(j,\rho)$-equivariant map $\Conv{K}\rightarrow\HH^n$ using Proposition~\ref{prop:amenableKirszbraun}, as in step~(C).
Proceeding with two other horoballs $B'''_i\subsetneq B''_i$ to get a map~$f^{(2)}$ and averaging as in step~(C), we obtain a map~$f_0$ with the required properties.

To construct~$f_i$, we use explicit coordinates: in the upper half-space model $\R^{n-1}\times\R_+^{\ast}$ of~$\HH^n$, we may assume (using Remark~\ref{rem:CLipconj}) that $j(S_i)$ and~$\rho(S_i)$ both fix the point at infinity, that the horosphere $\partial B_i$ is $\R^{n-1}\times\{ 1\} $, and that $f$ fixes the point $(\underline{0},1)\in\R^{n-1}\times\R_+^{\ast}$.
Let $W_i$ be the orthogonal projection to $\R^{n-1}$ of $\Conv{K}\subset\R^n$; the group $j(S_i)$ preserves and acts cocompactly on any set $W_i\times\{ b\}$ with $b\in\R_+^{\ast}$ (use Fact~\ref{fact:geomfinite}).
The restriction of $f$ to $W_i\times\{ 1\}$ may be written as
$$f(\underline{a},1) = \big(f'(\underline{a}),f''(\underline{a})\big)$$
for all $\underline{a}\in W_i$, where $f' : W_i\rightarrow\R^{n-1}$ and $f'' : W_i\rightarrow\R_+^{\ast}$.
Let
$$L := \max\big(1,\Lip(f')\big),$$
where $\Lip(f')$ is measured with respect to the Euclidean metric $\mathrm{d}s_{\R^{n-1}}$ of~$\R^{n-1}$, and let $B'_i\subset B_i$ be a horoball $\R^{n-1}\times [b_0,+\infty)$, with large $b_0>L$ to be adjusted later.
The map $f_i : W_i\times [b_0,+\infty)\rightarrow\HH^n$ given by
$$f_i(\underline{a},b) := \big(f'(\underline{a}), Lb\big)$$
is $(j|_{S_i},\rho|_{S_i})$-equivariant, since $f$ is and the groups $j(S_i)$ and~$\rho(S_i)$ both preserve the horospheres $\R^{n-1}\times\{ b\}$ (see Figure~\ref{fig:D}).
Moreover, $f_i$ is $1$-Lipschitz, since by construction it preserves the directions of $\R^{n-1}$ (horizontal) and $\R_+^{\ast}$ (vertical) and it stretches by a factor $\leq 1$ in the $\R^{n-1}$-direction and $1$ in the $\R_+^{\ast}$-direction, for the hyperbolic metric
$$\mathrm{d}s^2 = \frac{\mathrm{d}s_{\R^{n-1}}^2 + \mathrm{d}b^2}{b^2}.$$
Let $\D_i\subset W_i\times\{ 1\}$ be a compact fundamental domain for the action of $j(\Gamma_0)$ on $\partial B_i\cap\Conv{K}$, and let $R:= \max_{x\in\D_i} d((\underline{0},1),x)>\nolinebreak 0$.

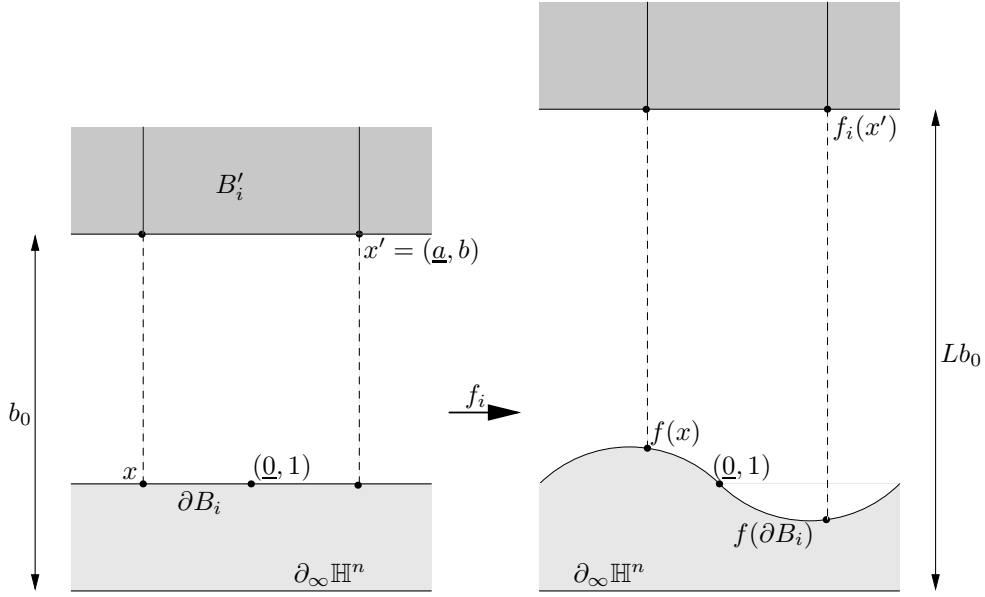
\begin{figure}[h!]
\begin{center}
\labellist
\small\hair 2pt
\pinlabel{$b_0$} at -5 90
\pinlabel{$x$} at 50 60
\pinlabel{$(\underline{0},1)$} at 126 63
\pinlabel{$(\underline{0},1)$} at 356 63
\pinlabel{$\partial B_i$} at 85 45
\pinlabel{$B'_i$} at 100 205
\pinlabel{$\partial_\infty\HH^n$} at 150 9
\pinlabel{$\partial_\infty\HH^n$} at 290 9
\pinlabel{$x'=(\underline{a},b)$} at 196 172
\pinlabel{$f_i$} at 223 99
\pinlabel{$f_i(x')$} at 418 234
\pinlabel{$f(x)$} at 323 80
\pinlabel{$L b_0$} at 466 120
\pinlabel{$f(\partial B_i)$} at 372 29
\endlabellist
\includegraphics[width=12cm]{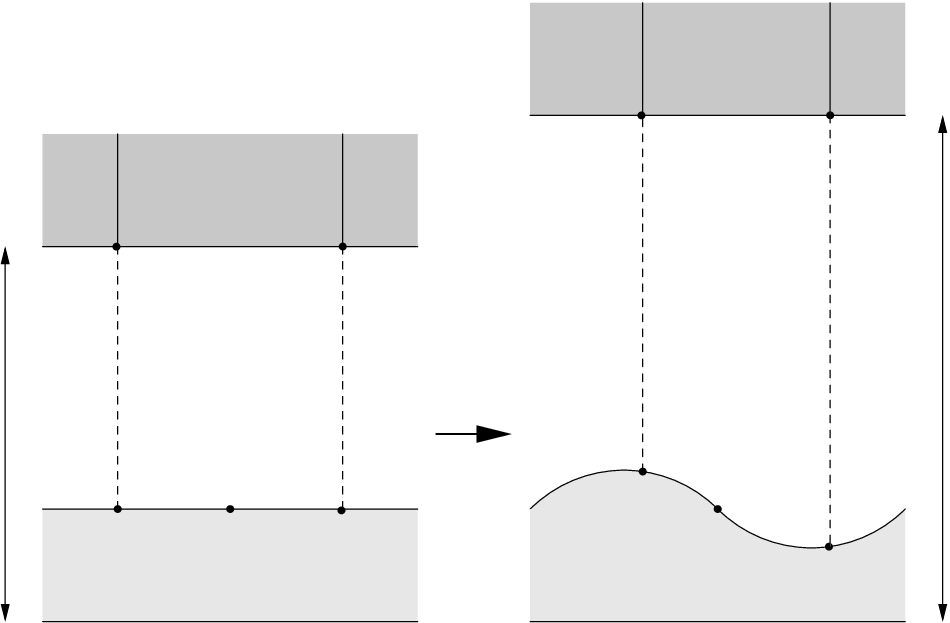}
\caption{Definition of a $1$-Lipschitz extension $f_i$ in the cusp in Step~(D).}
\label{fig:D}
\end{center}
\end{figure}

Recall (see \eqref{eqn:distinH2}) that for any $(\underline{a},b)\in\R^{n-1}\times\R_+^{\ast}$,
$$d\big((\underline{0},1),(\underline{a},b)\big) = \arccosh \bigg(\frac{\Vert\underline{a}\Vert^2+b^2+1}{2b}\bigg).$$
In particular,
$$\left| d\big((\underline{0},1),(\underline{a},b)\big) - \log \left(\frac{\Vert\underline{a}\Vert^2}{b}+b\right)\right| \leq 1$$
as soon as $b$ exceeds some constant, which we shall assume from now on.
Therefore, for any $x\in\D_i$ and $x'=(\underline{a},b)\in B'_i\cap\Conv{K}$,
\begin{equation}\label{eqn:logbound1}
d(x,x') \geq \log\left ( \frac{\Vert\underline{a}\Vert^2}{b}+b \right ) -1- R,
\end{equation}
and (using the expression of~$f_i$, and the fact that $f$ fixes $(\underline{0},1)$ and $f'$ is $L$-Lipschitz)
\begin{eqnarray}\label{eqn:logbound2}
d\big(f(x),f_i(x')\big) & \leq & d\big(f(x),f(\underline{0},1)\big) + d\big(f(\underline{0},1),f_i(x')\big)\\
& \leq & \log\left ( \frac{\Vert L\underline{a}\Vert^2}{Lb}+Lb \right ) + 1 + C^{\ast} R \nonumber\\
& = & \log\left ( \frac{\Vert\underline{a}\Vert^2}{b}+b \right ) + \log(L)+1+ C^{\ast} R. \nonumber
\end{eqnarray}
In particular, if $B'_i$ is far enough from $\partial B_i$ (\ie $b_0>0$ is large enough), then the log term dominates in \eqref{eqn:logbound1} and \eqref{eqn:logbound2} (where $b\geq b_0$), and so
$$d(f(x),f_i(x')) \leq C^{\ast}\,d(x,x')$$
for all $x\in\D_i$ and $x'\in B'_i\cap\Conv{K}$ (recall $C^{\ast}>1$).
Therefore, the $(j|_{S_i},\rho|_{S_i})$-equivariant map
$$f^{(1)} : \big(\Conv{K}\smallsetminus j(\Gamma_0)\cdot B_i\big) \cup (B'_i\cap\Conv{K}) \longrightarrow \HH^n$$
that agrees with $f$ on $\Conv{K}\smallsetminus j(\Gamma_0)\cdot B_i$ and with $f_i$ on $B'_i\cap\Conv{K}$ satisfies $\Lip(f^{(1)})\leq C^{\ast}$.
This completes the proof of~(D), hence of Proposition~\ref{prop:goodincusps}.
\end{proof}

\section{An optimized, equivariant Kirszbraun--Valentine theorem}\label{sec:Kirszbraunopt}

The goal of this section is to prove the following analogue and extension of Proposition~\ref{prop:amenableKirszbraun}.
We refer to Definitions \ref{def:stretchlocusgen} and~\ref{def:CE} for the notion of stretch locus.
We denote by $\Lambda_{j(\Gamma_0)}\subset\partial_{\infty}\HH^n$ the limit set of~$j(\Gamma_0)$.
Recall that for geometrically finite~$j$, the sets $\F_{K,\varphi}^{j,\rho}$ and $E_{K,\varphi}(j,\rho)$ of Definition~\ref{def:CE} are nonempty as soon as $K$ is nonempty or $\rho$ is reductive, except possibly if $C_{K,\varphi}(j,\rho)=1$ and $\rho$ is not cusp-deteriorating (Lemma~\ref{lem:Fnonempty} and Corollary~\ref{cor:Enonempty}).

\begin{theorem}\label{thm:Kirszbraunopt}
Let $\Gamma_0$ be a discrete group, $(j,\rho)\in\Hom(\Gamma_0,G)^2$ a pair of representations of $\Gamma_0$ in~$G$ with $j$ geometrically finite, $K$ a $j(\Gamma_0)$-invariant subset of~$\HH^n$ whose image in $j(\Gamma_0)\backslash\HH^n$ is compact, and $\varphi : K\rightarrow\HH^n$ a $(j,\rho)$-equivariant Lipschitz map.
Suppose that $\F_{K,\varphi}^{j,\rho}$ and $E_{K,\varphi}(j,\rho)$ are nonempty.
Set
$$C_0 :=
  \left \{ \begin{array}{lll}  \Lip(\varphi) & \text{if }K\neq\emptyset, \\ C'(j,\rho) & \text{if }K=\emptyset,
  \end{array} \right .$$
where $C'(j,\rho)$ is given by \eqref{eqn:defClambda}.

$\bullet$ If $C_0\geq 1$, then there exists a $(j,\rho)$-equi\-variant extension $f : \HH^n\rightarrow\HH^n$ of~$\varphi$ with Lipschitz constant~$C_0$, optimal in the sense of Definition~\ref{def:relstretchlocus}, whose stretch locus is the union of 
the stretch locus $E_{\varphi}$ of~$\varphi$ (defined to be empty if $K=\emptyset$) and of a closed set~$E'$ such~that:
\begin{enumerate}
  \item[--] if $C_0>1$, then $E'$ is equal to the closure of a geodesic lamination of $\HH^n\smallsetminus K$ that is maximally stretched by~$f$, and $j(\Gamma_0)\backslash E'$ is compact;
  \item[--] if $C_0=1$, then $E'$ is a union of convex sets, each isometrically preserved by~$f$, with extremal points only in the union of $K$ and of the limit set $\Lambda_{j(\Gamma_0)}\subset\partial_\infty\HH^n$; moreover, $j(\Gamma_0)\backslash E'$ is compact provided that $\rho$ is cusp-deteriorating.
\end{enumerate}
In particular, in these two cases $C_{K,\varphi}(j,\rho)=C_0$ and $E_{K,\varphi}(j,\rho)=E_{\varphi}\cup E'$.

$\bullet$ If $C_0<1$ then $C_{K,\varphi}(j,\rho)<1$.
\end{theorem}

By a geodesic lamination of $\HH^n\smallsetminus K$ we mean a nonempty disjoint union $\LL$ of geodesic intervals of $\HH^n\smallsetminus K$ (called leaves), with no endpoint in $\HH^n\smallsetminus K$, such that $\LL$ is closed for the $\mathcal{C}^1$ topology (\ie any Hausdorff limit of segments of leaves of~$\LL$ is a segment of a leaf of~$\LL$).
By ``maximally stretched by~$f$'' we mean that $f$ multiplies all distances by~$C_0$ on any leaf of the lamination.

For $\Gamma_0=\{ 1\}$ and $K\neq\emptyset$, Theorem~\ref{thm:Kirszbraunopt} improves the classical Kirszbraun--Valentine theorem (Proposition~\ref{prop:classicalKirszbraun}) by adding a control on the local Lipschitz constant of the extension (through a description of its stretch locus).

We shall give a proof of Theorem~\ref{thm:Kirszbraunopt} in Sections \ref{subsec:EforC>1} to~\ref{subsec:proofElamin}, and then a proof of Theorem~\ref{thm:Kirszbraunequiv}, as well as Corollary~\ref{cor:CC'} under the extra assumption $E(j,\rho)\neq\emptyset$, in Section~\ref{subsec:coroElamin} (this extra assumption will be removed in Section~\ref{subsec:proofadmred}).
For $K=\emptyset$, we shall finally examine how far the stretch locus $E(j,\rho)$ goes in the cusps in Section~\ref{subsec:injradius}.

\subsection{The stretch locus when $C_{K,\varphi}(j,\rho)>1$}\label{subsec:EforC>1}

We now fix $(j,\rho)$ and $(K,\varphi)$ as in Theorem~\ref{thm:Kirszbraunopt}.
To simplify notation, we set
\begin{align}\label{eqn:CEE}
C \,:=\, C_{K,\varphi}(j,\rho) & \quad\geq C_0,\nonumber\\
E \,:=\, E_{K,\varphi}(j,\rho) & \quad\subset\HH^n,\\
\widetilde{E} \,:=\, \widetilde{E}_{K,\varphi}(j,\rho) & \quad\subset \HH^n\times\HH^n\nonumber
\end{align}
(see Definition~\ref{def:CE}).
Recall that $E\subset\Conv{K}$ and $\widetilde{E}\subset\Conv{K}\times\Conv{K}$ as soon as $C>0$ (Lemma~\ref{lem:proj}).
In order to prove Theorem~\ref{thm:Kirszbraunopt}, we first establish the following.

\begin{lemma}\label{lem:MaxStretchedLam}
In the setting of Theorem~\ref{thm:Kirszbraunopt}, if $C>1$, then $E\smallsetminus K$ is a geodesic lamination of $\HH^n\smallsetminus K$, and any $f\in\F_{K,\varphi}^{j,\rho}$ multiplies arc length by~$C$ on the leaves of this lamination.
\end{lemma}

Note that the projection of~$E$ to $j(\Gamma_0)\backslash\HH^n$ is compact (even in the presence of cusps) by Corollary~\ref{cor:Ecompact}.

The proof is a refinement of the classical Kirszbraun--Valentine theorem (Proposition~\ref{prop:classicalKirszbraun}).

\begin{proof}[Proof of Lemma~\ref{lem:MaxStretchedLam}]
By Lemma~\ref{lem:optimalmap}, there exists an optimal $f\in\F_{K,\varphi}^{j,\rho}$, whose stretch locus is exactly~$E$.
Fix $p\in E\smallsetminus K$ and consider a small closed ball $B\subset\HH^n\smallsetminus K$, of radius $r>0$, centered at~$p$, which projects injectively to $j(\Gamma_0)\backslash\HH^n$.
By Lemma~\ref{lem:moteur} with $(\mathbf{K},{\boldsymbol\varphi}):=(\partial B,f|_{\partial B})$, we can find points $q\in\HH^n$ and $k_1,k_2\in \partial B$ such that $C_q:=\max_{k\in\partial B} d(q,f(k))/d(p,k)$ is minimal and such that $d(q,f(k_i))=C_q\,d(p,k_i)$ for $i\in\{1,2\}$ and $\widehat{k_1pk_2}\leq \widehat{f(k_1) q f(k_2)} \neq 0$.
By minimality of~$C_q$, we have $C_q\leq C_{f(p)}\leq\Lip(f)=\nolinebreak C$.

We claim that $C_q=C$.
Indeed, suppose by contradiction that $C_q<C$.
Let $B'\subset B$ be another ball centered at~$p$, of radius $r'\in (0,r)$ small enough so that $C_qr/(r-r')\leq C$.
Then, as in the proof of Lemma~\ref{lem:optimallocallycst}, the extension of $f|_{\partial B}$ to $\partial B\cup B'$ which is constant equal to $q$ on~$B'$ is still $C$-Lipschitz.
Using Proposition~\ref{prop:classicalKirszbraun} (see also Remark~\ref{rem:classicalKirszbraun<1}), we extend it to a $C$-Lipschitz map $f' : B\to\HH^n$.
Working by equivariance, we obtain an element of $\F_{K,\varphi}^{j,\rho}$, agreeing with~$f$ on $\HH^n\smallsetminus j(\Gamma_0)\cdot B$, which is constant on a neighborhood of~$p$, contradicting $p\in E$.
Thus $C_q=C$.

We claim that $k_1,k_2\in\partial B$ are diametrically opposite, that the geodesic segment $[k_1,k_2]$ is maximally stretched by~$f$, that $f(p)=q$, and that the set $\mathbf{K}'$ of points $k\in\partial B$ such that $d(q,f(k))=C\,d(p,k)$ is reduced to $\{k_1,k_2\}$.
Indeed, since $f$ is $C$-Lipschitz, we have
$$\frac{d(f(k_1),f(k_2))}{d(k_1, k_2)} \leq C = \frac{d(q,f(k_1))}{d(p,k_1)} = \frac{d(q,f(k_2))}{d(p,k_2)},$$
and since $C>1$, Toponogov's theorem \cite[Lem.\,II.1.13]{bh99} implies that $d(f(k_1),f(k_2))=C\,d(k_1, k_2)$ and $\widehat{k_1 p k_2}=\pi$ (the case $\widehat{k_1 p k_2}=0$ is ruled out since $k_1,k_2\in\partial B$ are distinct).
In particular, the geodesic segment $[k_1,k_2]$ has midpoint~$p$ and is maximally stretched by~$f$, and $f(p)=q$ by Remark~\ref{rem:pathlength}.(1).
For any $k\in\mathbf{K}'$ we have $\widehat{k_1 p k}\leq\widehat{f(k_1) q f(k)}\neq 0$ or $\widehat{k p k_2}\leq\widehat{f(k) q f(k_2)}\neq 0$ (since $\widehat{k_1 p k} + \widehat{k p k_2} = \widehat{f(k_1) q f(k)} + \widehat{f(k) q f(k_2)} = \pi$), and so the above reasoning shows that $k\in\{k_1,k_2\}$.

Taking $B$ arbitrarily small, we see that there are exactly two germs of geodesic rays through $p$ in~$E$ that are maximally stretched by~$f$, and they are diametrically opposite.
Let $\ell$ be the largest geodesic interval in $E\smallsetminus K$ through~$p$ that is maximally stretched by~$f$, corresponding to these two germs.
We claim that $\ell$ terminates on the union of $K$ and of the limit set $\Lambda_{j(\Gamma_0)}$.
Indeed, any infinite end of~$\ell$ terminates on $\Lambda_{j(\Gamma_0)}$ since $\ell\subset E_{\varphi,K}(j,\rho)\subset\Conv{K}$ by Lemma~\ref{lem:proj}; moreover, $\ell$ cannot terminate on a point $p'\in E\smallsetminus K$ since there are exactly two germs of geodesic rays through $p'$ in~$E$ that are maximally stretched by~$f$ and they are diametrically opposite.
This implies that $E\smallsetminus K$ is a geodesic lamination of $\HH^n\smallsetminus K$, maximally stretched by~$f$.
\end{proof}

It is possible for a point $p\in K$ to belong to the stretch locus~$E$ without being an endpoint of a leaf of $E$, or even without belonging to any closed $C$-stretched segment of $f$ at all (for instance if $x\mapsto \Lip_x(\varphi)$ immediately drops away from~$p$).
However, the following holds.

\begin{lemma}\label{lem:partsofK}
In the setting of Theorem~\ref{thm:Kirszbraunopt}, if $C>1$, then any $p\in E\cap K$ lies either in the closure of $E\smallsetminus K$ or in $E'_{\varphi}:=\{k\in K~|~\Lip_k(\varphi)=C\}$.
\end{lemma}

(Note that we have not yet proved that the inequality $C_0:=\Lip(\varphi)\leq C$ is an equality; this will be done in Proposition~\ref{prop:C_0=C}, and will imply by definition that $E'_{\varphi}$ is the stretch locus $E_{\varphi}$ of~$\varphi$.)

\begin{proof}
Suppose $C>1$ and consider $p\in E\cap K$.
Assuming $p\notin E'_{\varphi}$, we shall prove that $p$ lies in the closure of $E\smallsetminus K$.
Since $p\notin E'_{\varphi}$, there is a small closed ball $B$ of radius $r>0$ centered at~$p$, projecting injectively to $j(\Gamma_0)\backslash\HH^n$, such that $\Lip_B(\varphi)<\nolinebreak C$ (Lemma~\ref{lem:localLip}).
By Proposition~\ref{prop:classicalKirszbraun} and Remark~\ref{rem:classicalKirszbraun<1}, the map $\varphi|_{B\cap K}$ admits an extension $\overline{\varphi}$ to~$B$ with $\Lip_B(\overline{\varphi})<C$.
Consider an optimal $f\in\mathcal{F}_{K,\varphi}^{j,\rho}$, whose stretch locus is exactly~$E$ (Lemma~\ref{lem:optimalmap}), and let
$$C^{\ast}:=\sup_{q\in (K\cap B)\cup\partial B} \frac{d(\varphi(p),f(q))}{d(p,q)}\leq C\, .$$

We claim that $C^{\ast}=C$.
Indeed, suppose by contradiction that $C^{\ast}<C$.
For any ball $B'\subset B$ centered at~$p$, with radius $r'>0$ small enough, the map $f' : K\cup\partial B\cup B'\rightarrow\HH^n$ that coincides with $f$ on $K\cup\partial B$ and with $\varphi$ on~$B'$ is still $C$-Lipschitz.
Indeed, for any $x\in K\cup \partial B$ and $y\in B'$, if $x$ lies in the interior of $B$ then $\Lip_{\{x,y\}}(f')=\Lip_{\{x,y\}}(\overline{\varphi})<C$, and otherwise the triangle inequality gives
$$\frac{d(f'(x),f'(y))}{d(x,y)} \leq \frac{C^* r + \Lip(\overline{\varphi})r'}{r-r'}$$
as in \eqref{eqn:cutup}, which is $\leq C$ if $r'$ is small enough.
Therefore $f'$ admits a $C$-Lipschitz extension to~$B$ by Remark~\ref{rem:classicalKirszbraun<1}; working by equivariance, we obtain an element $f'\in\F_{K,\varphi}^{j,\rho}$ agreeing with~$f$ on $\HH^n\smallsetminus j(\Gamma_0)\cdot B$, such that $\Lip_p(f')\leq \Lip(\overline{\varphi})<C$, contradicting $p\in E$.
Thus $C^{\ast}=C$.

If the upper bound $C^{\ast}$ is approached by a sequence $(q_i)_{i\in\N}$ of $(K\cap B)\cup\partial B$ with $q_i \rightarrow p$, then $q_i\in K$ for all large enough~$i$ and $p\in E'_{\varphi}$, contradicting the assumption.
Therefore $(q_i)_{i\in\N}$ has an accumulation point $q\neq p$.
The geodesic segment $[p,q]$ is maximally stretched by~$f$ (Remark~\ref{rem:pathlength}.(1)), hence $[p,q]\smallsetminus K\subset E\smallsetminus K$ and any accumulation point of $[p,q]\cap K$ lies in~$E'_{\varphi}$.
Since $p\notin E'_{\varphi}$, we obtain that $p$ lies in the closure of $E\smallsetminus K$.
\end{proof}

\subsection{The stretch locus when $C_{K,\varphi}(j,\rho)=1$}\label{subsec:EforC=1}

We define $C, E, \widetilde{E}$ as in \eqref{eqn:CEE}.
When $C=1$, the stretch locus $E$ may contain pieces larger than lines that are isometrically preserved by all elements of~$\F_{K,\varphi}^{j,\rho}$.
Here is the counterpart of Lemma~\ref{lem:MaxStretchedLam} in this case.

\begin{lemma}\label{lem:1StretchedLam}
In the setting of Theorem~\ref{thm:Kirszbraunopt}, if $C=1$, then there is a canonical family $(\Omega_p)_{p\in E}$ of closed convex subsets of~$\HH^n$, of varying dimensions, with the following properties:
\begin{enumerate}[\rm (i)]
  \item any $p\in E$ lies in the interior of the corresponding~$\Omega_p$ (where we see $\Omega_p$ as a subset of its own affine span --- in particular, a point is equal to its own interior);\label{i}
  \item the interiors of $\Omega_p$ and $\Omega_q$ are either equal or disjoint for $p,q\in E$;\label{ii}
  \item the restriction to $\Omega_p$ of any $f\in\F_{K,\varphi}^{j,\rho}$ is an isometry;\label{iii}
  \item whenever two points $x\neq y$ in~$\HH^n$ satisfy $d(f(x),f(y))=d(x,y)$ for some (hence any) optimal $f\in\F_{K,\varphi}^{j,\rho}$ (Definition~\ref{def:relstretchlocus}), the geodesic segment $[x,y]$ (called a \emph{$1$-stretched} segment) is contained in some~$\Omega_p$;\label{iv}
  \item all extremal points of~$\Omega_p$ lie in the union of~$K$ and of the limit set~$\Lambda_{j(\Gamma_0)}$ of $j(\Gamma_0)$;\label{v}
  \item the intersection of~$\Omega_p$ with any supporting hyperplane is an~$\Omega_q$;\label{vi}
  \item $E=\bigcup_{p\in E\smallsetminus K} \Omega_p \cup E'_{\varphi}$ where $E'_{\varphi}=\{k\in K~|~\Lip_k(\varphi)=1\}$.\label{vii}
\end{enumerate}
\end{lemma}

Properties \eqref{i}--\eqref{vii} are reminiscent of the \emph{stratification} of the boundary of a convex object, with $1$-stretched segments of~$E$ replacing segments contained in the boundary of the convex object; we shall call the interiors of the sets~$\Omega_p$ \emph{strata} of~$E$, and the sets~$\Omega_p$ \emph{closed strata}.

\begin{remark}\label{rem:nonconvexE}
In dimension $n\geq 3$, the connected components of $E=E_{K,\varphi}(j,\rho)$ can be nonconvex.
Indeed, take $n=3$.
Let $\Gamma_0$ be the fundamental group of a closed surface, let $j\in\Hom(\Gamma_0,G)$ be geometrically finite, obtained by bending slightly a geodesic copy of~$\HH^2$ inside~$\HH^3$ along some geodesic lamination~$\LL$, and let $\rho\in\Hom(\Gamma_0,G)$ be obtained by bending even a little more along the same lamination~$\LL$.
Then $E$ is the first bent copy of~$\HH^2$, which can be nonconvex (though connected).
\end{remark}

\begin{proof}[Proof of Lemma~\ref{lem:1StretchedLam}]
Consider an optimal $f\in \F_{K,\varphi}^{j,\rho}$ (Lemma~\ref{lem:optimalmap}): by definition, the stretch locus $E_f=E$ is minimal, and so is the \emph{enhanced} stretch locus $\widetilde{E}_f=\widetilde{E}$.

For $p\in E$, let $W_p\subset \mathbb{P}(T_p\HH^n)$ be the set of directions of $1$-stretched segments containing $p$ in their interior.
This set is independent of~$f$ because $\widetilde{E}_f=\widetilde{E}$.
(It is for this independence property that we use the \emph{enhanced} stretch locus $\widetilde{E}$ here.)
Since $f$ is $1$-Lipschitz, the convex hull of any two such $1$-stretched segments is isometrically preserved by~$f$.
Therefore the set $W_p$ is a full projective subspace (possibly empty), equal to the projectivization of a vector subspace $V_p\subset T_p\HH^n$.
Moreover, there is a neighborhood of $p$ in $\exp_p(V_p)$ on which $f$ coincides with an isometric embedding $\psi_p : \exp_p(V_p)\rightarrow\nolinebreak\HH^n$, and the closed set
$$\Omega_p := \big\{x \in \exp_p (V_p)~|~f(x)=\psi_p(x)\big\} \subset E$$
is convex and contains $p$ in its interior.
The isometric embedding $\psi_p$ may depend on~$f$, but the set~$\Omega_p$ depends only on the data $(j,\rho, K,\varphi)$, because so does the enhanced stretch locus~$\widetilde{E}$.
We shall denote by $d_p\geq 0$ the dimension of~$V_p$.

Conditions \eqref{i} and~\eqref{iii} are satisfied by construction, and so is \eqref{iv} by taking $p$ in the interior of the given $1$-stretched segment $[x,y]$.
(Note that $\Omega_p$ may contain points of~$K$ in its interior, even when $p\notin K$.)

For any $x$ belonging to the interior of $\Omega_p$ in $\exp_p(V_p)$, we have $V_x=T_x\Omega_p$.
Indeed, $V_x\supset T_x\Omega_p$ is clear since $f|_{\mathrm{Int}(\Omega_p)}$ is an isometry; and if $x$ were in the interior of any $1$-stretched segment $s$ \emph{not} contained in $\exp_p(V_p)$, then $f$ would be isometric on the $(d_p+1)$-dimensional convex hull of $\Omega_p\cup s$, which contains $p$ in its interior: this would violate the definition of~$V_p$.
From $V_x=T_x \Omega_p$ we deduce in particular $\psi_x=\psi_p$ and $\Omega_x=\Omega_p$.

It follows that given $q\in E$, if the interiors of $\Omega_p$ and $\Omega_q$ intersect at a point $x$, then 
$\psi_p=\psi_x=\psi_q$ and $\Omega_p=\Omega_q$: thus \eqref{ii} holds.

Any $1$-stretched segment $s=[x,y]$ with an interior point $q$ in $\Omega_p$ is contained in $\Omega_p$.
Indeed, $f$ must preserve all angles $\widehat{xqp'}$ and $\widehat{yqp'}$ for $p'\in \Omega_p$, hence $f$ is an isometry on the convex hull of $s\cup \Omega_p$, which contains $p$ in its interior: therefore $s\subset \exp_p(V_p)$ by definition of $d_p$ and $s\subset \Omega_p$ by definition of $\Omega_p$.

In $\exp_p(V_p)$, the intersection of $\Omega_p$ with any supporting hyperplane $\Pi$ at a point of $\partial \Omega_p$ is the closure of an open convex subset $Q$ of some~$\HH^d$, where $0\leq d<d_p$ (with $\HH^0$ being a point).
Consider a point $q\in Q$.
By the previous paragraph, any open $1$-stretched segment through $q$ is in $\Omega_p$, hence in~$\Pi$, hence in~$Q$ (see Figure~\ref{fig:E}).
Therefore, $d=d_q$ and $\psi_q=\psi_p|_{\exp_q(V_q)}$.
It follows that $\Omega_q$ is the closure of~$Q$.
This gives~(\ref{vi}).

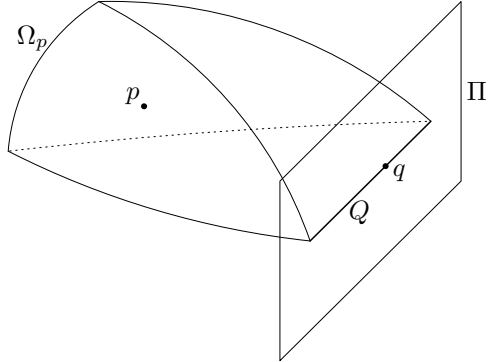
\begin{figure}[h!]
\begin{center}
\labellist
\small\hair 2pt
\pinlabel{$p$} at 75 160
\pinlabel{$q$} at 234 114
\pinlabel{$\Omega_p$} at 15 195
\pinlabel{$\Pi$} at 280 163
\pinlabel{$Q$} at 210 90
\endlabellist
\includegraphics[width=6cm]{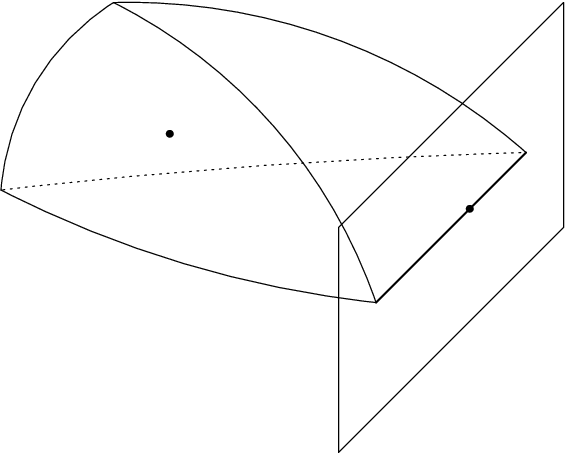}
\caption{A $3$-dimensional convex stratum $\Omega_p$ with a supporting plane~$\Pi$.}
\label{fig:E}
\end{center}
\end{figure}

By \eqref{vi}, extremal points $q\in \Omega_p$ in$\HH^n$ satisfy $\Omega_q=\{q\}$, hence $q\in K$ by Lemma~\ref{lem:1stretched} below; this gives~\eqref{v}.

We obtain \eqref{vii} by replacing $C$ by~$1$ and the call to Proposition~\ref{prop:classicalKirszbraun} by Proposition~\ref{prop:KirszbraunC<1} in the proof of Lemma~\ref{lem:partsofK}.
\end{proof}

\begin{lemma}\label{lem:1stretched}
Any $p\in E\smallsetminus K$ is contained in the interior of a $1$-stretched segment.
\end{lemma}

\begin{proof}
Consider an optimal $f\in\F_{K,\varphi}^{j,\rho}$ (Lemma~\ref{lem:optimalmap}), whose stretch locus is exactly~$E$.
Fix $p\in E\smallsetminus K$ and consider a small closed ball $B\subset\HH^n\smallsetminus K$, centered at~$p$, which projects injectively to $j(\Gamma_0)\backslash\HH^n$.
By Lemma~\ref{lem:moteur} with $(\mathbf{K},{\boldsymbol\varphi}):=(\partial B,f|_{\partial B})$, we can find a point $q\in\HH^n$ such that $C_q:=\max_{k\in\partial B} d(q,f(k))/d(p,k)$ is minimal.
In particular, $C_q\leq\nolinebreak  C_{f(p)}\leq\nolinebreak \Lip(f)=\nolinebreak 1$.

In fact $C_q=1$: otherwise, using Proposition~\ref{prop:classicalKirszbraun} as in the proofs of Lemmas \ref{lem:optimallocallycst} and~\ref{lem:MaxStretchedLam}, we could construct an element of $\mathcal{F}_{K,\varphi}^{j,\rho}$ that would be locally constant near~$p$, contradicting the fact that $p\in E$.

In particular, there cannot exist $k_1,k_2\in\partial B$ such that $d(q,f(k_i))=C_q\,d(p,k_i)$ for $i\in\{1,2\}$ and $0\leq\widehat{k_1 p k_2}<\widehat{f(k_1) q f(k_2)}\leq\pi$: otherwise we would have $d(f(k_1),f(k_2))>d(k_1,k_2)$ by basic trigonometry, contradicting $\Lip(f)=1$.

Therefore, Lemma~\ref{lem:moteur} implies the existence of a probability measure $\nu$ on $\mathbf{K}':=\{k\in\partial B\,|\,d(q,f(k))=d(p,k)\}$ such that $q$ belongs to the convex hull of the support of $f_*\nu$ and such that $\widehat{k_1 p k_2}=\widehat{f(k_1) q f(k_2)}$ for $(\nu\times\nu)$-almost all $(k_1,k_2)\in\mathbf{K}'\times\mathbf{K}'$.
This means that the continuous map $f : \HH^n\rightarrow \HH^n$ is an isometry on the support of~$\nu$ (which is contained in $\partial B$), hence has a unique $1$-Lipschitz extension (the isometric one, with which $f$ must agree) to the convex hull $X$ of this support.
Since $f(X)$ contains~$q$, we have $f(p)=q$ and there is a $1$-stretched segment through~$p$.
This completes the proof of Lemma~\ref{lem:1stretched}, hence of Lemma~\ref{lem:1StretchedLam}.
\end{proof}

\begin{remark}\label{rem:higherlamin}
When $K=\emptyset$, the closed strata~$\Omega_p$ of the lowest dimension (say $k\geq 1$) are always complete copies of~$\HH^k$: otherwise, they would by Lemma~\ref{lem:1StretchedLam}.\eqref{vi} admit supporting planes giving rise to closed strata of lower dimension.
In particular, the union of these closed strata is a $k$-dimensional geodesic lamination in the sense of Section~\ref{subsec:introstretchlocus}. In dimension $n=2$, we must have $k=1$ (unless $j$ and~$\rho$ are conjugate); this implies that the stretch locus is the union of a geodesic lamination and (possibly) certain connected components of its complement.
\end{remark}

\subsection{Proof of Theorem~\ref{thm:Kirszbraunopt}}\label{subsec:proofElamin}

Theorem~\ref{thm:Kirszbraunopt} is an immediate consequence of Corollary~\ref{cor:Ecompact}, of Lemmas \ref{lem:MaxStretchedLam} and~\ref{lem:1StretchedLam}, and of the following proposition.

\begin{proposition}\label{prop:C_0=C}
In the setting of Theorem~\ref{thm:Kirszbraunopt}, if $C_{K,\varphi}(j,\rho)\geq 1$, then $C_{K,\varphi}(j,\rho)=C_0$.
In particular, $C_0<1$ implies $C_{K,\varphi}(j,\rho)<1$.
\end{proposition}

\begin{proof}
Recall that $C_0\leq C:=C_{K,\varphi}(j,\rho)$ (see \eqref{eqn:CEE}).
Let us prove the converse inequality when $C\geq 1$.

In the particular situation where $C:=C_{K,\varphi}(j,\rho)=1$ and $\rho$ is not cusp-deteriorating, we have $C_0\geq 1$ by Lemma~\ref{lem:parabdet}.
We now assume that we are not in this particular situation.
By Corollary~\ref{cor:Enonempty}, the set $E:=E_{K,\varphi}(j,\rho)$ is nonempty.

Suppose $E\subset K$.
By Lemma~\ref{lem:partsofK} or Lemma~\ref{lem:1StretchedLam}.\eqref{vii}, for any $p\in E\subset K$ we have $\Lip_p(\varphi)=C$, and so $C_0=\Lip(\varphi)\geq C$. 

Suppose $E\not\subset K$.
By Lemma~\ref{lem:MaxStretchedLam} or Lemma~\ref{lem:1StretchedLam}.(iii)--(v), any point $p\in\nolinebreak E\smallsetminus\nolinebreak K$ belongs, either to the convex hull of a subset of~$K$ on which $f$ multiplies all distances by~$C$, or to a maximally stretched geodesic ray with an endpoint in $\Lambda_{j(\Gamma_0)}$ whose image in the quotient $j(\Gamma_0)\backslash\HH^n$ is bounded (Proposition~\ref{prop:goodincusps}.(1)).
In the first case, $C_0=\Lip(\varphi)\geq C$ since $f$ coincides with $\varphi$ on~$K$.
In the second case, Lemma~\ref{lem:ClambdaCLipLipf} yields $C_0\geq C$.
\end{proof}

Theorem~\ref{thm:lamin} is contained in Lemmas \ref{lem:Fnonempty} and~\ref{lem:optimalmap}, Corollary~\ref{cor:Enonempty}, Theorem~\ref{thm:Kirszbraunopt}, and Remark~\ref{rem:higherlamin}.

\subsection{Some easy consequences of Theorem~\ref{thm:Kirszbraunopt}}\label{subsec:coroElamin}

We first prove Theorem~\ref{thm:Kirszbraunequiv}, which concerns the case where $K$ is nonempty and possibly noncompact modulo~$j(\Gamma_0)$.

\begin{proof}[Proof of Theorem~\ref{thm:Kirszbraunequiv}]
Let $K\neq\emptyset$ be a $j(\Gamma_0)$-invariant subset of~$\HH^n$.
We can always extend~$\varphi$ to the closure $\overline{K}$ of~$K$ by continuity, with the same Lipschitz constant~$C_0$.
Suppose the image of $\overline{K}$ in $j(\Gamma_0)\backslash\HH^n$ is compact.
If $C_0\geq 1$, then Theorem~\ref{thm:Kirszbraunequiv} is contained in Theorem~\ref{thm:Kirszbraunopt}.
If $C_0<1$, then $C(j,\rho)<1$ by Theorem~\ref{thm:Kirszbraunopt}, which implies Theorem~\ref{thm:Kirszbraunequiv} since $\F^{j,\rho}\neq\emptyset$ (Lemma~\ref{lem:Fnonempty}).
Now, for $C_0\geq 1$, consider the general case where the image of $\overline{K}$ in $j(\Gamma_0)\backslash\HH^n$ is not necessarily compact.
Let $(\mathscr{C}_k)_{k\in\N}$ be a sequence of $j(\Gamma_0)$-invariant subsets of~$\HH^n$ whose images in $j(\Gamma_0)\backslash\HH^n$ are compact, with $\mathscr{C}_k\subset\mathscr{C}_{k+1}$ and $\bigcup_{k\in\N} \mathscr{C}_k=\HH^n$.
For any~$k$, Theorem~\ref{thm:Kirszbraunopt} gives a $(j,\rho)$-equivariant extension $f_k : \HH^n\rightarrow\HH^n$ of $\varphi|_{\overline{K}\cap\mathscr{C}_k}$ with $\Lip(f_k)=C_0$, and we conclude using the Arzel\`a--Ascoli theorem as in Remark~\ref{rem:Knoncompact}.
\end{proof}

We then turn to Corollary~\ref{cor:CC'}, for which $K$ is empty.
Corollary~\ref{cor:CC'} for $E(j,\rho)\neq\emptyset$ is an immediate consequence of Lemma~\ref{lem:ClambdaCLipLipf} and Theorem~\ref{thm:Kirszbraunopt}.
For $E(j,\rho)=\emptyset$, we shall prove Corollary~\ref{cor:CC'} in Section~\ref{subsec:proofadmred} (Lemma~\ref{lem:ClambdaCLipred}), using a \emph{Cartan projection} $\mu$ of~$G$.
Recall however from Corollary~\ref{cor:Enonempty} that $E(j,\rho)=\emptyset$ may only happen if $C(j,\rho)=1$ and $\rho$ is not cusp-deteriorating; in that case, a direct proof of Corollary~\ref{cor:CC'} could also be obtained by considering a sequence of closed geodesics of $j(\Gamma_0)\backslash\HH^n$ that spend more and more time in a cusp whose stabilizer contains an element $\gamma\in\Gamma_0$ with both $j(\gamma)$ and~$\rho(\gamma)$ parabolic.

In dimension $n=2$, for torsion-free $\Gamma_0$, let $C'_s(j,\rho)$ be the supremum of $\lambda(\rho(\gamma))/\lambda(j(\gamma))$ over all elements $\gamma\in\Gamma_0$ corresponding to \emph{simple} closed curves in the hyperbolic surface $j(\Gamma_0)\backslash\HH^2$.
(As for $C'(j,\rho)$, we define $C'_s(j,\rho)$ to be $C(j,\rho)$ in the degenerate case when $j(\Gamma_0)\backslash\HH^2$ has no essential closed curve.)
Then $C'_s(j,\rho)\leq C'(j,\rho)\leq C(j,\rho)$ (see \eqref{eqn:ClambdaCLip}).
Here is another consequence of Theorem~\ref{thm:Kirszbraunopt}.

\begin{lemma}[$n=2$, torsion-free $\Gamma_0$]
Suppose $E(j,\rho)\neq\emptyset$.
If $C'_s(j,\rho)<1$, then~$C'(j,\rho)<\nolinebreak 1$.
\end{lemma}

\begin{proof}
If $C(j,\rho)\geq 1$, then
\begin{equation}\label{eqn:Clambdasimple}
C'_s(j,\rho) = C'(j,\rho) = C(j,\rho).
\end{equation}
Indeed, by Theorem~\ref{thm:Kirszbraunopt} and Remark~\ref{rem:higherlamin}, the image of $E(j,\rho)$ in $j(\Gamma_0)\backslash\HH^n$ contains a nonempty geodesic lamination $\LL$ with compact image.
If $\LL$ contains a simple closed curve, then \eqref{eqn:Clambdasimple} is clear; otherwise we can argue as in the proof of Lemma~\ref{lem:ClambdaCLipLipf}, but with the axis of $j(\gamma)$ projecting to a \emph{simple} closed geodesic nearly carried by~$\LL$.
\end{proof}

Note that if $E(j,\rho)=\emptyset$, then it is possible to have $C'_s(j,\rho)<1=C'(j,\rho)=C(j,\rho)$: see Section~\ref{ex:nondeteriorating}.

\subsection{The recurrent set of maximally stretched laminations}

For empty~$K$, Theorem~\ref{thm:Kirszbraunopt} states that the stretch locus $E(j,\rho)$ contains a maximally stretched $k$-dimensional geodesic lamination with compact (nonempty) image in $j(\Gamma_0)\backslash\HH^n$ as soon as $\F_{K,\varphi}^{j,\rho}$ and $E_{K,\varphi}(j,\rho)$ are nonempty and $C'(j,\rho)\geq\nolinebreak 1$.
Conversely, we make the following observation.

\begin{lemma}\label{lem:maxstretchedlamin}
Let $\Gamma_0$ be a discrete group and $(j,\rho)\in\Hom(\Gamma_0,G)^2$ a pair of representations with $j$ geometrically finite.
Let $\LL$ be a $j(\Gamma_0)$-invariant $k$-dimensional geodesic lamination of~$\HH^n$ with a compact image in $j(\Gamma_0)\backslash\HH^n$.
If $\LL$ is maximally stretched by some $(j,\rho)$-equivariant Lipschitz map $f : \HH^n\rightarrow\nolinebreak\HH^n$, then the recurrent set of~$\LL$ is contained in the stretch locus $E(j,\rho)$.
\end{lemma}

By \emph{recurrent set of~$\LL$}, we mean the projection to $j(\Gamma_0)\backslash\HH^n$ of the recurrent set of the geodesic flow $(\Phi_t)_{t\in\R}$ restricted to vectors tangent to~$\LL$.
By compactness, this recurrent set is nonempty.

Recall that in this setting we have $C'(j,\rho) = C(j,\rho) = \Lip(f)$ by Lemma~\ref{lem:ClambdaCLipLipf}.

\begin{proof}
Set $C:=C(j,\rho)=\Lip(f)$.
In order to prove that the recurrent set of~$\LL$ is contained in $E(j,\rho)$, it is sufficient to prove that for any geodesic line $(p_t)_{t\in\R}$ of~$\HH^n$ contained in~$\LL$ and projecting to a geodesic which is recurrent in $j(\Gamma_0)\backslash\HH^n$, we have
$$d(f'(p_0),f'(p_1)) = C = C\,d(p_0,p_1)$$
for all $f'\in\F^{j,\rho}$ (hence $[p_0,p_1]\subset E(j,\rho)$).
Fix $\varepsilon>0$.
By recurrence, we can find $\gamma\in\Gamma_0$ and $t>1$ with $[j(\gamma)\cdot p_0,j(\gamma)\cdot p_1]$ arbitrarily close to $[p_t,p_{t+1}]$, so that by the closing lemma (Lemma~\ref{lem:closinglemma}) the translation axis of $j(\gamma)$ passes within~$\varepsilon$ of the four points $p_0, p_1, p_t, p_{t+1}$ and the axis of $\rho(\gamma)$ within $C\varepsilon$ of their four images under~$f$.
Choose $q_0,q_1 \in\HH^n$ within $\varepsilon$ of $p_0,p_1$, respectively, on the axis of $j(\gamma)$. 
For any $f'\in\F^{j,\rho}$,
\begin{eqnarray*}
d(f'(q_0),f'(q_1)) & \geq & d\big(f'(q_0),\rho(\gamma)\cdot f'(q_0)\big) - d\big(\rho(\gamma)\cdot f'(q_0),f'(q_1)\big)\\
& \geq & \lambda(\rho(\gamma)) - d\big(f'(j(\gamma)\cdot q_0),f'(q_1)\big)\\
& \geq & C \cdot (\lambda(j(\gamma))-4\varepsilon) - \Lip(f') \cdot \big(\lambda(j(\gamma))-d(q_0,q_1)\big)\\
& = & C \cdot (d(q_0,q_1)- 4\varepsilon)
\end{eqnarray*}
since $\Lip(f')=C$.
But $p_0, p_1$ are $\varepsilon$-close to $q_0, q_1$; therefore 
\begin{eqnarray*}
d(f'(p_0),f'(p_1)) & \geq & d(f'(q_0), f'(q_1))-d(f'(p_0),f'(q_0))-d(f'(p_1),f'(q_1))\\
& \geq & C \cdot (d(q_0,q_1)- 4\varepsilon) - 2\,\Lip(f')\,\varepsilon \\
& \geq & C \cdot (d(p_0,p_1)- d(p_0,q_0)-d(p_1,q_1)-4\varepsilon) - 2C\varepsilon \\
&\geq&  C \cdot (1- 8\varepsilon).
\end{eqnarray*}
This holds for any $\varepsilon>0$, hence $d(f'(p_0),f'(p_1))=C$.
\end{proof}

\subsection{How far the stretch locus goes into the cusps}\label{subsec:injradius}

Suppose that $j$ is geometrically finite but \emph{not} convex cocompact.
For empty~$K$, we can control how far the stretch locus $E_{K,\varphi}(j,\rho)=E(j,\rho)$ goes into the cusps.

\begin{proposition}\label{prop:margulis}
There is a nondecreasing function $\Psi : (1,+\infty)\rightarrow\R_+^{\ast}$ such that for any discrete group~$\Gamma_0$, any pair $(j,\rho)\in\Hom(\Gamma_0,G)^2$ with $j$ geometrically finite and $C(j,\rho)> 1$, and any $x\in E(j,\rho)$ whose image in $j(\Gamma_0)\backslash\HH^n$ belongs to a standard cusp region of the convex core (Definition~\ref{def:standardcusp}), the cusp thickness at~$x$ is $\geq\Psi(C(j,\rho))$.
\end{proposition}

Here we use the following terminology, where $N\subset\HH^n$ is the preimage of the convex core of $j(\Gamma_0)\backslash\HH^n$.

\begin{definition} \label{def:thickness}
Let $B$ be a horoball of~$\HH^n$ such that $B\cap N$ projects to a standard cusp region of $j(\Gamma_0)\backslash\HH^n$.
Given a point $x\in B$, let $\partial B_x$ denote the horosphere concentric to~$B$ running through~$x$.
The \emph{cusp thickness} of $j(\Gamma_0)\backslash\HH^n$ at~$x$ is the Euclidean diameter in $j(\mathrm{Stab}(B))\backslash \partial B_x$ of the orthogonal projection of $N$ to $\partial B_x$.
\end{definition}

By \emph{Euclidean diameter} we mean the diameter for the metric induced by the intrinsic, Euclidean metric of $\partial B_x$; it varies exponentially~with~the~depth of $x$ in the cusp region (see \eqref{eqn:horomu} for conversion to a hyperbolic distance).
Note that the orthogonal projection of~$N$ is convex inside the Euclidean space $\partial B_x\simeq\R^{n-1}$.

We believe that an analogue of Proposition~\ref{prop:margulis} should also hold for $C(j,\rho)<1$, see Appendix \ref{app:cusps}.
It is false for $C(j,\rho)=1$ (take $j=\rho$).

Proposition~\ref{prop:margulis} will be a consequence of the following lemma, which applies to $C=C(j,\rho)$ and to leaves $\ell_0,\ell_1$ of the geodesic lamination $E(j,\rho)$.
It implies that any two leaves of $E(j,\rho)$ coming close to each other must be nearly parallel.
This is always the behavior of \emph{simple} closed curves and geodesic laminations in dimension $n=2$. 

\begin{lemma}\label{lem:Esimple}
For any $C>1$, there exists $\delta_0>0$ with the following property.
Let $\ell_0,\ell_1$ be disjoint geodesic lines of~$\HH^n$.
Suppose there exists a $C$-Lipschitz map $f : \ell_0\cup\ell_1\rightarrow\HH^n$ multiplying all distances by~$C$ on~$\ell_0$ and on~$\ell_1$.
If $\ell_0$ and~$\ell_1$ pass within $\delta\leq\delta_0$ of each other near some point $x\in\HH^n$, then they stay within distance~$1$ of each other on a length $\geq |\log \delta|-10$ before and after~$x$.
\end{lemma}

(The constant $10$ is of course far from optimal.)

\begin{proof}
We can restrict to dimension $n=3$ because the geodesic span of two lines has dimension at most~$3$.
Fix $C>1$ and let $\ell_0$, $\ell_1$, and~$f$ be as above.
The images $\ell'_0:=f(\ell_0)$ and $\ell'_1:=f(\ell_1)$ are geodesic lines of~$\HH^3$.
Fix orientations on $\ell_0,\ell_1,\ell'_0,\ell'_1$ so that $f$ is orientation-preserving.
For $i\in\{ 0,1\}$, let $x_i$ be a point of~$\ell_i$ closest to~$\ell_{1-i}$, so that the geodesic segment $[x_0,x_1]$ is orthogonal to both $\ell_0$ and~$\ell_1$; let $\sigma$ be the rotational symmetry of~$\HH^3$ around the line $(x_0,x_1)$.
Similarly, let $x'_i \in \ell'_i$ be closest to~$\ell'_{1-i}$, so that the segment $[x'_0,x'_1]$ is orthogonal to $\ell'_0$ and~$\ell'_1$; let $\sigma'$ be the rotational symmetry of~$\HH^3$ around $(x'_0,x'_1)$.
Up to replacing~$f$ by
$$\frac{1}{2}\,f + \frac{1}{2}\,\sigma'\circ f\circ\sigma,$$
which is still $C$-Lipschitz (Lemma~\ref{lem:baryLipschitz}), which preserves the orientations of $\ell_0,\ell_1,\ell'_0,\ell'_1$, and which multiplies all distances by~$C$ on~$\ell_0$ and on~$\ell_1$, we may assume that $f(x_0)=x'_0$ and $f(x_1)=x'_1$.
Let $\eta$ (\resp $\eta'$) be the length of $[x_0,x_1]$ (\resp of $[x'_0,x'_1]$), and $\theta$ (\resp $\theta'$) the angle between the positive directions of $\ell_0$ and~$\ell_1$ (\resp of $\ell'_0$ and~$\ell'_1$), measured by projecting orthogonally to a plane perpendicular to $[x_0,x_1]$ (\resp to $[x'_0,x'_1]$ if $\eta'>0$).
We claim that

\emph{$(\ast)$ there exists $\Delta_0>0$, depending only on $C$, such that if $\eta\leq\Delta_0$, then}
$$\min\{ \theta,\pi-\theta\} \leq 1.005\,\eta.$$
Indeed, for $i\in\{ 0,1\}$, let $t_i>0$ be the linear coordinate of a point $p_i\in\ell_i$, measured from~$x_i$ with the chosen orientation.
By \eqref{eqn:mixedmu},
\begin{equation}\label{eqn:prismdistance}
\cosh d(p_0,p_1)=\cosh \eta \cdot \cosh t_0 \cosh t_1 - \cos \theta \cdot \sinh t_0 \sinh t_1.
\end{equation}
Therefore, using $\cosh t\sim e^t/2$, we obtain that for $t_0,t_1\rightarrow +\infty$,
$$d(p_0,p_1) = t_0 + t_1 + \log\left(\frac{\cosh\eta - \cos\theta}{2}\right) + o(1).$$
Similarly, since $f$ stretches $\ell_0$ and~$\ell_1$ by a factor of~$C$ and $f(x_i)=x'_i$,
$$d\big(f(p_0),f(p_1)\big) = Ct_0 + Ct_1 + \log\left(\frac{\cosh\eta' - \cos\theta'}{2}\right) + o(1).$$
Since $f$ is $C$-Lipschitz, we must have
$$\log\left(\frac{\cosh\eta' - \cos\theta'}{2}\right) \leq C \log\left(\frac{\cosh\eta - \cos\theta}{2}\right).$$
Note that this must also hold if we replace $\theta,\theta'$ with their complements to~$\pi$, because we can reverse the orientations of $\ell_1$ and~$\ell'_1$. 
We thus obtain
$$\frac{\cosh\eta' \pm \cos\theta'}{2} \leq \left(\frac{\cosh\eta \pm \cos\theta}{2}\right )^C.$$
Since $\cosh \eta'\geq 1$, adding the two inequalities yields
\begin{equation}\label{eqn:coshcos}
\left(\frac{\cosh\eta + \cos\theta}{2}\right)^C + \left(\frac{\cosh\eta - \cos\theta}{2}\right)^C \geq 1.
\end{equation}
Inequality \eqref{eqn:coshcos} means that $(\cos\theta,\cosh\eta)$ lies in $\R^2$ outside of a $\frac{\pi}{4}$-rotated and $\sqrt{2}$-scaled copy of the unit ball of~$\R^2$ for the $\mathrm{L}^C$-norm.
(See also \eqref{eqn:complexdistance} for an interpretation of \eqref{eqn:coshcos} in terms of the cross-ratio of the endpoints of $\ell_0, \ell_1$.)
Since $\cos\theta\in [-1,1]$ and $\cosh\eta\geq 1$, we obtain that $(\cos\theta,\cosh\eta)$ lies above some concave curve through the points $(-1,1)$ and $(1,1)$, with respective slopes $1$ and $-1$ at these points (recall that $C>1$).
In particular, if $\cosh\eta$ is very close to~$1$, then $|\cos\theta|$ must be about as close (or closer) to~$1$ (see Figure~\ref{fig:F}). 
We obtain $(\ast)$ by using the Taylor expansions of $\cosh$ and $\cos$ (of course $1.005$ can be replaced by any number $>1$).

\begin{figure}[h!]
\begin{center}
\labellist
\small\hair 2pt
\pinlabel{$x_0$} at 3 32
\pinlabel{$x_1$} at 3 73
\pinlabel{$\ell_0$} at 150 19
\pinlabel{$\ell_1$} at 190 85
\pinlabel{$p_0$} at 305 -2
\pinlabel{$p_1$} at 384 104
\pinlabel{$\theta$} at 195 30
\pinlabel{$\eta$} at -1 52
\endlabellist
\includegraphics[width=10cm]{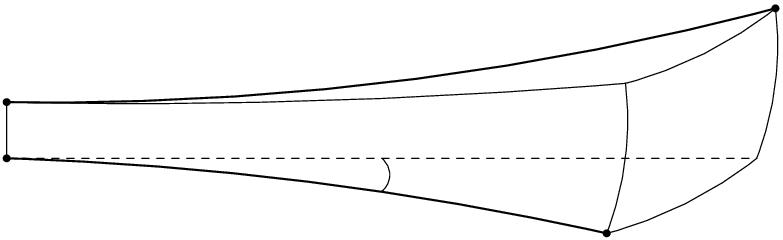}
\caption{At distance $d(p_0,x_0)=t_0<|\log \eta|$ from a point of closest approach of the two lines $\ell_0, \ell_1$, their angular drift $\approx \theta e^{t_0}$ cannot much exceed their height drift $\approx \eta e^{t_0}$.}
\label{fig:F}
\end{center}
\end{figure}
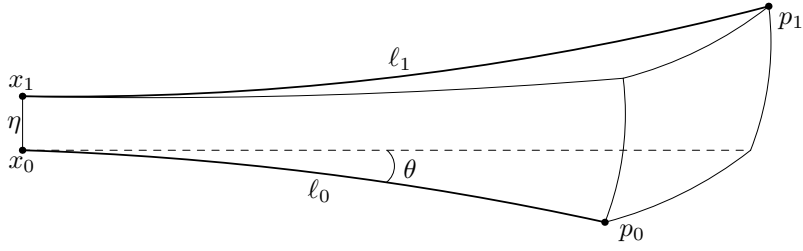

To deduce the lemma from $(\ast)$, we can minimize \eqref{eqn:prismdistance} in $t_1$ alone to find
$$\sinh^2 d(p_0, \ell_1)= \sinh^2 \eta + (\sinh^2\eta+\sin^2\theta) \sinh^2 t_0\,.$$
By $(\ast)$, for small enough~$\eta$ we have 
$$\eta^2 \leq \sinh^2 \eta + \sin^2 \theta \leq 2.004\,\eta^2,$$
hence
\begin{equation}\label{eqn:pointtoline}
\sinh^2 \eta + \eta^2 \sinh^2 t_0 \leq \sinh^2 d(p_0, \ell_1) \leq \sinh^2 \eta + 2.004\,\eta^2 \sinh^2 t_0\,.
\end{equation}
If $t_0\in [0,|\log\eta|]$ (for small~$\eta$), we have
$$\frac{\sinh^2 \eta}{\eta^2 e^{2t_0}} + 2.004\,\frac{\sinh^2 t_0}{e^{2t_0}} \leq 1.005+\frac{2.004}4 = 1.506,$$ 
hence, on the upper side of \eqref{eqn:pointtoline},
\begin{equation}\label{eqn:upperpointtoline}
\sinh^2\eta + 2.004\,\eta^2\sinh^2 t_0 \leq 1.506\,\eta^2 e^{2 t_0} \leq \sinh^2\big(\sqrt{1.506}\,\eta e^{t_0}\big)
\end{equation}
by multiplying by $\eta^2 e^{2t_0}$ and using the inequality $x\leq\sinh x$ for $x\in\R_+$.
Note that $\sqrt{1.506}\leq 1.23$.
On the other hand, using again $\sinh x\geq x$,
$$\frac{\sinh^2 \eta}{\eta^2 e^{2t_0}} + \frac{\sinh^2 t_0}{e^{2t_0}} \geq \frac{\cosh^2 t_0}{e^{2t_0}} \geq \frac{1}{4}\,,$$
hence, on the lower side of \eqref{eqn:pointtoline},
\begin{equation} \label{eqn:lowerpointtoline}
\sinh^2\eta + \eta^2\sinh^2 t_0 \geq (\eta e^{t_0} \sinh 0.48 )^2 \geq \sinh^2 (0.48\,\eta\,e^{t_0}) 
\end{equation}
by multiplying by $\eta^2 e^{2t_0}$ and using the inequality $\sinh^2 0.48<1/4$ and the convexity of $\sinh$ (recall $\eta e^{t_0}\leq 1$).
From \eqref{eqn:pointtoline}, \eqref{eqn:upperpointtoline}, and \eqref{eqn:lowerpointtoline}, it follows that for $\eta$ smaller than some $\delta_0\in (0,1)$ (depending only on~$C$),
$$0.48\,\eta\,e^{|t_0|} \leq d(p_0,\ell_1)\leq 1.23\,\eta\,e^{|t_0|}$$
as soon as $|t_0| \leq |\log \eta|$. This two-sided exponential bound means that $p_0\mapsto \log d(p_0, \ell_1)$ is essentially a $1$-Lipschitz function of $p_0$ (plus a bounded correction), which easily implies the lemma.
\end{proof}

\begin{proof}[Proof of Proposition~\ref{prop:margulis}]
Let $\Gamma_0$ be a discrete group, $(j,\rho)\in\Hom(\Gamma_0,G)^2$ a pair of representations with $j$ geometrically finite, and $B$ an open horoball of~$\HH^n$ whose image in $j(\Gamma_0)\backslash\HH^n$ intersects the convex core in a standard cusp region.
The stabilizer $S\subset\Gamma_0$ of $B$ under~$j$ has a normal subgroup $S'$ isomorphic to~$\Z^m$ for some $0<m<n$, and of index $\leq\nu(n)$ in~$S$, where $\nu(n)<+\infty$ depends only on~$n$ (see Section~\ref{subsec:geo-finiteness}).
In the upper half-space model $\R^{n-1}\times\R_+^{\ast}$ of~$\HH^n$, where $\partial_{\infty}\HH^n$ identifies with $\R^{n-1}\cup\{\infty\}$, we may assume that $B$ is centered at~$\infty$.
Let $\Omega$ be the convex hull of $\Lambda_{j(\Gamma_0)}\smallsetminus\{\infty\}$ in~$\R^{n-1}$, where $\Lambda_{j(\Gamma_0)}$ is the limit set of $j(\Gamma_0)$.
The ratio of the Euclidean diameter of $j(S')\backslash \Omega$ to that of $j(S)\backslash\Omega$ is bounded by~$2\nu(n)$.
We renormalize the metric on~$\R^{n-1}$ so that $j(S')\backslash \Omega$ has Euclidean diameter~$1$: then, by definition of cusp thickness, it is sufficient to prove that the height of points of $E(j,\rho)$ in $\R^{n-1}\times\R_+^{\ast}$ is bounded in terms of $C(j,\rho)$ alone.

There is an $m$-dimensional affine subspace $V\subset\Omega$ of~$\R^{n-1}$ which is preserved by $j(S')$ and on which $j(S')$ acts as a lattice of translations (see Section~\ref{subsec:geo-finiteness}).
Any point of~$\Omega$ lies within distance~$1$ of~$V$.

If $C(j,\rho)>1$, then by Theorem~\ref{thm:Kirszbraunopt} the stretch locus $E(j,\rho)$ is a disjoint union of geodesic lines of~$\HH^n$.
Let $\ell \subset E(j,\rho)$ be such a line, reaching a height $h$ in the upper half-space model.
We must bound~$h$.
The endpoints $\xi, \eta \in \Omega$ of $\ell$ are $2h$ apart in~$\R^{n-1}$.
Let $\xi',\eta'\in V$ be within distance $1$ from $\xi, \eta$ respectively.
There exists $\gamma\in S'$ such that $d_{\R^{n-1}}(j(\gamma)\cdot \xi',\frac{\xi'+\eta'}{2})\leq 1$.
Since $\xi', \eta'$ and their images under $j(\gamma)$ form a parallelogram, we also have $d_{\R^{n-1}}(j(\gamma)\cdot \frac{\xi'+\eta'}{2},\eta')\leq 1$.
By the triangle inequality,
$$d_{\R^{n-1}}\bigg(j(\gamma)\cdot \xi\,,\frac{\xi+\eta}{2}\bigg) \leq 3 \quad\mathrm{and}\quad d_{\R^{n-1}}\bigg(j(\gamma)\cdot \frac{\xi+\eta}{2}\,,\eta\bigg) \leq 3.$$
Adding up, it follows that the points $\frac{\xi+ 3\eta}{4}$ and $j(\gamma)\cdot \frac{3\xi+\eta}{4}$ are at Euclidean distance $\leq 3$ from each other.
But the leaves $\ell$ and $j(\gamma)\cdot \ell$ of $E(j,\rho)$ contain points at height $h\sqrt{3}/2$ above these two points, and are therefore $\leq 2\sqrt{3}/h$ apart in the hyperbolic metric.
However, $\ell$ and $j(\gamma)\cdot \ell$ form an angle close to $\pi/3$ (see Figure~\ref{fig:G}): by Lemma~\ref{lem:Esimple} (or $(\ast)$ in its proof), this places an upper bound on~$h$ (depending only on $C(j,\rho)$).
\end{proof}

\begin{figure}[h!]
\begin{center}
\labellist
\small\hair 2pt
\pinlabel{$h$} at 18 80
\pinlabel{$\Omega$} at 290 31
\pinlabel{$V$} at 288 49
\pinlabel{$\ell$} at 155 120
\pinlabel{$\xi$} at 62 40
\pinlabel{$\eta$} at 255 35
\pinlabel{$j(\gamma)\cdot \ell$} at 249 123
\pinlabel{$j(\gamma)\cdot\xi$} at 155 35
\pinlabel{$j(\gamma)\cdot\eta$} at 335 35
\pinlabel{$\partial_\infty \HH^n$} at 240 8
\pinlabel{\footnotesize{Cusp group $j(S')$}} at 137 13
\endlabellist
\includegraphics[width=12cm]{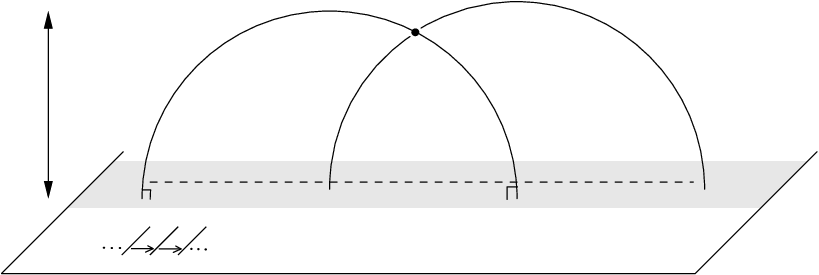}
\caption{Two leaves $\ell$ and $j(\gamma)\cdot \ell$ which nearly intersect, at an angle close to $\pi/3$ (in the upper half-space model of~$\HH^n$).}
\label{fig:G}
\end{center}
\end{figure}
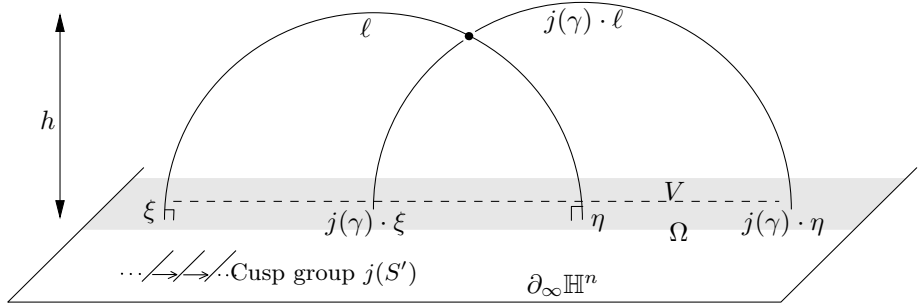

In Section~\ref{subsec:C-highcontinuous}, in order to prove the upper semicontinuity of $(j,\rho)\mapsto C(j,\rho)$ where $C\geq 1$ when all the cusps of~$j$ have rank $\geq n-2$, we shall need the following consequence of Proposition~\ref{prop:margulis}.

\begin{corollary}\label{cor:uniform-thick}
Let $\Gamma_0$ be a discrete group and $(j_k,\rho_k)_{k\in\N^{\ast}}$ a sequence of elements of $\Hom(\Gamma_0,G)^2$ converging to some $(j,\rho)\in\Hom(\Gamma_0,G)^2$, where
\begin{itemize}
  \item $j$ and the $j_k$ are all geometrically finite, of the same cusp type, with all cusps of rank $\geq n-2$,
  \item there exists $C^*>1$ such that $C(j_k, \rho_k)\geq C^*$ for all $k\in\N^{\ast}$,
  \item the stretch loci $E(j_k, \rho_k)$ are nonempty (this is the case for instance if $\rho_k$ is reductive, see Lemma~\ref{lem:Fnonempty} and Corollary~\ref{cor:Enonempty}).
\end{itemize}
Then for any $k\in\N^{\ast}$ we can find a fundamental domain $\mathcal{E}_k$ of $E(j_k,\rho_k)$ for the action of $j_k(\Gamma_0)$ so that all the $\mathcal{E}_k$ are contained in some compact subset of $\HH^n$ independent of~$k$.
\end{corollary}

\begin{proof}
By Proposition~\ref{prop:rem-gfopen}, there exist a compact set $\mathcal{C}\subset\HH^n$ and, for any large enough $k\in\N^{\ast}$, horoballs $H_1^k,\dots,H_c^k$ of~$\HH^n$, such that the union $\mathcal{G}$ of all geodesic rays from $\mathcal{C}$ to the centers of $H^k_1,\dots, H^k_c$ contains a fundamental domain of the convex core of $j_k(\Gamma_0)\backslash\HH^n$.
In particular, the cusp thickness of $j_k(\Gamma_0)\backslash\HH^n$ at any point of $\bigcup_{1\leq i\leq c} \partial H_i^k$ is uniformly bounded from above by some constant independent of~$k$.
On the other hand, by Proposition~\ref{prop:margulis}, the cusp thickness of $j_k(\Gamma_0)\backslash\HH^n$ at any point of $E(j_k,\rho_k)$ is uniformly bounded from below by some constant independent of~$k$.
Since cusp thickness decreases uniformly to~$0$ in all cusps (at exponential rate), this means that $E(j_k,\rho_k)\cap\mathcal{G}$ (which contains a fundamental domain of $E(j_k,\rho_k)$ for the action of $j_k(\Gamma_0)$) remains in some compact subset of $\HH^n$ independent of~$k$.
\end{proof}

\section{Continuity of the minimal Lipschitz constant}\label{sec:lipcont}

In this section we examine the continuity of the function $(j,\rho)\mapsto C(j,\rho)$ for geometrically finite~$j$ (the set $K$ of Sections \ref{sec:stretchlocus} and~\ref{sec:Kirszbraunopt} is empty).
We endow $\Hom(\Gamma_0,G)$ with its natural topology: a sequence $(j_k,\rho_k)$ converges to $(j,\rho)$ if and only if $j_k(\gamma)\rightarrow j(\gamma)$ and $\rho_k(\gamma)\rightarrow\rho(\gamma)$ for all $\gamma$ in some (hence any) finite generating subset of~$\Gamma_0$.

We first prove Proposition~\ref{prop:contCcc}, which states the \emph{continuity of $(j,\rho)\mapsto C(j,\rho)$ for convex cocompact~$j$}.
When $j$ is not convex cocompact, continuity, and even semicontinuity, fail in any dimension $n\geq 2$: see Sections \ref{ex:discontinu} and~\ref{ex:discontinu2} for counterexamples.
However, we will prove the following.

\begin{proposition}\label{prop:contpropertiesC}
Let $\Gamma_0$ be a discrete group and $j_0\in\Hom(\Gamma_0,G)$ a geometrically finite representation.
If all the cusps of~$j_0$ have rank $\geq n-2$ (for instance if we are in dimension $n\leq 3$), then
\begin{enumerate}
  \item the set of pairs $(j,\rho)$ with $C(j,\rho)<1$ is open in $\Hom_{j_0}(\Gamma_0,G)\times\Hom_{j_0\text{-}\mathrm{det}}(\Gamma_0,G)$,
  \item the set of pairs $(j,\rho)$ with $1<C(j,\rho)$ is open in $\Hom_{j_0}(\Gamma_0,G)\times\Hom(\Gamma_0,G)$,
  \item the map $(j,\rho)\mapsto C(j,\rho)$ is continuous on the set of pairs $(j,\rho)\in\Hom_{j_0}(\Gamma_0,G)\times\Hom(\Gamma_0,G)$ with $1\leq C(j,\rho)<+\infty$.
\end{enumerate}
If the cusps of~$j_0$ have arbitrary ranks, then condition (2) holds, as well as:
\begin{enumerate}
  \item[(1')] the set of $\rho$ with $C(j_0,\rho)<1$ is open in $\Hom_{j_0\text{-}\mathrm{det}}(\Gamma_0,G)$,
  \item[(3')] the map $(j,\rho)\mapsto C(j,\rho)$ is lower semicontinuous on the set of pairs $(j,\rho)\in\Hom_{j_0}(\Gamma_0,G)\times\Hom(\Gamma_0,G)$ with $1\leq C(j,\rho)$:
  $$C(j,\rho) \leq \liminf_k C(j_k,\rho_k)$$
  for any sequence $(j_k,\rho_k)$ of such pairs converging to such a pair $(j,\rho)$,
  \item[(3'')] the map $\rho\mapsto C(j_0,\rho)$ is upper semicontinuous on the set of representations $\rho\in\Hom(\Gamma_0,G)$ with $1\leq C(j_0,\rho)<+\infty$:
  $$C(j_0,\rho) \geq \limsup_k C(j_0,\rho_k)$$
  for any sequence $(\rho_k)$ of such representations converging to such a representation~$\rho$.
\end{enumerate}
\end{proposition}

Here we denote by
\begin{itemize}
  \item $\Hom_{j_0}(\Gamma_0,G)$ the space of geometrically finite representations of $\Gamma_0$ in~$G$ with the same cusp type as the fixed representation~$j_0$;
  \item $\Hom_{j_0\text{-}\mathrm{det}}(\Gamma_0,G)$ the space of representations that are cusp-deterio\-rating with respect to~$j_0$, in the sense of Definition~\ref{def:typedet}.
\end{itemize}
These two sets are endowed with the induced topology from $\Hom(\Gamma_0,G)$.
In (3)--(3')--(3''), we endow the set of pairs $(j,\rho)$ satisfying $1\leq C(j,\rho)<+\infty$ or $1\leq C(j,\rho)$ with the induced topology from $\Hom(\Gamma_0,G)^2$.
Note that $\Hom_{j_0\text{-}\mathrm{det}}(\Gamma_0,G)$ is a semi-algebraic subset of $\Hom(\Gamma_0,G)$; it is equal to $\Hom(\Gamma_0,G)$ if and only if $j_0$ is convex cocompact.

When $j_0$ is \emph{not} convex cocompact, the condition $C(j,\rho)<1$ is \emph{not} open in $\Hom_{j_0}(\Gamma_0,G)\times\Hom(\Gamma_0,G)$ or even in $\{j_0\}\times\Hom(\Gamma_0,G)$, since the constant representation~$\rho$ (for which $C(j,\rho)=0$) may be approached by non-cusp-deteriorating representations~$\rho$ (for which $C(j,\rho)\geq\nolinebreak 1$); see also Section~\ref{ex:discontinu} for a related example.
This is why we need to restrict to cusp-deteriorating~$\rho$ in Proposition~\ref{prop:contpropertiesC}.(1).

In dimension $n\geq 4$, when $j_0$ has cusps of rank $<n-2$, conditions (1) and~(3) of Proposition~\ref{prop:contpropertiesC} do not hold: see Sections \ref{ex:dim4upper} and~\ref{ex:dim4C<1} for counterexamples.
The reason, in a sense, is that the convex core of a small geometrically finite deformation of~$j$ can be ``much larger'' than that of~$j$, due to the presence of parabolic elements that are not unipotent.
(Such discontinuous behavior of the convex core also explains why being geometrically finite is not an open condition in the presence of cusps of rank $<n-2$, even among representations of a given cusp type \cite[\S\,5]{bow98}.)

Note finally that $C(j,\rho)=+\infty$ must be ruled out in (3) and (3'') due to Lemma~\ref{lem:C<infty}: parabolic elements can be approached by hyperbolic ones.

Proposition~\ref{prop:contCcc} is proved in Section~\ref{subsec:contCcc} using a partition-of-unity argument based on Lemma~\ref{lem:partofunity}, together with a control on fundamental domains for converging convex cocompact representations (see Appendix~\ref{sec:conv-fundamental}).
Proposition~\ref{prop:contpropertiesC}.(1)--(1') is proved in Section~\ref{subsec:C<1open} following the same approach but using also a comparison between distances in horospheres and spheres of~$\HH^n$ (Lemma~\ref{lem:horosph-sph}).
Proposition~\ref{prop:contpropertiesC}.(2) and (3)--(3')--(3'') are proved in Section~\ref{subsec:C-highcontinuous}; for reductive~$\rho$, they are a consequence of the existence of a maximally stretched lamination when $C(j,\rho)\geq 1$ (Theorem~\ref{thm:lamin}).
The case of nonreductive~$\rho$ follows from the reductive case by using again a partition-of-unity argument, as we explain in Section~\ref{subsec:nonred}.

\subsection{Continuity in the convex cocompact case}\label{subsec:contCcc}

In this section we prove Proposition~\ref{prop:contCcc}.
We fix a pair of representations $(j,\rho)\in\Hom(\Gamma_0,G)^2$ with $j$ convex cocompact and a sequence $(j_k,\rho_k)_{k\in\N^{\ast}}$ of elements of $\Hom(\Gamma_0,G)^2$ converging to $(j,\rho)$.
We may and shall assume that $\Gamma_0$ is torsion-free, using Lemma~\ref{lem:finiteindex} and the Selberg lemma \cite[Lem.\,8]{sel60}.

\subsubsection{Upper semicontinuity}\label{subsubsec:uppersemicont}

We first prove that
$$C(j,\rho) \,\geq\, \limsup_{k\rightarrow +\infty} \, C(j_k, \rho_k).$$
Fix $\varepsilon>0$ and let $f : \HH^n\rightarrow \HH^n$ be a $(j,\rho)$-equivariant, $(C(j,\rho)+\varepsilon)$-Lipschitz map.
We explain how for any large enough~$k$ we can modify~$f$ into a $(j_k,\rho_k)$-equivariant map $f_k$ with $\Lip(f_k)\leq\Lip(f)+\varepsilon$.
By Lemma~\ref{lem:proj}, we only need to define~$f_k$ on the preimage $N_k\subset\HH^n$ of the convex core of $j_k(\Gamma_0)\backslash\HH^n$.
In order to build $f_k$, we will paste together shifted ``pieces'' of~$f$ using Lemma~\ref{lem:partofunity}.

Let $N\subset\HH^n$ be the preimage of the convex core of $j(\Gamma_0)\backslash\HH^n$.
By Proposition~\ref{prop:rem-ccopen}, there exists a compact set $\mathcal{C}\subset\HH^n$ such that 
$$N\subset j(\Gamma_0)\cdot\mathcal{C}~\text{ and }~N_k\subset j_k(\Gamma_0)\cdot\mathcal{C}$$ 
for all large enough $k\in\N^{\ast}$, and the injectivity radius of $j(\Gamma_0)\backslash\HH^n$ and $j_k(\Gamma_0)\backslash\HH^n$ is bounded from below by some constant $\delta>0$ independent of~$k$.
Let $B_1,\dots,B_r$ be open balls of~$\HH^n$ covering~$\mathcal{C}$, of radius $<\delta$.
For any $1\leq i\leq r$, let $\psi_i : \HH^n\rightarrow [0,1]$ be a Lipschitz, $j(\Gamma_0)$-equivariant function supported on $j(\Gamma_0)\cdot B_i$, such that $(\psi_i)_{1\leq i\leq r}$ restricts to a partition of unity on $j(\Gamma_0)\cdot\mathcal{C}$, subordinated to the covering $(j(\Gamma_0)\cdot (B_i\cap\mathcal{C}))_{1\leq i\leq r}$.
For $1\leq i\leq r$ and $k\in\N^{\ast}$, let
$$\psi_{i,k} := \frac{\Psi_{i,k}}{\sum_{i'=1}^r \Psi_{i',k}},$$
where $\Psi_{i,k} : \HH^n\rightarrow [0,1]$ is the $j_k(\Gamma_0)$-invariant function supported on $j_k(\Gamma_0)\cdot\nolinebreak B_i$ that coincides with $\psi_i$ on~$B_i$.
Then, for $k\in\N^{\ast}$ large enough, $(\psi_{i,k})_{1\leq i\leq r}$ induces a $j_k(\Gamma_0)$-equivariant partition of unity on $j_k(\Gamma_0)\cdot\mathcal{C}$, subordinated to the covering $(j_k(\Gamma_0)\cdot (B_i\cap\mathcal{C}))_{1\leq i\leq r}$.
Note that there is a constant $L>0$ such that $\psi_{i,k}$ is $L$-Lipschitz on $j_k(\Gamma_0)\cdot\mathcal{C}$ for all $1\leq i\leq r$ and large $k\in\N^{\ast}$; indeed, the $j_k(\Gamma_0)$-invariant function $\sum_{i'} \Psi_{i',k}$ is Lipschitz with constant $\leq\sum_{i'} \Lip(\psi_{i'})$ and it converges uniformly to~$1$ on each $B_i\cap\mathcal{C}$ as $k\rightarrow +\infty$, by compactness.
For $1\leq i\leq r$ and $k\in\N^{\ast}$, let
$$f_{i,k} :\, j_k(\Gamma_0)\cdot B_i \longrightarrow \HH^n$$
be the $(j_k,\rho_k)$-equivariant map that coincides with $f$ on~$B_i$.
For $k\in\N^{\ast}$ and $p\in j_k(\Gamma_0)\cdot\mathcal{C}$ , let $I_{p,k}$ be the set of indices $1\leq i\leq r$ such that $p\in j_k(\Gamma_0)\cdot B_i$.
The function
$$p \,\longmapsto\, R_{p,k} := \mathrm{diam}\{ f_{i,k}(p)~|~i\in I_{p,k} \} ,$$
defined on $j_k(\Gamma_0)\cdot\mathcal{C}$, is $j_k(\Gamma_0)$-invariant and converges uniformly to~$0$ on $\mathcal{C}$ as $k\rightarrow +\infty$.
By Lemma~\ref{lem:partofunity}, the $(j_k,\rho_k)$-equivariant map
$$f_k := \sum_{i=1}^r \psi_{i,k} f_{i,k}\ :\ j_k(\Gamma_0)\cdot\mathcal{C} \longrightarrow \HH^n$$
satisfies
\begin{eqnarray*}
\Lip_p(f_k) & \leq & \sum_{i=1}^r \big(L R_{p,k} + \psi_{i,k}(p) \,\Lip_p(f_{i,k})\big)\\
& \leq & r L \bigg(\sup_{p'\in\mathcal{C}} R_{p',k}\bigg) + \Lip(f)
\end{eqnarray*}
for all $p\in\mathcal{C}$, hence for all $p\in N_k\subset j_k(\Gamma_0)\cdot\mathcal{C}$ by equivariance.
We have seen that $\sup_{p'\in\mathcal{C}} R_{p',k}\rightarrow 0$ as $k\rightarrow +\infty$.
Therefore, for large enough~$k$, the $(j_k,\rho_k)$-equivariant map $\HH^n\rightarrow\HH^n$ obtained by precomposing $f_k$ with the closest-point projection onto~$N_k$ has Lipschitz constant $\leq\sup_{p\in N_k} \Lip_p(f_k)\leq\Lip(f)+\varepsilon$ by Lemma~\ref{lem:localLip}.
This shows that $C(j_k,\rho_k)\leq C(j,\rho)+2\varepsilon$, and we conclude by taking the $\limsup$ over~$k$ and letting $\varepsilon$ tend to~$0$.

\subsubsection{Lower semicontinuity}\label{subsubsec:lowersemicont}

Let us now prove that
$$C(j,\rho) \,\leq\, \liminf_{k\rightarrow +\infty} \, C(j_k, \rho_k).$$
If $\rho(\Gamma_0)$ has a fixed point $p$ in~$\HH^n$, then $C(j,\rho)=0$ (Remark~\ref{rem:boundedrho}) and there is nothing to prove.
We thus assume that $\rho(\Gamma_0)$ has no fixed point in~$\HH^n$.

\smallskip

$\bullet$ \textbf{Generic case.}
Consider the case where $\rho(\Gamma_0)$ has no fixed point in $\partial_{\infty}\HH^n$ and does not preserve any geodesic line of~$\HH^n$.
Then $\rho(\Gamma_0)$ contains two hyperbolic elements $\rho(\gamma_1),\rho(\gamma_2)$ whose translation axes have no common endpoint in $\partial_{\infty}\HH^n$.
For large enough~$k$, the elements $\rho_k(\gamma_1),\rho_k(\gamma_2)\in\rho_k(\Gamma_0)$ are hyperbolic too and their translation axes converge to the respective axes of $\rho(\gamma_1),\rho(\gamma_2)$.
For any $k\in\N^{\ast}$, let $f_k : \HH^n\rightarrow\HH^n$ be a $(j_k,\rho_k)$-equivariant, $(C(j_k,\rho_k)+2^{-k})$-Lipschitz map.
The same argument as in the proof of Lemma~\ref{lem:Fnonempty} shows that for any compact subset $\mathscr{C}$ of~$\HH^n$, the sets $f_k(\mathscr{C})$ all lie inside some common compact subset of~$\HH^n$.
By the Arzel\`a--Ascoli theorem, some subsequence of $(f_k)_{k\in\N^{\ast}}$ converges to a $(j,\rho)$-equivariant map $f : \HH^n\rightarrow\HH^n$.
(Here we use that $(C(j_k,\rho_k))_{k\in\N^{\ast}}$ is bounded, a consequence of the upper semicontinuity proved in Section~\ref{subsubsec:uppersemicont}.)
This implies $C(j,\rho)\leq\liminf_k C(j_k, \rho_k)$.

\smallskip

$\bullet$ \textbf{Degenerate reductive case.}
Consider the case where $\rho(\Gamma_0)$ preserves a geodesic line $\mathcal{A}$ of $\HH^n$.
The following observation is interesting in its own right.

\begin{lemma}\label{lem:Ewith2fixedpoints}
If the group $\rho(\Gamma_0)$ preserves a geodesic line $\A\subset\HH^n$ without fixing any point in~$\HH^n$, then the stretch locus $E(j,\rho)$ is a geodesic lamination whose projection to $j(\Gamma_0)\backslash\HH^n$ is compact, contained in the convex core, and whose leaves are maximally stretched.
\end{lemma}

\begin{proof}
After passing to a subgroup of index two (which does not change the stretch locus by Lemma~\ref{lem:finiteindex}), we may assume that $\rho(\Gamma_0)$ fixes both endpoints of~$\A$ in $\partial_{\infty}\HH^n$: in other words, $\rho(\Gamma_0)$ is contained in $\underline{M}\underline{A}$, where $\underline{M}$ is the subgroup of~$G$ that (pointwise) fixes~$\A$ and $\underline{A}$ is the group of pure translations along~$\A$.
The groups $\underline{M}$ and~$\underline{A}$ commute and have a trivial intersection; let $\underline{\pi}_{\underline{A}} : \underline{M}\underline{A}\rightarrow\underline{A}$ be the natural projection.
We claim that $\rho_{\scriptscriptstyle{\underline{A}}}:=\underline{\pi}_{\underline{A}}\circ\rho$ satisfies
$$C(j,\rho_{\scriptscriptstyle{\underline{A}}}) = C(j,\rho) \quad\quad\mathrm{and}\quad\quad E(j,\rho_{\scriptscriptstyle{\underline{A}}}) = E(j,\rho).$$
Indeed, any element of $\F^{j,\rho}$ (\resp of~$\F^{j,\rho_{\scriptscriptstyle{\underline{A}}}}$) remains in $\F^{j,\rho}$ (\resp in~$\F^{j,\rho_{\scriptscriptstyle{\underline{A}}}}$) after postcomposing with the closest-point projection onto~$\A$, and for a map $\HH^n\rightarrow\A$ it is equivalent to be $(j,\rho)$-equivariant or $(j,\rho_{\scriptscriptstyle{\underline{A}}})$-equivariant.
Since $\rho_{\scriptscriptstyle{\underline{A}}}(\Gamma_0)\subset\underline{A}$ is commutative, for any $m\in\Z$ we can consider the representation $\rho_{\scriptscriptstyle{\underline{A}}}^m : \gamma\mapsto\rho_{\scriptscriptstyle{\underline{A}}}(\gamma)^m$.
We claim that for $m\geq 1$,
$$C(j,\rho_{\scriptscriptstyle{\underline{A}}}^m) = m\,C(j,\rho_{\scriptscriptstyle{\underline{A}}}) \quad\quad\mathrm{and}\quad\quad E(j,\rho_{\scriptscriptstyle{\underline{A}}}^m)=E(j,\rho_{\scriptscriptstyle{\underline{A}}}).$$
Indeed, let $h_m$ be an orientation-preserving homeomorphism of $\A\simeq\R$ such that
$$d(h_m(p),h_m(q)) = m\,d(p,q)$$
for all $p,q\in\A$; for any $C>0$, the postcomposition with~$h_m$ realizes a bijection between the $(j,\rho_{\scriptscriptstyle{\underline{A}}})$-equivariant, $C$-Lipschitz maps and the $(j,\rho_{\scriptscriptstyle{\underline{A}}}^m)$-equivariant, $mC$-Lipschitz maps from $\HH^n$ to~$\A$, which preserves the stretch locus.
Since $C(j,\rho_{\scriptscriptstyle{\underline{A}}})>0$ (because $\rho(\Gamma_0)$ has no fixed point in~$\HH^n$), we have $C(j,\rho_{\scriptscriptstyle{\underline{A}}}^m)>1$ for large enough~$m$, hence we can apply Theorem~\ref{thm:lamin} to the stretch locus $E(j,\rho_{\scriptscriptstyle{\underline{A}}}^m)=E(j,\rho_{\scriptscriptstyle{\underline{A}}})$.
\end{proof}

Suppose the group $\rho(\Gamma_0)$ preserves a geodesic line of~$\HH^n$ without fixing any point in~$\HH^n$.
By Lemmas \ref{lem:ClambdaCLipLipf} and~\ref{lem:Ewith2fixedpoints}, for any $\varepsilon>0$ there exists $\gamma\in\Gamma_0$ with $j(\gamma)$ hyperbolic such that
\begin{equation}\label{eq:almostcarried1}\frac{\lambda(\rho(\gamma))}{\lambda(j(\gamma))} \geq C(j,\rho) - \varepsilon.\end{equation}
It follows that $j_k(\gamma)$ is hyperbolic and
\begin{equation}\label{eq:almostcarried2}C(j_k,\rho_k) \,\geq\, \frac{\lambda(\rho_k(\gamma))}{\lambda(j_k(\gamma))} \,\geq\, C(j,\rho) - 2\varepsilon\end{equation}
for all large enough~$k$.
We conclude by taking the $\liminf$ over~$k$ and letting $\varepsilon$ tend to~$0$.

\smallskip

$\bullet$ \textbf{Nonreductive case.}
Finally, we consider the case where the group $\rho(\Gamma_0)$ has a unique fixed point $\xi$ in $\partial_{\infty}\HH^n$, \ie $\rho$ is nonreductive (Definition~\ref{def:reductive}).
Choose an oriented geodesic line $\A$ of~$\HH^n$ with endpoint~$\xi$.
For any $\gamma\in\Gamma_0$ we can write in a unique way $\rho(\gamma)=gu$ where $g\in G$ preserves~$\A$ (\ie belongs to $\underline{M}\underline{A}$ with the notation above) and $u$ is unipotent; setting $\rho^{\mathrm{red}}(\gamma):=g$ defines a representation $\rho^{\mathrm{red}}\in\Hom(\Gamma_0,G)$ which is reductive (with image in $\underline{M}\underline{A}$).
Note that changing the line~$\A$ only modifies $\rho^{\mathrm{red}}$ by conjugating it; this does not change the constant $C(j,\rho^{\mathrm{red}})$ by Remark~\ref{rem:CLipconj}.
When $\rho$ is reductive, we set $\rho^{\mathrm{red}}:=\rho$.
Then the following holds.

\begin{lemma}\label{lem:nonred-lip-cc}
For any pair of representations $(j,\rho)\in\Hom(\Gamma_0,G)^2$ with $j$ convex cocompact,
$$C(j,\rho^{\mathrm{red}}) = C(j,\rho).$$
\end{lemma}

\begin{proof}
We may assume that $\rho$ is nonreductive.
Let $\xi$ and~$\A$ be as above and let $\mathrm{pr} : \HH^n\rightarrow\A$ be the ``horocyclic projection'' collapsing each horosphere centered at $\xi$ to its intersection point with~$\A$; it is $1$-Lipschitz.
For any $(j,\rho)$-equivariant Lipschitz map $f : \HH^n\rightarrow\HH^n$, the map $\mathrm{pr}\circ f$ is $(j,\rho^{\mathrm{red}})$-equivariant with $\Lip(\mathrm{pr}\circ f)\leq\Lip(f)$, hence
$$C(j,\rho^{\mathrm{red}}) \leq C(j,\rho).$$
Let $a\in G$ be a hyperbolic element acting as a pure translation along~$\A$, with repelling fixed point~$\xi$ at infinity.
Then $\rho^{(i)}:=a^i\rho(\cdot)a^{-i}\rightarrow\rho^{\mathrm{red}}$ as $i\rightarrow +\infty$.
By Remark~\ref{rem:CLipconj}, we have $C(j,\rho^{(i)})=C(j,\rho)$ for all $i\in\N$.
By upper semicontinuity (proved in Section~\ref{subsubsec:uppersemicont}),
$$C(j,\rho^{\mathrm{red}}) \geq \limsup_{i\rightarrow +\infty}\ C(j,\rho^{(i)}) = C(j,\rho).\qedhere$$
\end{proof}

We now go back to our sequence $(j_k,\rho_k)_{k\in\N^{\ast}}$ converging to $(j,\rho)$.
Since $\rho_k\rightarrow \rho$ and $\rho$ has conjugates converging to $\rho^{\mathrm{red}}$ (see above), a diagonal argument shows that there are conjugates $\rho'_k$ of~$\rho_k$ such that $\rho'_k\rightarrow\rho^{\mathrm{red}}$.
By the reductive case above, $\liminf_k C(j_k,\rho'_k)\geq C(j,\rho^{\mathrm{red}})$, and we conclude using Remark~\ref{rem:CLipconj} and Lemma~\ref{lem:nonred-lip-cc}.
This completes the proof of Proposition~\ref{prop:contCcc}.

\subsection{Openness of the condition $C<1$ on cusp-deteriorating pairs}\label{subsec:C<1open}

In this section we prove Proposition~\ref{prop:contpropertiesC}.(1)--(1').
The strategy is analogous to the proof of upper semicontinuity in Section~\ref{subsubsec:uppersemicont}.
The partition-of-unity argument in that proof fails in the presence of cusps, since the convex core (when nonempty) is not compact anymore.
However, we shall see that the argument can be adapted as long as convex cores deform continuously.
Such continuous behavior is ensured under the assumptions of Prop.~\ref{prop:contpropertiesC}.(1') (constant convex core) or Prop.~\ref{prop:contpropertiesC}.(1) (all cusps of rank $\geq n-2$, see Proposition~\ref{prop:rem-gfopen}).

Consider a pair $(j,\rho)\in\Hom_{j_0}(\Gamma_0,G)\times\Hom_{j_0\text{-}\mathrm{det}}(\Gamma_0,G)$ with $C(j,\rho)<1$, and a sequence $(j_k,\rho_k)_{k\in\N^{\ast}}$ of elements of $\Hom_{j_0}(\Gamma_0,G)\times\Hom_{j_0\text{-}\mathrm{det}}(\Gamma_0,G)$ converging to $(j,\rho)$.
\emph{If $j_0$ has a cusp of rank $<n-2$, we assume that $j_k=j$ for all $k\in\N^{\ast}$.}
We shall prove that $C(j_k,\rho_k)<1$ for all large enough~$k$.

We can and shall assume that $\Gamma_0$ is torsion-free (using Lemma~\ref{lem:finiteindex} and the Selberg lemma \cite[Lem.\,8]{sel60}).
We can also always assume that the convex core of $j_k(\Gamma_0)\backslash\HH^n$ is nonempty: otherwise the group $j_k(\Gamma_0)$ is elementary with a fixed point in~$\HH^n$ or a unique fixed point in $\partial_{\infty}\HH^n$, and $C(j_k,\rho_k)=0$ by Remark~\ref{rem:boundedrho}.
Therefore the convex core of $M:=j(\Gamma_0)\backslash\HH^n$ is nonempty too (because $j$ and the~$j_k$ have the same cusp type).

Let $f : \HH^n\rightarrow\HH^n$ be a $(j,\rho)$-equivariant map with $0<\Lip(f)<1$.
We shall modify~$f$ into a $(j_k,\rho_k)$-equivariant map~$f_k$ with $\Lip(f_k)<1$ for all large enough~$k$.
As usual, by Lemma~\ref{lem:proj} we only need to define~$f_k$ on the preimage $N_k\subset\HH^n$ of the convex core of $j_k(\Gamma_0)\backslash\HH^n$.
In order to build $f_k$, we shall proceed as in Section~\ref{subsubsec:uppersemicont} and paste together shifted ``pieces'' of~$f$ using Lemma~\ref{lem:partofunity}.

By Proposition~\ref{prop:goodincusps}.(3) we may assume that $f$ is constant on neighborhoods of some horoballs $H_1,\dots,H_c$ of~$\HH^n$ whose images in $M=j(\Gamma_0)\backslash\HH^n$ are disjoint and intersect the convex core of~$M$ in standard cusp regions (Definition~\ref{def:standardcusp}), representing all the cusps.
For $1\leq\ell\leq c$, let $S_{\ell}\subset\Gamma_0$ be the stabilizer of~$H_{\ell}$ under the $j$-action.
Let $N\subset\HH^n$ be the preimage of the convex core of $j(\Gamma_0)\backslash\HH^n$.
By Proposition~\ref{prop:rem-gfopen}, if the horoballs $H_1,\dots,H_c$ are small enough, then there exist a compact set $\mathcal{C}\subset\HH^n$ and, for any $k\in\N^{\ast}$, horoballs $H_1^k,\dots,H_c^k$ of~$\HH^n$, such that
\begin{itemize}
  \item the images of $H_1^k,\dots,H_c^k$ in $j_k(\Gamma_0)\backslash\HH^n$ are disjoint and intersect the convex core in standard cusp regions, for all large enough $k\in\N^{\ast}$;
  \item the stabilizer in~$\Gamma_0$ of $H_{\ell}^k$ under~$j_k$ is~$S_{\ell}$;
  \item the horoballs $H_{\ell}^k$ converge to~$H_{\ell}$ for all $1\leq\ell\leq c$;
  \item $N\subset j(\Gamma_0)\cdot (\mathcal{C}\cup\bigcup_{1\leq\ell\leq c} H_{\ell})$ and, for all large enough $k\in\N^{\ast}$,
  \begin{equation} \label{eqn:triche}
N_k\subset j_k(\Gamma_0)\cdot \bigg(\mathcal{C} \cup \bigcup_{1\leq\ell\leq c} H_{\ell}^k\bigg);   
  \end{equation}
  \item the cusp thickness (Definition~\ref{def:thickness}) of $j_k(\Gamma_0)\backslash\HH^n$ at any point of $\partial H_{\ell}^k$ is uniformly bounded by some constant $\Theta>0$ independent of~$k$;
  \item the injectivity radius of $j_k(\Gamma_0)\backslash\big(\HH^n\smallsetminus\bigcup_{\ell=1}^k j_k(\Gamma_0)\cdot H_{\ell}^k\big)$ is bounded from below by some constant $\delta>0$ independent of~$k$.
\end{itemize}
(If $j_0$ has a cusp of rank $<n-2$, then $j_k=j$ and we take $H_{\ell}^k=H_{\ell}$ for all $k\in\N^{\ast}$.)
For any $1\leq\ell\leq c$, by convergence of the horoballs $H_{\ell}^k$, the map $f$ is constant on some neighborhood of $\partial H_{\ell}^k\cap\mathcal{C}$ for large enough~$k$, which implies
\begin{equation}\label{eqn:fnearconst}
\sup_{p\in\partial H^k_{\ell}\cap\mathcal{C}}\, \Lip_p(f)=0.
\end{equation}

Let $B_1,\dots,B_r$ be open balls of~$\HH^n$ covering~$\mathcal{C}$, of radius $<\delta$.
For any $1\leq i\leq r$, let $\psi_i : \HH^n\rightarrow [0,1]$ be a Lipschitz, $j(\Gamma_0)$-equivariant function supported on $j(\Gamma_0)\cdot B_i$, such that $(\psi_i)_{1\leq i\leq r}$ restricts to a partition of unity on $j(\Gamma_0)\cdot\mathcal{C}$, subordinated to the covering $(j(\Gamma_0)\cdot (B_i\cap\mathcal{C}))_{1\leq i\leq r}$.
As in Section~\ref{subsubsec:uppersemicont}, for large enough~$k$ we can perturb the $\psi_i$ to a $j_k(\Gamma_0)$-equivariant partition of unity $(\psi_{i,k})_{1\leq i\leq r}$ of $j_k(\Gamma_0)\cdot\mathcal{C}$, subordinated to the covering $(j_k(\Gamma_0)\cdot  B_i)_{1\leq i\leq r}$, such that all the functions~$\psi_{i,k}$ are $L$-Lipschitz for some constant $L>0$ independent of $i$ and~$k$.
For $1\leq i\leq r$ and $k\in\N^{\ast}$, let
$$f_{i,k} : j_k(\Gamma_0)\cdot B_i \longrightarrow \HH^n$$
be the $(j_k,\rho_k)$-equivariant map that coincides with $f$ on~$B_i$.
As in Section~\ref{subsubsec:uppersemicont}, it follows from Lemma~\ref{lem:partofunity} that the $(j_k,\rho_k)$-equivariant map
$$f'_k := \sum_{i=1}^r \psi_{i,k}\,f_{i,k} :\ j_k(\Gamma_0)\cdot\mathcal{C} \,\longrightarrow\, \HH^n$$
satisfies
\begin{equation}\label{eqn:Lipf'k}
\Lip_p(f'_k) \leq r L R_{p,k} + \Lip_p(f)
\end{equation}
for all $p\in j_k(\Gamma_0)\cdot\mathcal{C}$, where $p\mapsto R_{p,k}$ is a $j_k(\Gamma_0)$-invariant function converging uniformly to~$0$ on~$\mathcal{C}$ as $k\rightarrow +\infty$.
By equivariance,
$$\limsup_{k\rightarrow +\infty}\, \sup_{p\in j_k(\Gamma_0)\cdot\mathcal{C}}\,\Lip_p(f'_k) \,\leq\, \Lip(f) \,<\, 1.$$
It only remains to prove that for any $1\leq\ell\leq c$ we can extend $f'_k|_{\partial H^k_{\ell}\cap N_k}$ to $H^k_{\ell}\cap N_k$ in a $(j_k|_{S_{\ell}},\rho_k|_{S_{\ell}})$-equivariant way with Lipschitz constant $<1$.
Indeed, then we can extend $f'_k$ to the orbit $j_k(\Gamma_0)\cdot (H^k_{\ell}\cap N_k)$ in a $(j_k,\rho_k)$-equivariant way; piecing together these maps for varying~$\ell$, and taking $f'_k$ on the complement of $\bigcup_{\ell=1}^c j_k(\Gamma_0) \cdot H^k_{\ell}$ in~$N_k$ (which is contained in $j_k(\Gamma_0)\cdot\mathcal{C}$), we will obtain a $(j_k,\rho_k)$-equivariant map $f_k : N_k\rightarrow\HH^n$ with $\Lip(f_k)<1$ for all large enough~$k$, which will complete the proof.

Fix $1\leq\ell\leq c$.
By Theorem~\ref{thm:Kirszbraunequiv}, in order to prove that $f'_k|_{\partial H^k_{\ell}\cap N_k}$ extends to $H^k_{\ell}\cap N_k$ in a $(j_k|_{S_{\ell}},\rho_k|_{S_{\ell}})$-equivariant way with Lipschitz constant $<1$, it is sufficient to prove that $\Lip_{\partial H^k_{\ell}\cap N_k}(f'_k)<1$.
By \eqref{eqn:fnearconst} and \eqref{eqn:Lipf'k}, for any $\varepsilon>0$ we have
\begin{equation}\label{eqn:yetanotherlabel}
\sup_{p\in\partial H^k_{\ell}\cap N_k} \Lip_p(f'_k) \leq \varepsilon
\end{equation}
for all large enough~$k$, since $\partial H^k_{\ell}\cap N_k\subset j_k(\Gamma_0)\cdot\mathcal{C}$ and the $j_k(\Gamma_0)$-invariant functions $p\mapsto R_{p,k}$ converge uniformly to~$0$ on~$\mathcal{C}$ as $k\rightarrow +\infty$.
Note that \eqref{eqn:yetanotherlabel} does not immediately give a bound on the global constant $\Lip_{\partial H^k_{\ell}\cap N_k}(f'_k)$, since the subset of horosphere $\partial H^k_{\ell} \cap N_k$ is not convex for the hyperbolic metric.
However, such a bound follows from Lemmas \ref{lem:horosph-sph} and~\ref{lem:subconvex} below, which are based on a comparison between the intrinsic metrics of horospheres and spheres in~$\HH^n$.

The idea of Lemma~\ref{lem:horosph-sph} is that \eqref{eqn:yetanotherlabel} controls the Lipschitz constant at short range, while the fixed point of $\rho_k(S_\ell)$ implies control at long range.
The difficulty is that there can be an arbitrarily large ``medium range'' to handle inbetween, since the fixed point of $\rho_k(S_\ell)$ can lie arbitrarily far out.
In dimension $n\geq 4$ this is compounded by the fact that generally $\partial H^k_{\ell} \cap N_k$ is not even convex for the \emph{Euclidean} metric of the horosphere $\partial H^k_\ell$; Lemma~\ref{lem:subconvex} deals with that issue.

For $t\geq 1$, we say that a subset $X$ of a Euclidean space is \emph{$t$-subconvex} if for any $x,y\in X$ there exists a path from $x$ to $y$ in~$X$ whose length is at most $t$ times the Euclidean distance from $x$ to~$y$.

\begin{lemma}\label{lem:horosph-sph}
Let $S$ be a discrete group.
For any $R>0$, there exists $\varepsilon>0$ with the following property: if $(j,\rho)\in\Hom(S,G)^2$ is a pair of representations with $j$ injective such that
\begin{itemize}
  \item the group $j(S)$ is discrete and preserves a horoball $H$ of~$\HH^n$,
  \item the group $\rho(S)$ has a fixed point in~$\HH^n$,
  \item there exists a closed, $j(S)$-invariant, $2$-subconvex set $\mathcal{N} \subset \partial H$ such that the quotient $j(S)\backslash\mathcal{N}$ has Euclidean diameter $\leq R$,
\end{itemize}
then any $(j,\rho)$-equivariant map $f' : \mathcal{N} \rightarrow\HH^n$ satisfying $\Lip_p(f')\leq\varepsilon$ for all $p\in \mathcal{N}$ satisfies $\Lip(f')<1$.
\end{lemma}

\begin{lemma}\label{lem:subconvex}
In our setting, up to enlarging the compact set $\mathcal{C}\subset \HH^n$ and replacing the horoballs $H_1,\dots,H_c$ and $H_1^k,\dots,H_c^k$ with smaller horoballs, still satisfying the list of six properties \eqref{eqn:triche}, we may assume that $\partial H^k_{\ell}\cap N_k$ is $2$-subconvex in $\partial H^k_{\ell}\simeq\R^{n-1}$ for all $1\leq\ell\leq c$ and large enough $k\in\N^{\ast}$.
\end{lemma}

Here Lemma~\ref{lem:horosph-sph} applies to $\mathcal{N}:=N_k\cap\partial H^k_{\ell}$ (which is $2$-subconvex by Lemma~\ref{lem:subconvex}) and to $f':=f'_k|_{\mathcal{N}}$ (which satisfies \eqref{eqn:yetanotherlabel}).
Note that $\rho_k(S_{\ell})$ has a fixed point in~$\HH^n$ by Fact~\ref{fact:fixedpoint}, since $\rho_k$ is cusp-deteriorating with respect to~$j_k$, and that the Euclidean diameter of $j_k(S_{\ell})\backslash (N_k\cap\partial H^k_{\ell})$ is uniformly bounded for $k\in\N^{\ast}$, by the uniform bound $\Theta$ on cusp thickness.
Therefore it is sufficient to prove Lemmas \ref{lem:horosph-sph} and~\ref{lem:subconvex} to complete the proof of Proposition~\ref{prop:contpropertiesC}.(1)--(1').

\begin{proof}[Proof of Lemma~\ref{lem:horosph-sph}]
Fix $R>0$ and let $j,\rho,H,\mathcal{N}$ be as in the statement.
Consider a $(j,\rho)$-equivariant map $f' : \mathcal{N}\rightarrow\HH^n$ such that $\Lip_p(f')\leq\varepsilon$ for all $p\in\mathcal{N}$, for some $\varepsilon>0$.
Let us show that if $\varepsilon$ is smaller than some constant independent of~$f'$, then $\Lip(f')<1$.
Let $d_{\partial H}$ be the natural Euclidean metric on $\partial H$.
By \eqref{eqn:horomu}, for any $p,q\in\mathcal{N}$,
\begin{equation}\label{eqn:horoline}
d(p,q) = 2\,\arcsinh\left(\frac{d_{\partial H}(p,q)}{2}\right).
\end{equation}
If $d(p,q)\leq 1$, then $d(p,q)\geq\kappa\,d_{\partial H}(p,q) $ for some universal $\kappa>0$ (specifically, $\kappa=(2\sinh(1/2))^{-1}$ by concavity of $\arcsinh$). 
On the other hand, $d(f'(p),f'(q))\leq 2\varepsilon\,d_{\partial H}(p,q)$ by Remark~\ref{rem:pathlength}.(3) and $2$-subconvexity, hence
$$\frac{d(f'(p),f'(q))}{d(p,q)} \leq \frac{2\varepsilon}{\kappa} < \frac{2}{3}$$
for all $p,q\in\mathcal{N}$ with $0<d(p,q)\leq 1$ as soon as $\varepsilon<\kappa/3$.
We now assume that this is satisfied and consider pairs of points $p,q\in\mathcal{N}$ with $d(p,q)\geq 1$.

Let $\partial B$ be a sphere of~$\HH^n$ centered at a fixed point of~$\rho(S)$, and containing $f'(p_0)$ for some $p_0\in\mathcal{N}$ (see Figure~\ref{fig:H}).
Then $f'(j(\gamma)\cdot p_0)\in\partial B$ for all $\gamma\in S$, since $f'$ is $(j,\rho)$-equivariant and $\rho(S)$ preserves $\partial B$.
By Remark~\ref{rem:pathlength}.(3) and $2$-subconvexity, the set $f'(\mathcal{N})$ is contained in the $2\varepsilon R$-neighborhood of $\partial B$.
If the radius of $\partial B$ is $\leq 1/3$, then as soon as $\varepsilon<\frac{1}{24R}$ we have
$$d(f'(p),f'(q)) \,\leq\, \frac{2}{3} + 4\varepsilon R \,\leq\, \frac{5}{6}\,d(p,q)$$
for all $p,q\in\mathcal{N}$ with $d(p,q)\geq 1$, hence $\Lip(f')\leq 5/6<1$.
We now assume that the radius $r$ of $\partial B$ is $>1/3$ (possibly very large!).
There exists a universal constant $\eta>0$ such that the closest-point projection onto any sphere of~$\HH^n$ of radius $>1/3$ is $2$-Lipschitz on the $\eta$-neighborhood (inner and outer) of this sphere.
In particular, if $\varepsilon\leq\frac{\eta}{2R}$, which we shall assume from now on, then the projection onto $\partial B$ is $2$-Lipschitz on the set $f'(\mathcal{N})$.

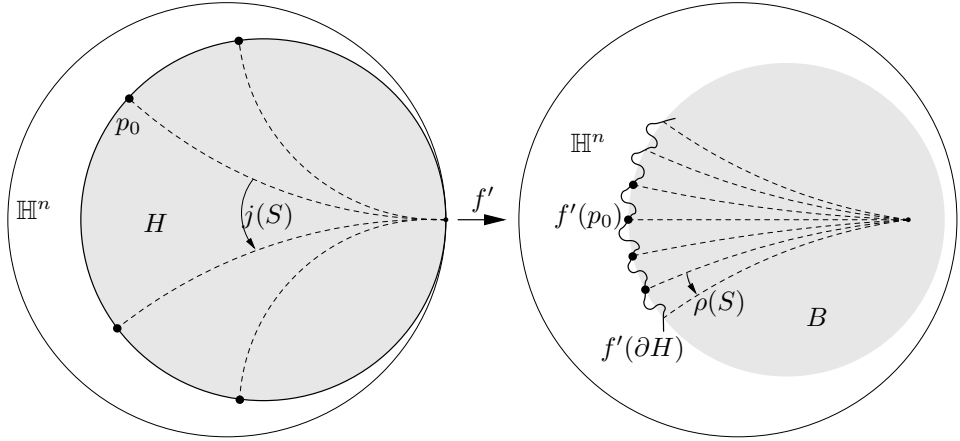
\begin{figure}[h!]
\begin{center}
\labellist
\small\hair 2pt
\pinlabel{$\HH^n$} at 20 170
\pinlabel{$\HH^n$} at 430 220
\pinlabel{$p_0$} at 90 232
\pinlabel{$H$} at 110 160
\pinlabel{$j(S)$} at 193 165
\pinlabel{$f'$} at 352 178
\pinlabel{$f'(\partial H)$} at 470 66
\pinlabel{$f'(p_0)$} at 430 165
\pinlabel{$\rho(S)$} at 527 100
\pinlabel{$B$} at 599 90
\endlabellist
\includegraphics[width=12.5cm]{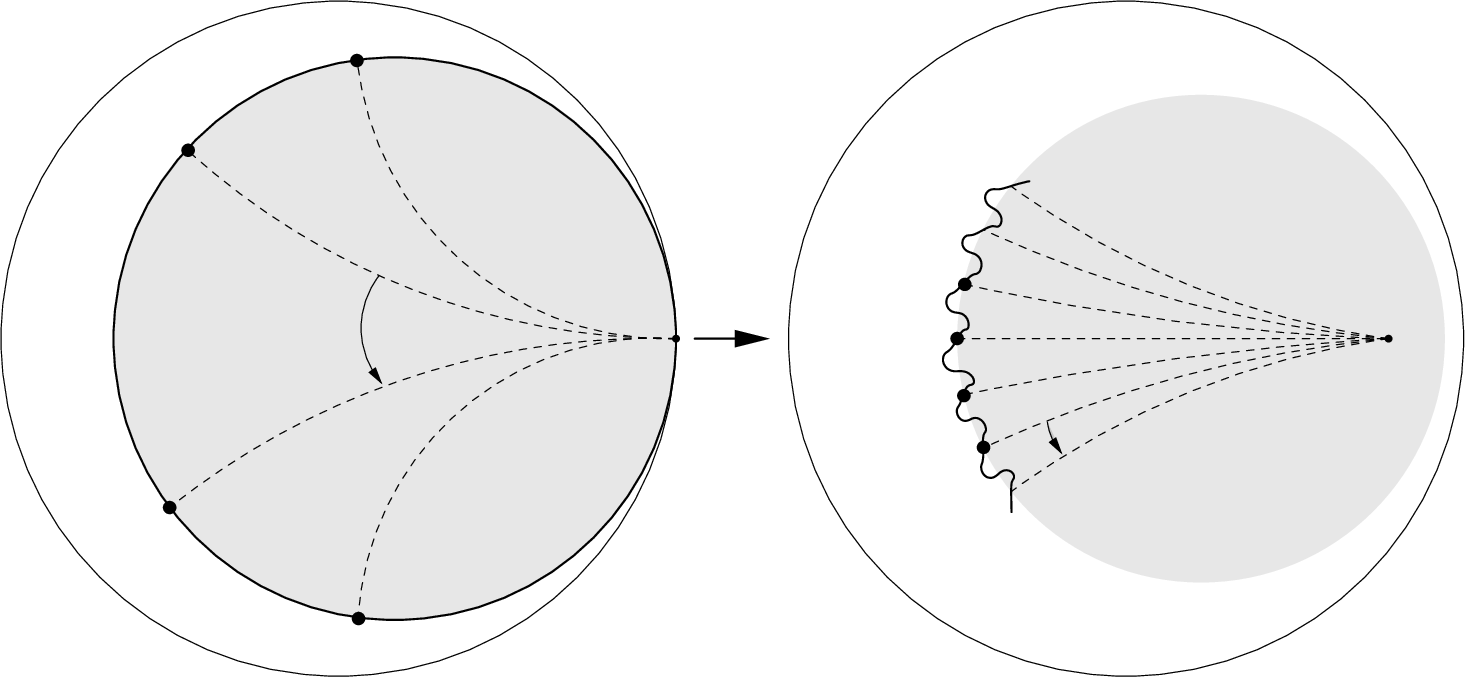}
\caption{An equivariant map $f'$, contracting at small scale, taking a horosphere to (or near) a sphere, is contracting at all scales.}
\label{fig:H}
\end{center}
\end{figure}

Let $x,y\in\partial B$ be the respective projections of $f'(p),f'(q)$; the distances $d(x,f'(p))$ and $d(y,f'(q))$ are bounded from above by $2\varepsilon R$.
Let $d_{\partial B}(x,y)$ be the length of the shortest path from $x$ to~$y$ that is contained in the sphere~$\partial B$.
The formulas \eqref{eqn:circlearclength} and \eqref{eqn:circle} yield
\begin{equation}\label{eqn:distsphere}
d(x,y) = 2\,\arcsinh\left(\sinh(r) \cdot \sin\left(\frac{d_{\partial B}(x,y)}{2\sinh(r)}\right)\right).
\end{equation}
On the other hand, by $2$-subconvexity, we can find a path $\omega$ from $p$ to $q$ in $\mathcal{N}$ of length at most $2\,d_{\partial H}(p,q)$. Then $d_{\partial B}(x,y)$ is bounded from above by the length of the projection of the path $f'(\omega)$ to $\partial B$, hence, by Remark~\ref{rem:pathlength}.(3),
\begin{equation}\label{eqn:distineq}
d_{\partial B}(x,y) \leq 4\varepsilon\,d_{\partial H}(p,q).
\end{equation}
Using $\sin(t)\leq\min\{1,t\}$ for $t\geq 0$, it follows from \eqref{eqn:distsphere} and \eqref{eqn:distineq} that
\begin {eqnarray*}
d(f'(p),f'(q)) & \leq & d(f'(p),x) + d(x,y) + d(y,f'(q))\\
& \leq & \min\big\{ 2r, 2\,\arcsinh(d_{\partial B}(x,y)/2)\big\} + 4\varepsilon R\\
& \leq & \min\big\{2r, 2\,\arcsinh\big(2\varepsilon\,d_{\partial H}(p,q)\big)\big\} + 4\varepsilon R.
\end{eqnarray*}
Comparing with \eqref{eqn:horoline}, we see that if $\varepsilon$ is smaller than some constant depending only on~$R$, then
$$d(f'(p),f'(q)) < d(p,q)$$
for all $p,q\in\mathcal{N}$ with $d(p,q)\geq 1$.
Since $d(f'(p),f'(q))$ is bounded independently of $p$ and~$q$, the ratio $d(f'(p),f'(q))/d(p,q)$ is uniformly bounded away from~$1$ by compactness of $\mathcal{N}$ modulo $j(S)$.
This proves that $\Lip_{\mathcal{N}}(f')<\nolinebreak 1$.
\end{proof}

\begin{proof}[Proof of Lemma~\ref{lem:subconvex}]
Fix $1\leq\ell\leq c$, where $c$ is still the number of cusps.

\smallskip
\noindent
$\bullet$ \textbf{Subconvexity for $\partial H_{\ell}\cap N$.}
We first prove that, up to replacing $H_{\ell}$ with some smaller, concentric horoball, the set $\partial H_{\ell}\cap N$ is $2$-subconvex~in~$\partial H_{\ell}$.
This will prove Lemma~\ref{lem:subconvex} when $j_0$ has a cusp of rank $<n-2$, since in that case $j_k=j$ and $H_{\ell}^k=H_{\ell}$ for all~$k$ by assumption.

The stabilizer $S_{\ell}\subset\Gamma_0$ of $H_{\ell}$ under~$j$ has a finite-index normal subgroup~$S'$ isomorphic to~$\Z^m$ for some $0<m<n$ (see Section~\ref{subsec:geo-finiteness}).
Consider the upper half-space model $\R^{n-1}\times\R_+^{\ast}$ of~$\HH^n$, so that $\partial_{\infty}\HH^n$ identifies with $\R^{n-1}\cup\{\infty\}$.
We may assume that $H_{\ell}$ is centered at infinity, so that $\partial H_{\ell}=\R^{n-1}\times\{ b\} $ for some $b>0$.
Let $\Omega$ be the convex hull of $\Lambda_{j(\Gamma_0)}\smallsetminus\{\infty\}$ in~$\R^{n-1}$, where $\Lambda_{j(\Gamma_0)}$ is the limit set of $j(\Gamma_0)$.
The group $j(S')$ acts on~$\R^{n-1}$ by Euclidean isometries and there exists an $m$-dimensional affine subspace $V\subset\Omega$, preserved by $j(S')$, on which $j(S')$ acts as a lattice of translations (see Section~\ref{subsec:geo-finiteness}).

We claim that $N$ contains $V\times [b_0,+\infty)$ for some $b_0>0$.
Indeed, since $V\subset\Omega$, some point $p_0\in V\times\R_+^{\ast}\subset\HH^n$ belongs to~$N$.
The convex hull in~$\HH^n$ of the orbit $j(S')\cdot p_0$ is also contained in~$N$.
This convex hull contains all the $j(S')$-translates of the (compact) convex hull of
$$\big\{ j(\gamma_1^{\varepsilon_1}\dots \gamma_m^{\varepsilon_m})\cdot p_0\, | \, (\varepsilon_1,\dots,\varepsilon_m) \in \{0,1\}^m \big\} ,$$
where $(\gamma_1,\dots,\gamma_m)$ is a generating subset of~$S'$; the union $X$ of these $j(S')$-translates projects vertically onto the whole of~$V$ and has bounded height since $j(S')$ preserves the horospheres centered at~$\infty$.
Then $N$ contains $V\times [b_0,+\infty)$ where $b_0>0$ is the maximal height of~$X$.

Up to replacing $H_{\ell}$ with some smaller, concentric horoball, we may assume that $b\geq\max\{ b_0,7\delta\}$, where $\delta>0$ is the Euclidean diameter of $j(S')\backslash\Omega$.
Let us show that $\partial H_{\ell}\cap N$ is then $2$-subconvex.
Consider $p,q\in\partial H_{\ell}\cap N$, with respective orthogonal projections $\zeta_p,\zeta_q$ to~$\R^{n-1}$.
We have $d_{\partial H_{\ell}}(p,q)=d_{\R^{n-1}}(\zeta_p,\zeta_q)/b$.

Suppose $d_{\R^{n-1}}(\zeta_p,\zeta_q)\leq 6\delta$.
By definition of~$\delta$, we can find a point $\zeta\in\nolinebreak\Lambda_{j(\Gamma_0)}\smallsetminus\{ \infty\} \subset\R^{n-1}$ with $d_{\R^{n-1}}(\zeta,\zeta_p)\leq\delta$.
The hyperbolic triangle $(p,q,\zeta)$ is contained in~$N$.
Since $b\geq 7\delta$, both edges $(p,\zeta]$ and $(q,\zeta]$ lie outside $H_{\ell}=\R^{n-1}\times[b,+\infty)$. It follows that the intersection of this triangle $(p,q,\zeta)$ with $\partial H_{\ell}$ is an arc of Euclidean circle from $p$ to~$q$, of angular measure $\leq \pi$, and hence~has~Euclidean length at most $\frac{\pi}{2}\,d_{\partial H_{\ell}}(p,q)\leq 2\,d_{\partial H_{\ell}}(p,q)$.

Suppose $d_{\R^{n-1}}(\zeta_p,\zeta_q)\geq 6\delta$.
Since $\zeta_p,\zeta_q\in\Omega$, by definition of~$\delta$ we can find points $p',q'$ in $N\cap (V\times\{ b\})$ whose orthogonal projections $\zeta_{p'},\zeta_{q'}$ to~$\R^{n-1}$ satisfy
$$d_{\R^{n-1}}(\zeta_p,\zeta_{p'}) \leq \delta \quad\quad\mathrm{and}\quad\quad d_{\R^{n-1}}(\zeta_q,\zeta_{q'}) \leq \delta.$$
Then $d_{\partial H_{\ell}}(p,p')=d_{\R^{n-1}}(\zeta_p,\zeta_{p'})/b\leq\delta/b$, and similarly $d_{\partial H_{\ell}}(q,q')\leq\delta/b$.
As above, there is an arc of Euclidean circle from $p$ to~$p'$ in $\partial H_{\ell}\cap N$, of length at most $2\,d_{\partial H}(p,p')\leq 2\delta/b$.
Similarly, there is an arc of Euclidean circle from $q'$ to~$q$ in $\partial H_{\ell}\cap N$, of Euclidean length $\leq 2\delta/b$.
Concatenating these arcs with the Euclidean segment $[p',q']\subset V\times\{b\}$, which is contained in $\partial H_{\ell}\cap N$ and has Euclidean length $b^{-1}\,d_{\R^{n-1}}(\zeta_{p'},\zeta_{q'})$, we find a path from $p$ to~$q$ in $\partial H_{\ell}\cap N$ of Euclidean length at most 
$$\frac{d_{\R^{n-1}}(\zeta_{p'},\zeta_{q'})+4\delta}{b} \leq \frac{d_{\R^{n-1}}(\zeta_p,\zeta_q)+6\delta}{b} \leq 2\,d_{\partial H_{\ell}}(p,q).$$
This proves that $\partial H_{\ell}\cap N$ is $2$-subconvex in~$\partial H_{\ell}$.

\smallskip
\noindent
$\bullet$ \textbf{Convexity for $\partial H_{\ell}^k\cap N_k$ in the case of cusps of rank $\geq n-2$.}
Finally, we suppose that all cusps of~$j_0$ have rank $\geq n-2$, in which case the representation $j_k$ is allowed to vary with~$k$.
Recall that the cusp thickness of $j_k(\Gamma_0)\backslash\HH^n$ at $\partial H_k^{\ell}$ is bounded by some constant $\Theta>0$ independent of $\ell$ and~$k$.
If we replace every horoball $H_k^{\ell}$ with the smaller, concentric horoball at distance $\log(3 \Theta)$ from $\partial H_k^{\ell}$ (and correspondingly enlarge the compact set $\mathcal{C}\subset \HH^n$), we obtain new horoballs $H_k^{\ell}$ still satisfying the list of six properties \eqref{eqn:triche}, such that the cusp thickness of $j_k(\Gamma_0)\backslash\HH^n$ at $\partial H_k^{\ell}$ is $\leq 1/3$\ for all $\ell$ and~$k$.
Then $\partial H_{\ell}^k\cap N_k$ is convex in $\partial H_{\ell}^k$ by Lemma~\ref{lem:convexcusp}, hence in particular $2$-subconvex.
\end{proof}

\subsection{The constant $C(j,\rho)$ for nonreductive~$\rho$}\label{subsec:nonred}

In order to prove conditions (2), (3), (3)', (3)'' of Proposition~\ref{prop:contpropertiesC} (in Section~\ref{subsec:C-highcontinuous}), we shall rely on the existence of a maximally stretched lamination for $C(j,\rho)\geq 1$, given by Theorem~\ref{thm:lamin}.
However, Theorem~\ref{thm:lamin} assumes that the space $\F^{j,\rho}$ of equivariant maps realizing the best Lipschitz constant $C(j,\rho)$ is nonempty: this holds for reductive~$\rho$ (Lemma~\ref{lem:Fnonempty}), but may fail otherwise (see Section~\ref{ex:nonreductive2}).
In order to deal with nonreductive~$\rho$, we first establish the following lemma, which extends Lemma~\ref{lem:nonred-lip-cc}.

\begin{lemma}\label{lem:nonred-lip}
For any pair of representations $(j,\rho)\in\Hom(\Gamma_0,G)^2$ with $j$ geometrically finite,
$$C(j,\rho) = C(j,\rho^{\mathrm{red}}),$$
unless the representation $\rho$ is not cusp-deteriorating and $C(j,\rho^{\mathrm{red}})<1$, in which case $C(j,\rho)=1$.
\end{lemma}

Here $\rho^{\mathrm{red}}\in\Hom(\Gamma_0,G)$ is the ``reductive part'' of~$\rho$, defined in Section~\ref{subsubsec:lowersemicont}: if $\rho$ is nonreductive, then the group $\rho^{\mathrm{red}}(\Gamma_0)$ preserves some geodesic line of $\HH^n$ with an endpoint in $\partial_{\infty}\HH^n$ equal to the fixed point of $\rho(\Gamma_0)$.
Since $\rho^{\mathrm{red}}$ is well defined up to conjugation, the constant $C(j,\rho^{\mathrm{red}})$ is well defined by Remark~\ref{rem:CLipconj}.
If $\rho$ is reductive, then $\rho^{\mathrm{red}}:=\rho$.

\begin{proof}[Proof of Lemma~\ref{lem:nonred-lip}]
We may assume that $\rho$ is nonreductive, with fixed point $\xi\in\partial_{\infty}\HH^n$.
Then $\rho^{\mathrm{red}}$ is cusp-deteriorating and preserves an oriented geodesic line $\A$ of~$\HH^n$ with endpoint~$\xi$.
If the group $j(\Gamma_0)$ is elementary and fixes a unique point in $\partial_{\infty}\HH^n$, then $C(j,\rho)=1$ by Corollary~\ref{cor:jrhononred} and $C(j,\rho^{\mathrm{red}})=0$ by Remark~\ref{rem:boundedrho}.
We now assume that we are not in this case, which means that the convex core of $M:=j(\Gamma_0)\backslash\HH^n$ is nonempty.
As in the proof of Lemma~\ref{lem:nonred-lip-cc}, by using the projection onto~$\A$ along concentric horocycles we see that
$$C(j,\rho^{\mathrm{red}}) \leq C(j,\rho),$$
and there is a sequence $(a_k)_{k\in\N^{\ast}}$ of pure translations along~$\A$, with repelling fixed point~$\xi$, such that the conjugates $\rho_k:=a_k\rho(\cdot)a_k^{-1}$ (which still fix~$\xi$) converge to~$\rho^{\mathrm{red}}$ as $k\rightarrow +\infty$.
By invariance of $C(j,\rho)$ under conjugation (Remark~\ref{rem:CLipconj}), it is sufficient to prove that
$$\limsup_{k\rightarrow +\infty}\, C(j,\rho_k) \leq
  \left \{ \begin{array}{lll}  C(j,\rho^{\mathrm{red}}) & \text{if $\rho$ is cusp-deteriorating}, \\ \max\big(1,C(j,\rho^{\mathrm{red}})\big) & \text{otherwise}.
  \end{array} \right .$$
To prove this, we use a partition-of-unity argument as in Sections \ref{subsubsec:uppersemicont} and~\ref{subsec:C<1open}.
Fix $\varepsilon>0$.
By using Proposition~\ref{prop:goodincusps} and postcomposing with the closest-point projection onto~$\A$, we can find a $(j,\rho^{\mathrm{red}})$-equivariant map $f : \HH^n\rightarrow \A$ with $\Lip(f)\leq C(j,\rho^{\mathrm{red}})+\varepsilon/2$ that is constant on neighborhoods of some horoballs $B_1,\dots,B_c$ of~$\HH^n$ whose images in $M=j(\Gamma_0)\backslash\HH^n$ are disjoint and intersect the convex core in standard cusp regions (Definition~\ref{def:standardcusp}), representing all the cusps.
We shall use~$f$ to build $(j,\rho_k)$-equivariant maps~$f_k$ with $\Lip(f_k)$ bounded from above by $\Lip(f)+\varepsilon$ or $1 +\varepsilon$, as the case may be, for all large enough~$k$.
Let $S_1,\dots,S_c\subset\Gamma_0$ be the respective stabilizers of $B_1,\dots,B_c$ under~$j$; the singleton $f(B_i)$ is fixed by $\rho(S_i)$ for all $1\leq i\leq c$.
Let also $B_{c+1},\dots,B_r$ be open balls of~$\HH^n$, each projecting injectively to $j(\Gamma_0)\backslash\HH^n$, such that $\bigcup_{i=1}^r j(\Gamma_0)\cdot B_i$ contains the preimage $N\subset\HH^n$ of the convex core of~$M$.
For $c<i\leq r$, let $f_{i,k} : j(\Gamma_0)\cdot B_i\rightarrow\HH^n$ be the $(j,\rho_k)$-equivariant map that coincides with $f$ on~$B_i$.

We first assume that $\rho$ is cusp-deteriorating.
For $1\leq i \leq c$, all the elements of $\rho(S_i)$ are elliptic, hence $\rho(S_i)$ fixes a point in~$\HH^n$ (Fact~\ref{fact:fixedpoint}).
Since it also fixes $\xi\in\partial_{\infty}\HH^n$, it fixes pointwise a full line $\A'$ with endpoint~$\xi$.
Then $\rho_k(S_i)=a_k\rho(S_i)a_k^{-1}$ fixes pointwise the line $a_k\cdot\A'$, which converges to~$\A$ as $k\rightarrow +\infty$.
In particular, we can find a sequence $(p_{i,k})_{k\in\N^{\ast}}$ that converges to the singleton $f(B_i)\in\A$ as $k\rightarrow +\infty$, with $p_{i,k}$ fixed by $\rho_k(S_i)$ for all~$k$.
For $1\leq i \leq c$ and $k\in\N^{\ast}$, let
$$f_{i,k} :\, j(\Gamma_0)\cdot B_i \longrightarrow \HH^n$$
be the $(j,\rho_k)$-equivariant map that is constant equal to~$p_{i,k}$ on the horoball~$B_i$.
Let $(\psi_i)_{1\leq i\leq r}$ be a Lipschitz partition of unity subordinated to the covering $(j(\Gamma_0)\cdot B_i)_{1\leq i\leq r}$ of~$N$, and let $L:=\max_{1\leq i\leq r} \Lip(\psi_i)$.
By Lemma~\ref{lem:partofunity}, the $(j,\rho_k)$-equivariant map
$$f_k := \sum_{i=1}^r \psi_i\,f_{i,k} :\ N \longrightarrow \HH^n$$
satisfies
$$\Lip_p(f_k) \leq r L R_{p,k} + \Lip_p(f)$$
for all $p\in N$, where the $j(\Gamma_0)$-invariant function
$$p \longmapsto R_{p,k} := \max_{i,i'}\ d\big(f_{i,k}(p),f_{i',k}(p)\big)$$
converges uniformly to~$0$ for $p\in\bigcup_{i=1}^r B_i$, as $k\rightarrow +\infty$.
For large enough~$k$ this yields $\Lip_N(f_k)\leq\Lip(f)+\varepsilon/2$ by \eqref{eqn:supLip}, hence
$$C(j,\rho_k) \leq C(j,\rho^{\mathrm{red}})+\varepsilon$$
by Lemma~\ref{lem:proj}.
Letting $\varepsilon$ go to~$0$, we obtain
$\limsup_k C(j,\rho_k) \leq C(j,\rho^{\mathrm{red}})$
as desired.

Suppose now that $\rho$ is \emph{not} cusp-deteriorating.
We proceed as in the cusp-deteriorating case, but work with the union of balls $\bigcup_{c<i\leq r} j(\Gamma_0)\cdot B_i$ instead of the union of balls and horoballs~$\bigcup_{1\leq i\leq r} j(\Gamma_0)\cdot B_i$.
Let $(\psi_i)_{c< i\leq r}$ be a Lipschitz partition of unity of $N':=N\smallsetminus \bigcup_{1\leq\ell\leq c} j(\Gamma_0)\cdot B_{\ell}$ subordinated to the covering $(j(\Gamma_0)\cdot B_i)_{c< i\leq r}$, and let $L:=\max_{c< i\leq r} \Lip(\psi_i)$.
As in the cusp-deteriorating case, by Lemma~\ref{lem:partofunity}, the $(j,\rho_k)$-equivariant map
$$f'_k := \sum_{c<i\leq r} \psi_i\,f_{i,k} :\ N' \longrightarrow \HH^n$$
satisfies
$$\Lip_p(f'_k) \leq \Lip_p(f) + \varepsilon/2$$
for all $p\in N'$ when $k$ is large enough.
In particular, for $1\leq\ell\leq c$, since $f$ is constant on a neighborhood of the horoball~$B_{\ell}$, we obtain $\Lip_p(f'_k)\leq\varepsilon/2$ for all $p\in N\cap \partial B_{\ell}$.
It is sufficient to prove that
\begin{equation}\label{eqn:Liphorospherefk}
\Lip_{N\cap \partial B_{\ell}}(f'_k) \leq 1
\end{equation}
for all $1\leq\ell\leq c$, since Theorem~\ref{thm:Kirszbraunequiv} (or Proposition~\ref{prop:amenableKirszbraun}) then lets us extend $f'_k|_{N\cap \partial B_{\ell}}$ to a $1$-Lipschitz, $(j|_{S_{\ell}},\rho_k|_{S_{\ell}})$-equivariant map $(B_{\ell}\cup\partial B_{\ell})\cap N \rightarrow\HH^n$.
We can then extend $f'_k$ to the orbit $j(\Gamma_0)\cdot (B_{\ell}\cup\partial B_{\ell})\cap N$ in a $(j,\rho_k)$-equivariant way.
Piecing together these maps for varying $1\leq\ell\leq c$, and taking $f'_k$ on~$N'$, we then obtain a $(j,\rho_k)$-equivariant map $f_k : \HH^n\rightarrow\HH^n$ with $\Lip(f_k)\leq\max(1,\Lip(f)+\varepsilon/2)$ for all large enough~$k$ (using \eqref{eqn:supLip}).
Letting $\varepsilon$ go to~$0$, we obtain $\limsup_k C(j,\rho_k)\leq\max(1,C(j,\rho^{\mathrm{red}}))$, as desired.
To prove \eqref{eqn:Liphorospherefk}, it is sufficient to establish the following analogue of Lemma~\ref{lem:horosph-sph}, which together with Lemma~\ref{lem:subconvex} completes the proof of Lemma~\ref{lem:nonred-lip}.
\end{proof}

\begin{lemma}\label{lem:horosph-horosph}
Let $S$ be a discrete group.
For any $R>0$, there exists $\varepsilon>0$ with the following property: if $(j,\rho)\in\Hom(S,G)^2$ is a pair of representations with $j$ injective such that
\begin{itemize}
  \item the group $j(S)$ is discrete and preserves a horoball $H$ of~$\HH^n$,
  \item the group $\rho(S)$ has a fixed point in~$\partial_{\infty}\HH^n$,
  \item there exists a closed, $j(S)$-invariant, $2$-subconvex set $\mathcal{N}\subset \partial H$ such that the quotient $j(S)\backslash \mathcal{N}$ has (Euclidean) diameter $\leq R$,
\end{itemize}
then any $(j,\rho)$-equivariant map $f' : \mathcal{N}\rightarrow\HH^n$ satisfying $\Lip_p(f')\leq\varepsilon$ for all $p\in \mathcal{N}$ satisfies $\Lip(f')\leq 1$.
\end{lemma}

\begin{proof}
We proceed as in the proof of Lemma~\ref{lem:horosph-sph}, but the sphere $\partial B$ will now be a horosphere.
Fix $R>0$ and let $j,\rho,H,\mathcal{N}$ be as in the statement.
Consider a $(j,\rho)$-equivariant map $f' : \mathcal{N}\rightarrow\HH^n$ such that $\Lip_p(f')\leq\varepsilon$ for all $p\in\mathcal{N}$, for some $\varepsilon>0$.
Let us show that if $\varepsilon$ is smaller than some constant independent of~$f'$, then $\Lip(f')\leq 1$.
As in the proof of Lemma~\ref{lem:horosph-sph}, if $\varepsilon$ is smaller than some universal constant, then $d(f'(p),f'(q))\leq d(p,q)$ for all $p,q\in\mathcal{N}$ with $d(p,q)\leq 1$.
We now consider $p,q\in\mathcal{N}$ with $d(p,q)\geq 1$.
Let $\partial B$ be a horosphere centered at the fixed point of~$\rho(S)$ in $\partial_{\infty}\HH^n$, containing $f'(p_0)$ for some $p_0\in \mathcal{N}$.
As in the proof of Lemma~\ref{lem:horosph-sph}, the set $f'(\mathcal{N})$ is contained in the $2\varepsilon R$-neighborhood of $\partial B$.
We now use the existence of a universal constant $\eta>0$ such that the closest-point projection onto any horosphere of~$\HH^n$ is $2$-Lipschitz on the $\eta$-neighborhood (inner and outer) of this horosphere.
In particular, if $\varepsilon\leq\frac{\eta}{2R}$, which we shall assume from now on, then the projection onto $\partial B$ is $2$-Lipschitz on the set $f'(\mathcal{N})$.

Denoting by $x,y\in\partial B$ the projections of $f'(p),f'(q)$, the (in)equalities \eqref{eqn:horoline} and \eqref{eqn:distineq} still hold, but \eqref{eqn:distsphere} becomes
$$d(x,y) = 2\,\arcsinh\left(\frac{d_{\partial B}(x,y)}{2}\right),$$
where $d_{\partial B}$ is the natural Euclidean metric on $\partial B$.
We obtain
\begin{eqnarray*}
d(f'(p),f'(q)) & \leq & d(f'(p),x) + d(x,y) + d(y,f'(q))\\
& \leq & 2\,\arcsinh\big(2\varepsilon\,d_{\partial H}(p,q)\big) + 2\varepsilon R.
\end{eqnarray*}
Comparing with \eqref{eqn:horoline} we see that if $\varepsilon$ is small enough then $d(f'(p),f'(q))\leq d(p,q)$ for all $p,q\in\partial H\cap N$ with $d(p,q)\geq 1$.
Hence, $\Lip(f')\leq 1$.
\end{proof}

\subsection{Semicontinuity for $C(j,\rho)\geq 1$ in the general geometrically finite case}\label{subsec:C-highcontinuous}

We now complete the proof of Proposition~\ref{prop:contpropertiesC}.
Condition~(1) when all the cusps of~$j_0$ have rank $\geq n-2$ and condition~(1') in general have already been proved in Section~\ref{subsec:C<1open}.
We now show that for pairs $(j,\rho)$ of representations with $j$ geometrically finite representations of fixed cusp type,
\begin{itemize}
  \item[(2)] the condition $1<C$ is open,
  \item[(3')] the function $(j,\rho)\mapsto C(j,\rho)$ is lower semicontinuous on the set of pairs where $1\leq C$,
  \item[(3'')] it is upper semicontinuous on the set of pairs where $1\leq C<+\infty$ when either all the cusps of~$j_0$ have rank $\geq n-2$ (for instance $n\leq 3$) or the representation~$j$ is constant.
\end{itemize}
Upper semicontinuity on the set of pairs where $1\leq C<+\infty$ does not hold in general in dimension $n\geq 4$: see Section~\ref{ex:dim4upper}. 
The condition $C(j,\rho)=+\infty$ is open in $\Hom(\Gamma_0,G)_{j_0}\times\Hom(\Gamma_0,G)$ by Lemma~\ref{lem:C<infty}, hence in (2) and~(3') we may actually restrict to $1\leq C < +\infty$.

Let $(j_k,\rho_k)_{k\in\N^{\ast}}$ be a sequence of elements of $\Hom(\Gamma_0,G)_{j_0}\times\Hom(\Gamma_0,G)$ converging to an element $(j,\rho)\in\Hom(\Gamma_0,G)_{j_0}\times\Hom(\Gamma_0,G)$ with $C(j,\rho)<+\infty$.
It is sufficient to prove the following two statements:
\begin{enumerate}
  \item[(A)] if $C(j,\rho)>1$, then $C(j,\rho)\leq\liminf_k C(j_k, \rho_k)$,
  \item[(B)] if $C^{\ast}:=\limsup_k C(j_k,\rho_k)>1$ and if either all the cusps of~$j_0$ have rank $\geq n-2$ or $j_k=j$ for all $k\in\N^{\ast}$, then $C(j,\rho)\geq C^{\ast}$.
\end{enumerate}

If $\rho$ is reductive, then (A) is an easy consequence of Corollary~\ref{cor:CC'} (here $E(j,\rho)\neq\emptyset$ by Corollary~\ref{cor:Enonempty}, in which case Corollary~\ref{cor:CC'} has been proved in Section~\ref{subsec:coroElamin}):
namely, for any $\varepsilon>0$ there is an element $\gamma\in\Gamma_0$ with $j(\gamma)$ hyperbolic such that
$$\frac{\lambda(\rho(\gamma))}{\lambda(j(\gamma))} \geq C(j,\rho)-\varepsilon.$$
If $k$ is large enough, then $\lambda(j_k(\gamma))$ is hyperbolic and $\lambda(\rho_k(\gamma))/\lambda(j_k(\gamma))\geq C(j,\rho)-2\varepsilon$ by continuity of~$\lambda$, hence $C(j_k,\rho_k)\geq C(j,\rho)-2\varepsilon$ by \eqref{eqn:ClambdaCLip}.
We conclude by letting $\varepsilon$ tend to~$0$.
If $\rho$ is nonreductive, then $C(j,\rho)>1$ entails $C(j,\rho^{\mathrm{red}})=C(j,\rho)$ by Lemma~\ref{lem:nonred-lip}, and the $\rho_k$ have conjugates converging to $\rho^{\mathrm{red}}$ (see the end of Section~\ref{subsubsec:lowersemicont}), so we just apply the reductive case to obtain~(A).

To prove~(B), suppose that $C^{\ast}>1$ and that either all the cusps of~$j_0$ have rank $\geq n-2$ or $j_k=j$ for all $k\in\N^{\ast}$.
Up to passing to a subsequence, we may assume $C(j_{k},\rho_{k})>1$ for all $k\in\N^{\ast}$ and $C(j_{k},\rho_{k})\rightarrow C^{\ast}$.
Then
$$C(j_{k},\rho_{k}^{\mathrm{red}}) = C(j_{k},\rho_{k})$$
for all $k\in\N^{\ast}$ by Lemma~\ref{lem:nonred-lip}.
We now use Theorem~\ref{thm:lamin}, and either Proposition~\ref{prop:margulis} (if $j_k=j$) or Corollary~\ref{cor:uniform-thick} (if all the cusps of~$j_0$ have rank $\geq n-2$): in either case we obtain that the stretch locus $E(j_{k},\rho_{k}^{\mathrm{red}})$ is a (nonempty) geodesic lamination admitting a fundamental domain that remains in some \emph{compact} subset of~$\HH^n$, independent of~$k$.
This implies, up to passing to a subsequence, that $E(j_{k},\rho_{k}^{\mathrm{red}})$ converges to some (nonempty) $j(\Gamma_0)$-invariant geodesic lamination~$\LL$, with a compact image in $j(\Gamma_0)\backslash\HH^n$.
For any $\varepsilon>0$, a closed curve $\varepsilon$-nearly carried by~$\LL$ is also nearly carried by $E(j_{k},\rho_{k}^{\mathrm{red}})$ and will give (as in the proof of Lemma~\ref{lem:ClambdaCLipLipf}) an element $\gamma\in\Gamma_0$ such that $j_{k}(\gamma)$ is hyperbolic and
$$\frac{\lambda(\rho_{k}(\gamma))}{\lambda(j_{k}(\gamma))} = \frac{\lambda(\rho_{k}^{\mathrm{red}}(\gamma))}{\lambda(j_{k}(\gamma))} \geq C(j_{k},\rho_{k}^{\mathrm{red}}) - \varepsilon$$
for all large enough~$k$. The right hand side converges to $C^{\ast}-\varepsilon$, hence
by continuity of~$\lambda$,
$$\frac{\lambda(\rho(\gamma))}{\lambda(j(\gamma))} \geq C^{\ast}-\varepsilon,$$
therefore $C(j,\rho)\geq C^{\ast}-\varepsilon$ by \eqref{eqn:ClambdaCLip} (in particular, $C^{\ast}<+\infty$). We conclude by letting $\varepsilon$ tend to~$0$.

This completes the proof of Proposition~\ref{prop:contpropertiesC}.

\section{Application to properly discontinuous actions on $G=\PO(n,1)$}\label{sec:properactions}

In this section we prove the results of Section~\ref{subsec:introquotientsG} on the geometrically finite quotients of $G:=\PO(n,1)$, namely Theorem~\ref{thm:adm} (properness criterion) and Theorems \ref{thm:deformcompact} and~\ref{thm:deform} (deformation).
We adopt the notation and terminology of Section~\ref{subsec:introquotientsG}, and proceed as in \cite{kasPhD}.
Note that all the results remain true if $G$ is replaced by $\OO(n,1)$, $\SO(n,1)$, or $\SO(n,1)_0$.

In Section~\ref{subsec:Cmu} we start by introducing a constant $C_{\mu}(j,\rho)$, which we use in Section~\ref{subsec:implicationsproperness} to state a refinement of Theorem~\ref{thm:adm}.
This refinement is proved in Sections \ref{subsec:proofadmred} and~\ref{subsec:proofadmnonred}.
Before that, in Section~\ref{subsec:Cartanproj} we discuss the connection with the general theory of properly discontinuous actions on reductive homogeneous spaces, and in Section~\ref{subsec:topology} we make two side remarks.
Section~\ref{subsec:proofdeform} is devoted to the proof of Theorems \ref{thm:deformcompact} and~\ref{thm:deform}, and Section~\ref{subsec:(G,X)struct} to their interpretation in terms of completeness of geometric structures.

\subsection{The constant $C_{\mu}(j,\rho)$}\label{subsec:Cmu}

We shall refine Theorem~\ref{thm:adm} by characterizing properness, not only in terms of the constants $C(j,\rho)$ of \eqref{eqn:defC} and $C'(j,\rho)$ of \eqref{eqn:defClambda}, but also in terms of a third constant $C_{\mu}(j,\rho)$.
We start by introducing this constant.

Fix a basepoint $p_0\in\HH^n$ and let $\mu : G\rightarrow\R_+$ be the displacement function relative to~$p_0$:
\begin{equation}\label{eqn:defmu}
\mu(g) := d(p_0,g\cdot p_0)
\end{equation}
for all $g\in G$.
This function is continuous, proper, and surjective; we shall see in Section~\ref{subsec:Cartanproj} that it corresponds to a \emph{Cartan projection} of~$G$.
Note that $\mu(g^{-1})=\mu(g)$ for all $g\in G$ because $G$ acts on~$\HH^n$ by isometries.
By the triangle inequality,
\begin{equation}\label{eqn:triangineqmu}
\mu(gg') \leq \mu(g) + \mu(g'),
\end{equation}
and
\begin{equation}\label{eqn:lambdaleqmu}
\lambda(g) \leq \mu(g)
\end{equation}
for all $g,g'\in G$, where $\lambda : G\rightarrow\R_+$ is the translation length function of \eqref{eqn:deflambda}.
For hyperbolic~$g$, the function $k\mapsto\mu(g^k)$ grows linearly because $\mu(g^k)-k\lambda(g)$ is bounded (for instance by twice the distance from $g\cdot\nolinebreak p_0$ to the translation axis $\A_g$ of~$g$).
For parabolic~$g$, the function $k\mapsto\mu(g^k)$ grows logarithmically (Lemma~\ref{lem:disthorosphere}), while for elliptic~$g$ it is bounded.
Therefore,
\begin{equation}\label{eqn:lambdalimmu}
\lambda(g) = \lim_{k\rightarrow +\infty} \frac{1}{k}\,\mu(g^k)
\end{equation}
for all $g\in G$.

For any discrete group~$\Gamma_0$ and any pair $(j,\rho)\in\Hom(\Gamma_0,G)^2$ of representations, we denote by $C_{\mu}(j,\rho)$ the infimum of constants $t\geq 0$ for which the set $\{ \mu(\rho(\gamma))-t\,\mu(j(\gamma))\,|\,\gamma\in\Gamma_0\} $ is bounded from above, \ie for which $\mu(\rho(\gamma))\leq t\, \mu(j(\gamma))+O(1)$ as $\gamma$ ranges over $\Gamma_0$.
Note that
\begin{equation}\label{eqn:ClambdamuLip}
C'(j,\rho) \leq C_{\mu}(j,\rho) \leq C(j,\rho).
\end{equation}
Indeed, the left-hand inequality follows from \eqref{eqn:lambdalimmu}.
The right-hand inequality follows from the fact that for any $(j,\rho)$-equivariant map $f : \HH^n\rightarrow\HH^n$ and any $\gamma\in\Gamma_0$,
\begin{eqnarray*}
\mu(\rho(\gamma)) & = & d(p_0,\rho(\gamma)\cdot p_0)\\
& \leq & d\big(f(p_0),\rho(\gamma)\cdot f(p_0)\big) + 2\,d\big(p_0,f(p_0)\big)\\
& = & d\big(f(p_0),f(j(\gamma)\cdot p_0)\big) + 2\,d\big(p_0,f(p_0)\big)\\
& \leq & \Lip(f)\,d(p_0,j(\gamma)\cdot p_0) + 2\,d\big(p_0,f(p_0)\big) \\
&=& \Lip(f) \, \mu(j(\gamma)) +  2\,d\big(p_0,f(p_0)\big).
\end{eqnarray*}

\subsection{A refinement of Theorem~\ref{thm:adm}}\label{subsec:implicationsproperness}

Let $\Gamma_0$ be a discrete group.
In Sections \ref{subsec:proofadmred} and~\ref{subsec:proofadmnonred}, we shall refine Theorem~\ref{thm:adm} by establishing the following implications for any pair $(j,\rho)\in\Hom(\Gamma_0,G)^2$ with $j$ geometrically finite.
We refer to Definitions~\ref{def:typedet} and~\ref{def:admissible} for the notions of cusp-deterioration and left admissibility; recall that any $\rho$ is cusp-deteriorating if $j$ is convex cocompact.

$$\xymatrix{
{\scriptsize\circled{1}}\quad C(j,\rho)<1 
\ar@{<=>}[d] 
\ar@{=>}[r] & 
C(j,\rho^{\mathrm{red}})<1 
\ar@/^-2pc/@{=>}[l]_{\text{if \circled{5} holds}} \quad {\scriptsize\circled{i}} 
\ar@{<=>}[d]  \\
{\scriptsize\circled{2}} \quad (j,\rho)\text{ left admissible} 
\ar@{<=>}[d] 
\ar@{=>}[r] &
(j,\rho^{\mathrm{red}})\text{ left admissible} \quad{\scriptsize\circled{ii}} 
\ar@{<=>}[d]  \\
{\scriptsize\circled{3}} \quad C_{\mu}(j,\rho)<1 
\ar@{<=>}[d] 
\ar@{=>}[r] & 
C_{\mu}(j,\rho^{\mathrm{red}})<1 \quad {\scriptsize\circled{iii}} 
\ar@{<=>}[d]  \\
{\scriptsize\circled{4}} \quad C'(j,\rho)<1 \text{ and } \rho \text{ cusp-deteriorating} 
\ar@{=>}[d] 
\ar@{=>}[r] & 
C'(j,\rho^{\mathrm{red}})<1 \quad {\scriptsize\circled{iv}} 
\ar@{=>}[d]  \\
{\scriptsize\circled{5}} \quad \rho\text{ cusp-deteriorating} 
\ar@{=>}[r] & 
\rho^{\mathrm{red}}\text{ cusp-deteriorating} \quad {\scriptsize\circled{v}}
}$$

We define the ``reductive part'' $\rho^{\mathrm{red}}$ of~$\rho$ as in Section~\ref{subsubsec:lowersemicont}: in the generic case when $\rho$ is reductive (Definition~\ref{def:reductive}), we set $\rho^{\mathrm{red}}:=\rho$.
In the degenerate case when $\rho$ is nonreductive, we fix a Levi factor $\underline{M}\underline{A}$ of the stabilizer $P$ in~$G$ of the fixed point at infinity of~$\rho(\Gamma_0)$ (see Section~\ref{subsec:proofadmnonred}), denote by $\underline{\pi} : P\rightarrow\underline{M}\underline{A}$ the natural projection, and set $\rho^{\mathrm{red}}:=\underline{\pi}\circ\rho$, so that $\rho^{\mathrm{red}}$ is reductive and preserves an oriented geodesic line $\A\subset\HH^n$, depending only on~$\underline{M}\underline{A}$.

The implications $\scriptsize\circled{1}\Rightarrow\scriptsize\circled{3}$ and $\scriptsize\circled{i}\Rightarrow\scriptsize\circled{iii}\Rightarrow\scriptsize\circled{iv}$ are immediate consequences of \eqref{eqn:ClambdamuLip}, while $\scriptsize\circled{3}\Rightarrow\scriptsize\circled{4}\Rightarrow\scriptsize\circled{5}$ and $\scriptsize\circled{iii}\Rightarrow\scriptsize\circled{v}$ follow from \eqref{eqn:ClambdamuLip} and from the estimate $\mu(g^k)=2\log k+O(1)$ for parabolic~$g$ (Lemma~\ref{lem:disthorosphere}); these implications do not require any geometrical finiteness assumption on~$j$.
The implications $\scriptsize\circled{1}\Rightarrow\scriptsize\circled{i}$ and $\scriptsize\circled{i}\Rightarrow\scriptsize\circled{1}$, the latter assuming~$\scriptsize\circled{5}$, are immediate consequences of Lemma~\ref{lem:nonred-lip}.
We shall explain:
\begin{itemize}
  \item $\scriptsize\circled{3}\Rightarrow\scriptsize\circled{2}$ and $\scriptsize\circled{iii}\Rightarrow\scriptsize\circled{ii}$ and $\scriptsize\circled{2}\Rightarrow\scriptsize\circled{5}$ in Section~\ref{subsec:Cartanproj},
  \item $\scriptsize\circled{iv}\Rightarrow\scriptsize\circled{iii}$ and $\scriptsize\circled{ii}\Rightarrow\scriptsize\circled{i}$ in Section~\ref{subsec:proofadmred},
  \item $\scriptsize\circled{2}\Rightarrow\scriptsize\circled{ii}$, $\scriptsize\circled{3}\Rightarrow\scriptsize\circled{iii}$, $\scriptsize\circled{4}\Rightarrow\scriptsize\circled{iv}$, and $\scriptsize\circled{5}\Rightarrow\scriptsize\circled{v}$ in Section~\ref{subsec:proofadmnonred}.
\end{itemize}

\subsection{General theory of properly discontinuous actions and sharpness}\label{subsec:Cartanproj}

Before proving the implications above, we discuss the connection with the general theory of properly discontinuous actions on reductive homogeneous spaces.

The group~$G$ endowed with the transitive action of $G\times G$ by right and left multiplication identifies with the homogeneous space $(G\times G)/\Diag(G)$, where $\Diag(G)$ is the diagonal of $G\times G$.
Let $\underline{K}$ be the stabilizer in~$G$ of the basepoint $p_0$ of \eqref{eqn:defmu}: it is a maximal compact subgroup of $G=\PO(n,1)$, isomorphic to $\OO(n)$.
Let $\underline{A}$ be a one-parameter subgroup of~$G$ whose nontrivial elements are hyperbolic, all pure translations, with a translation axis~$\A$ passing through~$p_0$.
Choose an endpoint $\xi\in\partial_{\infty}\HH^n$ of~$\A$ and let $\underline{A}^+$ be the subsemigroup of~$\underline{A}$ sending $p_0$ into the geodesic ray $[p_0,\xi)$.
Then the \emph{Cartan decomposition} $G=\underline{K}\underline{A}^+\underline{K}$ holds: any element $g\in G$ may be written as $g=kak'$ for some $k,k'\in\underline{K}$ and a unique $a\in\underline{A}^+$ (see \cite[Th.\,IX.1.1]{hel01}).
The Cartan projection $\mu$ of \eqref{eqn:defmu} is the projection onto~$\underline{A}^+$ composed with an appropriate identification of $\underline{A}^+$ with~$\R_+$ (namely the restriction of $\lambda$ to~$\underline{A}^+$). 
Likewise, the group $G\times G$ admits the Cartan decomposition
$$G\times G = (\underline{K}\times\underline{K})(\underline{A}^+\!\times\underline{A}^+)(\underline{K}\times\underline{K})\,,$$
with Cartan projection
$$\mu_{\bullet} = \mu\times\mu\ :\ G\times G \longrightarrow \R_+\times\R_+.$$
The general \emph{properness criterion} of Benoist \cite{ben96} and Kobayashi \cite{kob96} states, in this context, that a closed subgroup~$\Gamma$ of $G\times G$ acts properly on~$G$ by right and left multiplication if and only if the set $\mu_{\bullet}(\Gamma)$ ``drifts away from the diagonal at infinity'', in the sense that for any $R>0$, there is a compact subset of $\R_+\times\R_+$ outside of which any point of $\mu_{\bullet}(\Gamma)$ lies at distance $>R$ from the diagonal of $\R_+\times\R_+$.
Consider a group $\Gamma_0^{j,\rho}$ as in \eqref{eqn:Gamma(j,rho)}, with $j$ injective and discrete.
Then the properness criterion states that $\Gamma_0^{j,\rho}$ acts properly discontinuously on~$G$ (\ie $(j,\rho)$ is admissible in the sense of Definition~\ref{def:admissible}) if and only if
\begin{equation}\label{eqn:propcrit}
\text{for any $R\geq 0$, }\quad \quad |\mu(j(\gamma)) - \mu(\rho(\gamma))| > R \quad\text{for almost all }\gamma\in\Gamma_0
\end{equation}
(\ie for all $\gamma\in\Gamma_0$ but finitely many exceptions).
In particular, this gives the implications $\scriptsize\circled{3}\Rightarrow\scriptsize\circled{2}$ and $\scriptsize\circled{iii}\Rightarrow\scriptsize\circled{ii}$ of Section~\ref{subsec:implicationsproperness} above. 
It also gives $\scriptsize\circled{2}\Rightarrow\scriptsize\circled{5}$ by the contrapositive: if $\rho$ is not cusp-deteriorating, then there exists an element $\gamma\in\Gamma_0$ with $j(\gamma), \rho(\gamma)$ both parabolic, hence $j(\gamma^k)=2\log k +O(1)$ and $\rho(\gamma^k)=2\log k+O(1)$ as $k\rightarrow +\infty$, violating \eqref{eqn:propcrit}.
(Note that we needed no geometrical finiteness assumption on~$j$ so far.)

By \cite[Th.\,1.3]{kas08}, if $\Gamma_0$ is residually finite (for instance finitely generated) and $\Gamma_0^{j,\rho}$ acts properly discontinuously on~$G$, then the set $\mu_{\bullet}(\Gamma_0^{j,\rho})$ lies on \emph{one side only} of the diagonal of $\R_+\times\R_+$, up to a finite number of points.
This means, up to switching $j$ and~$\rho$, that condition \eqref{eqn:propcrit} is in fact equivalent to the following stronger condition:
\begin{equation}\label{eqn:propcritbelow}
\text{for any $R\geq 0$, }\quad \quad  \mu(\rho(\gamma)) < \mu(j(\gamma)) - R \quad\text{for almost all }\gamma\in\Gamma_0,
\end{equation}
and that properness implies $\lambda(\rho(\gamma))<\lambda(j(\gamma))$ for all $\gamma\in\Gamma_0$ (using \eqref{eqn:lambdalimmu}).
Condition~\eqref{eqn:propcritbelow} is a necessary and sufficient condition for \emph{left admissibility} in the sense of Definition~\ref{def:admissible}; right admissibility is obtained by switching $j$ and~$\rho$.

The implication $\scriptsize\circled{2}\Rightarrow\scriptsize\circled{3}$ of Section~\ref{subsec:implicationsproperness} for geometrically finite~$j$ (which will be proved in Sections \ref{subsec:proofadmred} and~\ref{subsec:proofadmnonred} below) can be interpreted as follows.

\begin{theorem}\label{thm:sharp}
Let $\Gamma$ be a discrete subgroup of $G\times G$ such that the set~$\mu_{\bullet}(\Gamma)$ lies below the diagonal of $\R_+\times\R_+$ (up to a finite number of points) and such that the projection of~$\Gamma$ to the first factor of $G\times G$ is geometrically finite.
Then $\Gamma$ acts properly discontinuously on~$G$ by right and left multiplication if and only if there are constants $C<1$ and $D\in \mathbb{R}$ such that
$$\mu(\gamma_2) \leq C\,\mu(\gamma_1) + D$$
for all $\boldsymbol\gamma=(\gamma_1,\gamma_2)\in\Gamma$.
\end{theorem}

The point of Theorem~\ref{thm:sharp} is that if $\Gamma$ acts properly discontinuously on~$G$, then the set $\mu_{\bullet}(\Gamma)$ ``drifts away from the diagonal at infinity'' \emph{linearly}; in other words, $\Gamma$ is \emph{sharp} in the sense of \cite[Def.\,4.2]{kk12}.
In particular, Theorem~\ref{thm:sharp} corroborates the conjecture \cite[Conj.\,4.10]{kk12} that any discrete group acting properly discontinuously \emph{and cocompactly} on a reductive homogeneous space should be sharp.
Sharpness has analytic consequences on the discrete spectrum of the (nonelliptic) Laplacian defined by the natural pseudo-Riemannian structure of signature $(n,n(n-1)/2)$ on the quotients of~$G$: see~\cite{kk12}.

\subsection{Properness and the topology of the quotients of $G=\PO(n,1)$}\label{subsec:topology}

Let us make two side remarks.

\medskip

\noindent $\bullet$ First, here is for convenience a short proof of the properness criterion \eqref{eqn:propcrit} of Benoist and Kobayashi in our setting. Note that there is no geometrical finiteness assumption here.

\begin{proof}[Proof of the properness criterion of Benoist and Kobayashi]
Suppose that condition \eqref{eqn:propcrit} holds.
Let $\mathscr{C}$ be a compact subset of~$G$ and let
$$R := \max_{g\in\mathscr{C}} \mu(g).$$
By the subadditivity \eqref{eqn:triangineqmu} of~$\mu$, for any $g\in\mathscr{C}$ and $\gamma\in\Gamma_0$,
$$\mu\left ( \rho(\gamma)g j(\gamma)^{-1} \right ) \geq |\mu(j(\gamma)) - \mu(\rho(\gamma))| - \mu(g).$$
By \eqref{eqn:propcrit}, the right-hand side is $>R$ for almost all $\gamma\in\Gamma_0$, hence $\mathscr{C}\cap \rho(\gamma)\mathscr{C}j(\gamma)^{-1}=\emptyset$ for almost all $\gamma\in\Gamma_0$.
Thus the action of $\Gamma_0^{j,\rho}$ on~$G$ is properly discontinuous.
Conversely, suppose that \eqref{eqn:propcrit} does \emph{not} hold, \ie there exists $R>0$ and a sequence $(\gamma_m)_{m\in\N}$ of pairwise distinct elements of~$\Gamma_0$ such that
$$|\mu(j(\gamma_m)^{-1}) - \mu(\rho(\gamma_m)^{-1})| \leq R$$
for all $m\in\N$.
By definition \eqref{eqn:defmu} of~$\mu$, this means that for any $m\in\N$ there is an element $k_m\in\underline{K}$ such that $d(\rho(\gamma_m)^{-1}\cdot p_0,k_m j(\gamma_m)^{-1}\cdot p_0)\leq R$.
Since $\rho(\gamma_m)$ acts on~$\HH^n$ by an isometry, we obtain
$$\mu(\rho(\gamma_m)k_m j(\gamma_m)^{-1}) = d(p_0,\rho(\gamma_m)k_m j(\gamma_m)^{-1}\cdot p_0) \leq R.$$
Therefore $\mathscr{C}\cap \rho(\gamma_m)\mathscr{C} j(\gamma_m)^{-1}\neq\emptyset$, where $\mathscr{C}$ is the compact subset of~$G$ consisting of the elements~$g$ with $\mu(g)\leq R$. This shows that the action of $\Gamma_0^{j,\rho}$ on~$G$ is not properly discontinuous.
\end{proof}

\noindent $\bullet$ Second, still without any geometrical finiteness assumption, here is a topological consequence of the inequality $C(j,\rho)< 1$; we refer to \cite{dgk12} for further developments and applications.

\begin{proposition} \label{prop:quotients}
Let $\Gamma_0$ be a discrete group and $(j,\rho)\in\Hom(\Gamma_0,G)^2$ a pair of representations with $j$ injective and discrete.
If $C(j,\rho)<1$, then the group
$$\Gamma_0^{j,\rho} = \{ (j(\gamma),\rho(\gamma))~|~\gamma\in\Gamma_0\} $$
acts properly discontinuously on~$G$ by right and left multiplication and the quotient is homeomorphic to a $\underline{K}$-bundle over $M:=j(\Gamma_0)\backslash\HH^n$, where $\underline{K}\cong\OO(n)$ is a maximal compact subgroup of $G=\PO(n,1)$.  
\end{proposition}

\begin{proof}
The group~$\underline{K}$ is the stabilizer in~$G$ of some point of~$\HH^n$.
Choose a $(j,\rho)$-equivariant map $f : \HH^n\rightarrow\HH^n$ with $\Lip(f)<1$.
For any $p\in\HH^n$,
$$\mathcal{L}_p := \{g\in G ~|~  g\cdot p = f(p) \} $$
is a right-and-left coset of~$\underline{K}$.
An element $g\in G$ belongs to~$\mathcal{L}_p$ if and only if $p$ is a fixed point of $g^{-1}\circ f$; since $\Lip(g^{-1}\circ f)=\Lip(f)<1$, such a fixed point exists and is unique, which shows that $g$ belongs to exactly one set~$\mathcal{L}_p$.
We denote this $p$ by $\Pi(g)$.
The fibration $\Pi : G\rightarrow\HH^n$ is continuous: if $h\in G$ is close enough to~$g$ so that $d(\Pi(g),h^{-1}\circ f\circ\Pi(g))\leq (1-\Lip(f))\,\varepsilon$, then $h^{-1}\circ f$ takes the $\varepsilon$-ball centered at $\Pi(g)$ to itself, hence $\Pi(h)$ is within $\varepsilon$ from~$\Pi(g)$.
Moreover, $\Pi : G\rightarrow\HH^n$ is by construction $(\Gamma_0^{j,\rho},j(\Gamma_0))$-equivariant:
$$\rho(\gamma)\,\mathcal{L}_p\,j(\gamma)^{-1} = \mathcal{L}_{j(\gamma)\cdot p}$$
for all $\gamma\in\Gamma_0$ and $p\in\HH^n$.
Since the action of $j(\Gamma_0)$ on~$\HH^n$ is properly discontinuous, the action of $\Gamma_0^{j,\rho}$ on~$G$ is properly discontinuous.
The fibration~$\Pi$ descends to a topological fibration of the quotient of $G$ by $\Gamma_0^{j,\rho}$, with base $M=j(\Gamma_0)\backslash\HH^n$ and fiber~$\underline{K}$.
Note that for constant~$\rho$, \ie $\rho(\Gamma_0)=\{1\}$, this fibration naturally identifies with the orthonormal frame bundle of~$M$.
\end{proof}

\subsection{Proof of the implications of Section~\ref{subsec:implicationsproperness} when $\rho$ is reductive}\label{subsec:proofadmred}

We first consider the generic case where $\rho$ is reductive (Definition~\ref{def:reductive}), \ie $\rho^{\mathrm{red}}=\rho$.
We have already explained the easy implications $\scriptsize\circled{i}\Rightarrow\scriptsize\circled{iii}\Rightarrow\scriptsize\circled{iv}$ and $\scriptsize\circled{iii}\Rightarrow\scriptsize\circled{v}$ (Section~\ref{subsec:implicationsproperness}), as well as $\scriptsize\circled{iii}\Rightarrow\scriptsize\circled{ii}$ which is an immediate consequence of the properness criterion of Benoist and Kobayashi (Section~\ref{subsec:Cartanproj}).
We now explain $\scriptsize\circled{iv}\Rightarrow\scriptsize\circled{iii}$ and $\scriptsize\circled{ii}\Rightarrow\scriptsize\circled{i}$.

The implication $\scriptsize\circled{iv}\Rightarrow\scriptsize\circled{iii}$ is an immediate consequence of the following equality.

\begin{lemma}\label{lem:ClambdaCmured}
Let $\Gamma_0$ be a discrete group and $(j,\rho)\in\Hom(\Gamma_0,G)^2$ a pair of representations.
If $\rho$ is reductive, then $C'(j,\rho)=C_{\mu}(j,\rho)$.
\end{lemma}

\begin{proof}
By \eqref{eqn:ClambdamuLip}, we always have $C'(j,\rho)\leq C_{\mu}(j,\rho)$.
Let us prove the converse inequality.
If $\rho$ is reductive, then by \cite[Th.\,4.1]{ams95} and \cite[Lem.\,2.2.1]{ben97} there are a finite subset $F$ of~$\Gamma_0$ and a constant $D\geq 0$ with the following property:
for any $\gamma\in\Gamma_0$ there is an element $f\in F$ such that
$$|\mu(\rho(\gamma f)) - \lambda(\rho(\gamma f))| \leq D$$
(the element $\gamma f$ is \emph{proximal} --- see \cite{ben97}).
Then \eqref{eqn:triangineqmu} and \eqref{eqn:lambdaleqmu} imply
\begin{eqnarray*}
\mu(\rho(\gamma)) & \leq & \mu(\rho(\gamma f)) + \mu(\rho(f))\\
& \leq & \lambda(\rho(\gamma f)) + D + \mu(\rho(f))\\
& \leq & C'(j,\rho)\,\lambda(j(\gamma f)) + D + \mu(\rho(f))\\
& \leq & C'(j,\rho)\,\mu(j(\gamma f)) + D + \mu(\rho(f))\\
& \leq & C'(j,\rho)\,\mu(j(\gamma)) + c,
\end{eqnarray*}
where we set
$$c := D + \max_{f\in F} \big(C'(j,\rho)\,\mu(j(f)) + \mu(\rho(f))\big) <\infty.$$
Thus $C_{\mu}(j,\rho)\leq C'(j,\rho)$, which completes the proof.
\end{proof}

The implication $\scriptsize\circled{ii}\Rightarrow\scriptsize\circled{i}$ (or its contrapositive) for geometrically finite~$j$ is a consequence of the existence of a maximally stretched lamination when $C(j,\rho)\geq 1$ (Theorem~\ref{thm:lamin}).
We first establish the following.

\begin{lemma}\label{lem:ClambdaCLipred}
Let $\Gamma_0$ be a discrete group and $(j,\rho)\in\Hom(\Gamma_0,G)^2$ a pair of representations with $j$ geometrically finite.
If $\rho$ is reductive and $C(j,\rho)\geq 1$, then there is a sequence $(\gamma_k)_{k\in\N}$ of pairwise distinct elements of~$\Gamma_0$ such that $\mu(\rho(\gamma_k))-C(j,\rho)\,\mu(j(\gamma_k))$ is uniformly bounded from below; in particular (using Lemma~\ref{lem:ClambdaCmured} and \eqref{eqn:ClambdamuLip}),
\begin{equation}\label{eqn:ClambdaCLipred}
C'(j,\rho) = C_{\mu}(j,\rho) = C(j,\rho).
\end{equation}
\end{lemma}

The equality $C'(j,\rho)=C(j,\rho)$ is Corollary~\ref{cor:CC'}, which has already been proved in Section~\ref{subsec:coroElamin} when the stretch locus $E(j,\rho)$ is nonempty.
Here we do not make any assumption on $E(j,\rho)$.

\begin{proof}[Proof of Lemma~\ref{lem:ClambdaCLipred}]
If $C(j,\rho)=1$ and $\rho$ is \emph{not} cusp-deteriorating, then there exists $\gamma\in\Gamma_0$ with $j(\gamma)$ and~$\rho(\gamma)$ both parabolic.
Since $\mu(\rho(\gamma^k))$ and $\mu(j(\gamma^k))$ are both equal to $2\log(k)+O(1)$ as $k\rightarrow +\infty$ (Lemma~\ref{lem:disthorosphere}), the sequence $(\mu(\rho(\gamma_k))-C(j,\rho)\,\mu(j(\gamma_k)))_{k\in\N}$ is uniformly bounded from below.

If $C(j,\rho)>1$ or if $C(j,\rho)=1$ and $\rho$ is cusp-deteriorating, then by Theorem~\ref{thm:lamin} there is a $(j,\rho)$-equivariant map $f : \HH^n\rightarrow\HH^n$ with minimal Lipschitz constant $C(j,\rho)$ that stretches maximally some geodesic line $\ell$ of~$\HH^n$ whose image in $j(\Gamma_0)\backslash\HH^n$ lies in a compact part of the convex core.
Consider a sequence $(\gamma_k)_{k\in\N}$ of pairwise distinct elements of~$\Gamma_0$ such that $d(j(\gamma_k)\cdot p_0,\ell)$ is uniformly bounded by some constant $R>0$.
For $k\in\N$, let $y_k$ be the closest-point projection to~$\ell$ of $j(\gamma_k)\cdot p_0$.
If $p_0\in\HH^n$ is the basepoint defining~$\mu$ in \eqref{eqn:defmu} and if we set $\Delta:=d(p_0,f(y_0))+d(p_0,f(p_0))$, then the triangle inequality implies
\begin{eqnarray*}
\mu(\rho(\gamma_k)) 
& =    & d(p_0,\rho(\gamma_k)\cdot p_0)\\
& \geq & d\big(f(y_0),\rho(\gamma_k)\cdot f(p_0)\big) - \Delta\\
& =    & d\big(f(y_0),f(j(\gamma_k) \cdot p_0)\big)   - \Delta \\
& \geq & d\big(f(y_0),f(y_k)\big) - d\big(f(y_k),f(j(\gamma_k) \cdot p_0)\big) - \Delta \\
& \geq & C(j,\rho)\,d(y_0,y_k) - C(j,\rho)\,R - \Delta \\
& \geq & C(j,\rho) \big(d(p_0,j(\gamma_k)\cdot p_0)-2R\big) - C(j,\rho)\,R - \Delta \\
& =    & C(j,\rho)\,\mu(j(\gamma_k)) - 3\,C(j,\rho)\,R - \Delta.
\end{eqnarray*}
Thus the sequence $(\mu(\rho(\gamma_k))-C(j,\rho)\,\mu(j(\gamma_k)))_{k\in\N}$ is uniformly bounded from below.
\end{proof}

We can now prove the implication $\scriptsize\circled{ii}\Rightarrow\scriptsize\circled{i}$.

\begin{corollary}
Let $\Gamma_0$ be a discrete group and $(j,\rho)\in\Hom(\Gamma_0,G)^2$ a pair of representations with $j$ geometrically finite.
If $\rho$ is reductive and $(j,\rho)$ left admissible, then $C(j,\rho)<1$.
\end{corollary}

\begin{proof}
Assume that $\rho$ is reductive and $(j,\rho)$ left admissible.
By \cite[Th.\,1.3]{kas08}, condition \eqref{eqn:propcritbelow} is satisfied, hence $C_{\mu}(j,\rho)\leq 1$.
Suppose by contradiction that $C(j,\rho)\geq 1$.
By Lemma~\ref{lem:ClambdaCLipred}, we have $C_{\mu}(j,\rho)=C(j,\rho)=1$ and there is a sequence $(\gamma_k)\in(\Gamma_0)^{\N}$ of pairwise distinct elements with $|\mu(\rho(\gamma_k))-\mu(j(\gamma_k))|$ uniformly bounded.
This contradicts the properness criterion \eqref{eqn:propcrit} of Benoist and Kobayashi.
\end{proof}

\subsection{Proof of the implications of Section~\ref{subsec:implicationsproperness} when $\rho$ is nonreductive}\label{subsec:proofadmnonred}

We now prove the implications of Section~\ref{subsec:implicationsproperness} for geometrically finite~$j$ when $\rho$ is nonreductive (Definition~\ref{def:reductive}).
We have already explained the easy implications $\scriptsize\circled{1}\Rightarrow\scriptsize\circled{3}\Rightarrow\scriptsize\circled{4}\Rightarrow\scriptsize\circled{5}$ (Section~\ref{subsec:implicationsproperness}) and $\scriptsize\circled{3}\Rightarrow\scriptsize\circled{2}\Rightarrow\scriptsize\circled{5}$ (Section~\ref{subsec:Cartanproj}), as well as $\scriptsize\circled{1}\Rightarrow\scriptsize\circled{i}$ and $\scriptsize\circled{i}\Rightarrow\scriptsize\circled{1}$ under the cusp-deterioration assumption~$\scriptsize\circled{5}$ (Section~\ref{subsec:implicationsproperness}).
Moreover, in Section~\ref{subsec:proofadmred} we have established the implications $\scriptsize\circled{i}\Leftrightarrow\scriptsize\circled{ii}\Leftrightarrow\scriptsize\circled{iii}\Leftrightarrow\scriptsize\circled{iv}\Rightarrow\scriptsize\circled{v}$ for the ``reductive part'' $\rho^{\mathrm{red}}$ of~$\rho$.
Therefore, we only need to explain the ``horizontal'' implications $\scriptsize\circled{2}\Rightarrow\scriptsize\circled{ii}$, $\scriptsize\circled{3}\Rightarrow\scriptsize\circled{iii}$, $\scriptsize\circled{4}\Rightarrow\scriptsize\circled{iv}$, and $\scriptsize\circled{5}\Rightarrow\scriptsize\circled{v}$

Let $\rho\in\Hom(\Gamma_0,G)$ be nonreductive.
The group $\rho(\Gamma_0)$ has a unique fixed point $\xi$ in the boundary at infinity $\partial_{\infty}\HH^n$ of~$\HH^n$.
Let $P$ be the stabilizer of $\xi$ in~$G$: it is a proper parabolic subgroup of~$G$.
Choose a Levi decomposition $P=(\underline{M}\underline{A})\ltimes\underline{N}$, where $\underline{A}\cong\R_+^{\ast}$ is a Cartan subgroup of~$G$ (\ie a one-parameter subgroup of purely translational, commuting hyperbolic elements), $\underline{M}\cong\OO(n-1)$ is a compact subgroup of~$G$ such that $\underline{M}\underline{A}$ is the centralizer of $\underline{A}$ in~$G$, and $\underline{N}\cong\R^{n-1}$ is the unipotent radical of~$P$.
For instance, for $n=2$ (\resp $n=3$), the identity component of the group~$G$ identifies with $\PSL_2(\R)$ (\resp with $\PSL_2(\C)$), and we can take~$\underline{A}$ to be the projectivized real diagonal matrices, $\underline{N}$ the projectivized upper triangular unipotent matrices, and the identity component of~$\underline{M}$ to be the projectivized diagonal matrices with entries of module~$1$.
We set $\rho^{\mathrm{red}}:=\underline{\pi}\circ\rho$, where $\underline{\pi} : P\rightarrow\underline{M}\underline{A}$ is the natural projection.

The implications $\scriptsize\circled{2}\Rightarrow\scriptsize\circled{ii}$, $\scriptsize\circled{3}\Rightarrow\scriptsize\circled{iii}$, $\scriptsize\circled{4}\Rightarrow\scriptsize\circled{iv}$, and $\scriptsize\circled{5}\Rightarrow\scriptsize\circled{v}$ of Section~\ref{subsec:implicationsproperness} are consequences of the following easy observation.

\begin{lemma}\label{lem:muparabolic}
After possibly changing the basepoint $p_0\in\HH^n$ of \eqref{eqn:defmu} (which modifies $\mu$ only by a bounded additive amount, by the triangle inequality), we have
$$\lambda(g) = \lambda(\underline{\pi}(g)) = \mu(\underline{\pi}(g)) \leq \mu(g)$$
for all $g\in P$.
\end{lemma}

\begin{proof}
Take the basepoint $p_0\in\HH^n$ on the geodesic line~$\A$ preserved by~$\underline{M}\underline{A}$, which is pointwise fixed by~$\underline{M}$ and on which the elements of~$A$ act by translation. Then $\lambda(\underline{\pi}(g))=\mu(\underline{\pi}(g))$ for all $g\in P$. 
The projection onto~$\A$ along horospheres centered at~$\xi$ is $(P,\underline{\pi}(P))$-equivariant, and restricts to an isometry on any line ending at~$\xi$: therefore, if $g$ is hyperbolic, preserving such a line, then $\underline{\pi}(g)$ translates along $\A$ by $\lambda(g)$ units of length, yielding $\lambda(g)=\lambda(\underline{\pi}(g))$.
If $g$ is parabolic or elliptic, then $\lambda(g)=0$ and $g$ preserves each horosphere centered at~$\xi$, hence $\lambda(\underline{\pi}(g))=0$.
\end{proof}

\begin{proof}[Proof of $\scriptsize\circled{3}\Rightarrow\scriptsize\circled{iii}$ and $\scriptsize\circled{4}\Rightarrow\scriptsize\circled{iv}$]
Lemma~\ref{lem:muparabolic} and \eqref{eqn:ClambdamuLip} imply
$$C'(j,\rho) = C'(j,\rho^{\mathrm{red}}) = C_{\mu}(j,\rho^{\mathrm{red}}) \leq C_{\mu}(j,\rho),$$
which immediately yields the implications.
\end{proof}

\begin{proof}[Proof of $\scriptsize\circled{2}\Rightarrow\scriptsize\circled{ii}$]
Assume that $(j,\rho)$ is left admissible.
By \cite{kas08}, the pair $(j,\rho)$ is not right admissible and the stronger form \eqref{eqn:propcritbelow} of the properness criterion of Benoist and Kobayashi holds.
Since $\mu\circ\underline{\pi}\leq\mu$ by Lemma~\ref{lem:muparabolic}, the condition \eqref{eqn:propcritbelow} also holds for $\rho^{\mathrm{red}}=\underline{\pi}\circ\rho$, hence $(j,\rho^{\mathrm{red}})$ is left admissible.
\end{proof}

\begin{proof}[Proof of $\scriptsize\circled{5}\Rightarrow\scriptsize\circled{v}$]
Assume that $\rho$ is cusp-deteriorating.
For any $\gamma\in\Gamma_0$ with $j(\gamma)$ parabolic, $\rho^{\mathrm{red}}(\gamma)$ is not hyperbolic, otherwise $\rho(\gamma)$ would be hyperbolic too by Lemma~\ref{lem:muparabolic}, and it is not parabolic since $\rho^{\mathrm{red}}$ takes values in the group $\underline{M}\underline{A}$ which has no parabolic element.
\end{proof}

\subsection{Deformation of properly discontinuous actions}\label{subsec:proofdeform}

Theorems \ref{thm:deformcompact} and~\ref{thm:deform} follow from Theorem~\ref{thm:adm} (properness criterion) and Proposition~\ref{prop:contCcc} (continuity of $(j,\rho)\mapsto C(j,\rho)$ for convex cocompact~$j$), together with a classical cohomological argument for cocompactness.

\begin{proof}[Proof of Theorems \ref{thm:deformcompact} and~\ref{thm:deform}]
Let $\Gamma$ be a finitely generated discrete subgroup of $G\times G$ acting properly discontinuously on~$G$ by right and left multiplication.
By the Selberg lemma \cite[Lem.\,8]{sel60}, there is a finite-index subgroup $\Gamma'$ of~$\Gamma$ that is torsion-free.
By \cite{kr85} and \cite{sal00} (case $n=2$) and \cite{kas08} (general case), up to switching the two factors of $G\times G$, the group~$\Gamma'$ is of the form $\Gamma_0^{j,\rho}$ as in \eqref{eqn:Gamma(j,rho)}, where $\Gamma_0$ is a torsion-free discrete subgroup of~$G$ and $j,\rho\in\Hom(\Gamma_0,G)$ are two representations of $\Gamma_0$ in~$G$ with $j$ injective and discrete, and $(j,\rho)$ is left admissible in the sense of Definition~\ref{def:admissible}.
Assume that $j$ is convex cocompact.
Then $C(j,\rho)<1$ by Theorem~\ref{thm:adm}.
By Proposition~\ref{prop:contCcc} and the fact that being convex cocompact is an open condition (see \cite[Prop.\,4.1]{bow98} or Proposition~\ref{prop:rem-ccopen}), there is a neighborhood $\mathcal{U}\subset\Hom(\Gamma,G\times G)$ of the natural inclusion such that for any $\varphi\in\mathcal{U}$, the group $\varphi(\Gamma')$ is of the form $\Gamma_0^{j',\rho'}$ for some $(j',\rho')\in\Hom(\Gamma_0,G)^2$ with $j'$ convex cocompact and $C(j',\rho')<1$.
In particular, $\varphi(\Gamma')$ is discrete in $G\times G$ and acts properly discontinuously on~$G$ by Theorem~\ref{thm:adm}, and the same conclusion holds for~$\varphi(\Gamma)$.

Assume that the action of $\Gamma$ on~$G$ is cocompact.
We claim that the action of $\varphi(\Gamma)$ on~$G$ is cocompact for all $\varphi\in\mathcal{U}$.
Indeed, let $\Gamma'=\Gamma_0^{j,\rho}$ be a finite-index subgroup of~$\Gamma$ as above; it is sufficient to prove that the action of $\varphi(\Gamma')$ is cocompact for all $\varphi\in\mathcal{U}$.
Since $\varphi(\Gamma')$ is of the form $\Gamma_0^{j',\rho'}$ with $j'$ injective, the group $\varphi(\Gamma')$ has the same cohomological dimension as~$\Gamma'$.
We then use the fact that when a torsion-free discrete subgroup of $G\times G$ acts properly discontinuously on~$G$, it acts cocompactly on~$G$ if and only if its cohomological dimension is equal to the dimension of the Riemannian symmetric space of~$G$, namely~$n$ in our case (see \cite[Cor.\,5.5]{kob89}).
(Alternatively, the cocompactness of the action of $\Gamma_0^{j',\rho'}$ on~$G$ also follows from Proposition~\ref{prop:quotients} and from the cocompactness of~$j'(\Gamma_0)$.)

Finally, assume that the action of $\Gamma$ on~$G$ is free.
This means that for any $\boldsymbol\gamma\in\Gamma\smallsetminus\{ 1\}$, the elements $\mathrm{pr}_1(\boldsymbol\gamma)$ and $\mathrm{pr}_2(\boldsymbol\gamma)$ are \emph{not} conjugate in~$G$, where $\mathrm{pr}_i : G\times G\rightarrow G$ denotes the $i$-th projection.
In fact, since the action of $\Gamma$ on~$G$ is properly discontinuous, $\mathrm{pr}_1(\boldsymbol\gamma)$ and $\mathrm{pr}_2(\boldsymbol\gamma)$ can never be conjugate in~$G$ when $\boldsymbol\gamma$ is of infinite order.
Therefore freeness is seen exclusively on torsion elements.
We claim that $\Gamma$ has only finitely many conjugacy classes of torsion elements.
Indeed, $\Gamma$ has a finite-index subgroup of the form $\Gamma_0^{j,\rho}$ with $j$ injective and convex cocompact (up to switching the two factors of $G\times G$), and a convex cocompact subgroup of~$G$ (or more generally a geometrically finite subgroup) has only finitely many conjugacy classes of torsion elements (see \cite{bow93}).
For any nontrivial torsion element $\boldsymbol\gamma\in\Gamma$, there is a neighborhood $\mathcal{U}_{\boldsymbol\gamma}\subset\Hom(\Gamma,G\times G)$ such that for all $\varphi\in\mathcal{U}_{\boldsymbol\gamma}$, the elements $\mathrm{pr}_1(\varphi(\boldsymbol\gamma))$ and $\mathrm{pr}_2(\varphi(\boldsymbol\gamma))$ are not conjugate in~$G$; then $\mathrm{pr}_1(\varphi(\boldsymbol\gamma'))$ and $\mathrm{pr}_2(\varphi(\boldsymbol\gamma'))$ are also not conjugate for any $\Gamma$-conjugate $\boldsymbol\gamma'$ of~$\boldsymbol\gamma$.
\end{proof}

The same argument, replacing Proposition~\ref{prop:contCcc} (continuity of $(j,\rho)\mapsto C(j,\rho)$ for convex cocompact~$j$) by Proposition~\ref{prop:contpropertiesC}.(1) (openness of the condition $C(j,\rho)<1$ for geometrically finite~$j$ and cusp-deteriorating~$\rho$ in dimension $n\leq 3$), yields the following.

\begin{theorem}\label{thm:deformwithcusps}
For $G=\PO(2,1)$ or $\PO(3,1)$ (\ie $\PSL_2(\R)$ or $\PSL_2(\C)$ up to index two), let $\Gamma$ be a discrete subgroup of $G\times G$ acting properly discontinuously on~$G$, with a geometrically finite quotient (Definition~\ref{def:cococo}).
There is a neighborhood $\mathcal{U}\subset\Hom_{\mathrm{det}}(\Gamma,G\times G)$ of the natural inclusion such that for any $\varphi\in\mathcal{U}$, the group $\varphi(\Gamma)$ is discrete in $G\times G$ and acts properly discontinuously on~$G$, with a geometrically finite quotient; moreover, this quotient is compact (\resp is convex cocompact, \resp is a manifold) if the initial quotient of~$G$ by~$\Gamma$ was.
\end{theorem}

The set $\Hom_{\mathrm{det}}(\Gamma,G\times G)$ is defined as follows.
We have seen that the group~$\Gamma$ has a finite-index subgroup~$\Gamma'$ of the form $\Gamma_0^{j,\rho}$ or $\Gamma_0^{\rho,j}$, where $\Gamma_0$ is a discrete subgroup of~$G$ and $j,\rho\in\Hom(\Gamma_0,G)$ are two representations of $\Gamma_0$ in~$G$ with $j$ injective and geometrically finite and $(j,\rho)$ left admissible in the sense of Definition~\ref{def:admissible}.
By Lemma~\ref{lem:parabdet}, the representation~$\rho$ is cusp-deteriorating with respect to~$j$ in the sense of Definition~\ref{def:typedet}.
We define $\Hom_{\mathrm{det}}(\Gamma,G\times G)$ to be the set of group homomorphisms from $\Gamma$ to $G\times G$ whose restriction to $\Gamma'$ is of the form $(j',\rho')$ (if $\Gamma'\cong\Gamma_0^{j,\rho}$) or $(\rho',j')$ (if $\Gamma'\cong\Gamma_0^{\rho,j}$) with $j'$ injective and geometrically finite, of the cusp type of~$j$, and $\rho'$ cusp-deteriorating with respect to~$j$.

It is necessary to restrict to $\Hom_{\mathrm{det}}(\Gamma,G\times G)$ in Theorem~\ref{thm:deformwithcusps}, for the following reasons:
\begin{itemize}
  \item as mentioned in the introduction, for a given $j$ with cusps, the constant representation $\rho=1$ can have small, non-cusp-deteriorating deformations $\rho'$, for which $(j,\rho')$ is nonadmissible;
  \item if we allow for small deformations $j'$ of~$j$ with a \emph{different} cusp type than~$j$ (fewer cusps), then the pair $(j,\rho)$ can have small, \emph{nonadmissible} deformations $(j',\rho')$ with $\rho'$ cusp-deteriorating with respect to~$j'$: this shows that we must fix the cusp type.
\end{itemize}

Note that properly discontinuous actions on $G=\PO(3,1)$ of finitely generated groups $\Gamma=\Gamma_0^{j,\rho}$ with $j$ geometrically infinite do \emph{not} deform into properly discontinuous actions in general, for the group $j(\Gamma_0)$ (\emph{e.g.}\ the fiber group of a hyperbolic surface bundle over the circle) may have small deformations $j'(\Gamma_0)$ that are not even discrete (\emph{e.g.}\ small perturbations of a nearby cusp group in the sense of \cite{mcm91}).
It would be interesting to know whether in this case one can define an analogue of $\Hom_{\mathrm{det}}(\Gamma,G\times G)$ in which admissibility becomes again an open condition; see also Section~\ref{sec:geom-infinite}.

\subsection{Interpretation of Theorem~\ref{thm:deformcompact} in terms of $(\mathbf{G},\mathbf{X})$-structures}\label{subsec:(G,X)struct}

We can translate Theorem~\ref{thm:deformcompact} in terms of geometric structures, in the sense of Ehresmann and Thurston, as follows.
We set $\mathbf{X}=G=\PO(n,1)$ and $\mathbf{G}=G\times G$, where $G\times G$ acts on~$G$ by right and left multiplication.
Let $N$ be a manifold with universal covering~$\widetilde{N}$.
Recall that a \emph{$(\mathbf{G},\mathbf{X})$-structure} on~$N$ is a (maximal) atlas of charts on~$N$ with values in~$\mathbf{X}$ such that the transition maps are given by elements of~$\mathbf{G}$.
Such a structure is equivalent to a pair $(h,D)$ where $h : \pi_1(N)\rightarrow\mathbf{G}$ is a group homomorphism called the \emph{holonomy} and $D : \widetilde{N}\rightarrow\mathbf{X}$ an $h$-equivariant local diffeomorphism called the \emph{developing map}; the pair $(h,D)$ is unique modulo the natural action of~$\mathbf{G}$ by
$$\mathbf{g}\cdot (h,D) = \left (\mathbf{g}h(\cdot)\mathbf{g}^{-1},\mathbf{g} D\right ).$$
A $(\mathbf{G},\mathbf{X})$-structure on~$N$ is said to be \emph{complete} if the developing map is a covering; this is equivalent to a notion of geodesic completeness for the natural pseudo-Riemannian structure induced by the Killing form of the Lie algebra of~$G$ (see \cite{gol85}).
For $n>2$, the fundamental group of $G_0=\PO(n,1)_0$ is finite, hence completeness is equivalent to the fact that the $(\mathbf{G},\mathbf{X})$-structure identifies~$N$ with the quotient of~$\mathbf{X}$ by some discrete subgroup~$\Gamma$ of~$\mathbf{G}$ acting properly discontinuously and freely on~$\mathbf{X}$, up to a finite covering.
For $n=2$, this characterization of completeness still holds for \emph{compact} manifolds~$N$ \cite[Th.\,7.2]{kr85}.
Therefore, Theorem~\ref{thm:deformcompact} can be restated as follows.

\begin{corollary}\label{cor:open(G,X)}
Let $\mathbf{X}=G=\PO(n,1)$ and $\mathbf{G}=G\times G$, acting on~$\mathbf{X}$ by right and left multiplication.
The set of holonomies of complete $(\mathbf{G},\mathbf{X})$-structures on any compact manifold~$N$ is open in $\Hom(\pi_1(N),\mathbf{G})$.
\end{corollary}

We note that for a compact manifold~$N$, the so-called \emph{Ehresmann--Thurston principle} asserts that the set of holonomies of \emph{all} (not necessarily complete) $(\mathbf{G},\mathbf{X})$-structures on~$N$ is open in $\Hom(\pi_1(N),\mathbf{G})$ (see \cite{thu77}).
For $n=2$, Klingler \cite{kli96} proved that all $(\mathbf{G},\mathbf{X})$-structures on~$N$ are complete, which implies Corollary~\ref{cor:open(G,X)}.
For $n>2$, it is not known whether all $(\mathbf{G},\mathbf{X})$-structures on~$N$ are complete; it has been conjectured to be true at least for $n=3$ \cite{dz09}.
The question is nontrivial since the Hopf--Rinow theorem does not hold for non-Riemannian manifolds.

\section{A generalization of the Thurston metric on Teichm\"uller~space}\label{sec:Thurston}

In this section we prove Proposition~\ref{prop:generalThurston}, which generalizes the Thurston metric on Teichm\"uller space to higher dimension, in a geometrically finite setting.
(Corollary~\ref{cor:CC'} has already been proved as part of Lemma~\ref{lem:ClambdaCLipred}; see also Section~\ref{subsec:coroElamin}.)

\subsection{An asymmetric metric on the space of geometrically finite structures}\label{subsec:generalized-dTh}

Let $M$ be a hyperbolic $n$-manifold and $\T(M)$ the space of conjugacy classes of geometrically finite representations of $\Gamma_0:=\pi_1(M)$ into $G=\PO(n,1)$ with the homeomorphism type and cusp type of~$M$.
In this section we assume that $M$ contains more than one essential closed curve.
We define a function $d_{\mathrm{Th}} : \T(M)\times\T(M)\to\R$ as in \eqref{eqn:thurston-distance-corrected}: for any $j,\rho\in\T(M)$,
$$d_{\mathrm{Th}}(j,\rho) := \log \left ( C(j,\rho) \frac{\delta(\rho)}{\delta(j)}\right ),$$
where $\delta(j)$ is the \emph{critical exponent} of~$j$.
Recall that, by definition,
\begin{equation}\label{eqn:defcritexp}
\delta(j) := \limsup_{R\rightarrow +\infty}\ \frac{1}{R}\,\log\#\big(j(\Gamma_0)\cdot p\cap B_p(R)\big),
\end{equation}
where $p$ is any point of~$\HH^n$ and $B_p(R)$ denotes the ball of radius~$R$ centered at~$p$ in~$\HH^n$.
We have $\delta(j)\in (0,n-1]$ \cite{bea66,sul79} and the limsup is in fact a limit \cite{pat76,sul79,rob02}.
The Poincar\'e series $\sum_{\gamma\in\Gamma_0} e^{-s\,d(p,j(\gamma)\cdot p)}$ converges for $s>\delta(j)$ and diverges for $s\leq\delta(j)$.
Equivalently, $\delta(j)$ is the Hausdorff dimension of the limit set of $j(\Gamma_0)$ \cite{sul79,sul84}.

If $M$ is a surface of finite volume, then $\T(M)$ is the Teichm\"uller space of~$M$, we have $\delta \equiv 1$ on $\T(M)$, and $d_{\mathrm{Th}}$ is the \emph{Thurston metric} on Teichm\"uller space, which was introduced in \cite{thu86} (see Section~\ref{subsec:Thurstonmetric}).
This metric is \emph{asymmetric}: in general $d_{\mathrm{Th}}(j,\rho)\neq d_{\mathrm{Th}}(\rho,j)$, see \cite[\S\,2]{thu86}.

\begin{lemma} \label{lem:dThcont}
The function $d_{\mathrm{Th}} : \T(M)\times\T(M)\rightarrow\R$ is continuous as soon as $M$ is convex cocompact or all the cusps have rank $\geq n-2$.
In particular, it is always continuous if $n\leq 3$.
\end{lemma}

\begin{proof}
If $M$ is convex cocompact or all the cusps have rank $\geq n-2$, then the function $(j,\rho)\mapsto C(j,\rho)$ is continuous on $\T(M)\times\T(M)$ by Proposition~\ref{prop:contCcc}, Lemma~\ref{lem:parabdet}, and Proposition~\ref{prop:contpropertiesC}.(3).
Moreover, convergence in $\T(M)$ implies geometric convergence (see Proposition~\ref{prop:rem-gfopen}), and so $\delta$ is continuous on $\T(M)$ by \cite[Th.\,7.3]{mcm99}.
\end{proof}

The following remark justifies the introduction of the correcting factor $\delta(\rho)/\delta(j)$ in the definition \eqref{eqn:thurston-distance-corrected} of~$d_{\mathrm{Th}}$.

\begin{remark}\label{rem:dTh<0}
If $M$ has infinite volume, then $\log C(j,\rho)$ can take negative values, and $\log C(j,\rho)=0$ does \emph{not} imply $j=\rho$.
\end{remark}

\begin{proof}
The following example is taken from \cite[proof of Lemma~3.4]{thu86}; see also \cite{pt10}.
Let $M$ be a pair of pants, \ie a hyperbolic surface of genus~$0$ with three funnels.
Let $\alpha$ be an infinite embedded geodesic of~$M$ whose two ends go out to infinity in the same funnel, and let $\alpha'$ be another nearby geodesic (see Figure~\ref{fig:I}).
\begin{figure}[h!]
\begin{center}
\labellist
\small\hair 2pt
\pinlabel{$\alpha$} at 193 240
\pinlabel{$\alpha'$} at 237 220
\endlabellist
\includegraphics[width=7cm]{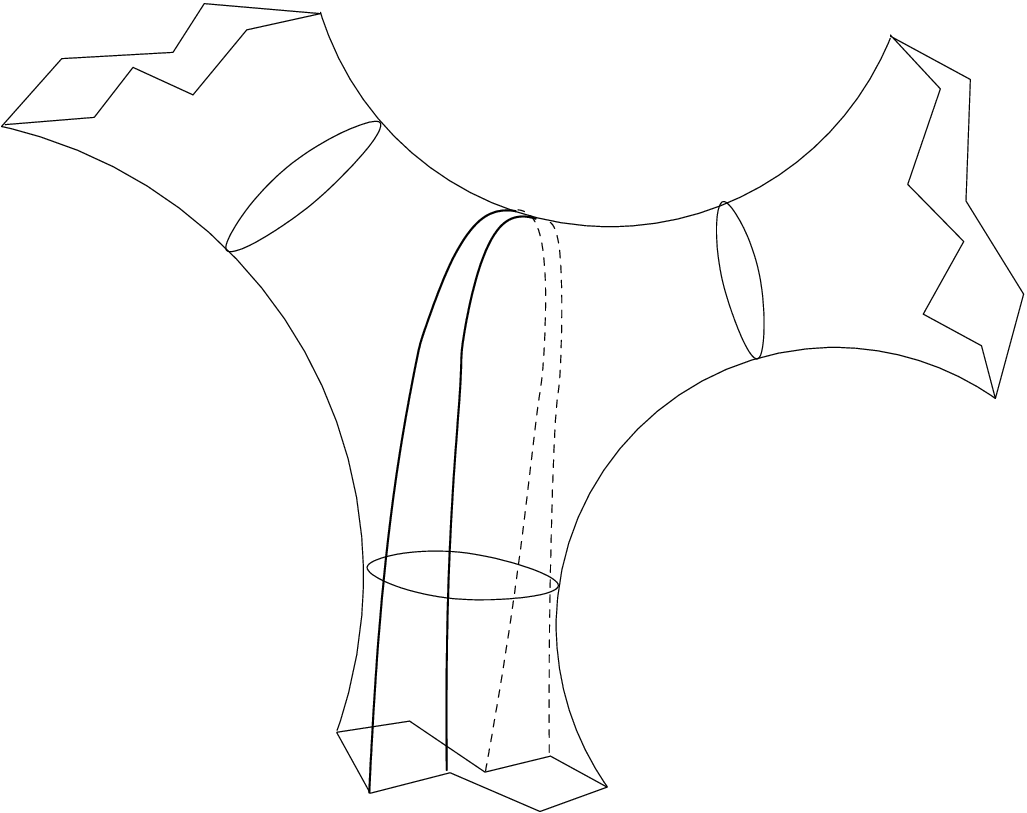}
\caption{The strip between the geodesics $\alpha$ and~$\alpha'$ can be collapsed to create a new hyperbolic metric with shorter curves. In general, the closed geodesic at the bottom (met by $\alpha, \alpha'$) will \emph{not} collapse to a closed geodesic of the new metric.}
\label{fig:I}
\end{center}
\end{figure}
Cutting out the strip between $\alpha$ and~$\alpha'$ and gluing back so that the endpoints of the common perpendicular to $\alpha, \alpha'$ are identified yields a new hyperbolic surface $M'$ such that two boundary components of the convex core of~$M'$ have the same lengths as in~$M$, and the third one is shorter.
There is a $1$-Lipschitz map between $M$ and~$M'$, and the corresponding holonomies $j\neq\rho$ satisfy $\log C(j,\rho)=0$.
In fact, it is easy to see that after repeating the process with all three funnels we obtain an element $\rho'\in\T(M)$ with $\log C(j,\rho')<0$.
\end{proof}

With the correcting factor $\delta(\rho)/\delta(j)$, the following holds.

\begin{lemma}\label{lem:dThgeq0}
For any $j,\rho,j_1,j_2,j_3\in\T(M)$,
\begin{enumerate}
  \item $d_{\mathrm{Th}}(j,\rho)\geq 0$,
  \item $d_{\mathrm{Th}}(j_1,j_3)\leq d_{\mathrm{Th}}(j_1,j_2)+d_{\mathrm{Th}}(j_2,j_3)$.
\end{enumerate}
\end{lemma}

\begin{proof}
Let $f : \HH^n\rightarrow\HH^n$ be a $(j,\rho)$-equivariant Lipschitz map.
For any $p\in\HH^n$, $\gamma\in\Gamma_0$, and $R>0$, if $j(\gamma)\cdot p\in B_p(R/\Lip(f))$, then
$$\rho(\gamma)\cdot f(p) = f\big(j(\gamma)\cdot p\big) \in B_{f(p)}(R),$$
hence $\delta(\rho)\geq\delta(j)/\Lip(f)$ by definition \eqref{eqn:defcritexp} of~$\delta$.
Then (1) follows by letting $\Lip(f)$ tend to $C(j,\rho)$ and taking the logarithm.
The triangle inequality (2) follows from the general inequality $\Lip(f_1\circ f_2)\leq\Lip(f_1)\,\Lip(f_2)$.
\end{proof}

As in Section~\ref{subsec:Thurstonmetric}, let $\T(M)_{\mathrm{Zs}}$ be the subset of $\T(M)$ consisting of elements $j$ such that the Zariski-closure of $j(\Gamma_0)$ in~$G$ is simple (\eg equal to~$G$).
In order to prove Proposition~\ref{prop:generalThurston}, it remains to prove the following.

\begin{proposition}\label{prop:critexp}
If $j, \rho\in \T(M)_{\mathrm{Zs}}$ are distinct, then $d_{\mathrm{Th}}(j,\rho)>0$.
\end{proposition}

Note that Proposition~\ref{prop:critexp} is not true when $j$ or~$\rho$ does not belong to $\T(M)_{\mathrm{Zs}}$: for instance, if $j\in\T(M)_{\mathrm{Zs}}$ takes values in $\PO(2,1)\subset G=\PO(n,1)$, then we may multiply it by a homomorphism with values in the centralizer $\mathrm{O}(n-2)$ of $\PO(2,1)$ in~$G$ without changing the set\linebreak $\{ \mu(j(\gamma))\,|\,\gamma\in\nolinebreak\Gamma_0\}$; this yields a homomorphism $\rho$ in $\T(M)\smallsetminus\T(M)_{\mathrm{Zs}}$ with $d_{\mathrm{Th}}(j,\rho)=d_{\mathrm{Th}}(\rho,j)=0$.

\subsection{Proof of Proposition~\ref{prop:critexp}}

In our setting, the rigidity of the marked length spectrum $\{ \lambda(j(\gamma))\,|\,\gamma\in\Gamma_0\}$ is well known: see \cite[Th.\,2]{kim01}.
We shall use a slightly stronger property, namely the rigidity of the \emph{projectivized} marked length spectrum.

\begin{lemma}\label{lem:rigid-proj-length-spec}
Let $j,\rho\in\T(M)_{\mathrm{Zs}}$ and $C>0$.
If $\lambda(\rho(\gamma))=C\,\lambda(j(\gamma))$ for all $\gamma\in\Gamma_0$, then $j=\rho$ in $\T(M)_{\mathrm{Zs}}$.
\end{lemma}

\begin{proof}
For $i\in\{ 1,2\}$, let $\mathrm{pr}_i : G\times G\rightarrow G$ be the $i$-th projection.
Let $H_1$ (\resp $H_2$) be the identity component (for the real topology) of the Zariski closure of $j(\Gamma_0)$ (\resp $\rho(\Gamma_0)$) in~$G$, and let $H$ be the identity component of the Zariski closure of $(j,\rho)(\Gamma_0)$ in $G\times G$.
Since $H_1$ is simple, the kernel of $\mathrm{pr}_2|_H$ is either $\{ 1\}$ or~$H_1$.
It cannot be $H_1$, otherwise $H=H_1\times H_2$ would be semisimple of real rank~$2$, and so by \cite{ben97} the cone spanned by $\{ (\lambda(\rho(\gamma)),\lambda(j(\gamma)))\,|\,\gamma\in\Gamma_0\}$ would have nonempty interior, contradicting the fact that $\lambda(\rho(\gamma))=C\,\lambda(j(\gamma))$ for all $\gamma\in\Gamma_0$.
Therefore $\mathrm{pr}_2|_H$ is injective, and similarly $\mathrm{pr}_1|_H$ is injective, and so there is a $(j,\rho)$-equivariant isomorphism $H_1\rightarrow H_2$.

Note that any noncompact simple subgroup $H_i$ of $G=\PO(n,1)$ is conjugate to some $\PO(k,1)$ embedded in the standard way, for $k\geq 2$, up to finite index.
Indeed, any $H_i$-orbit in $\HH^n$ is a totally geodesic subspace of~$\HH^n$ \cite[Chap.\,IV, Th.\,7.2]{hel01}, hence is a copy of~$\HH^k$ for some $k\geq 2$, and there is only one way to embed $\HH^k$ into~$\HH^n$ up to isometry (this is clear in the hyperboloid model).

Therefore the $(j,\rho)$-equivariant isomorphism $H_1\rightarrow H_2$ above is given by conjugation by some element of~$G$, and $j$ and~$\rho$ are conjugate under~$G$.
\end{proof}

\begin{proof}[Proof of Proposition~\ref{prop:critexp}]
Assume that $d_{\mathrm{Th}}(j,\rho)=0$ for some $j\neq\rho$ in $\T(M)_{\mathrm{Zs}}$, and let $C:=C(j,\rho)=\delta(j)/\delta(\rho)$.
By Lemma~\ref{lem:rigid-proj-length-spec}, there exists $\gamma_0\in \Gamma_0$ such that $\lambda(\rho(\gamma_0))<C\, \lambda(j(\gamma_0))$.
Let $f : \HH^n\rightarrow\HH^n$ be a $(j,\rho)$-equivariant, $C$-Lipschitz map (Lemma~\ref{lem:Fnonempty}).
The translation axis $\A\subset\HH^n$ of $j(\gamma_0)$ cannot be $C$-stretched by~$f$ since $\lambda(\rho(\gamma_0))< C\lambda(j(\gamma_0))$. 
Therefore we can find $p,q\in\A$ and $\Delta>0$ such that $d(f(p),f(q))\leq C\, d(p,q)-(2+C)\Delta$.
Let $B_p$ (\resp $B_q$) be the ball of diameter $\Delta$ centered at $p$ (\resp $q$), so that $d(f(p'),f(q'))\leq C\, d(p',q')-\Delta$ for all $p'\in B_p$ and $q'\in B_q$.
We can assume moreover that $p,q$ are close enough in the sense that no segment $[p',q']$ with $p'\in B_p$ and $q'\in B_q$ intersects any ball $j(\gamma)\cdot B_p$ or $j(\gamma)\cdot B_q$ with $\gamma\in\Gamma_0\smallsetminus\{1\}$.

Let $\widetilde{\mathcal{U}}$ be the open set of all vectors $(x,\overrightarrow{v})$ in the unit tangent bundle $T^1\HH^n$ such that $x\in B_p$ and $\exp_x(\R_+\overrightarrow{v})$ intersects $B_q$.
Let $X:=j(\Gamma_0)\backslash T^1\HH^n$ be the unit tangent bundle of the quotient manifold $j(\Gamma_0)\backslash\HH^n$, and $\mathcal{U}\subset X$ the projection of~$\widetilde{\mathcal{U}}$.
For $\gamma\in\Gamma_0$ with $j(\gamma)$ hyperbolic, let $N_{\gamma}$ be the number of times that the axis of~$j(\gamma)$ traverses $\mathcal{U}$ in~$X$ (see Figure~\ref{fig:J}); the triangle inequality yields
\begin{equation}\label{eqn:count-passages}
\lambda(\rho(\gamma)) \leq C \, \lambda(j(\gamma)) - N_{\gamma} \Delta.
\end{equation}
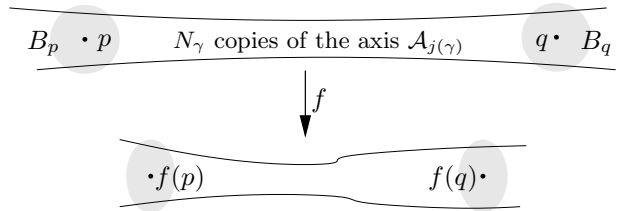
\begin{figure}[h!]
\begin{center}
\labellist
\small\hair 2pt
\pinlabel{$p$} at 70 125
\pinlabel{$f(p)$} at 125 25
\pinlabel{$B_p$} at 25 125
\pinlabel{$q$} at 390 125
\pinlabel{$f(q)$} at 325 25
\pinlabel{$B_q$} at 430 125
\pinlabel{$f$} at 227 83
\pinlabel{\footnotesize{$N_\gamma$ copies of the axis $\A_{j(\gamma)}$}} at 227 125
\endlabellist
\includegraphics[width=8cm]{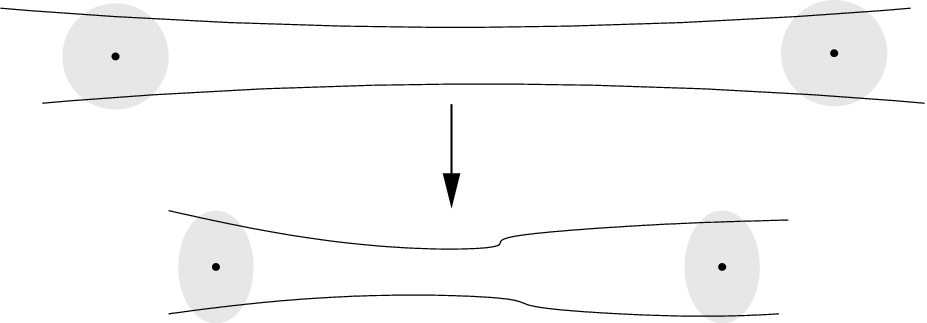}
\caption{Illustration of the proof of Proposition~\ref{prop:critexp} when $N_\gamma=2$.}
\label{fig:J}
\end{center}
\end{figure}
Let $\nu$ be the Bowen--Margulis--Sullivan probability measure on $X$ (see \cite[\S\,1.C]{rob03}).
We have $\nu(\mathcal{U})>0$ since $\mathcal{U}$ intersects the projection of the axis of $j(\gamma_0)$; therefore we can find a continuous function $\psi : X \rightarrow [0,1]$ with compact support contained in~$\mathcal{U}$ such that $\|\psi\|_\infty=1$ and $\varepsilon:=\int_X \psi\,\mathrm{d}\nu >0$.
For any $\gamma\in\Gamma_0$ with $j(\gamma)$ hyperbolic and primitive (\ie not a power of any other $j(\gamma')$), we denote by $\nu_{\gamma}$ the uniform probability measure on~$X$ supported on the axis of~$j(\gamma)$.
Since the support of~$\psi$ is contained in~$B_p$,
\begin{equation}\label{eqn:intpsinugamma}
\int_X \psi\,\mathrm{d}\nu_{\gamma} \leq \|\psi\|_{\infty} \cdot N_{\gamma} \frac{\mathrm{diam}(B_p)}{\lambda(j(\gamma))} = \Delta \frac{N_{\gamma}}{\lambda(j(\gamma))}\,.
\end{equation}
For $R>0$, let $\Gamma_0^{R,j}$ be the set of elements $\gamma\in\Gamma_0$ such that $j(\gamma)$ is primitive hyperbolic and $\lambda(j(\gamma))\leq R$.
By \cite[Th.\,5.1.1]{rob03} (see also \cite{lal89,dp96} for special cases),
\begin{equation}\label{eqn:convergencegeod}
\delta(j)R \, e^{-\delta(j)R} \sum_{\gamma\in\Gamma_0^{R,j}}\ \int_X \psi\,\mathrm{d}\nu_{\gamma} \underset{\scriptscriptstyle R\rightarrow +\infty}{\longrightarrow} \int_X \psi\,\mathrm{d}\nu = \varepsilon.
\end{equation}
Moreover, this convergence is still true if we replace $\psi$ with the constant function equal to~$1$ on~$X$ \cite[Cor.\,5.3]{rob03}, yielding
\begin{equation}\label{eqn:cardGamma0jR}
\#(\Gamma_0^{j,R}) \,\underset{\scriptscriptstyle R\rightarrow +\infty}{\sim}\, \frac{e^{\delta(j)R}}{\delta(j)R}.
\end{equation}
Combined, formulas \eqref{eqn:intpsinugamma}, \eqref{eqn:convergencegeod}, \eqref{eqn:cardGamma0jR} imply that the average value of $N_{\gamma}/\lambda(j(\gamma))$, for $\gamma$ ranging over $\Gamma_0^{R,j}$, is $\geq\frac{\varepsilon}{2\Delta}$ for all large enough~$R$.
Since
$$\frac{N_{\gamma}}{\lambda(j(\gamma))} \,\leq\, \frac{1}{d(p,q)-\Delta} \,\leq\, \frac{1}{2\Delta}$$
for all $\gamma\in\Gamma_0^{R,j}$, this classically implies that a proportion $\geq\frac{\varepsilon}{2}$ of elements $\gamma\in\Gamma_0^{R,j}$ satisfy $N_{\gamma}/\lambda(j(\gamma))\geq\frac{\varepsilon}{4\Delta}$, which by \eqref{eqn:count-passages} entails
$$\lambda(\rho(\gamma)) \,\leq\, C \lambda(j(\gamma)) - N_{\gamma} \Delta \,\leq\, \Big(C-\frac{\varepsilon}{4}\Big)\,\lambda(j(\gamma)) \,\leq\, \Big(C-\frac{\varepsilon}{4}\Big) R.$$
Thus
$$\#\Big(\Gamma_0^{(C-\frac{\varepsilon}{4})R,\rho}\Big) \,\geq\, \frac{\varepsilon}{2}\,\#(\Gamma_0^{R,j})$$ 
for all large enough~$R$. 
Then \eqref{eqn:cardGamma0jR} yields $(C-\frac{\varepsilon}{4})\,\delta(\rho)\geq\delta(j)$, hence $C\frac{\delta(\rho)}{\delta(j)}>\nolinebreak 1$.
\end{proof}

\section{The stretch locus in dimension 2}\label{sec:dim2}

We now focus on results specific to dimension $n=2$.
We first consider the case $C(j,\rho)>1$, for which we recover and extend two aspects of the classical theory \cite{thu86} of the Thurston metric on Teichm\"uller space.
The first aspect is the \emph{chain recurrence} of the lamination $E(j,\rho)$, which we prove in Section~\ref{subsec:chainrec}.
Building on chain recurrence, the second aspect is the \emph{upper semicontinuity} of $E(j,\rho)$ for the Hausdorff topology, namely
$$E(j,\rho) \supset \limsup_{k\rightarrow +\infty}\,E(j_k,\rho_k)$$
for any $(j_k,\rho_k)\rightarrow (j,\rho)$ with $\rho$ and $\rho_k$ reductive, which we prove in Section~\ref{subsec:semicontE}.

We also consider the case $C(j,\rho)<1$ and provide some evidence for Conjecture~\ref{conj:gramination} (describing the stretch locus $E(j,\rho)$) in Section~\ref{subsec:EforC<1}.

In fact, we believe that chain recurrence (suitably defined) should probably also hold in higher dimension for $C(j,\rho)>1$, but we shall use the classification of geodesic laminations on surfaces to prove it here.
Semicontinuity should also hold in higher dimension, not only for $C(j,\rho)>1$ but also in some form for $C(j,\rho)\leq 1$: this is natural to expect in view of Propositions \ref{prop:contCcc} and~\ref{prop:contpropertiesC}.(3) (if the minimal Lipschitz constant varies continuously, so should the stretch locus).
However, our proof hinges on chain recurrence and on the fact that $f$ multiplies arc length along the leaves of the stretch locus: this property does not obviously have a counterpart when $C(j,\rho)<1$ (the stretch locus being no longer a lamination in general), and is at any rate harder to prove (the Kirszbraun--Valentine theorem no longer applies as in Lemma~\ref{lem:MaxStretchedLam}).

\subsection{Chain recurrence in the classical setting}\label{subsec:Thurstonchainrec}

We first recall the notion of chain recurrence and, for readers interested in the more technical aspects of \cite{thu86}, we make the link between the ``maximal, ratio-maximizing, chain recurrent lamination'' $\mu(j,\rho)$ introduced by Thurston in the latter paper, and the stretch locus $E(j,\rho)$ introduced in the present paper.

On a hyperbolic surface~$S$, a geodesic lamination is called \emph{recurrent} if every half-leaf returns arbitrarily often, arbitrarily close to its starting point.
In \cite{thu86}, Thurston introduced the weaker notion of \emph{chain recurrence}.

\begin{definition}\label{def:chainrec}
A geodesic lamination $\dot{\LL}$ on~$S$ is called \emph{chain recurrent} if for every $\dot{p}\in\dot{\LL}$ and $\varepsilon>0$, there exists a simple closed geodesic~$\mathcal{G}$ passing within $\varepsilon$ of~$\dot{p}$ and staying $\varepsilon$-close to~$\dot{\LL}$ in the $\mathcal{C}^1$ sense.
\end{definition}

By ``$\varepsilon$-close in the $\mathcal{C}^1$ sense'' we mean that any unit-length segment of~$\mathcal{G}$ lies $\varepsilon$-close to a segment of~$\dot{\LL}$ (for the Hausdorff metric).
In particular, any recurrent lamination is chain recurrent.
The following is well known.

\begin{fact}\label{fact:classif-lamin}
Any geodesic lamination on~$S$ consists of finitely many disjoint recurrent components, together with finitely many isolated leaves spiraling from one recurrent component to another (possibly the same).
The total number of recurrent components and of isolated leaves can be bounded by an integer depending only on the topology of~$S$.
\end{fact}

By Fact~\ref{fact:classif-lamin}, chain recurrence implies that for any $\varepsilon>0$, any $\dot{p}\in\dot{\LL}$, and any direction of travel along $\dot{\LL}$ from~$\dot{p}$, one can return to~$\dot{p}$ (with the same direction of travel) by following leaves of~$\dot{\LL}$ and occasionally jumping to nearby leaves within distance $\varepsilon$.
(For example, if $\dot{\LL}$ has an isolated leaf spiraling to a simple closed curve and no leaf spiraling out, then $\dot{\LL}$ is \emph{not} chain recurrent.)
By Fact~\ref{fact:classif-lamin}, the number of necessary $\varepsilon$-jumps can be bounded by a number~$m$ depending only on the topology of the surface, and the distances in-between the jumps can be taken arbitrarily large.
In the sequel, we shall call a sequence of leaf segments, separated by a number $\leq m$ of $\varepsilon$-jumps, an \emph{$\varepsilon$-quasi-leaf} of~$\dot{\LL}$.
The closing lemma (Lemma~\ref{lem:closinglemma}) implies that conversely any $\varepsilon$-quasi-leaf can be $\varepsilon$-approximated, in the $\mathcal{C}^1$ sense, by a simple closed geodesic.

We now briefly discuss the relation to \cite{thu86} (this will not be needed from Lemma~\ref{lem:mu=E} onwards).
Let $S$ be a hyperbolic surface of finite volume.
In \cite{thu86}, Thurston associated to any pair $(j,\rho)$ of distinct elements of the Teichm\"uller space $\mathcal{T}(S)$ of~$S$ (\ie type-preserving, geometrically finite representations of $\Gamma_0:=\pi_1(S)$ into $\PO(2,1)\cong\PGL_2(\R)$, of finite covolume, up to conjugation) a subset $\mu(j,\rho)$ of~$S$, defined as the union of all chain recurrent laminations $\dot{\LL}$ that are \emph{ratio-maximizing}, in the sense that there exists a locally $C(j,\rho)$-Lipschitz map from a neighborhood of $\dot{\LL}$ in $(S,j)$ to a neighborhood of $\dot{\LL}$ in $(S,\rho)$, in the correct homotopy class, that multiplies arc length by $C(j,\rho)$ on each leaf of~$\dot{\LL}$.
He proved that $\mu(j,\rho)$ is a lamination \cite[Th.\,8.2]{thu86}, necessarily chain recurrent, and that this lamination is $C(j,\rho)$-stretched by some $C(j,\rho)$-Lipschitz homeomorphism $(S,j)\rightarrow (S,\rho)$, in the correct homotopy class, whose local Lipschitz constant is $<C(j,\rho)$ everywhere outside of $\mu(j,\rho)$.
Indeed, this last property follows from the existence of a concatenation of ``stretch paths'' going from $j$ to~$\rho$ in $\mathcal{T}(S)$ \cite[Th.\,8.5]{thu86} and from the definition of stretch paths in terms of explicit homeomorphisms of minimal Lipschitz constant \cite[\S\,4]{thu86}.
Therefore, the preimage $\widetilde{\mu}(j,\rho)\subset\HH^2$ of Thurston's chain recurrent lamination $\mu(j,\rho)\subset S\simeq j(\Gamma_0)\backslash\HH^2$ contains the stretch locus $E(j,\rho)$ that we have introduced in this paper.
In fact, this inclusion is an equality, as the following variant of Lemma~\ref{lem:maxstretchedlamin} shows (with $\Gamma_0=\pi_1(S)$ and $j,\rho\in\T(S)$).

\begin{lemma}\label{lem:mu=E}
(in dimension $n=2$)
Let $\Gamma_0$ be a torsion-free discrete group and $(j,\rho)\in\Hom(\Gamma_0,G)^2$ a pair of representations with $j$ geometrically finite.
Let $\widetilde{\mu}\subset\HH^2$ be the preimage of some \emph{chain recurrent} lamination on~$S$.
If $\widetilde{\mu}$ is maximally stretched by some $(j,\rho)$-equivariant Lipschitz map $f : \HH^2\rightarrow\HH^2$, then $\widetilde{\mu}$ is contained in the stretch locus $E(j,\rho)$.
\end{lemma}

\begin{proof}
We proceed as in the proof of Lemma~\ref{lem:maxstretchedlamin}, but using closed quasi-leaves instead of recurrent leaves.
Set $C:=C(j,\rho)$.
Consider a geodesic segment $[x,y]$ contained in $\mu$.
By chain recurrence and by the closing lemma (Lemma~\ref{lem:closinglemma}), for any $\varepsilon>0$ there is a simple closed geodesic $\mathcal{G}$ on $(S,j)$ that passes within $\varepsilon$ of~$x$ and is $\varepsilon$-close to an $\varepsilon$-quasi-leaf $\mathcal{L}$ of $\mu$. 
We may assume that $\mathcal{L}$ consists of $m$ or fewer leaf segments, of which one contains $[x,y]$.
Let $\gamma\in\Gamma_0$ correspond to the closed geodesic~$\mathcal{G}$.
Then $\lambda(j(\gamma))=\mathrm{length}(\mathcal{G})\leq\mathrm{length}(\mathcal{L})+m\varepsilon$, and since each leaf segment of~$\mathcal{L}$ is $C$-stretched by~$f$ we see, using the closing lemma again, that
$$\lambda(\rho(\gamma)) \geq C \cdot \big(\mathrm{length}(\mathcal{L}) - 3m\varepsilon\big) \geq C \cdot \big(\lambda(j(\gamma)) - 4m\varepsilon\big).$$
By considering $p,q,p',q'\in\HH^2$ such that $p,q$ project to~$x,y\in \mathcal{L}$ and $p',q'$ to points within $\varepsilon$ from $x,y$ in $\mathcal{G}$, we obtain, exactly as in the proof of Lemma~\ref{lem:maxstretchedlamin}, that for any $f'\in\F^{j,\rho}$,
$$d(f'(p),f'(q)) \geq C \cdot d(p,q) - (4m+4)C\varepsilon.$$
This holds for any $\varepsilon>0$, hence $d(f'(p),f'(q))=Cd(p,q)$ and $p$ belongs to the stretch locus of~$f'$.
\end{proof}

\subsection{Chain recurrence for $C(j,\rho)>1$ in general}\label{subsec:chainrec}

We now prove that the stretch locus $E(j,\rho)$ is chain recurrent in a much wider setting than \cite{thu86}, allowing $j(\Gamma_0)$ to have infinite covolume in~$G$ and $\rho$ to be any representation of $\Gamma_0$ in~$G$ with $C(j,\rho)>1$ (not necessarily injective or discrete).

\begin{proposition}\label{prop:chainrec}
(in dimension $n=2$)
Let $\Gamma_0$ be a torsion-free discrete group and $(j,\rho)\in\Hom(\Gamma_0,G)^2$ a pair of representations with $j$ geometrically finite and $\rho$ reductive (Definition~\ref{def:reductive}).
If $C(j,\rho)>1$, then the image in $S:=j(\Gamma_0)\backslash\HH^2$ of the stretch locus $E(j,\rho)$ is a (nonempty) chain recurrent lamination.
\end{proposition}

\begin{proof}
Let $f_0\in\F^{j,\rho}$ be optimal (Definition~\ref{def:relstretchlocus}), with stretch locus $E:=E(j,\rho)$. 
By Theorem~\ref{thm:lamin} and Lemma~\ref{lem:Fnonempty}, we know that $E$ is a nonempty, $j(\Gamma_0)$-invariant geodesic lamination.
Suppose by contradiction that its image $\dot{E}$ in $S=j(\Gamma_0)\backslash\HH^2$ is not chain recurrent.
We shall ``improve'' $f_0$ by decreasing its stretch locus, which will be absurd.

Given $\dot{p}\in\dot{E}$ and a direction (``forward'') of travel from~$\dot{p}$, define the \emph{forward chain closure} $\dot{E}_p$ of $\dot{p}$ in~$\dot{E}$ as the subset of~$\dot{E}$ that can be reached from~$\dot{p}$, starting forward, by following $\varepsilon$-quasi-leaves of~$\dot{E}$ for positive time, for any $\varepsilon>0$.
Clearly, $\dot{E}_p$ is the union of a closed sublamination of~$\dot{E}$ and of an open half-leaf issued from~$\dot{p}$. If $\dot{E}_p$ contains $\dot{p}$ for all $\dot{p}\in\dot{E}$ and choices of forward direction, then for any $\varepsilon>0$ we can find a closed $\varepsilon$-quasi-leaf of $\dot{E}$ through~$\dot{p}$. 
Since $\dot{E}$ is not chain recurrent by assumption, this is not the case: we can therefore choose a point $\dot{p}\in\dot{E}$ and a direction of travel such that $\dot{E}_p$ does \emph{not} contain~$\dot{p}$.

Then $\dot{E}_p$ is orientable: otherwise for any $\varepsilon>0$ we could find an $\varepsilon$-quasi-leaf of~$\dot{E}$ through~$\dot{p}$ by following a quasi-leaf from~$\dot{p}$, ``jumping'' onto another (quasi)-leaf with the reverse orientation, and getting back to~$\dot{p}$, which would contradict the fact that $\dot{E}_p$ does not contain~$\dot{p}$.

Let $\dot{\Upsilon}$ be the lamination of~$S$ obtained by removing from~$\dot{E}_p$ the (isolated) half-leaf issued from~$\dot{p}$. Then $\dot{\Upsilon}$ inherits an orientation from the ``forward'' orientation of $\dot{E}_p$.
No leaf of $\dot{E}\smallsetminus\dot{\Upsilon}$ can be \emph{outgoing} from~$\dot{\Upsilon}$, otherwise it would automatically belong to~$\dot{\Upsilon}$.
But at least one leaf of $\dot{E}\smallsetminus\dot{\Upsilon}$ is \emph{incoming} towards~$\dot{\Upsilon}$: namely, the leaf $\dot{\ell}$ containing~$\dot{p}$.

The geodesic lamination~$\dot{\Upsilon}$ fills some subsurface $\Sigma\subset S$ with geodesic boundary (possibly reduced to a single closed geodesic).
Let $\dot{\mathcal{U}}\subset S$ be a uniform neighborhood of~$\Sigma$, with the same topological type as~$\Sigma$, such that $\dot{\mathcal{U}}\cap\dot{E}$ is the union of the oriented lamination~$\dot{\Upsilon}$ and of some (at least one) incoming half-leaves.
Up to shifting the point~$\dot{p}$ along its leaf~$\dot{\ell}$, we may assume that $\dot{p}\in\partial\dot{\mathcal{U}}$.

Let $\mathcal{U}$ and $\Upsilon$ be the (full) preimages of $\dot{\mathcal{U}}$ and $\dot{\Upsilon}$ in~$\HH^2$.
To reach a contradiction, we shall modify $f_0$ on~$\mathcal{U}$. 
The modification on~$\Upsilon$ itself is simply to replace $f|_{\Upsilon}$ with $f_{\varepsilon}:=f_0\circ\Phi_{-\varepsilon}$, where $(\Phi_t)_{t\in\R}$ is the flow on the oriented lamination~$\Upsilon$ and $\varepsilon>0$ is small enough.
We make the following two claims, for $C:=C(j,\rho)$:
\begin{enumerate}
  \item[$(i)$] the map~$f_{\varepsilon}$ is still $C$-Lipschitz on~$\Upsilon$, for all small enough $\varepsilon>0$;
  \item[$(ii)$] the map~$f_{\varepsilon}$ extends to a $C$-Lipschitz, $(j,\rho)$-equivariant map $f$ on~$\mathcal{U}$, that agrees with $f_0$ on~$\partial\mathcal{U}$, for all small enough $\varepsilon>0$.
\end{enumerate}
\begin{figure}[h!]
\begin{center}
\labellist
\small\hair 2pt
\pinlabel{$p$} at 92 206
\pinlabel{$\Upsilon$} at 170 145
\pinlabel{$\partial \mathcal{U}$} at 170 201
\pinlabel{$\ell$} at 107 170
\pinlabel{$q$} at 102 62
\pinlabel{$q'$} at 117 50
\pinlabel{$\Phi_{-\varepsilon}(q')$} at 139 70
\endlabellist
\includegraphics[width=5.5cm]{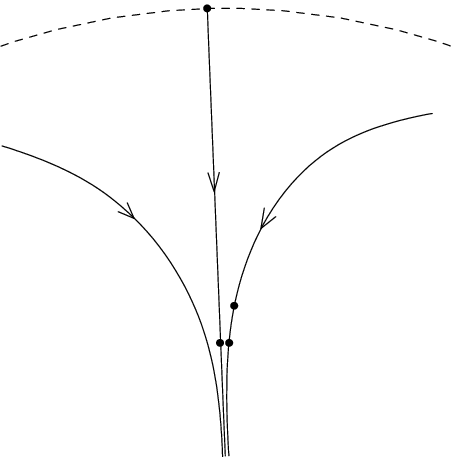}
\caption{Flowing back by $\Phi_{-\varepsilon}$ brings $q'$ closer to $p$.}
\label{fig:K}
\end{center}
\end{figure}
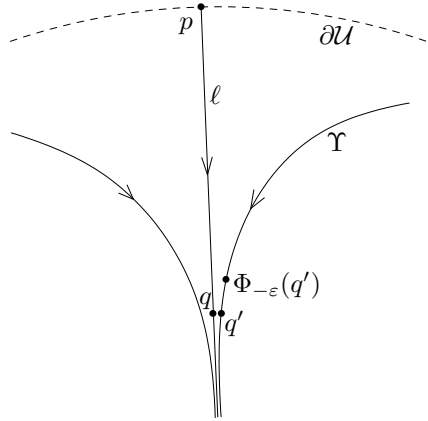
This will prove that the leaf $\ell$ of~$E$ containing a lift~$p$ of~$\dot{p}$ did not have to be maximally stretched after all, a contradiction: indeed, consider $q\in\ell$ far enough from $p$, at distance $<\varepsilon/4$ from some point $q'\in\Upsilon$, such that $\Phi_{-\varepsilon}(q')$ is still within $<\varepsilon/4$ from the point of~$\ell$ at distance $\varepsilon$ from $q$ (see Figure~\ref{fig:K}).
Then $d(p,\Phi_{-\varepsilon}(q'))\leq d(p,q)-\varepsilon/2$, which implies
\begin{eqnarray*}
d(f_\varepsilon(p),f_\varepsilon(q)) & \leq & d(f_\varepsilon(p),f_\varepsilon(q')) + d(f_\varepsilon(q'),f_\varepsilon(q))\\
& \leq & C\,\big(d(p,\Phi_{-\varepsilon}(q')) + \varepsilon/4\big)\\
& \leq & C\, (d(p,q)-\varepsilon/4) < C\, d(p,q).
\end{eqnarray*}

\smallskip

$\bullet$ \textbf{Proof of $(ii)$ assuming $(i)$.}
By Remark~\ref{rem:pathlength}.(2), it is sufficient to consider one connected component $A$ of $\mathcal{U}\smallsetminus\Upsilon$ in $\HH^2$ and, assuming~$(i)$, to prove that for any small enough $\varepsilon>0$ the map $f_{\varepsilon}$ extends to a $C$-Lipschitz, $(j,\rho)$-equivariant map $f$ on~$A$, that agrees with $f_0$ on~$\partial A$.
Fix such a connected component~$A$; its image in~$S$ is an annulus.
By Theorem~\ref{thm:Kirszbraunequiv}, it is sufficient to prove that $d(f_{\varepsilon}(x),f_0(y))\leq C d(x,y)$ for any geodesic segment $[x,y]$ across $A$ with $x\in\Upsilon$ and $y\in\partial A$.
Note that the length of such segments is uniformly bounded from below, by $d(\Upsilon,\partial A)$.

\begin{figure}[h!]
\begin{center}
\labellist
\small\hair 2pt
\pinlabel{$\ell'$} at 70 80
\pinlabel{$x$} at 105 50
\pinlabel{$y$} at 100 98
\pinlabel{$E_{\delta}$} at 150 84
\pinlabel{$A$} at 220 90
\pinlabel{$\Upsilon$} at 270 58
\pinlabel{$\partial \mathcal{U}$} at 320 115
\endlabellist
\includegraphics[width=12cm]{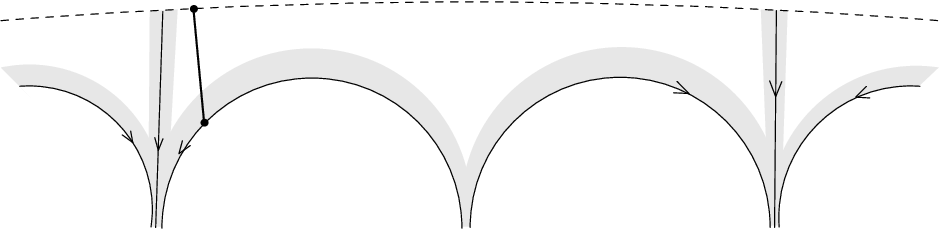}
\caption{A segment $[x,y]$ across the lifted annulus $A$.}
\label{fig:L}
\end{center}
\end{figure}
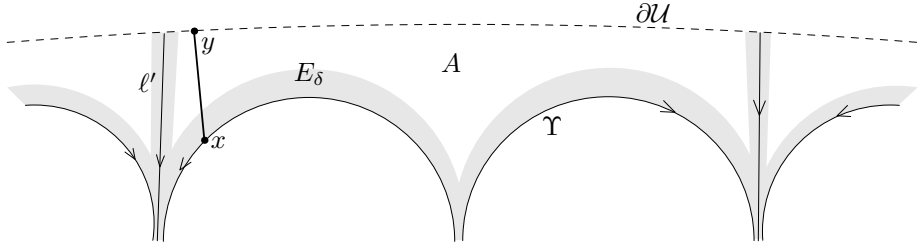

For any $0<\delta<d(\Upsilon,\partial A)$, let $E_{\delta}$ be the $\delta$-neighborhood of the lamination~$E$ in the lifted annulus~$A$ (see Figure~\ref{fig:L}).
By Lemma~\ref{lem:localLip}, since $f_0$ is optimal, 
\begin{equation}\label{eqn:LipAeta}
\sup_{x\in A\smallsetminus E_{\delta}} \Lip_{x}(f_0) < C.
\end{equation}
If no leaf of~$E$ entering~$\Upsilon$ meets~$A$, then all geodesic segments $[x,y]$ as above spend a definite amount of length (at least $d(\dot{\Upsilon},\partial A)-\delta$) in $A\smallsetminus E_{\delta}$, and so \eqref{eqn:LipAeta} implies
$$d(f_0(x),f_0(y)) \leq C\,(d(x,y) - \varepsilon_0)$$
for some $\varepsilon_0>0$ independent of $[x,y]$.
Therefore
$$d(f_{\varepsilon}(x),f_0(y)) \leq d(f_{\varepsilon}(x),f_0(x)) + d(f_0(x),f_0(y)) \leq C\,d(x,y)$$
for all $0<\varepsilon<\varepsilon_0$.
Now suppose that there are leaves of~$E$ entering~$\Upsilon$ that meet~$A$. 
The collection of such leaves is finite modulo the stabilizer of $A$.
There exists $\delta>0$ such that if $[x,y]$ is contained in the $\delta$-neighborhood of some leaf $\ell'$ of~$E$ entering~$A$, then the function $t\mapsto d(\Phi_{-t}(x),y)$ is decreasing for $t\in [0,1]$, because the direction of the flow $\Phi$ at $x$ is essentially the same as the direction of~$\ell'$; in particular,
$$d(f_{\varepsilon}(x),f_0(y)) \leq C\,d(\Phi_{-\varepsilon}(x),y) \leq C\,d(x,y)$$
for all $\varepsilon\in (0,1]$.
There also exists $\delta'\in (0,\delta)$ such that if a geodesic segment $[x,y]$ as above is \emph{not} contained in the $\delta$-neighborhood of one of the finitely many leaves entering~$\Upsilon$, then it meets $A\smallsetminus E_{\delta'}$; in particular, it spends a definite amount of length (at least $\delta'/2$) in $A\smallsetminus E_{\delta'/2}$, and we conclude as above, using \eqref{eqn:LipAeta} with $\delta'/2$ instead of~$\delta$.

\smallskip

$\bullet$ \textbf{Proof of $(i)$.}
By Remark~\ref{rem:pathlength}.(2), it is enough to consider one connected component $A$ of $\mathcal{U}\smallsetminus\Upsilon$ in $\HH^2$ and prove that $d(f_{\varepsilon}(x),f_{\varepsilon}(y))\leq C d(x,y)$ for all $x,y\in\Upsilon\cap\partial A$.
If $\Upsilon\cap\partial A$ is a geodesic line (corresponding to a closed geodesic of~$\dot{\Upsilon}$), then $(\Phi_{-\varepsilon})|_{\Upsilon\cap\partial A}$ is an isometry and so $\Lip_{\Upsilon\cap\partial A}(f_{\varepsilon})\leq\Lip_{\Upsilon\cap\partial A}(f)=C$ (in fact $\Upsilon\cap\partial A$ is $C$-stretched by~$f_{\varepsilon}$).
Otherwise, $\Upsilon\cap\partial A$ is a countable union of geodesic lines $D_i$, $i\in\Z$, with $D_i$ and~$D_{i+1}$ asymptotic to each other, both oriented in the direction of the ideal spike they bound if $i$ is odd, and both oriented in the reverse direction if $i$ is even; the leaves of~$E$ entering~$\Upsilon$ do so in the spikes.
Suppose by contradiction that there is a sequence $(\varepsilon_k)\in (\R_+^{\ast})^{\N}$ tending to~$0$ and, for every $k\in\N$, a pair $(x_k,y_k)$ of points of $\Upsilon\cup\partial A$ such that
\begin{equation}\label{eqn:chainrecurrent}
d(f_{\varepsilon_k}(x_k),f_{\varepsilon_k}(y_k)) > C\cdot d(x_k,y_k).
\end{equation}
Note that the Hausdorff distance from $[x_k,y_k]$ to the nearest leaf segment of~$\Upsilon$ tends to zero as $k\rightarrow +\infty$.
Indeed, as above, for any $\delta>0$, if a geodesic segment $[x,y]$ is \emph{not} contained in the $\delta$-neighborhood $E_{\delta}$ of $E$ in~$A$, then it spends a definite amount of length (at least $\delta/2$) in $A\smallsetminus E_{\delta/2}$, and \eqref{eqn:LipAeta} with $\delta/2$ instead of~$\delta$ forces $d(f_{\varepsilon}(x),f_{\varepsilon}(y))\leq C d(x,y)$ for small enough~$\varepsilon$.
This proves that the Hausdorff distance from $[x_k,y_k]$ to the nearest segment of~$E$ tends to zero as $k\rightarrow +\infty$, and this segment actually lies in $\Upsilon$ because $x_k$ and~$y_k$ both belong to~$\Upsilon$ and there are locally only finitely many leaves of~$E$ entering~$\Upsilon$.

Up to replacing $x_k$ and~$y_k$ by $j(\Gamma_0)$-translates and passing to a subsequence, we can in fact suppose that there exists $i\in\Z$ such that both $d(x_k,D_i)$ and $d(y_k,D_i)$ tend to zero as $k\rightarrow +\infty$; indeed, the set of lines~$D_i$ is finite modulo the stabilizer of~$A$.
Up to switching $x_k$ and~$y_k$ and passing to a subsequence, we can suppose that either $(x_k,y_k)\in D_i\times D_{i+1}$ for all~$k$, or $(x_k,y_k)\in D_{i-1}\times D_{i+1}$ for all~$k$; the case $(x_k,y_k)\in D_i\times D_i$ is excluded by the assumption \eqref{eqn:chainrecurrent} since $D_i$ is $C$-stretched under~$f_{\varepsilon_k}$. 

Let $y'_k$ be the point of~$D_i$ on the same horocycle as $y_k$ in the ideal spike of~$A$ bounded by~$D_i$ and $D_{i+1}$, and let $\eta_k\geq 0$ be the length of the piece of horocycle from $y_k$ to~$y'_k$. 
If $x_k\in D_{i-1}$, define similarly an arc of horocycle from $x_k$ to $x'_k\in D_i$, of length $\xi_k$; otherwise, set $(x'_k,\xi_k)=(x_k,0)$.
Since $d(x_k,D_i)$ and $d(y_k,D_i)$ tend to zero as $k\rightarrow +\infty$, so do $\xi_k$ and~$\eta_k$.

\begin{figure}[h!]
\begin{center}
\labellist
\small\hair 2pt
\pinlabel{$x^\ast$} at 87 153
\pinlabel{$D_i$} at 150 140
\pinlabel{$D_i$} at 200 29
\pinlabel{$x_k$} at 185 157
\pinlabel{$x_k$} at 85 58
\pinlabel{$y_k$} at 249 176
\pinlabel{$y_k$} at 322 62
\pinlabel{$y'_k$} at 245 142
\pinlabel{$y'_k$} at 305 18
\pinlabel{$\eta_k$} at 245 158
\pinlabel{$\eta_k$} at 312 40
\pinlabel{$D_{i+1}$} at 295 160
\pinlabel{$D_{i+1}$} at 350 40
\pinlabel{$D_{i-1}$} at 48 36
\pinlabel{$x'_k$} at 103 20
\pinlabel{$\xi_k$} at 90 40
\endlabellist
\includegraphics[width=12cm]{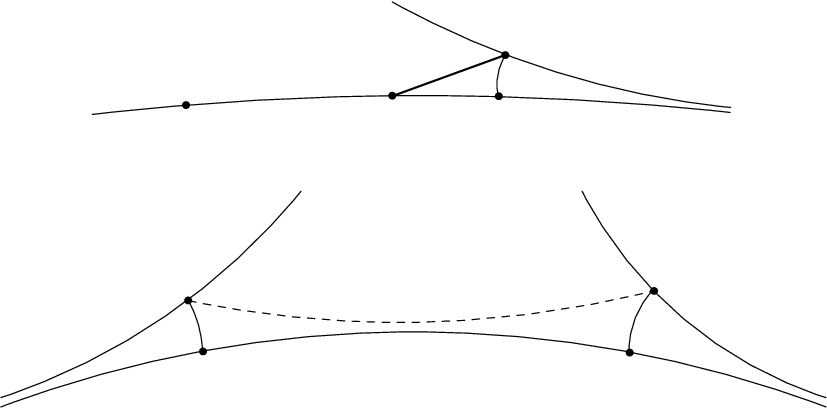}
\caption{Distance estimates between points of $D_{i-1}, D_i,D_{i+1}$.}
\label{fig:M}
\end{center}
\end{figure}
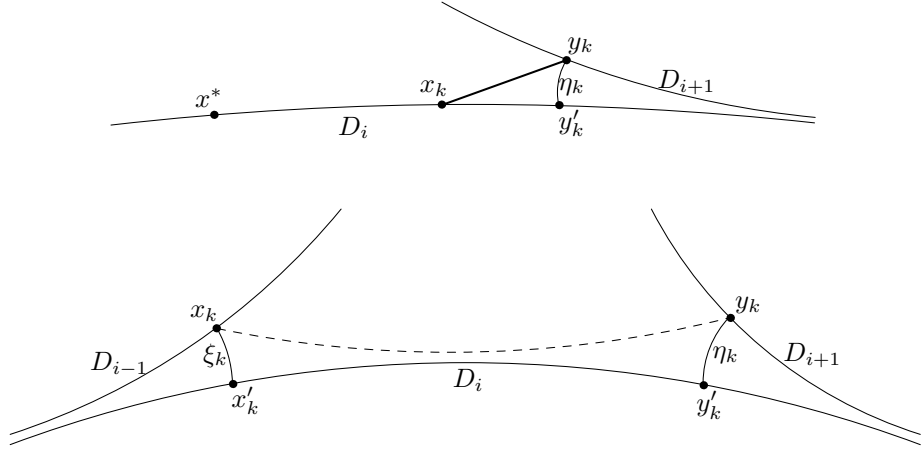

We claim that, up to passing to a subsequence and replacing $x_k$ and~$y_k$ by other points on the same leaves, still subject to \eqref{eqn:chainrecurrent}, we can assume that $d(x_k,y_k)\rightarrow\nolinebreak +\infty$ as $k\rightarrow +\infty$.
Indeed, this is already the case if $(x_k,y_k)\in D_{i-1}\times D_{i+1}$ for all $k$, because $\xi_k,\eta_k\rightarrow 0$. If $(x_k,y_k)\in D_i\times D_{i+1}$ for all $k$, note that the $C$-Lipschitz map~$f$ stretches $D_i$ and~$D_{i+1}$ maximally and sends them to two geodesic lines of~$\HH^2$, necessarily asymptotic.
Moreover, $f(y_k)$ and $f(y'_k)$ lie at the same depth in the spike bounded by $f(D_i)$ and $f(D_{i+1})$: indeed, the distance between the horocycles through $f(y_k)$ and through $f(y'_k)$ is independent of~$k$; if it were nonzero, then for large enough~$k$ we would obtain a contradiction with the fact that $f$ is Lipschitz (recall $\eta_k\rightarrow 0$).
Using \eqref{eqn:exphorodist}, we see that there exists $Q\geq 0$ such that for any integer~$k$, the piece of horocycle connecting $f(y_k)$ to $f(y'_k)$ has length $Q\eta_k^C$, and the piece of horocycle connecting $f_{\varepsilon_k}(y_k)$ to $f_{\varepsilon_k}(y'_k)$ has length $Qe^{\pm \varepsilon_k}\eta_k^C$, which is $\ll \eta_k$ since $C>1$.
In particular, $x_k\neq y'_k$ for all large enough~$k$ (since $x_k$ and~$y_k$ satisfy \eqref{eqn:chainrecurrent}); in other words, $x_k$ and~$y_k$ lie at distinct depths inside the spike of~$A$ bounded by $D_i$ and~$D_{i+1}$ (see Figure~\ref{fig:M}, top).
If $y_k$ lies deeper than~$x_k$ (which we can assume by symmetry), then for any $x^{\ast}\in D_i$ less deep than~$x_k$,
$$\frac{\pi}{2} < \widehat{x^{\ast} x_k y_k} < \widehat{f_{\varepsilon_k}(x^{\ast}) f_{\varepsilon_k}(x_k) f_{\varepsilon_k}(y_k)} < \pi.$$
Moreover, $d(f_{\varepsilon_k}(x^{\ast}),f_{\varepsilon_k}(x_k))=Cd(x^{\ast},x_k)$ and \eqref{eqn:chainrecurrent} holds, hence
$$d(f_{\varepsilon_k}(x^{\ast}),f_{\varepsilon_k}(y_k)) > Cd(x^{\ast},y_k)$$
by Toponogov's theorem \cite[Lem.\,II.1.13]{bh99}.
Thus, up to replacing $x_k$ by some fixed~$x^{\ast}$, we may assume that $d(x_k,y_k)\rightarrow +\infty$.

Using \eqref{eqn:twospikes}, we see that
$$d(x_k,y_k) = d(x'_k,y'_k) + (\xi_k^2 + \eta_k^2)(1+o(1))$$
(see Figure~\ref{fig:M}, bottom).
Similarly, given that the length of the piece of horocycle from $f(x_k)$ to~$f(x'_k)$ in the spike bounded by $D_{i-1}$ and~$D_i$ is $Q\xi_k^C$ for some $Q\geq 0$ independent of~$k$ (see above) and that the length of the piece of horocycle from $f(y_k)$ to~$f(y'_k)$ in the spike bounded by $D_i$ and~$D_{i+1}$ is $Q'\eta_k^C$ for some $Q'\geq 0$ independent of~$k$, we obtain
$$d(f_{\varepsilon_k} (x_k),f_{\varepsilon_k}(y_k))\leq 
d(f_{\varepsilon_k}(x'_k),f_{\varepsilon_k}(y'_k)) + (Q^2\xi_k^{2C} + {Q'}^2\eta_k^{2C})(1+o(1)).$$
Since $d(f_{\varepsilon_k}(x'_k),f_{\varepsilon_k}(y'_k))= C d(x'_k,y'_k)$ and since $\xi_k^C=o(\xi_k)$ and $\eta_k^C=o(\eta_k)$, we find that $d(f_{\varepsilon_k} (x_k),f_{\varepsilon_k}(y_k))\leq C d(x_k,y_k)$ for all large enough~$k$, contradicting \eqref{eqn:chainrecurrent}.
This completes the proof of~$(i)$.
\end{proof}

\subsection{Semicontinuity for $C(j,\rho)>1$}\label{subsec:semicontE}

The notion of chain recurrence (Definition~\ref{def:chainrec}) is closed for the Hausdorff topology: any compactly-supported lamination which is a Hausdorff limit of chain recurrent laminations is chain recurrent \cite[Prop.~6.1]{thu86}.
It is therefore relevant to consider (semi)continuity issues.

In the classical setting, Thurston \cite[Th.\,8.4]{thu86} proved that his maximal ratio-maximizing chain recurrent lamination $\mu(j,\rho)$ varies in an upper semicontinuous way as $j$ and~$\rho$ vary over the Teichm\"uller space $\mathcal{T}(S)$ of~$S$.
In other words, by Lemma~\ref{lem:mu=E}, the stretch locus $E(j,\rho)$ varies in an upper semicontinuous way over $\mathcal{T}(S)$: for any sequence $(j_k,\rho_k)_{k\in\N}$ of elements of $\mathcal{T}(S)^2$ converging to $(j,\rho)$,
$$E(j,\rho) \supset \limsup_{k\rightarrow +\infty} E(j_k,\rho_k),$$
where the limsup is defined with respect to the Hausdorff topology.

We now work in a more general setting and show how the chain recurrence of the stretch locus $E(j,\rho)$ (Proposition~\ref{prop:chainrec}) implies upper semicontinuity.

\begin{proposition}\label{prop:semicontE}
In dimension $n=2$, the stretch locus $E(j,\rho)$ is upper semicontinuous on the open subset of $\Hom_{j_0}(\Gamma_0,G)\times\Hom(\Gamma_0,G)^{\mathrm{red}}$ where $C(j,\rho)>1$.
\end{proposition}

Here we denote by $\Hom_{j_0}(\Gamma_0,G)$ the space of geometrically finite representations of $\Gamma_0$ in~$G$ with the same cusp type as the fixed representation~$j_0$ (as in Section~\ref{sec:lipcont}) and by $\Hom(\Gamma_0,G)^{\mathrm{red}}$ the space of reductive representations $\rho\in\Hom(\Gamma_0,G)$ (Definition~\ref{def:reductive}).
These two sets are endowed with the induced topology from $\Hom(\Gamma_0,G)$.
The condition $C(j,\rho)>1$ is open by Proposition~\ref{prop:contpropertiesC}.(2).

\begin{proof}[Proof of Proposition~\ref{prop:semicontE}]
By Lemma~\ref{lem:finiteindex}, we may assume that $\Gamma_0$ is torsion-free.
Let $(j_k,\rho_k)_{k\in\N}$ be a sequence of elements of $\Hom_{j_0}(\Gamma_0,G)\times\Hom(\Gamma_0,G)^{\mathrm{red}}$ with $C(j_k,\rho_k)>1$ converging to some $(j,\rho)\in\Hom_{j_0}(\Gamma_0,G)\times\Hom(\Gamma_0,G)^{\mathrm{red}}$ with $C(j,\rho)>1$.
Since the stretch locus $E(j,\rho)$ is empty with our definitions when $C(j,\rho)=+\infty$, and the condition $C(j,\rho)=+\infty$ is open by Lemma \ref{lem:C<infty}, we may assume that the $C(j_k, \rho_k)$ and $C(j,\rho)$ are all finite.
Recall from Section~\ref{subsec:C-highcontinuous} the proof of the fact (labelled (B) there) that $\limsup C(j_k, \rho_k)$, if greater than $1$, gives a lower bound for $C(j,\rho)$. By the same argument as in that proof, up to passing to a subsequence, the stretch loci $E(j_k,\rho_k)$ are $j_k(\Gamma_0)$-invariant geodesic laminations that converge to some $j(\Gamma_0)$-invariant geodesic lamination~$\LL$, compact in $j(\Gamma_0)\backslash\HH^n$.
Moreover, the image of $\LL$ in $j(\Gamma_0)\backslash\HH^n$ nearly carries simple closed curves corresponding to elements $\gamma\in\Gamma_0$ with $\lambda(\rho(\gamma))/\lambda(j(\gamma))$ arbitrarily close to $C(j,\rho)$.
However, this does not immediately imply that $\LL$ is contained in $E(j,\rho)$: we need to improve the ``multiplicative error'' to an ``additive error''.
The idea is similar to Lemmas \ref{lem:maxstretchedlamin} and~\ref{lem:mu=E}, but with varying $j,\rho$.

Suppose by contradiction that $\LL$ contains a point $p\notin E(j,\rho)$.
According to Lemma~\ref{lem:optimallocallycst}, there is an element $f\in\F^{j,\rho}$ that is constant on some ball centered at~$p$, with radius~$\delta>0$.
Since the $E(j_k,\rho_k)$ are chain recurrent (Proposition~\ref{prop:chainrec}), so is~$\LL$.
Let $\mathcal{G}$ be a simple closed geodesic in $j(\Gamma_0)\backslash\HH^2$ passing within $\delta/2$ of $p$ and approached by a $\frac{\delta}{16mC}$-quasi-leaf of~$\LL$ (in the sense of Section~\ref{subsec:Thurstonchainrec}), made of at most $m$ leaf segments of~$\LL$, where $m$ is the integer (depending only on the topology of $S=j(\Gamma_0)\backslash\HH^2$) defined after Fact~\ref{fact:classif-lamin}.
Let $\gamma\in\Gamma_0$ correspond to~$\mathcal{G}$.
By Hausdorff convergence, for large enough~$k$, the geodesic representative of~$\mathcal{G}$ in $j_k(\Gamma_0)\backslash\HH^2$ is approached by a $\frac{\delta}{8mC}$-quasi-leaf of $E(j_k,\rho_k)$, made of at most $m$ leaf segments.
Since $E(j_k,\rho_k)$ is maximally stretched by a factor $C_k:=C(j_k, \rho_k)$ by any element of~$\F^{j_k,\rho_k}$, it follows, as in the proof of Lemma~\ref{lem:mu=E}, that
$$\big|\lambda(\rho_k(\gamma)) - C_k\,\lambda(j_k(\gamma))\big| \leq 4mC_k \cdot \frac{\delta}{8mC}\,.$$
Since $C_k$ tends to~$C$ by Proposition~\ref{prop:contpropertiesC}.(2)--(3) and since $\lambda$ is continuous, the left-hand side converges to $|\lambda(\rho(\gamma))-C\,\lambda(j(\gamma))|$ as $k\rightarrow +\infty$, while the right-hand side converges to $\delta/2$.
However, this left-hand limit is $\geq C\delta$ since $f$ is constant on the ball of radius~$\delta$ centered at~$p$, which contains a segment of~$\mathcal{G}$ of length~$\delta$.
This is absurd, hence $\LL\subset E(j,\rho)$.
\end{proof}

\subsection{The stretch locus for $C(j,\rho)<1$} \label{subsec:EforC<1}

Still in dimension $n=2$, let $\Gamma_0$, $(j,\rho)$, $K\subset\HH^n$ compact, and~$\varphi:K\rightarrow \HH^n$ be as in Section~\ref{sec:stretchlocus} (with $K$ possibly empty).
The relative stretch locus $E_{K,\varphi}(j,\rho)$ behaves very differently depending on whether $C_{K,\varphi}(j,\rho)$ is smaller than, equal to, or larger than~$1$.
Let us give a simple example to illustrate the contrast.

\begin{example}\label{ex:K=3points}
We take $\Gamma_0$ to be trivial.
Fix $o\in \HH^2$ and let $(a_s)_{s\geq 0}$, $(b_s)_{s\geq 0}$, and $(c_s)_{s\geq 0}$ be three geodesic rays issued from~$o$, parametrized at unit speed, forming angles of $2\pi/3$ at~$o$.
Let $K=\{ a_t,b_t,c_t\}$ for some $t>0$ and let $\varphi : K\rightarrow\HH^2$ be given by $\varphi(a_t)=a_T$, $\varphi(b_t)=b_T$, and $\varphi(c_t)=c_T$ for some $T>0$.
Then $\Lip(\varphi)=g(T)/g(t)$, where $g : \R_+\rightarrow\R_+$ is given by $g(s)=d(a_s,b_s)$.
By convexity of the distance function, the function~$g$ is \emph{strictly convex}, asymptotic to $\sqrt{3}s$ for $s\rightarrow 0$ and to $2s$ for $s\rightarrow +\infty$.
(Explicitly, $g(s) = 2\,\arcsinh (\sqrt{3/4}\sinh s)$ by \eqref{eqn:circle}.)
\begin{itemize}
  \item If $t<T$, then $\Lip(\varphi)>T/t>1$ by strict convexity of~$g$.
  By Theorem~\ref{thm:Kirszbraunopt}, the map~$\varphi$ extends to~$\HH^2$ with the same Lipschitz constant and with stretch locus the perimeter of the triangle $a_t b_t c_t$; this stretch locus is the smallest possible by Remark~\ref{rem:pathlength}.(1).
  \item If $t=T$, then $\varphi$ is $1$-Lipschitz and has a unique $1$-Lipschitz extension to the (filled) triangle $a_t b_t c_t$, namely the identity map.
  An optimal extension to~$\HH^2$ is obtained by precomposing with the closest-point projection onto the triangle $a_t b_t c_t$; the stretch locus is this triangle.
  \item If $t>T$, then $\Lip(\varphi)<T/t<1$ by strict convexity of~$g$.
  However, the optimal Lipschitz constant of an extension of $\varphi$ to~$\HH^2$ cannot be less than $T/t$: indeed, such an extension may be assumed to fix~$o$ by symmetry, and $\frac{d(o,a_T)}{d(o,a_t)}=T/t$.
  It follows from the construction used in Section~\ref{ex:C'neqC} below that a $(T/t)$-Lipschitz extension of $\varphi$ to~$\HH^2$ does indeed exist, and the stretch locus is equal to the union of the geodesic segments $[o,a_t]$, $[o,b_t]$, $[o,c_t]$.
\end{itemize}
Although the stretch locus may vary abruptly in the above, note that this variation is upper semicontinuous in $(t,T)$ for the Hausdorff topology, in agreement with a potential generalization of Proposition~\ref{prop:semicontE} to $C\leq 1$.
\end{example}

We now consider the case when $K$ is empty.
Here is some evidence in favor of Conjecture~\ref{conj:gramination}, which claims that for $C(j,\rho)<1$ the stretch locus $E(j,\rho)$ should be what we call a \emph{gramination}:
\begin{itemize}
  \item In Section~\ref{ex:C'neqC}, we give a construction, for certain Coxeter groups~$\Gamma_0$, of pairs $(j,\rho)$ with $j$ convex cocompact, $j(\Gamma_0)\backslash\HH^2$ compact, and $C(j,\rho)<1$, for which the stretch locus $E(j,\rho)$ is a trivalent graph.
  \item Consider the examples constructed in \cite[\S\,4.4]{sal00}: for any compact hyperbolic surface~$S$ of genus~$g$ and any integer~$k$ with $|k|\leq 2g-2$, Salein constructed highly symmetric pairs $(j,\rho)\in\Hom(\pi_1(S),G)^2$ with $j$ Fuchsian such that $\rho$ has Euler number~$k$; a construction similar to Section~\ref{ex:C'neqC} shows that the stretch locus of such a pair $(j,\rho)$ is a regular graph of degree~$4g$.
  \item In Section~\ref{ex:C'neqCnoncompact}, we give a construction of pairs $(j,\rho)$ with $j$ convex cocompact, $j(\Gamma_0)\backslash\HH^2$ \emph{noncompact}, and $C(j,\rho)<1$, for which the stretch locus $E(j,\rho)$ is a trivalent graph.
  It is actually possible to generalize this construction and obtain, for any given convex cocompact hyperbolic surface~$S$ of infinite volume and any given trivalent graph $\mathcal{G}$ retract of~$S$, an \emph{open} set of pairs $(j,\rho)\in\Hom(\Gamma_0,G)^2$ with $j$ convex cocompact for which the stretch locus $E(j,\rho)$ is a trivalent graph of~$\HH^2$, with geodesic edges, whose projection to $j(\Gamma_0)\backslash\HH^2$ is a graph isotopic to~$\mathcal{G}$ (see Remarks~\ref{rem:graph}).
  \item It is also possible to construct examples of pairs $(j,\rho)\in\Hom(\Gamma_0,G)^2$ with $j$ geometrically finite and $C(j,\rho)<1$ for which the stretch locus $E(j,\rho)$ is a geodesic lamination: see Sections \ref{ex:nonreductive1} and~\ref{ex:nonreductive2} for instance.
\end{itemize}

Here is perhaps a first step towards proving Conjecture~\ref{conj:gramination}.

\begin{lemma}\label{lem:Ecoconvex}
In dimension $n=2$, let $(j,\rho)\in\Hom(\Gamma_0,G)^2$ be a pair of representations with $j$ geometrically finite and $\F^{j,\rho}\neq\emptyset$.
Each connected component of the complement of the stretch locus $E(j,\rho)\subset\HH^2$ is convex.
\end{lemma}

\begin{proof}
Suppose by contradiction that the open set $\HH^2\smallsetminus E(j,\rho)$ has a nonconvex component~$\mathcal{U}$.
There exists a smooth arc $A_0\subset \mathbb{H}^2\smallsetminus E(j,\rho)$ whose endpoints $x_0, y_0$ are connected by a segment intersecting $E(j,\rho)$. 
We can assume that $A_0\cup [x_0, y_0]$ is a Jordan curve, and is arbitrarily small: this can be seen by moving the segment $[x_0, y_0]$ until it stops intersecting $E(j, \rho)$, and then doing a small perturbative argument. 
In particular, we may assume that $A_0\cup [x_0, y_0]$ embeds into $j(\Gamma_0)\backslash \mathbb{H}^2$ under the quotient map.

We can perturb the Jordan curve $A_0\cup [x_0,y_0]$ to a Jordan curve $A\cup [x,y]$ whose inner (open) disk $D$ intersects $E(j,\rho)$ in some point $z$, with $A\subset\mathcal{U}$ (see Figure~\ref{fig:N}).

\begin{figure}[h!]
\begin{center}
\labellist
\small\hair 2pt
\pinlabel{$x$} at 15 10
\pinlabel{$y$} at 265 65
\pinlabel{$z$} at 140 78
\pinlabel{$A$} at 170 150
\pinlabel{$A'$} at 160 123
\pinlabel{$E(j,\rho)$} at 200 20
\pinlabel{$B'$} at 50 135
\pinlabel{$D'$} at 90 90
\pinlabel{$V$} at 38 30
\pinlabel{$V'$} at 240 77
\endlabellist
\includegraphics[width=6cm]{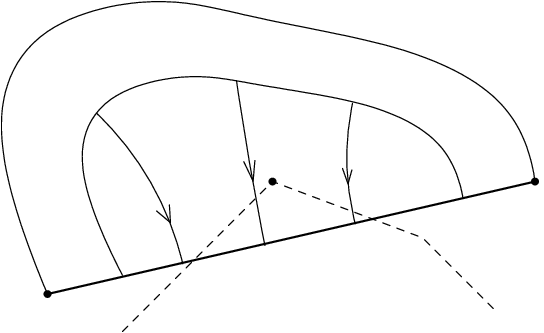}
\caption{Improving the local Lipschitz constant at~$z$.}
\label{fig:N}
\end{center}
\end{figure}
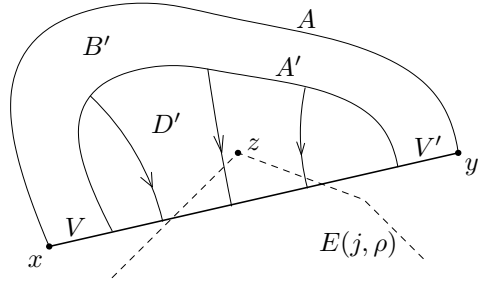

Let $B\subset\mathcal{U}$ be a compact neighborhood of~$A$.
Choose an optimal, $C$-Lipschitz equivariant map~$f$ (Lemma~\ref{lem:optimalmap}): by construction, $\Lip_B(f)=:C^{\ast}<C$.
We define a $(j,\rho)$-equivariant map $g_{\varepsilon}:=f\circ J_{\varepsilon} : \HH^2\rightarrow\HH^2$, where $J_{\varepsilon}$ is the small deformation of the identity map~$\mathrm{id}_{\HH^2}$ given as follows:
\begin{itemize}
  \item on $\HH^2\smallsetminus j(\Gamma_0)\cdot D$, take $J_{\varepsilon}$ to be the identity map;
  \item on $D':=D\smallsetminus B$, take
  $$J_{\varepsilon} := \varepsilon \cdot \pi_{[x,y]} + (1-\varepsilon)\cdot \text{id}_{D'}$$
  where $\pi_{[x,y]}$ denotes the closest-point projection onto $[x,y]$; extend $(j,\rho)$-equivariantly to $j(\Gamma_0)\cdot D'$;
  \item on $B':= D\cap B$, note that $J_{\varepsilon}$ is already defined on $\partial B'$ and use Proposition~\ref{prop:classicalKirszbraun} to find an optimal extension to $B'$; extend $(j,\rho)$-equivariantly to $j(\Gamma_0)\cdot B'$.
\end{itemize}
We have $\Lip_{D'}(J_{\varepsilon})\leq 1$ because $\Lip (\pi_{[x,y]})\leq 1$ (use Lemma~\ref{lem:baryLipschitz}). 
Also, we claim that $\Lip_{\partial B'}(J_{\varepsilon})\leq C/C^{\ast}$ for $\varepsilon$ small enough.
This is true because $\partial B'$ is the union of two subsegments $V,V'$ of $[x,y]$ and two disjoint arcs, namely $A$ and another, nearly parallel arc $A'$: the only pairs of points $(\xi,\xi')\in (\partial B')^2$ that $J_{\varepsilon}|_{\partial B'}$ can move \emph{apart} are in $A\times A'$ (up to order), but $d(\xi,\xi')$ is then bounded from below by a positive constant $d(A,A')$.
Thus, $\Lip_{\partial B'}(J_{\varepsilon})$ (and hence $\Lip(J_{\varepsilon})$) goes to $1$ as $\varepsilon$ goes to $0$, which yields $\Lip(g_{\varepsilon})\leq C$ as soon as $\Lip(J_{\varepsilon})\leq C/C^{\ast}$.
However, since $\pi_{[x,y]}$ is contracting near $z\in D'$, we have $\Lip_z(g_{\varepsilon})<C$, hence $z\notin E_{g_{\varepsilon}}$.
This contradicts the optimality of $f$, as $z\in\nolinebreak E(j,\rho)$.
\end{proof}

\section{Examples and counterexamples}\label{sec:ex}

All examples below are in dimension $n=2$, except the last two.
For $n=2$, we use the upper half-plane model of~$\HH^2$ and identify $G=\PO(2,1)$ with $\PGL_2(\R)$ and its identity component~$G_0$ with $\PSL_2(\R)$.

Example~\ref{ex:infinite} deals with infinitely generated~$\Gamma_0$.
Examples \ref{ex:nonreductive1} to~\ref{ex:C'neqCnoncompact} concern convex cocompact~$j$, while Examples \ref{ex:discontinu} to~\ref{ex:dim4C<1} illustrate phenomena that arise only in the presence of cusps.

\subsection{A left admissible pair $(j,\rho)$ with $C(j,\rho)=1$, for infinitely generated~$\Gamma_0$}\label{ex:infinite}

In this section, we give an example of an infinitely generated discrete subgroup $\Gamma$ of $G\times G$ that acts properly discontinuously on~$G$ but that does not satisfy the conclusion of Theorem~\ref{thm:sharp}; in other words, $\Gamma$ is not sharp in the sense of \cite[Def.\,4.2]{kk12}.

In the upper half-plane model of~$\HH^2$, let $A_k$ (\resp $B_k$) be the half-circle of radius 1 (\resp $\log k$) centered at~$k^2$, oriented clockwise.
Let $A'_k$ (\resp $B'_k$) be the half-circle of radius~$1$ (\resp $\log k$) centered at $k^2+k$, oriented counterclockwise (see Figure~\ref{fig:O}).
Let $\alpha_k\in G_0$ (\resp $\beta_k\in G_0$) be the shortest hyperbolic translation identifying the geodesic represented by $A_k$ with $A'_k$ (\resp $B_k$ with $B'_k$); its axis is orthogonal to $A_k$ and~$A'_k$ (\resp to $B_k$ and~$B'_k$), hence its translation length $\lambda(\alpha_k)$ (\resp $\lambda(\beta_k)$) is equal to the distance between $A_k$ and~$A'_k$ (\resp $B_k$ and~$B'_k$).
An elementary computation (see \eqref{eqn:rainbow} below) shows that
\begin{equation}\label{eqn:nocrossratio} \begin{array}{rclcl}
\lambda(\alpha_k) & = & 2\, \mathrm{arccosh} (k/2) &=&  2\log k + o(1),\\
\lambda(\beta_k) & = & 2\, \mathrm{arccosh} (\frac{k}{2\log k}) &=& 2\log k - 2 \log \log k +o(1)\,.
\end{array}\end{equation}
Consider the free group $\Gamma_N=\langle\gamma_k\rangle_{k\geq N}$ and its injective and discrete representations $j,\rho$ given by $j(\gamma_k)=\alpha_k$ and $\rho(\gamma_k)=\beta_k$.
Since $\lambda(\beta_k)/\lambda(\alpha_k)$ goes to~$1$, we have $C(j,\rho)\geq 1$ and $C(\rho,j)\geq 1$.
However, we claim that for $N$ large enough, the group $\Gamma_N^{j,\rho}=\{ (j(\gamma),\rho(\gamma))\,|\,\gamma\in\Gamma_N\}$ is left admissible (Definition~\ref{def:admissible}), acting properly discontinuously on~$G$.

Indeed, fix the basepoint $p_0=\sqrt{-1}\in\HH^2$ and consider a reduced word $\gamma=\gamma_{k_1}^{\varepsilon_1}\dots\gamma_{k_m}^{\varepsilon_m}\in\Gamma_N$, where $\varepsilon_i=\pm 1$.
Let $\D_A\subset \HH^2$ be the fundamental domain of~$\HH^2$ for the action of~$j(\Gamma_N)$ that is bounded by the geodesics $A_k, A'_k$ for $k\geq N$.
Let $\D_B$ be the fundamental domain for the action of~$\rho(\Gamma_N)$ that is bounded by the geodesics $B_k, B'_k$ for $k\geq N$.

The geodesic segment from $p_0$ to $j(\gamma)\cdot p_0$ projects in the fundamental domain~$\D_A$ to a union of $m+1$ geodesic segments $I_0, \dots, I_m$: namely, $I_i$ connects the half-circle $A_{k_i}$ or $A'_{k_i}$ (depending on~$\varepsilon_i$) to the half-circle $A_{k_{i+1}}$ or $A'_{k_{i+1}}$ (depending on~$\varepsilon_{i+1}$), unless $i=0$ or $m$, in which case one of the endpoints is~$p_0$ (see Figure~\ref{fig:O}).
Moreover, the geodesic line carrying $I_i$ hits $\partial_{\infty}\HH^2$ near the centers of these half-circles, since all half-circles $A_k, A'_k$ are far from one another and from $p_0$ (compared to their radii).
Therefore, the ends of~$I_i$ are nearly orthogonal to the $A_k, A'_k$ and the length of~$I_i$ can be approximated by the distance from some side of $\D_A$ to another (or to~$p_0$).
The error is $o(1)$ for each segment~$I_i$, uniformly as $N\rightarrow +\infty$.

\begin{figure}[h!]
\begin{center}
\labellist
\small\hair 2pt
\pinlabel{$p_0=\sqrt{-1}$} at 80 30
\pinlabel{$I_0$} at 185 180
\pinlabel{$I_1$} at 780 165
\pinlabel{$A_{n_1}$} at 410 30
\pinlabel{$B_{n_1}$} at 440 60
\pinlabel{$A'_{n_1}$} at 650 30
\pinlabel{$B'_{n_1}$} at 680 60
\pinlabel{$A_{n_2}$} at 990 30
\pinlabel{$B_{n_2}$} at 1010 85
\endlabellist
\includegraphics[width=12.5cm]{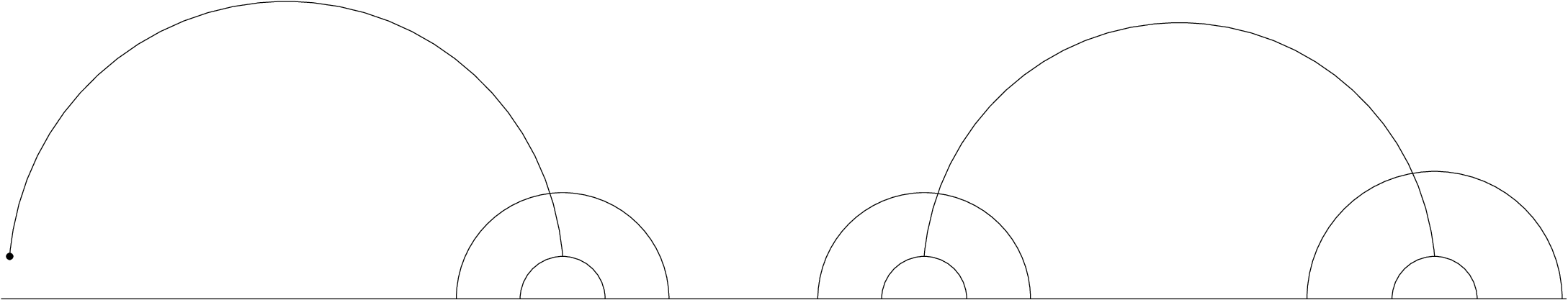}
\caption{For infinitely generated~$\Gamma_0$, construction of an admissible pair $(j,\rho)$ with $C(j,\rho)=1$.}
\label{fig:O}
\end{center}
\end{figure}
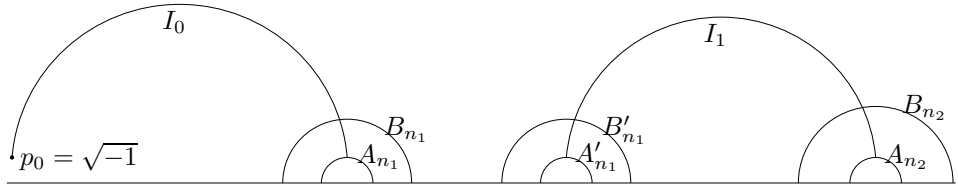

The distance from $p_0$ to $\rho(\gamma)\cdot p_0$ is likewise a sum of lengths of segments $J_0,\dots, J_m$ between boundary components of $\D_B$: the segment $J_i$ meets $B_{k_i}$ (\resp $B'_{k_i}, B_{k_{i+1}}, B'_{k_{i+1}}$) exactly when $I_i$ meets $A_{k_i}$ (\resp $A'_{k_i}, A_{k_{i+1}}, A'_{k_{i+1}}$).
Therefore,~$I_i$ is longer than~$J_i$ by roughly the sum of $d(A_{k_i},B_{k_i})=d(A'_{k_i},B'_{k_i})$ and $d(A_{k_{i+1}},B_{k_{i+1}})=d(A'_{k_{i+1}},B'_{k_{i+1}})$.
Using \eqref{eqn:nocrossratio}, we obtain
$$\mathrm{length}(I_i) - \mathrm{length}(J_i) = \log\log k_i + \log\log k_{i+1} + o(1)$$
(with one term stricken out for $i=0$ or~$m$), with uniform error as $N\rightarrow +\infty$.
In particular, for $N$ large enough the left-hand side is always $\geq 1$.
Finally,
\begin{eqnarray*}
\mu(j(\gamma)) - \mu(\rho(\gamma)) & = & d(p_0,j(\gamma)\cdot p_0) - d(p_0,\rho(\gamma)\cdot p_0)\\
& = & \sum_{i=0}^m \big(\mathrm{length}(I_i) - \mathrm{length}(J_i)\big)\\
& \geq & \max\Big\{m, \log \log \Big(\max_{1\leq i\leq m} k_i\Big)\Big\},
\end{eqnarray*}
which clearly diverges to~$+\infty$ as $\gamma=\gamma_{k_1}^{\varepsilon_1}\dots\gamma_{k_m}^{\varepsilon_m}$ exhausts the countable group~$\Gamma_N$.
Therefore the group~$\Gamma_N^{j,\rho}$ acts properly discontinuously on~$G$ by the properness criterion of Benoist and Kobayashi (Section~\ref{subsec:Cartanproj}).

\subsection{A nonreductive~$\rho$ with $\F^{j,\rho}\neq\emptyset$}\label{ex:nonreductive1}

Let $\Gamma_0$ be a free group on two generators $\alpha,\beta$ and let $j\in\Hom(\Gamma_0,G)$ be the holonomy representation of a hyperbolic one-holed torus~$S$ of infinite volume, such that the translation axes $\A_{j(\alpha)}$ and~$\A_{j(\beta)}$ of $\alpha$ and~$\beta$ meet at a right angle at a point $p\in\HH^2$ (see Figure~\ref{fig:P}).

\begin{figure}[h!]
\begin{center}
\labellist
\small\hair 2pt
\pinlabel{$\HH^2$} at 70 240
\pinlabel{$\mathcal{U}$} at 45 165
\pinlabel{$\mathcal{U}$} at 247 165
\pinlabel{$s$} at 160 50
\pinlabel{$t$} at 52 128
\pinlabel{$t'$} at 241 128
\pinlabel{$s'$} at 160 242
\pinlabel{$\A_{j(\alpha)}$} at 170 190
\pinlabel{$\A_{j(\beta)}$} at 100 126
\pinlabel{$p$} at 155 134
\pinlabel{$\mathcal{D}$} at 180 110
\pinlabel{$f$} at 315 155
\pinlabel{$f(s)$} at 467 40
\pinlabel{$f(t)$} at 473 120
\pinlabel{$f(p)$} at 467 142
\pinlabel{$f(s')$} at 465 247
\pinlabel{$f(t')$} at 528 125
\endlabellist
\includegraphics[width=12cm]{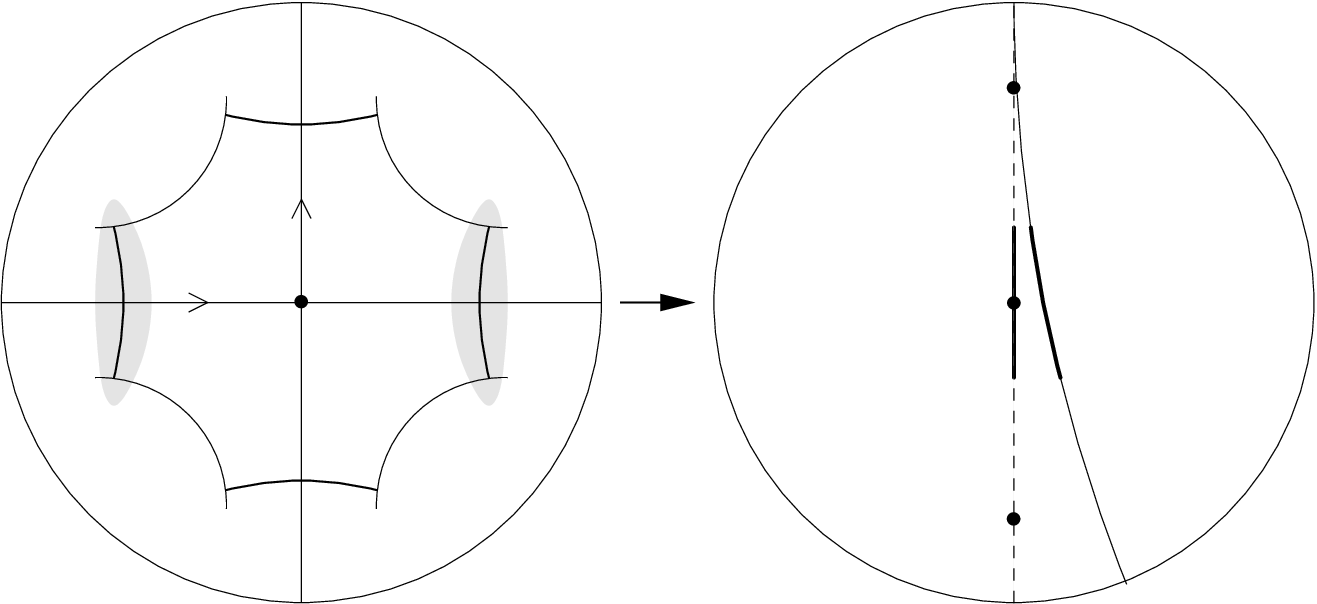}
\caption{A nonreductive representation $\rho$ such that the stretch locus $E(j,\rho)$ is the $j(\Gamma_0)$-orbit of the axis~$\A_{j(\alpha)}$.}
\label{fig:P}
\end{center}
\end{figure}
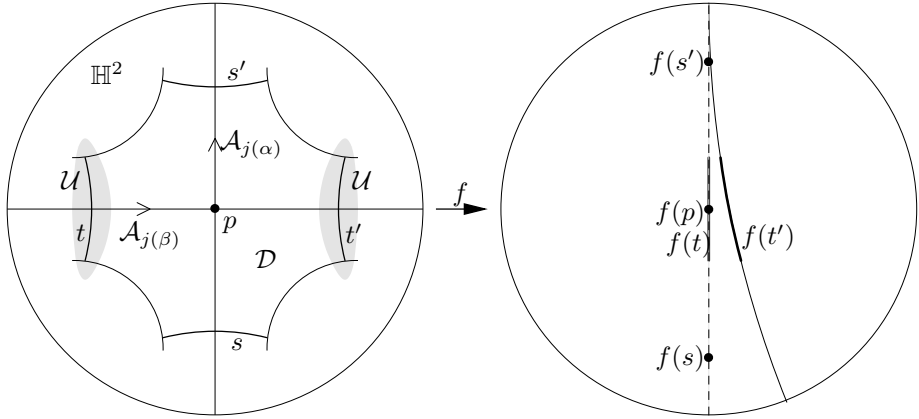

We first consider the representation $\rho_0\in\Hom(\Gamma_0,G)$ given by $\rho_0(\alpha)=j(\alpha)^2$ and $\rho_0(\beta)=1$.
It is reductive with two fixed points in $\partial_{\infty}\HH^2$.
We claim that $C(j,\rho_0)=2=\lambda(\rho_0(\alpha))/\lambda(j(\alpha))$ and that the image of the stretch locus $E(j,\rho)$ in $j(\Gamma_0)\backslash\HH^2$ is the closed geodesic corresponding to~$\alpha$.
Indeed, consider the Dirichlet fundamental domain $\mathcal{D}$ of the convex core centered at~$p$, for the action of~$j(\Gamma_0)$.
It is bounded by four segments of the boundary of the convex core, and by four other segments $s,s',t,t'$ such that $j(\alpha)$ maps $s$ to~$s'$ and $j(\beta)$ maps $t$ to~$t'$.
Let $\pi_{\A_{j(\alpha)}}$ be the closest-point projection onto~$\A_{j(\alpha)}$ and $h$ the orientation-preserving homeomorphism of~$\A_{j(\alpha)}$ such that $d(p,h(q)) = 2\,d(p,q)$ for all $q\in\A_{j(\alpha)}$.
The map $h\circ\pi_{\A_{j(\alpha)}} : \D\rightarrow\HH^2$ is $2$-Lipschitz and extends to a $2$-Lipschitz, $(j,\rho_0)$-equivariant map $f_0 : \HH^2\rightarrow\A_{j(\alpha)}$ whose stretch locus is exactly $j(\Gamma_0)\cdot\A_{j(\alpha)}$.

Consider a small, nonreductive deformation $\rho\in\Hom(\Gamma_0,G)$ of~$\rho_0$ such that $\rho(\alpha)=\rho_0(\alpha)=j(\alpha)^2$ and such that $\rho(\beta)$ has a parabolic fixed point in $\partial_{\infty}\HH^2$ common with $j(\alpha)$.
Then $C(j,\rho)=C(j,\rho_0)=2$ by Lemma~\ref{lem:nonred-lip-cc}.
We claim that $\F^{j,\rho}$ is nonempty if $\rho(\beta)$ is close enough to~$\mathrm{Id}_{\HH^2}$.
Indeed, let us construct a $(j,\rho)$-equivariant deformation $f$ of~$f_0$ which is still $2$-Lipschitz.
By Lemma~\ref{lem:localLip}, we have $\Lip_{\mathcal{U}}(f_0)<C$ for some neighborhood $\mathcal{U}$ of $t\cup t'$.
Therefore, the map~$f$ defined on $s\cup s'\cup t \cup t'$ by $f|_{t\cup s \cup s'}=f_0|_{t \cup s \cup s'}$ and $f|_{t'}=\rho(\beta)\circ f_0|_{t'}$ is still $2$-Lipschitz if $\rho(\beta)$ is close enough to~$\mathrm{Id}_{\HH^2}$.
This map~$f$ extends, with the same Lipschitz constant~$2$, to all of~$\D$ (by the Kirszbraun--Valentine theorem, Proposition~\ref{prop:classicalKirszbraun}), hence $(j,\rho)$-equivariantly to~$\HH^2$.

This construction can be adapted to any hyperbolic surface $S$ of infinite volume when the stretch locus $E(j,\rho_0)$ is a multicurve.

\subsection{A nonreductive~$\rho$ with $\F^{j,\rho}=\emptyset$}\label{ex:nonreductive2}

Let again $\Gamma_0$ be a free group on two generators $\alpha,\beta$ and let $j\in\Hom(\Gamma_0,G)$ be the holonomy representation of a hyperbolic one-holed torus~$S$ of infinite volume.

Let $\LL$ be the preimage in~$\HH^2$ of an oriented \emph{irrational} measured lamination of~$S$.
We first construct a reductive representation $\rho_0\in\Hom(\Gamma_0,G)$ with two fixed points in $\partial_{\infty}\HH^2$ such that $E(j,\rho_0)=\LL$ and $C(j,\rho_0)<1$.
It is sufficient to construct a differential $1$-form $\omega$ of class $\mathrm{L}^{\infty}$ on~$S$ with the following properties:
\begin{enumerate}
  \item $\omega$ is locally the differential of some $1$-Lipschitz function $\varphi$,
  \item $\int_I\omega= \mathrm{length}(I)$ for any segment of leaf $I$ of (the image in~$S$ of)~$\LL$, oriented positively.
\end{enumerate}
Indeed, if such an~$\omega$ exists, then for any geodesic line $\A$ of~$\HH^2$, any isometric identification $\A\simeq\R$, and any $C\in (0,1)$, we can define a representation $\rho_0\in\Hom(\Gamma_0,G)$ as follows: if $\gamma\in\Gamma_0\smallsetminus\{ 1\}$ corresponds to a loop $\mathcal{G}_{\gamma}$ on~$S$, then $\rho_0(\gamma)$ is the hyperbolic element of~$G$ translating along~$\A$ with length $C\int_{\mathcal{G}_{\gamma}}\omega\in\R$.
Such a representation~$\rho_0$ satisfies $E(j,\rho_0)=\LL$ and $C(j,\rho_0)=C$ because for any basepoint $p\in\HH^2$, the map
$$f_0 :\ q\in\HH^2 \,\longmapsto\, C \int_{[p,q]} \widetilde{\omega} \,\in \R \simeq \A$$
(where $\widetilde{\omega}$ is the $j(\Gamma_0)$-invariant $1$-form on~$\HH^2$ lifting~$\omega$) is $(j,\rho_0)$-equivariant, has Lipschitz constant exactly~$C$, and stretches $\LL$ maximally, and we can use Lemma~\ref{lem:maxstretchedlamin}.

\begin{figure}[h!]
\begin{center}
\labellist
\small\hair 2pt
\pinlabel{$\mathcal{B}$} at 230 250
\pinlabel{$\partial$} at 210 165
\pinlabel{$H$} at 170 190
\pinlabel{$H'$} at 290 110
\pinlabel{$\LL$} at 40 276
\pinlabel{$\LL$} at 605 175
\pinlabel{$f$} at 435 210
\pinlabel{$\{\xi\}$} at 445 250
\pinlabel{$\propto \! e^{-t}$} at 560 100
\pinlabel{$t$} at 615 106
\pinlabel{$J$} at 655 95
\endlabellist
\includegraphics[width=12cm]{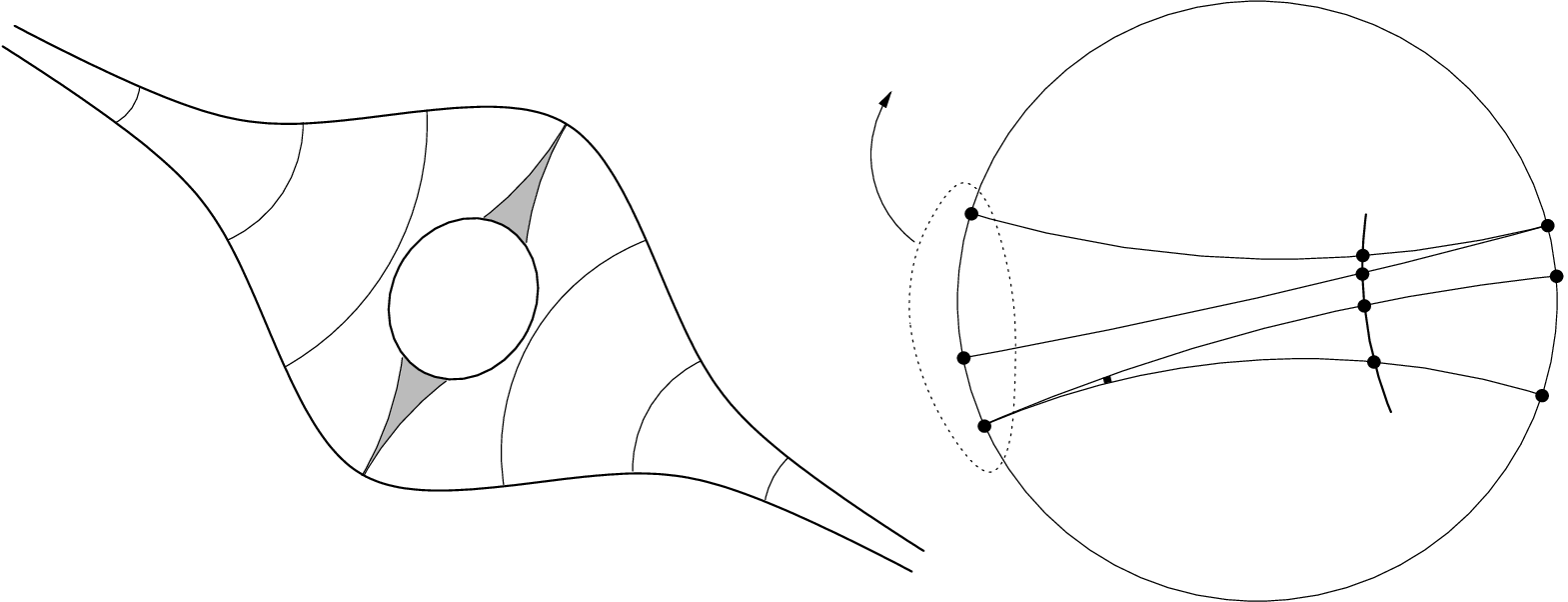}
\caption{\emph{Left}: the one-holed bigon $\mathcal{B}$ bounded by the irrational lamination~$\LL$. The symbol $\partial$ denotes the boundary of the convex core. The function $\varphi$ is constant on the shaded area; elsewhere its level curves are pieces of horocycles. \emph{Right}: the Lipschitz map $f$ must collapse all lines of~$\LL$ if $C<1$, because $e^{-Ct} \gg e^{-t}$ for large $t$.}
\label{fig:Q}
\end{center}
\end{figure}
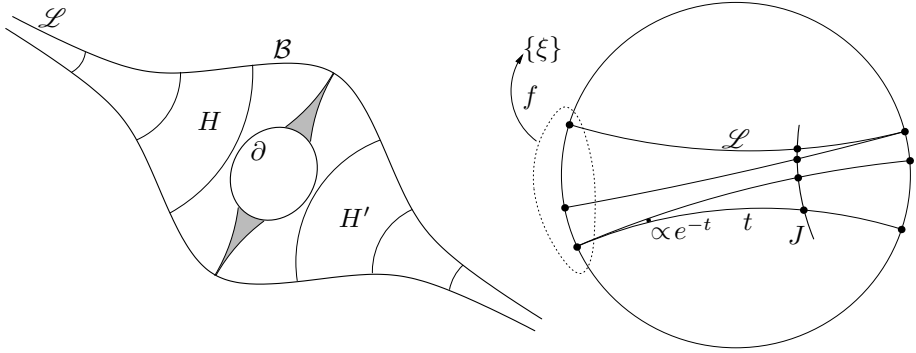

Let us construct a $1$-form~$\omega$ as above. The idea is similar to the ``stretch maps'' of \cite{thu86}.
The complement of the image of~$\LL$ in the convex core of~$S$ is a one-holed biinfinite bigon~$\mathcal{B}$; each of its two spikes can be foliated by pieces of horocycles (see Figure~\ref{fig:Q}).
Let $H$ and~$H'$ be horoball neighborhoods of the two spikes, tangent in two points of~$\LL$ (one for each side of~$\mathcal{B}$).
We take $\omega=d\varphi$ where
$$\varphi :=
\left \{ \begin{array}{ll}
0 & \text{on }\mathcal{B}\smallsetminus (H\cup H'),\\
d(\cdot,\partial H) & \text{on }H,\\
-d(\cdot,\partial H') & \text{on }H'.
\end{array} \right .$$
This form $\omega$ extends continuously to all of~$\LL$.

Let now $\rho\in\Hom(\Gamma_0,G)$ be a nonreductive representation such that the fixed point $\xi$ of $\rho(\Gamma_0)$ in $\partial_{\infty}\HH^2$ is one of the two fixed points of $\rho_0(\Gamma_0)$.
Then $C(j,\rho)=C:=C(j,\rho_0)$ by Lemma~\ref{lem:nonred-lip-cc}.
We claim that $\F^{j,\rho}=\emptyset$, \emph{i.e.}\ there exists no $C$-Lipschitz, $(j,\rho)$-equivariant map $f : \HH^2\rightarrow\HH^2$. Indeed, suppose by contradiction that such an $f$ exists.

We first note that $f$ stretches maximally every leaf of~$\LL$.
Indeed, the ``horocyclic projection'' taking any $p\in\HH^2$ to the intersection of the translation axis $\A$ of~$\rho_0$ with the horocycle through~$p$ centered at~$\xi$ is $1$-Lipschitz.
After postcomposing~$f$ with this horocyclic projection, we obtain a $C$-Lipschitz map $f_1$ which is $(j,\rho_0)$-equivariant, hence has to stretch maximally every leaf of $E(j,\rho_0)=\LL$ (Theorem~\ref{thm:lamin}).
Then $f$ also stretches maximally every leaf of~$\LL$.
In fact, this argument shows that on any leaf of~$\LL$, the map~$f$ coincides with~$f_1$ postcomposed with some parabolic (or trivial) isometry of~$\HH^2$ fixing~$\xi$, depending on the leaf; the leaf endpoint in $\partial_{\infty}\HH^2$ which is sent to~$\xi$ by~$f_1$ is also sent to~$\xi$ by~$f$.
(Actually, by density of leaves the $(j,\rho_0)$-equivariant restrictions $f_0|_{\LL}$ and $f_1|_{\LL}$ differ only by a translation along the axis $\A$ of~$\rho_0$, but we will not need this.)

Let us now prove that $f$ maps all the leaves of~$\LL$ to the same geodesic line of~$\HH^2$.
This will provide a contradiction since $f(\LL)$ is $\rho(\Gamma_0)$-invariant and $\rho(\Gamma_0)$ has only one fixed point in $\partial_{\infty}\HH^2$ (namely~$\xi$).
Let $J$ be a short geodesic segment of~$\HH^2$ transverse to the lamination~$\LL$, such that $J\smallsetminus\LL$ is the union of countably many open subintervals~$J_k$.
Then each~$J_k$ intercepts an ideal sector bounded by two leaves of~$\LL$ that are asymptotic to each other on one of the two sides, left or right, of~$J$.
Orient~$J$ so that all the half-leaves on the left of~$J$ are mapped under~$f_1$ to geodesic rays with endpoint~$\xi$.
Then any two leaves asymptotic on the right of~$J$ have the same image under~$f$: indeed, the right parts of the image leaves are asymptotic because $f$ is Lipschitz, and the left parts are asymptotic because $\rho$ sends all the left endpoints to~$\xi$.
Consider two leaves $\ell,\ell'$ of~$\LL$ that are asymptotic on the left of~$J$, bounding together an infinite spike.
At depth $t \gg 1$ inside the spike, $\ell$ and $\ell'$ approach each other at rate~$e^{-t}$ (see \eqref{eqn:expdistcusps}), and their images under~$f$, if distinct, approach each other at the slower exponential rate $e^{-Ct}$ (recall that $C<1$); since $f$ is Lipschitz, this forces $f(\ell)=f(\ell')$.
Therefore, all the sectors intercepted by the~$J_k$ are collapsed by~$f$.
Since by \cite{birman-series} the length of $J$ is the sum of the lengths of the $J_k$, passing to the limit we see that all the leaves of~$\LL$ meeting~$J$ have the same image under~$f$.
We conclude by observing that the projection of $J\cup\LL$ to~$S$ carries the full fundamental group of~$S$.

This proves that $\F^{j,\rho}=\emptyset$.
It is not clear whether the same can happen when $C(j,\rho)\geq 1$, but the natural conjecture would be that it does not.

\subsection{A pair $(j,\rho)$ with $C'(j,\rho)<C(j,\rho)<1$ and $j(\Gamma_0)\backslash\HH^n$ compact}\label{ex:C'neqC}

While the constants $C(j,\rho)$ and $C'(j,\rho)$ are equal above~$1$ (Corollary~\ref{cor:CC'}), they can differ below~$1$.
To prove this, we only need to exhibit a pair $(j,\rho)$ such that any closed geodesic of $j(\Gamma_0)\backslash\HH^2$ spends a definite (nonzero) proportion of its length in a compact set $V$  disjoint from the stretch locus (on~$V$ the local Lipschitz constant of an optimal Lipschitz equivariant map stays bounded away from $C(j,\rho)$, see Lemma~\ref{lem:localLip}).

In~$\HH^2$, consider a positively oriented hyperbolic triangle $ABC$ with angles
$$\widehat{A}=\frac{\pi}{3}, \quad \widehat{B}=\frac{\pi}{2}, \quad \widehat{C}=\frac{\pi}{14}$$
and another, smaller triangle $A'B'C'$ with the same orientation and with angles
$$\widehat{A'}=\frac{\pi}{3}, \quad \widehat{B'}=\frac{\pi}{2}, \quad \widehat{C'}=\frac{\pi}{7}\,.$$
The edge $[A',\!B']$ is shorter than $[A,\!B]$; let $\varphi : [A,\!B]\rightarrow [A',\!B']$ be the uniform parametrization, with $\varphi(A)=A'$ and $\varphi(B)=B'$, so that
$$C_0 := \Lip(\varphi) = \frac{d(A', B')}{d(A,B)} <1.$$

\begin{claim}\label{claim:exC'neqC}
The map~$\varphi$ admits a $C_0$-Lipschitz extension $f$ to the filled triangle $ABC$, taking the geodesic segment $[A,\!C]$ (\resp $[C,\!B]$) to the geodesic segment $[A',\!C']$ (\resp $[C',\!B']$), and with stretch locus the segment $[A,\!B]$.
\end{claim}

\begin{proof}
Let $\ell$ be the geodesic line of~$\HH^2$ containing $[A,\!B]$, oriented from $A$ to~$B$.
Any point $p\in\HH^2$ may be reached in a unique way from~$A$ by first applying a translation of length $v(p )\in\R$ along the geodesic line orthogonal to~$\ell$ at $A$, positively oriented with respect to~$\ell$ (``vertical direction''), then a translation of length $h(p )\in\R$ along ~$\ell$ itself (``horizontal direction'').
The real numbers $h(p)$ and~$v(p)$ are called the \emph{Fermi coordinates} of~$p$ with respect to $(\ell,A)$.
Similarly, let $h'$ and~$v'$ be the Fermi coordinates with respect to $(\ell ',A')$, where $\ell'$ is the geodesic line containing $[A',\!B']$, oriented from $A'$ to~$B'$ (see Figure~\ref{fig:R}).

\begin{figure}[h!]
\begin{center}
\labellist
\small\hair 2pt
\pinlabel{$A$} at -4 4
\pinlabel{$B$} at 93 4
\pinlabel{$C$} at 93 165
\pinlabel{$A'$} at 165 4
\pinlabel{$B'$} at 238 4
\pinlabel{$C'$} at 238 100
\pinlabel{$p$} at 65 50
\pinlabel{$h(p)$} at 35 10
\pinlabel{$v(p)$} at 68 25
\pinlabel{$f^{\ast}$} at 140 75
\pinlabel{$f^{\ast}(C)$} at 247 130
\endlabellist
\includegraphics[width=7cm]{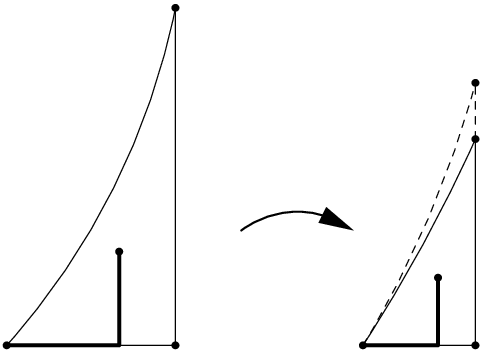}
\caption{Defining a contracting map between right-angled triangles.}
\label{fig:R}
\end{center}
\end{figure}
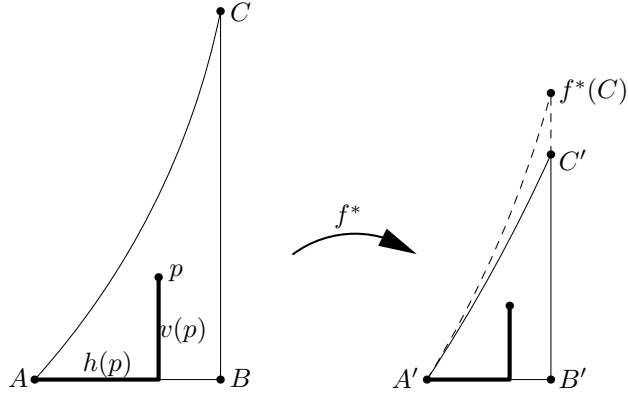

Let $\Psi : \R_+\rightarrow\R_+$ be a diffeomorphism whose derivative~$\Psi'$ is everywhere $<C_0$ on $\R_+^{\ast}$ and let $f^{\ast} : \HH^2\rightarrow\HH^2$ be given, in Fermi coordinates, by
$$h'(f^{\ast}(p)) = C_0\,h(p) \quad\quad\text{and}\quad\quad v'(f^{\ast}(p)) = \Psi(v(p)).$$
Then $\Lip_p(f^{\ast})<C_0$ for all $p\notin\ell$.
Indeed, the differential of $f^{\ast}$ at $p\notin\ell$ has principal value $\Psi'(v(p))<C_0$ in the vertical direction and, by \eqref{eqn:equidistpoints}, principal value $C_0\frac{\cosh\Psi(v(p))}{\cosh v(p)}<C_0$ in the horizontal direction.

We shall take $f:=\pi_{\scriptscriptstyle A'\!B'\!C'}\circ f^{\ast}$ for a suitable choice of~$\Psi$, where $\pi_{\scriptscriptstyle A'\!B'\!C'}$ is the closest-point projection onto the filled triangle $A'B'C'$.
Since we wish $f$ to map $[A,\!C]$ to $[A',\!C']$, we need to choose~$\Psi$ so that for any $p\in [A,\!C]$ the point $f^{\ast}(p)$ lies \emph{above} (or on) the edge $[A',\!C']$.
By \eqref{eqn:trigo},
$$\tan \widehat{pAB} = \frac{\tanh v(p)}{\sinh h(p)}$$
and
$$\tan \widehat{f^{\ast}(p)A'B'} = \frac{\tanh v'(f^{\ast}(p))}{\sinh h'(f^{\ast}(p))} = \frac{\tanh\Psi(v(p))}{\sinh(C_0 v(p))}\,.$$
Note that $\tanh(C_0t)>C_0\tanh(t)$ and $\sinh(C_0t)<C_0\sinh(t)$ for all $t>0$, by strict concavity of $\tanh$ and convexity of $\sinh$ (recall that $0<C_0<1$).
Therefore the function $\Psi : t\mapsto C_0t$ yields a map~$f^{\ast}$ with $f^{\ast}([A,\!C])$ above $[A',\!C']$.
We can decrease this function slightly to obtain~$\Psi$ with $\Psi'(t)<C_0$ for all $t>0$ while keeping $f^{\ast}([A,\!C])$ above $[A',\!C']$.

In fact, by the above formulas, we can also ensure $f^{\ast}([A,\!C])=[A',\!C']$ directly, by taking $\Psi=\Psi_{\widehat{A}}$ where
\begin{equation}\label{eqn:exactbisect}
\Psi_{\widehat{A}}(v)=\sigma^{-1}(C_0 \, \sigma (v))~\text{ with }~\sigma(v) = \arcsinh\left(\frac{\tanh v}{\tan\widehat{A}}\right):
\end{equation}
then $\Psi_{\widehat{A}}'<C_0$ (on $\R_+^{\ast}$) easily follows from $C_0<1$ and from the concavity of~$\sigma$.
\end{proof}

Let $\overline{\Gamma}_0$ be the group generated by the orthogonal reflections in the sides of $ABC$ and $\overline{j}$ its natural inclusion in $G=\PGL_2(\R)$.
Let $\overline{\rho}\in\Hom(\overline{\Gamma}_0,G)$ be the representation taking the reflections in $[A,\!B]$, $[B,\!C]$, $[C,\!A]$ to the reflections in $[A',\!B']$, $[B',\!C']$, $[C',\!A']$ respectively; it is well defined (relations are preserved) because $\pi/7$ is a multiple of $\pi/14$.
The group $\overline{\Gamma}_0$ has a finite-index normal subgroup $\Gamma_0$ which is torsion-free and such that $\overline{j}(\Gamma_0)$ and $\overline{\rho}(\Gamma_0)$ are orientation-preserving, \ie with values in $G_0=\PSL_2(\R)$.
Let $j,\rho\in\Hom(\Gamma_0,G)$ be the corresponding representations.
The map~$f$ given by Claim~\ref{claim:exC'neqC} extends, by reflections in the sides of $ABC$, to a $C_0$-Lipschitz, $(j,\rho)$-equivariant map on~$\HH^2$.
Its stretch locus is the $\overline{\Gamma}_0$-orbit of the segment $[A,\!B]$, which is the $1$-skeleton of a tiling of $\HH^2$ by regular $14$-gons meeting $3$ at each vertex.

We claim that $C(j,\rho)=C_0$ and that $f$ is an optimal element of~$\F^{j,\rho}$, in the sense of Definition~\ref{def:relstretchlocus}.
Indeed Lemma~\ref{lem:finiteindex}, applied to $\overline{\Gamma}_0$ and its finite-index subgroup~$\Gamma_0$, shows that there exists an element $\overline{f}\in\F^{j,\rho}$ which is optimal and $(\overline{j},\overline{\rho})$-equivariant.
In particular, if $p\in\HH^2$ is fixed by some $\overline{j}(\overline{\gamma})\in\overline{j}(\overline{\Gamma}_0)$, then $\overline{f}(p)$ is fixed by~$\overline{\rho}(\overline{\gamma})$.
Applying this to the three sides of the triangle $ABC$, we see that $\overline{f}$ sends $A,B,C$ to $A',B',C'$ respectively.
In particular,
$$C(j,\rho) \geq \frac{d(A', B')}{d(A,B)} = C_0.$$
Since $f$ is $C_0$-Lipschitz with stretch locus the $\overline{\Gamma}_0$-orbit of the segment $[A,\!B]$, this shows that $C(j,\rho)=C_0$ and that $E(j,\rho)$ is the stretch locus of~$f$; in other words, $f$ is an optimal element of~$\F^{j,\rho}$.

It is easy to see that no geodesic of $\HH^2$ can spend more than a bounded proportion of its length near the regular trivalent graph $E(j,\rho)$, which implies that $C'(j,\rho)<C(j,\rho)$.

\subsection{A pair $(j,\rho)$ with $C'(j,\rho)<C(j,\rho)<1$ and $j(\Gamma_0)\backslash\HH^n$ noncompact}\label{ex:C'neqCnoncompact}

Let $\Gamma_0$ be a free group on two generators and $j\in\Hom(\Gamma_0,G)$ the holonomy representation of a hyperbolic three-holed sphere~$S$ with three funnels.
Let $\mathcal{G}$ be a geodesic trivalent graph on~$S$, with two vertices $v,w$ and three edges, such that the natural symmetry of~$S$ switches $v$ and~$w$ and preserves each edge.
Let $\ell_1,\ell_2,\ell_3>0$ be the lengths of the three edges and $\theta_1,\theta_2,\theta_3 \in (0,\pi)$ the angles between consecutive edges at both vertices, so that $\theta_1+\theta_2+\theta_3=2\pi$.
The preimage of~$\mathcal{G}$ in~$\HH^2$ is an embedded trivalent tree~$T$.
For any $C_0\in (0,1)$, there exists an \emph{immersed} trivalent tree~$T'$ with the same combinatorics as~$T$, with all (oriented) angles between adjacent edges of~$T'$ the same as in~$T$, but with all edges of length~$\ell_i$ in~$T$ replaced by edges of length~$C_0\ell_i$ in~$T'$.
The natural map $\varphi : T\rightarrow T'$, multiplying all lengths along edges by $C_0$, is $(j, \rho)$-equivariant for a unique $\rho\in\Hom(\Gamma_0,G)$.
If $C_0<1$ is large enough, then the immersed tree~$T'$ is in fact embedded, and $\rho$ is convex cocompact.

\begin{claim}\label{claim:exC'neqCnoncompact}
Suppose $\ell_1, \ell_2, \ell_3$ are large enough so that the bisecting rays at the vertices of~$T$ meet either outside of the preimage $N\subset\HH^2$ of the convex core of~$S$, or not at all.
Then the map~$\varphi$ extends to a $(j,\rho)$-equivariant map $f : \HH^2\rightarrow\HH^2$ with Lipschitz constant~$C_0$ and stretch locus~$T$.
\end{claim}

\begin{proof}
Let $e$ be an edge of~$T$ and $e', e''$ two of its neighbors, so that $e',e,e''$ are consecutive edges of some complementary component of~$T$.
By symmetry of the pair of pants $S$, the edge $e$ forms the same angle $\theta_i$ with $e'$ and with $e''$.
Let $\beta'$ and $\beta''$ be the corresponding bisecting rays, issued from the endpoints of~$e$.
Let $Q\subset\HH^2$ be the compact quadrilateral bounded by $e$, $\beta$, $\beta'$, and a segment of the boundary of~$N$.
There is a similarly defined quadrilateral on each side of each edge of~$T$, and their union is~$N$: therefore, it is sufficient to define the map $f$ on~$Q$ in a way that is consistent (along the bisecting rays $\beta', \beta''$) for $Q$ and neighboring quadrilaterals $Q',Q''$ (see Figure~\ref{fig:S}).

\begin{figure}[h!]
\begin{center}
\labellist
\small\hair 2pt
\pinlabel{$e$} at 190 48
\pinlabel{$e'$} at 80 90
\pinlabel{$e''$} at 300 90
\pinlabel{$Q$} at 190 75
\pinlabel{$Q'$} at 110 100
\pinlabel{$Q''$} at 270 100
\pinlabel{$\beta'$} at 141 83
\pinlabel{$\beta''$} at 238 80
\pinlabel{$\partial N$} at 190 100
\pinlabel{$T$} at 30 135
\pinlabel{$\theta_i/2$} at 140 64
\endlabellist
\includegraphics[width=9cm]{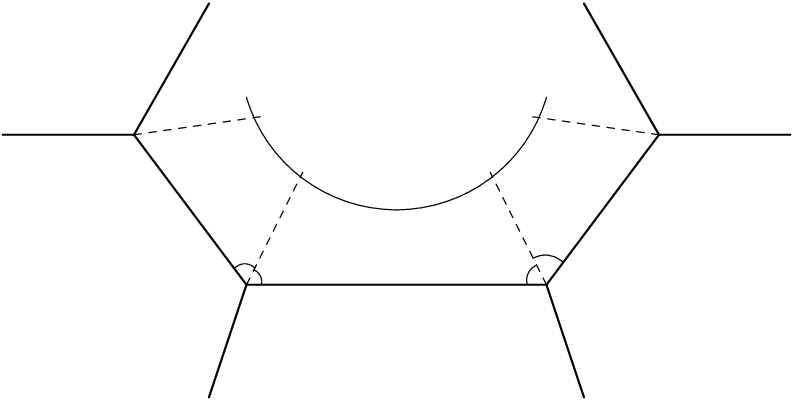}
\caption{Defining a contracting map on the convex hull $N$ of a tree~$T$, one quadrilateral $Q$ at a time.}
\label{fig:S}
\end{center}
\end{figure}
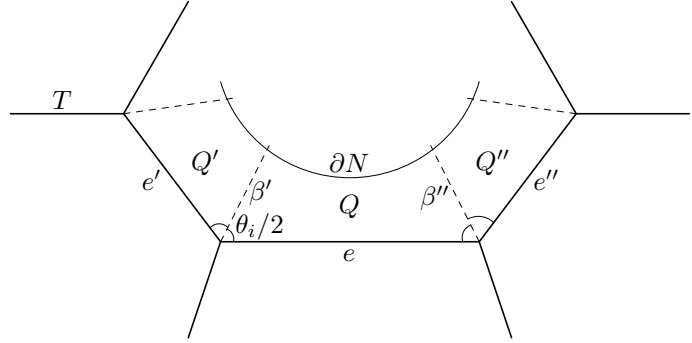

The construction is similar to Claim~\ref{claim:exC'neqC}, whose notation we borrow: let $(h,v):Q\rightarrow \R \times\R_+$ be the Fermi coordinates with respect to the edge~$e$, and $(h',v')$ the Fermi coordinates with respect to $\varphi(e)$.
Define $f|_Q$ by $h'(f(p))=C_0\, h(p)$ and $v'(f(p))=\Psi_{\theta_i/2}(v(p))$ for all $p\in Q$, where $\Psi_{\theta_i/2}$ is given by \eqref{eqn:exactbisect}.
Since the quadrilaterals $Q,Q',Q''$ have all their angles along $T$ equal to $\theta_i/2$, the map $f$ just defined takes the bisecting rays $\beta, \beta'$ to the bisecting rays of the corresponding angles of $T'$, in a well-defined manner.
The proof that $f$ is $C_0$-Lipschitz on~$N$ is the same as in Claim~\ref{claim:exC'neqC}.
\end{proof}

We claim that $C(j,\rho)=C_0$ and that $f$ is an optimal element of~$\F^{j,\rho}$, in the sense of Definition~\ref{def:relstretchlocus}.
Indeed, since $\mathcal{G}$ is invariant under the natural symmetry of~$S$, the group~$\Gamma_0$ is contained, with index two, in a discrete subgroup $\overline{\Gamma}_0$ of $G=\PGL_2(\R)$.
Let $\overline{j}\in\Hom(\overline{\Gamma}_0,G)$ be the natural inclusion and let $\overline{\rho}\in\Hom(\overline{\Gamma}_0,G)$ be the natural extension of~$\rho$. 
All reflections in perpendicular bisectors of edges of $T$ (\resp $T'$) belong to $\overline{j}(\overline{\Gamma}_0)$ (\resp $\overline{\rho}(\overline{\Gamma}_0)$).
Lemma~\ref{lem:finiteindex}, applied to $(\overline{\Gamma}_0,\Gamma_0)$, shows that there exists an element $\overline{f}\in\F^{j,\rho}$ which is optimal and $(\overline{j},\overline{\rho})$-equivariant.
Let us show that $\overline{f}$ agrees with $\varphi$ on~$T$.
Let $v$ be a vertex of~$T$ and let $e_1,e_2,e_3$ be the three incident edges of~$T$, connecting $v$ to $v_1, v_2, v_3$, with perpendicular bisectors $\mathcal{M}_1,\mathcal{M}_2,\mathcal{M}_3$.
For $1\leq i\leq 3$,  by $(\overline{j},\overline{\rho})$-equivariance,
$\overline{f}(v_i)$ is the symmetric image of~$\overline{f}(v)$ with respect to the perpendicular bisector $\mathcal{M}'_i$ of~$\varphi(e_i)$.
In particular, $d(\overline{f}(v),\overline{f}(v_i))=2\,d(\overline{f}(v),\mathcal{M}'_i)$.
Note that the convex function
$$q \,\longmapsto\, \max_{1\leq i\leq 3}\, \frac{d(q,\mathcal{M}'_i)}{d(v,\mathcal{M}_i)}$$
is always $\geq C_0$ on~$\HH^2$, with equality if and only if $q=\varphi(v)$, in which case all three ratios are equal to~$C_0$.
Therefore $C(j,\rho)=\Lip(\overline{f})\geq C_0$ and the constant $C_0$ is achieved, if at all, \emph{only} by maps that agree with $\varphi$ on the vertices of the tree $T$.
Since the map~$f$ of Claim~\ref{claim:exC'neqCnoncompact} is $C_0$-Lipschitz with stretch locus~$T$, this shows that $C(j,\rho)=C_0$ and that $E(j,\rho)=T$; in other words, $f$ is an optimal element of~$\F^{j,\rho}$.

As in Section~\ref{ex:C'neqC}, it is easy to see that no closed geodesic can spend more than a bounded proportion of its length near the trivalent graph~$\mathcal{G}$, which implies $C'(j,\rho)<C(j,\rho)=C_0$.

\begin{remarks}\label{rem:graph}
\begin{itemize}
  \item This construction actually gives an \emph{open} set of pairs $(j,\rho)\in\Hom(\Gamma_0,G)^2$ with $j$ convex cocompact and
  $$C'(j,\rho) < C(j,\rho) < \nolinebreak 1.$$
  Indeed, $\Hom(\Gamma_0,G)^2$ has dimension~$12$ and we have $12$ independent parameters, namely $\ell_1$, $\ell_2$, $\ell_3$, $\theta_1$, $\theta_2$, $C_0$, and a choice of a unit tangent vector in~$\HH^2$ for $T$ and for~$T'$ (\emph{i.e.}\ conjugation of~$j$ and~$\rho$).
  The map from this parameter space to $\Hom(\Gamma_0,G)^2$ is injective since different parameters give rise to different stretch loci; therefore its image is open by Brouwer's invariance of domain theorem.
  \item There is no constraint on $C_0\in (0,1)$: in particular, $\rho$ could be noninjective or nondiscrete.
  \item A similar construction works for any trivalent topological graph that is a retract of a convex cocompact hyperbolic surface.
  The invariance of the geodesic realizations under the natural symmetry of the three-holed sphere is replaced by a minimization property for the sum of weighted edge-lengths of the graph.
\end{itemize}
\end{remarks}

\medskip

All remaining examples show phenomena specific to the presence of cusps.

\subsection{The function $(j,\rho)\mapsto C(j,\rho)$ is not upper semicontinuous when $j(\Gamma_0)$ has parabolic elements}\label{ex:discontinu}

The following example shows that Proposition~\ref{prop:contCcc} fails in the presence of cusps, even if we restrict to $C<1$.
(It certainly fails for larger $C$ since the constant representation $\rho$ can have non-cusp-deteriorating deformations, for which $C\geq 1$.)

Let $\Gamma_0$ be a free group on two generators $\alpha,\beta$ and let $j\in\Hom(\Gamma_0,G)$ be given by
$$j(\alpha) = \begin{pmatrix} 1 & 3\\ 0 & 1\end{pmatrix} \ \text{ and }\ j(\beta) = \begin{pmatrix} 1 & 0\\ -3 & 1\end{pmatrix}.$$
The quotient $j(\Gamma_0)\backslash\HH^2$ is homeomorphic to a sphere with three holes, two of which are cusps (corresponding to the orbits of $0$ and $\infty$ in $\partial_{\infty}\HH^2$).
Let $\rho\in\Hom(\Gamma_0,G)$ be the constant representation, so that $C(j,\rho)=0$.
We shall exhibit a sequence $\rho_k\rightarrow\rho$ with $C(j,\rho_k)<1$ for all~$k$ and $C(j,\rho_k)\rightarrow 1$.

Define $\rho_k(\alpha)$ (\resp $\rho_k(\beta)$) to be the rotation centered at $A_k:=2^k \sqrt{-1}$ (\resp $B_k:=2^{-k} \sqrt{-1}$), of angle $2\pi/(2^k k)$.
Note that a circle of radius $r$ in~$\HH^2$ has circumference $2\pi\sinh(r)$ (see \eqref{eqn:circlearclength}), which is equivalent to $\pi e^r$ as $r\rightarrow +\infty$.
Therefore,
$$d\big(\sqrt{-1},\rho_k(\alpha)\cdot \sqrt{-1}\big) \ \underset{\scriptstyle k\rightarrow +\infty}{\sim}\  \frac{\pi}{k} \quad\underset{\scriptstyle k\rightarrow +\infty}\longrightarrow\quad 0,$$
hence $(\rho_k)_{k\in\N}$ converges to the constant representation~$\rho$.
By construction, $\rho_k(\alpha^{2^{k-1}k})$ is a rotation of angle~$\pi$ centered at $A_k$, and $\rho_k(\beta^{2^{k-1}k})$ a rotation of angle~$\pi$ centered at~$B_k$.
Therefore if $\omega_k=\alpha^{2^{k-1}k} \beta^{2^{k-1}k}$ then $\rho_k(\omega_k)$ is a translation of length $2d(A_k, B_k)=4k\log 2$.
On the other hand, one can compute explicitly $|\text{Tr}\, (j(\omega_k))|= (3\cdot 2^{k-1} k)^2-2$ which shows that $j(\omega_k)$ is a translation of length $4(k\log 2 + \log k)+O(1)$.
It follows that
$$C(j,\rho_k) \geq 1 - \frac{\log k}{k\log 2}+O\Big(\frac{1}{k}\Big),$$
which goes to $1$ as $k\rightarrow +\infty$.
See Figure~\ref{fig:T} for an interpretation of $\rho_k$ as the holonomy of a singular hyperbolic metric.

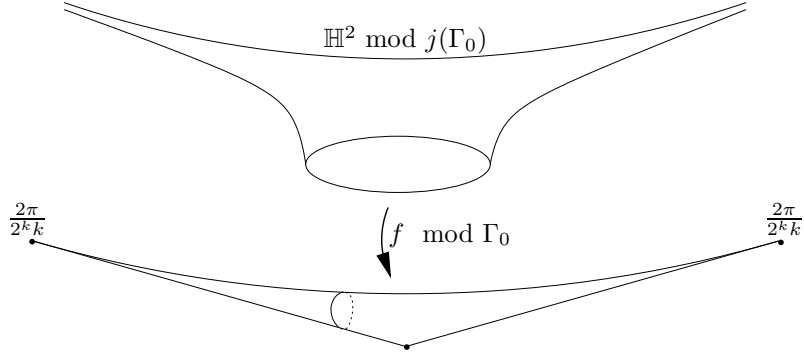
\begin{figure}[h!]
\begin{center}
\labellist
\small\hair 2pt
\pinlabel{$\HH^2\ \mathrm{mod}\ j(\Gamma_0)$} at 240 200
\pinlabel{$f\mod \Gamma_0$} at 268 75
\pinlabel{$\frac{2\pi}{2^k k}$} at 0 85
\pinlabel{$\frac{2\pi}{2^k k}$} at 480 85
\endlabellist
\includegraphics[width=10cm]{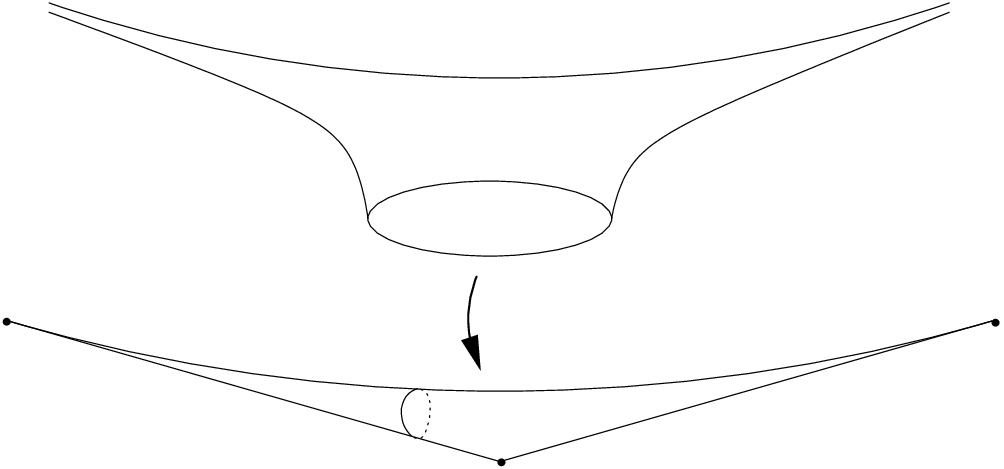}
\caption{The representation $\rho_k$ can be seen as the holonomy of a singular hyperbolic metric on a sphere with three cone points of angle $\frac{2\pi}{2^k k}, \frac{2\pi}{2^k k},$ and close to $2\pi$. The angle at the third cone point determines the distance between the other two, and is adjusted so that no equivariant map $f$ can be better than $(1-o(1))$-Lipschitz, as $k\rightarrow +\infty$.}
\label{fig:T}
\end{center}
\end{figure}

However, we have $C(j,\rho_k)<1$ for all~$k$: otherwise, by Corollary~\ref{cor:Enonempty} and Lemmas \ref{lem:Fnonempty}, \ref{lem:MaxStretchedLam}, and~\ref{lem:1StretchedLam}, the stretch locus $E(j,\rho_k)$ would contain a maximally stretched geodesic lamination $\LL_k$ with compact image $\dot{\LL}_k$ in $j(\Gamma_0)\backslash\HH^2$.
Necessarily any recurrent component of $\dot{\LL}_k$ would be a geodesic boundary component of the convex core (a three-holed sphere carries no other recurrent geodesic laminations!), corresponding to $\alpha\beta\in\Gamma_0$.
Therefore we would have $\lambda(\rho_k(\alpha\beta))=C(j,\rho)\,\lambda(j(\alpha\beta))\geq\lambda(j(\alpha\beta))>0$.
This is impossible since $\rho_k$ tends to the constant representation and $\lambda$ is continuous.

Note that by placing $A_k, B_k$ at $t^{\pm k}\sqrt{-1}$ for different values of $t$ in $(1,2]$ (without changing the rotation angle of $\rho_k(\alpha)$ and $\rho_k(\beta)$), we could also have forced $C(j,\rho_k)$ to converge to any value in $(0,1]$.

\subsection{The function $(j,\rho)\mapsto C(j,\rho)$ is not lower semicontinuous when $j(\Gamma_0)$ has parabolic elements}\label{ex:discontinu2}

Let $\Gamma_0$ be a free group on two generators $\alpha,\beta$ and $j\in\Hom(\Gamma_0,G)$ the holonomy representation of a hyperbolic metric on a once-punctured torus, with $j(\alpha\beta\alpha^{-1}\beta^{-1})$ parabolic.
We assume that $j(\Gamma_0)$ admits an ideal square $Q$ of~$\HH^2$ as a fundamental domain, with the axes of $j(\alpha)$ and $j(\beta)$ crossing the sides of~$Q$ orthogonally.
Fix two points $p,q\in \HH^2$ distance 1 apart.
For each $k\geq 1$, let $r_k\in \HH^2$ be the point at distance $k$ from $p$ and~$q$, so that $pqr_k$ is counterclockwise oriented.
Fix a small number $\delta>0$ and let $\rho_k$ be the representation of $\Gamma_0$ taking $\alpha$ (\resp $\beta$) to the translation of length~$\delta$ along the oriented geodesic line $(p,r_k)$ (\resp $(q,r_k)$) --- see Figure~\ref{fig:U}.

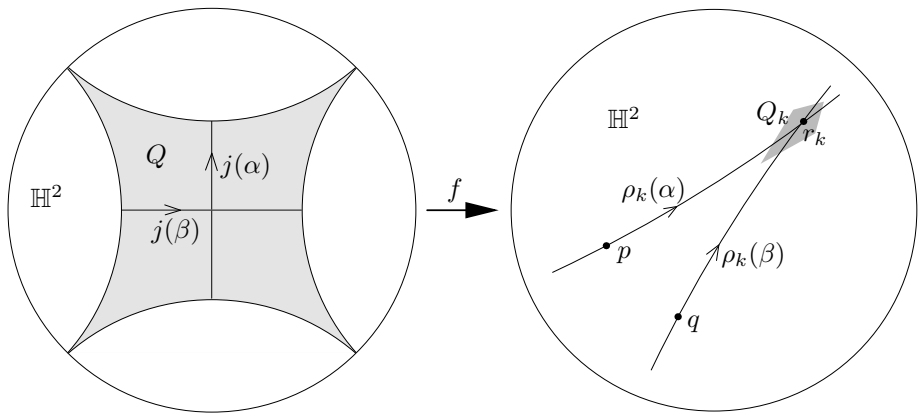
\begin{figure}[h!]
\begin{center}
\labellist
\small\hair 2pt
\pinlabel{$\HH^2$} at 20 110
\pinlabel{$\HH^2$} at 310 150
\pinlabel{$Q$} at 75 130
\pinlabel{$f$} at 225 113
\pinlabel{$j(\alpha)$} at 120 125
\pinlabel{$j(\beta)$} at 85 92
\pinlabel{$Q_k$} at 385 152
\pinlabel{$p$} at 310 80
\pinlabel{$q$} at 345 45
\pinlabel{$r_k$} at 406 140
\pinlabel{$\rho_k(\alpha)$} at 325 113
\pinlabel{$\rho_k(\beta)$} at 375 80
\endlabellist
\includegraphics[width=12cm]{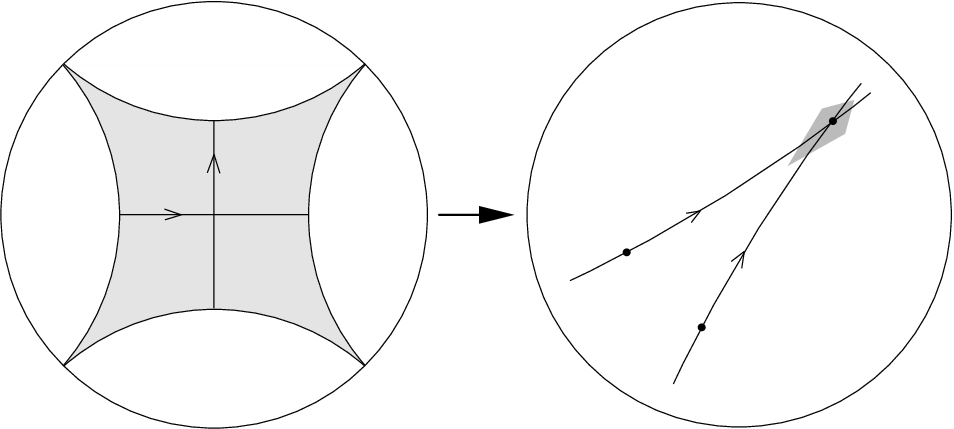}
\caption{If $\lambda(\rho_k(\alpha))$ and $\lambda(\rho_k(\beta))$ are small enough, then $C(j,\rho_k)$ stays small and bounded away from~$1$.}
\label{fig:U}
\end{center}
\end{figure}

As $k\rightarrow +\infty$, the representations $\rho_k$ converge to a representation~$\rho$ fixing exactly one point at infinity (the limit of $(r_k)_{k\geq 1}$), and $\rho(\alpha\beta\alpha^{-1}\beta^{-1})$ is parabolic: hence $C(j,\rho)\geq 1$. However, $C(j,\rho_k)$ is bounded away from~$1$ from above.
To see this, observe that the fixed points of $\rho_k(\alpha\beta\alpha^{-1}\beta^{-1})$, $\rho_k(\beta\alpha^{-1}\beta^{-1}\alpha)$, $\rho_k(\alpha^{-1}\beta^{-1}\alpha\beta)$, $\rho_k(\beta^{-1}\alpha\beta\alpha^{-1})$ are the vertices of a quadrilateral~$Q_k$ with four sides of equal length, centered at~$r_k$, of size roughly $2\delta$.
The maps $\rho_k(\alpha),\rho_k(\beta)$ identify pairs of opposite sides of~$Q_k$.
Taking $\delta$ very small, it is not difficult to construct maps $Q\rightarrow Q_k$ (taking whole neighborhoods of the ideal vertices of $Q$ to the vertices of $Q_k$) that are equivariant with very small Lipschitz constant.

Note however that the inequality $C(j,\rho)\leq \liminf_k C(j_k, \rho_k)$ of lower semicontinuity holds as soon as the Arzel\`a--Ascoli theorem applies for maps $f_k\in \mathcal{F}^{j_k,\rho_k}$, \ie as soon as the sequence $(f_k(p))_{k\geq 1}$ does not escape to infinity in~$\HH^2$: this fails only when $\rho$ fixes exactly one point at infinity.

\subsection{A reductive, non-cusp-deteriorating~$\rho$ with $E(j,\rho)=\emptyset$}\label{ex:nondeteriorating}

Let $S$ be a hyperbolic surface of infinite volume with at least one cusp and $j\in\Hom(\Gamma_0,G)$ its holonomy representation, where $\Gamma_0:=\pi_1(S)$.
Consider a collection of disjoint geodesics $\alpha_1,\dots,\alpha_m$ of~$S$ with both ends going out in the funnels, subdividing the convex core of~$S$ into contractible polygons and polygons with one puncture (cusp).
We apply Thurston's construction from the proof of Remark~\ref{rem:dTh<0}: for each~$\alpha_i$ we consider another geodesic~$\alpha'_i$ very close to, but disjoint from~$\alpha_i$, and construct the holonomy~$\rho$ of a new hyperbolic metric by cutting out the strips bounded by $\alpha_i\cup\alpha'_i$ and gluing back the boundaries, identifying the endpoints of the common perpendicular to $\alpha_i$ and~$\alpha'_i$ (see Figure~\ref{fig:V}).

\begin{figure}[h!]
\begin{center}
\labellist
\small\hair 2pt
\pinlabel{$\alpha_1$} at 285 350
\pinlabel{$\alpha'_1$} at 268 340
\pinlabel{$\alpha_2$} at 285 270
\pinlabel{$\alpha'_2$} at 250 260
\pinlabel{$f\ \mathrm{mod}\ \Gamma_0$} at 90 160
\pinlabel{$j(\Gamma_0)\backslash\HH^2$} at 80 290
\pinlabel{$\rho(\Gamma_0)\backslash\HH^2$} at 270 100
\endlabellist
\includegraphics[width=11cm]{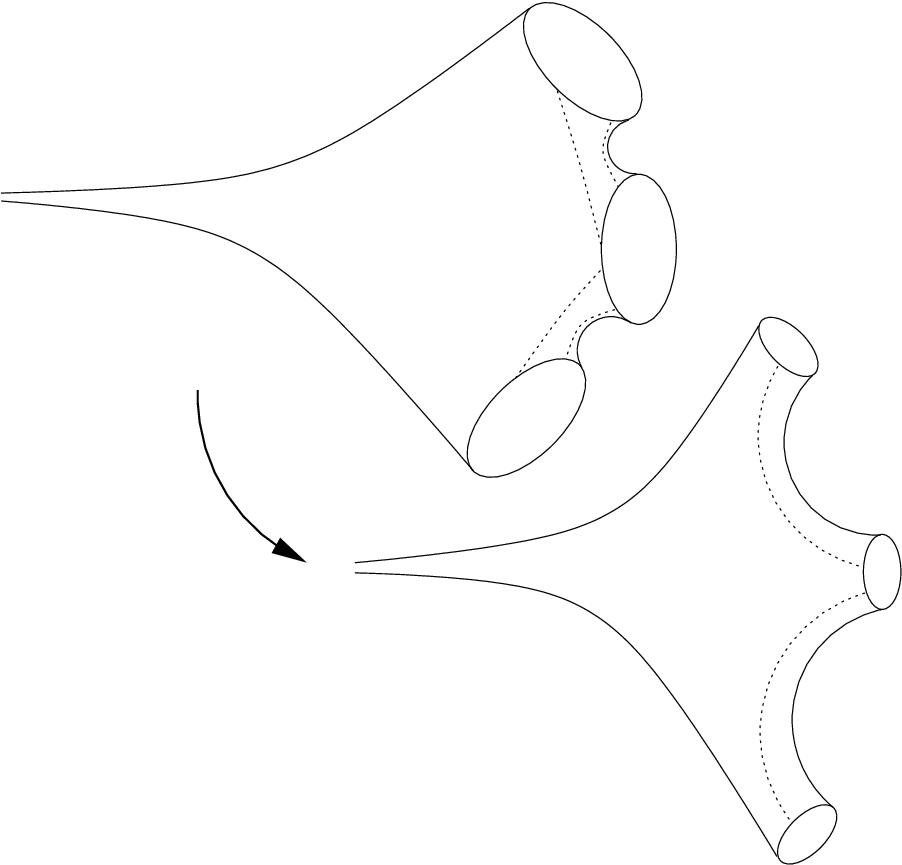}
\caption{In the second surface (with strips removed), simple closed curves are uniformly shorter than in the first.}
\label{fig:V}
\end{center}
\end{figure}
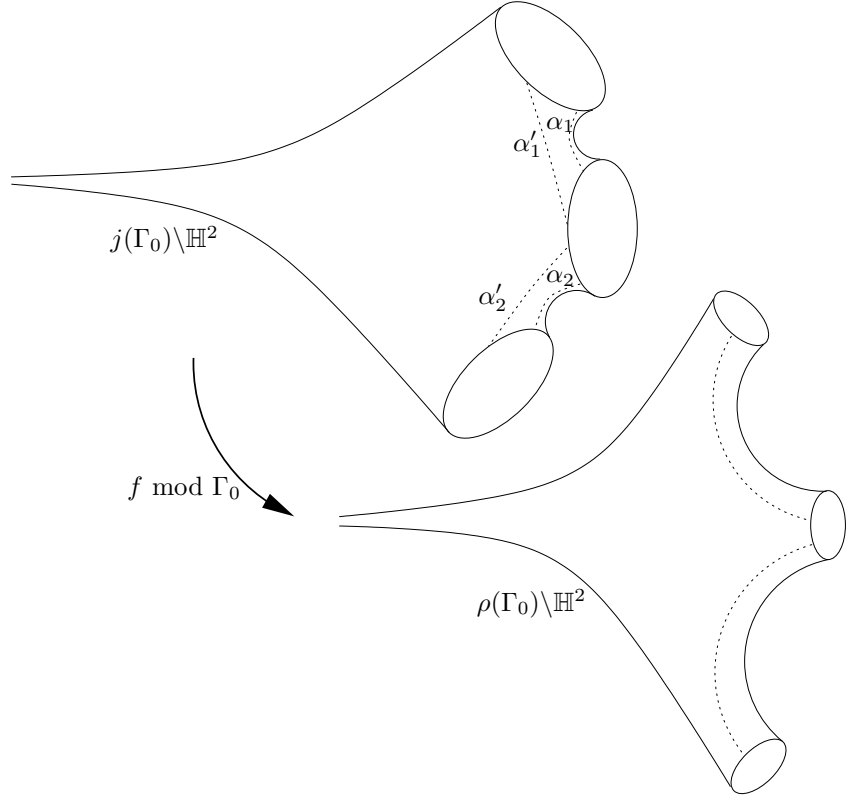

It is easy to check that the $(j,\rho)$-equivariant map $f : \HH^2\rightarrow\HH^2$ defined by this ``cut and paste'' procedure is $1$-Lipschitz, hence $C(j,\rho)\leq 1$.
In fact, $C(j,\rho)=1$ since $\rho$ is not cusp-deteriorating (Lemma~\ref{lem:parabdet}).
However, $E(j,\rho)=\emptyset$: otherwise, \eqref{eqn:Clambdasimple} would imply $C'_s(j,\rho)=1$, where $C'_s(j,\rho)$ is the supremum of $\lambda(\rho(\gamma))/\lambda(j(\gamma))$ over all elements $\gamma\in\Gamma_0$ corresponding to \emph{simple} closed curves $\mathcal{G}$ in~$S$. 
To see that this is impossible, first note that any such $\mathcal{G}$ intersects the arcs $\alpha_i$ nontrivially, yielding $\lambda(\rho(\gamma))<\lambda(j(\gamma))$. 
In fact, $\mathcal{G}$ stays in the complement of some cusp neighborhoods, which is compact: this means that $\mathcal{G}$ intersects the $\alpha_i$ a number of times roughly proportional to the length of $\mathcal{G}$. 
Moreover, each of these intersections is responsible for a definite (additive) drop in length between $\lambda(j(\gamma))$ and $\lambda(\rho(\gamma))$: this simply follows from the fact that $\alpha_i$ is a definite distance away from $\alpha'_i$, and forms with $\mathcal{G}$ an angle which can be bounded away from $0$ (again by compactness: $\alpha_i$ exits the convex core and $\mathcal{G}$ must not).
This implies $C'_s(j,\rho)<1$.
Therefore $E(j,\rho)=\emptyset$.
A similar argument can be found in \cite{pt10}.

(This is an example where $C'_s(j,\rho)<1=C(j,\rho)=C'(j,\rho)$, the last equality coming from Lemma~\ref{lem:ClambdaCLipred}.)

\subsection{A nonreductive, non-cusp-deteriorating~$\rho$ with $C'(j,\rho)<1=C(j, \rho)$ (and $E(j,\rho)=\emptyset$)}\label{ex:CC'nonreductive}

Let $\Gamma_0$ be a free group on two generators $\alpha,\beta$ and $j\in\Hom(\Gamma_0,G)$ the holonomy representation of a hyperbolic three-holed sphere with one cusp and two funnels, such that $j(\alpha)$ is hyperbolic and $j(\beta)$ parabolic.

For any nonreductive $\rho\in\Hom(\Gamma_0,G)$, if $\rho(\alpha)$ and $\rho(\beta)$ are \emph{not} hyperbolic (for instance if $\rho(\Gamma_0)$ is unipotent), then $C'(j,\rho)=0$; if moreover $\rho(\beta)$ is parabolic, then $\rho$ is not cusp-deteriorating and so $C(j,\rho)\geq 1$ by Lemma~\ref{lem:parabdet}, which implies $E(j,\rho)=\emptyset$ and $C(j,\rho)=1$ by Theorem~\ref{thm:lamin}.

Here is another example with $C'(j,\rho)>0$.
Let $\rho\in\Hom(\Gamma_0,G)$ be a nonreductive representation with $\rho(\alpha)$ hyperbolic and $\rho(\beta)$ parabolic; set $\varepsilon:=\lambda(\rho(\alpha))>0$.
There exists $L>0$ with the following property (see \cite[p.\,122]{dop00}, together with Lemma~\ref{lem:disthorosphere}): for any nontrivial cyclically reduced word $\gamma=\alpha^{m_1}\beta^{m_2}\alpha^{m_3}\beta^{m_4}\dots$ in $\Gamma_0$, with $m_2\cdots m_s\neq 0$ where $m_s$ is the last exponent,
$$\lambda(j(\gamma))\geq L\left( \sum_{i\in [1,s] \text{ odd}} |m_i| + \sum_{i\in [1,s] \text{ even}}(1+\log |m_i|) \right).$$
On the other hand, for such a~$\gamma$,
$$\lambda(\rho(\gamma))=\varepsilon \left | \sum_{i\in [1,s] \text{ odd}} m_i \right |,$$
hence $\lambda(\rho(\gamma))/\lambda(j(\gamma))\leq \varepsilon/L$.
This shows that $C'(j,\rho)\leq\varepsilon/L$, which is $<1$ for $\varepsilon$ small enough.
However, since $\rho(\beta)$ is parabolic we have $C(j,\rho)\geq 1$ as above, which implies $E(j,\rho)=\emptyset$ and $C(j,\rho)=1$ by Theorem~\ref{thm:lamin}.

\subsection{In dimension $n\geq 4$, the function $(j,\rho)\mapsto C(j,\rho)$ is not upper semicontinuous even above~$1$}\label{ex:dim4upper}

When $n\geq 4$, the existence of nonunipotent parabolic elements, coming from cusps of rank $<n-2$, destroys certain semicontinuity properties of~$C$.
We first give an example, in dimension $n=4$, where
$$1 \leq C(j,\rho) < \liminf_k C(j_k, \rho_k)$$
for some $(j_k,\rho_k)\rightarrow (j,\rho)$ with $j,j_k$ geometrically finite of the same cusp type, with a cusp of rank~$1$.
This shows that condition~(3) of Proposition~\ref{prop:contpropertiesC} is not satisfied in general for $n\geq 4$.

Identify $\partial_{\infty}\HH^4$ with $\R^3\cup\{\infty\}$ and let $G:=\PO(4,1)$.
For $\xi\in\R^3$, we denote by $P_{\xi}\subset\HH^4$ the copy of~$\HH^3$ bordered by the unit sphere of~$\R^3$ centered at~$\xi$.
Let $\Gamma_0$ be a free group on two generators $\alpha$ and~$\beta$, and let $j\in\Hom(\Gamma_0,G)$ be the representation such that
\begin{itemize}
 \item $j(\alpha)$ is the unipotent isometry of~$\HH^4$ fixing~$\infty$ and acting on~$\R^3$ by translation along the vector $(2\pi,0,0)$;
\item $j(\beta)$ is the pure translation (hyperbolic element) taking $\xi:=(3,0,0)$ to $\infty$, and $\infty$ to $\eta:=(0,0,0)$, and $P_{\xi}$ to $P_{\eta}$.
\end{itemize}
It is a standard argument (sometimes called ``ping pong'') that $j(\alpha)$ and $j(\beta)$ generate a free discrete group in~$G$; the representation $j$ is geometrically finite and the quotient manifold $j(\Gamma_0)\backslash\HH^4$ has one cusp, with stabilizer $\langle\alpha\rangle\subset\Gamma_0$.
Take $\rho=\rho_k=j$, so that $C(j,\rho)=1$.
Choose an integer $p\geq 2$ and, for $k\geq 1$, let $j_k\in\Hom(\Gamma_0,G)$ be the representation such that
\begin{itemize}
  \item $j_k(\alpha)$ is the parabolic element of~$G$ fixing~$\infty$ and acting on~$\R^3$ as the corkscrew motion preserving the line $\ell_k:=\{0\}\times \R\times \{k\}$, with rotation angle $2\pi/k$ around~$\ell_k$ and progression $\sqrt[p]{k}/k$ along~$\ell_k$;
  \item $j_k(\beta)=j(\beta)$.
\end{itemize}
It is an easy exercise to check that $j_k\rightarrow j$ as $k\rightarrow +\infty$.
Moreover, $j_k$ is geometrically finite with the same cusp type as~$j$ for large~$k$, by a standard ping pong argument.
(Note however that fundamental domains for the action of $j_k(\Gamma_0)$ do \emph{not} converge to a fundamental domain for the action of $j(\Gamma_0)$, but to a smaller set.)
The element $\rho_k(\alpha^k \beta)=j(\alpha^k \beta)$ takes $\xi$ to~$\infty$, and $\infty$ to $j(\alpha^k)(\eta)=(2k\pi,0,0)$, and $P_{\xi}$ to $P_{j(\alpha^k)(\eta)}$; by \eqref{eqn:hyperbolicH4},
\begin{equation}\label{eqn:lambdaj(a^kb)}
\lambda\big(j(\alpha^k \beta)\big) \geq 2\log 2\pi k - R
\end{equation}
for some $R>0$ independent of $k$.
On the other hand, $j_k(\alpha^k \beta)$ takes $\xi$ to~$\infty$, and $\infty$ to $j_k(\alpha^k)(\eta)=(0,\sqrt[p]{k},0)$, and $P_{\xi}$ to $P_{j_k(\alpha^k)(\eta)}$, hence
$$\lambda\big(j_k(\alpha^k \beta)\big) \leq 2\log \sqrt[p]{k} + R + 1$$
by \eqref{eqn:hyperbolicH4}.
It follows, by \eqref{eqn:ClambdaCLip}, that 
$$C(j_k,\rho_k)\geq \frac{\lambda(\rho_k(\alpha^k \beta))}{\lambda(j_k(\alpha^k \beta))} \geq \frac{2\log 2\pi k -R}{2\log \sqrt[p]{k}+ R+1}\,,$$
which accumulates only to values $\geq p$ as $k\rightarrow +\infty$.
Since $p$ was arbitrary, we see that $(j',\rho')\mapsto C(j',\rho')$ is not even bounded near $(j,\rho)$.

\subsection{The condition $C(j,\rho)<1$ is not open in dimension $n\geq 4$}\label{ex:dim4C<1}

We finally give an example, in dimension $n=4$, where
$$C(j,\rho) < 1 < \liminf_k C(j_k, \rho_k)$$
for $(j_k,\rho_k)\rightarrow (j,\rho)$ with $j,j_k$ geometrically finite of the same cusp type, with a cusp of rank~$1$, and with $\rho_k$ (and $\rho$) cusp-deteriorating.
This proves that condition~(1) of Proposition~\ref{prop:contpropertiesC} need not be satisfied for $n\geq 4$ when there is a cusp of rank $<n-2$.

Let $\Gamma_0$ be a free group on two generators $\alpha$ and~$\beta$, and let $j$ and~$j_k$ be as in Section~\ref{ex:dim4upper}.
We take a representation $\rho\in\Hom(\Gamma_0,G)$ such that
\begin{itemize}
  \item $\rho(\alpha)=1\in G$,
  \item $\rho(\beta)$ is a pure translation along some line $\ell$ of~$\HH^4$.
\end{itemize}
Since $\rho(\Gamma_0)$ is contained in the stabilizer of~$\ell$, multiplying the translation length of~$\rho(\beta)$ by some constant $\varepsilon>0$ multiplies the translation length of \emph{all} elements $\rho(\gamma)$ by~$\varepsilon$.
Therefore, up to taking $\lambda(\rho(\beta))$ small enough, we may assume $C(j,\rho)<1$.
Up to conjugating $\rho$, we can furthermore assume that there exist $\xi,\eta \in \R^3$ (distance $2\cosh \frac{\lambda(\rho(\beta))}{2}$ apart by \eqref{eqn:rainbow}) such that $\rho(\beta)$ takes $\xi$ to~$\infty$, and $\infty$ to~$\eta$, and $P_{\xi}$ to~$P_{\eta}$.
We still normalize to $\eta=(0,0,0)$ for convenience.
We then take $\rho_k\in\Hom(\Gamma_0,G)$ such that
\begin{itemize}
  \item $\rho_k(\alpha)$ is an elliptic transformation fixing pointwise the hyperbolic $2$-plane bordered by the line $\ell'_k:=\{0\}\times \R\times \{\sqrt{k}\}$ of $\R^3$ (compactified at~$\infty$), and acting as a rotation of angle $\pi/k$ in the orthogonal direction,
  \item $\rho_k(\beta)=\rho(\beta)$.
\end{itemize}
Clearly $\rho_k(\alpha)\rightarrow\rho(\alpha)$ as $k\rightarrow +\infty$, since this holds in restriction to any horosphere centered at~$\infty$ (such a horosphere is stable under~$\rho_k(\alpha)$).
This time, $\rho_k(\alpha^k \beta)$ takes $\xi$ to~$\infty$, and $\infty$ to $\rho_k(\alpha^k)(\eta)=(0,0,2\sqrt{k})$, and $P_{\xi}$ to $P_{\rho_k(\alpha^k)(\eta)}$, hence
$$\lambda\big(\rho_k(\alpha^k \beta)\big) \geq 2\log 2\sqrt{k} -R$$
by \eqref{eqn:hyperbolicH4}.
Using \eqref{eqn:lambdaj(a^kb)}, we obtain
$$C(j_k,\rho_k)\geq \frac{\lambda(\rho_k(a^k b))}{\lambda(j_k(a^k b))} \geq \frac{2\log 2\sqrt{k}-R}{2\log \sqrt[p]{k}+R+1}\,,$$
which accumulates only to values $\geq p/2$ as $k\rightarrow +\infty$.
Since $p$ was arbitrary, we see that $(j',\rho')\mapsto C(j',\rho')$ is not bounded near $(j,\rho)$, even in restriction to cusp-deteriorating~$\rho'$.

\appendix

\section{Some hyperbolic trigonometry}\label{sec:appendix}

We collect a few well-known formulas in hyperbolic trigonometry, from which we derive several formulas used at various places in the paper.

\subsection{Distances in $\HH^2$ and~$\HH^3$}

Let $n=2$ or~$3$.
We use the upper half-space model of~$\HH^n$: if $n=2$, then $\HH^n\simeq\{ z\in\C : \mathrm{Im}(z)>0\} $, the hyperbolic metric is given by
$$\mathrm{d}s^2 = \frac{\mathrm{d}|z|^2}{\mathrm{Im}(z)^2},$$
the isometry group $G$ of~$\HH^n$ identifies with $\PGL_2(\R)$ acting by M\"obius transformations, and $\partial_{\infty}\HH^n\simeq\R\cup\{ \infty\}$.
If $n=3$, then $\HH^n\simeq\C\times\R_+^{\ast}$, the hyperbolic metric is given by
$$\mathrm{d}s^2 = \frac{\mathrm{d}|a|^2 + \mathrm{d}b^2}{b^2}$$
for $(a,b)\in\C\times\R_+^{\ast}$, the identity component $G_0$ of~$G$ identifies with $\PSL_2(\C)$, which acts on the boundary $\partial_{\infty}\HH^n\simeq\C\cup\{ \infty\}$ by M\"obius transformations, and this action extends in a natural way to~$\HH^n$.
The matrix
$$T_{\eta} := \begin{pmatrix} e^{\eta/2} & 0 \\ 0 & e^{-\eta/2} \end{pmatrix} \,\in G_0$$
defines a translation of (complex) length~$\eta$ along the geodesic line with endpoints $0,\infty\in\partial_{\infty}\HH^n$.
Set $p_0:=\sqrt{-1}\in\HH^n$ if $n=2$, and $p_0:=(0,1)\in\HH^n$ if $n=3$.
Then
$$R_{\theta} : = \begin{pmatrix} \cos(\theta/2) & \sin(\theta/2) \\ -\sin(\theta/2) & \cos(\theta/2) \end{pmatrix} \,\in G_0$$
defines a rotation of angle~$\theta$ around~$p_0$ if $n=2$, and around the geodesic line (containing~$p_0$) with endpoints $\pm\sqrt{-1}\in\partial_{\infty}\HH^n$ if $n=3$.
The stabilizer of~$p_0$ in~$G_0$ is $\underline{K}=\mathrm{PSO}(2)$ if $n=2$, and $\underline{K}=\PSU(2)$ if $n=3$.
For any $g=\begin{pmatrix} a & b\\ c & d\end{pmatrix}\in\nolinebreak G_0$,
\begin{equation}\label{eqn:formulamu}
2\cosh d(p_0,g\cdot p_0) = \left\Vert\begin{pmatrix} a & b\\ c & d\end{pmatrix}\right\Vert^2 := |a|^2 + |b|^2 + |c|^2 + |d|^2.
\end{equation}
Indeed this holds for $g=T_{\eta}$ and the right-hand side is invariant under multiplication of~$g$ by elements of~$\underline{K}$ on either side (recall the Cartan decomposition $G_0=\underline{K}\underline{A}\underline{K}$ for $\underline{A}:=\{ T_{\eta}\,|\,\eta\in\R\} $, see Section~\ref{subsec:Cartanproj}).
Suppose $n=2$; applying \eqref{eqn:formulamu} to $g=\begin{pmatrix} v^{1/2} & uv^{-1/2} \\ 0 & v^{-1/2} \end{pmatrix}$, we find in particular that for any $u,v\in\R$ with $v>0$,
\begin{equation}\label{eqn:distinH2}
d\big(\sqrt{-1},u+\sqrt{-1}v) = \arccosh\left(\frac{u^2+v^2+1}{2v}\right).
\end{equation}

\subsubsection{Horospherical distances}

Applying \eqref{eqn:formulamu} to $g=\begin{pmatrix} 1 & L \\ 0 & 1 \end{pmatrix}$, we see that for any points $p,q$ on a common horosphere~$\partial H$, the distance $d(p,q)$ from $p$ to~$q$ in $\HH^n$ and the distance $L=d_{\partial H}(p,q)$ of the shortest path from $p$ to~$q$ contained in the horosphere~$\partial H$ (``horocyclic distance'') satisfy
\begin{equation}\label{eqn:horomu}
d(p,q) = \arccosh\Big(1+\frac{d_{\partial H}(p,q)^2}{2}\Big) = 2\,\arcsinh\Big(\frac{d_{\partial H}(p,q)}{2}\Big).
\end{equation}
Let $t\mapsto p_t$ and $t\mapsto q_t$ be the geodesic rays from~$p$ and $q$ to the center $\xi\in\partial_{\infty}\HH^n$ of the horosphere~$\partial H$, parametrized by arc length.
Then
\begin{equation}\label{eqn:exphorodist}
d_{\partial H_t}(p_t,q_t) = e^{-t}\,d_{\partial H}(p,q)
\end{equation}
for all $t\geq 0$, where $\partial H_t$ is the horocycle through $p_t$ and~$q_t$ centered at~$\xi$.
Using \eqref{eqn:horomu} and the concavity of $\arcsinh$, we find that there exists $D>1$ such that
\begin{equation}\label{eqn:expdistcusps}
e^{-t}\,d(p,q) \leq d(p_t,q_t) \leq D\, e^{-t}\,d(p,q)
\end{equation}
for all $t\geq 0$; moreover, an upper bound on $d(p,q)$ yields one on~$D$.

\subsubsection{Distances in two ideal spikes of~$\HH^2$}

The following situation is considered in the proof of Proposition~\ref{prop:chainrec}.
Let $\zeta_1\neq \zeta_2\neq \zeta_3\neq \zeta_4$ be points of $\partial_{\infty}\HH^2$, not necessarily all distinct.
Let $D_{i-1}$, $D_i$, $D_{i+1}$ be the geodesic lines of~$\HH^2$ running from~$\zeta_1$ to~$\zeta_2$, from~$\zeta_2$ to~$\zeta_3$, and from~$\zeta_3$ to~$\zeta_4$ respectively.
Consider two points $x\in D_{i-1}$ and $x'\in D_i$ on a common horocycle centered at~$\zeta_2$ and let $\xi\geq 0$ be their horocyclic distance.
Similarly, consider two points $y\in D_{i+1}$ and $y'\in D_i$ on a common horocycle centered at~$\zeta_3$ and let $\eta\geq 0$ be their horocyclic distance.
Setting $L:=d(x',y')$, we have
\begin{equation}\label{eqn:twospikes}
d(x,y) = L + \xi^2 + \eta^2 + o(\xi^2+\eta^2)
\end{equation}
as $\xi^2+\eta^2+e^{-L}\rightarrow 0$.
Indeed, by \eqref{eqn:formulamu},
\begin{eqnarray*}
\cosh d(x,y) & = & \frac{1}{2} \left\Vert \begin{pmatrix} 1 & 0\\ \xi & 1\end{pmatrix} T_L \begin{pmatrix} 1 & -\eta\\ 0 & 1\end{pmatrix}\right\Vert^2\\
& = & \cosh L + (\sinh L) \cdot (\xi^2 + \eta^2) (1+o(1))
\end{eqnarray*}
and we conclude using the degree-$1$ Taylor series of $\cosh$ at~$L$.

\subsubsection{Distances in a prism in~$\HH^3$}

The following situation is considered in the proof of Lemma~\ref{lem:Esimple}.
Consider a geodesic segment $I$ of~$\HH^3$, of length $\eta\geq\nolinebreak 0$, together with two oriented geodesic lines $\ell,\ell'$ of~$\HH^3$ meeting $I$ orthogonally at its endpoints, and forming an angle $\theta\in [0,\pi]$ with each other.
Note that $I$ is the shortest geodesic segment between $\ell$ and~$\ell'$; the complex number $L:=\eta+i\theta\in\C$ is called the \emph{complex distance} between $\ell$ and~$\ell'$ and will be expressed in terms of cross-ratios in Section~\ref{subsubsec:complexdistance}.
For now, let us compute the distance between points $p\in\ell$ and $q\in\ell'$, at respective signed distances $s$ and~$t$ from~$I$.
Note that $T'_s:=R_{\pi/2}T_s R_{-\pi/2}\in G_0$ is a translation of length~$s$ along the geodesic line from $-1\in\partial_{\infty}\HH^3$ to $1\in\partial_{\infty}\HH^3$, which intersects the translation axis of~$T_{L}$ (with endpoints $0,\infty\in\partial_{\infty}\HH^3$) perpendicularly at the basepoint $p_0=(0,1)\in\HH^3$.
Define $g:=T'_{-s}T_{L}T'_t$.
Without loss of generality, we may assume that $p=p_0$ and $q=g\cdot p$, and that $I=[T'_{-s}\cdot p_0,T'_{-s}T_{L}\cdot p_0]$.
Using \eqref{eqn:formulamu} and the identities $2|\cosh \frac{L}{2}|^2=\cosh \eta + \cos \theta$ and $2|\sinh \frac{L}{2}|^2=\cosh \eta - \cos \theta$, this gives
\begin{equation}\label{eqn:mixedmu}
 \cosh d(p,q)=\cosh \eta \cosh s \cosh t - \cos \theta \sinh s \sinh t~.
\end{equation}

When $\eta=0$ and $\theta=0$ or $\pi$ we recover the formulas for $\cosh(s\pm t)$.

When $\eta=0$ and $s=t$, we find that points $p,q$ on a circle of radius $s$, forming an angle $\theta$ from the center, are a distance $\sim \theta \sinh(s)$ apart when $\theta$ is small. 
This estimate is needed in the proof of Lemma~\ref{lem:horosph-sph}: approaching the arc of circle $\mathscr{C}$ from $p$ to~$q$ with a union of short geodesic segments, we find in the limit
\begin{equation}\label{eqn:circlearclength}
\mathrm{length}(\mathscr{C}) = \theta \sinh r.
\end{equation}

When $\theta=0$ and $s=t$, we find that points $p,q$ at (signed) distance $s$ from a straight line $\A$ of~$\HH^2$, whose projections to $\A$ are distance $\eta$ apart, satisfy 
\begin{equation}\label{eqn:equidistpoints}
d(p,q) \sim \eta \cosh s
\end{equation}
when $\eta$ is small.
(This situation is considered in the proof of Claim~\ref{claim:exC'neqC}.)

\subsubsection{Line-to-line distances}\label{subsubsec:complexdistance}

For any $\xi_-,\xi_+,\xi'_-,\xi'_+\in\partial_{\infty}\HH^3$, the complex distance $L=\eta+i\theta$ between the oriented lines $(\xi_-,\xi_+)$ and $(\xi'_-,\xi'_+)$ satisfies
\begin{equation} \label{eqn:complexdistance}
\cosh \eta = \frac{1+|c|}{|1-c|} \quad \text{and} \quad \cos \theta = \frac{1-|c|}{|1-c|},
\end{equation}
where $c:=[\xi_-:\xi_+:\xi'_-:\xi'_+]$ is the cross-ratio of $\xi_-,\xi_+,\xi'_-,\xi'_+$, defined so that $[\infty:0:1:\xi]=\xi$.
Indeed, by invariance under the action of~$G_0$, it is sufficient to check \eqref{eqn:complexdistance} when $(\xi_-,\xi_+)=(-1,1)$ and the shortest geodesic segment between $(\xi_-,\xi_+)$ and $(\xi'_-,\xi'_+)$ is contained in $(0,\infty)$.
In this case, $(\xi'_-,\xi'_+)$ is of the form $(-z,z)$ for some $z\in\C^{\ast}$ and we have
$$\left\{ \begin{array}{l}
\cosh \eta = (|z|+|z|^{-1})/2,\\
\cos \theta = \mathrm{Re}(z/|z|),\\
c = (1-z)^2/(1+z)^2;
\end{array}\right.$$
it is then elementary to check that \eqref{eqn:complexdistance} holds.

Now let $n\geq 2$ be arbitrary and consider as above the upper half space model of~$\HH^n$, so that $\partial_{\infty}\HH^n$ identifies with $\R^{n-1}\cup\{\infty\}$.
In Sections \ref{ex:dim4upper} and~\ref{ex:dim4C<1} we use the following consequence of \eqref{eqn:complexdistance}: there exists $R>0$ such that for any $D\geq 2$, any $\xi,\eta\in\R^{n-1}\subset\partial_{\infty}\HH^n$ distance $D$ apart for the Euclidean metric, and any $g\in\mathrm{Isom}(\mathbb{H}^n)$, if $g(\xi)=\infty=g^{-1}(\eta)$ and if $g$ maps the unit hemisphere (geodesic hyperplane) $P_{\xi}$ centered at~$\xi$ to the unit hemisphere $P_{\eta}$ centered at~$\eta$, then $g$ is hyperbolic and its translation length $\lambda(g)$ satisfies
\begin{equation}\label{eqn:hyperbolicH4}
|\lambda(g) - 2\log D| \leq R.
\end{equation}
Indeed, by \eqref{eqn:complexdistance} the hyperbolic distance between the closest points of $P_{\xi}$ and~$P_{\eta}$ is
\begin{equation}
\label{eqn:rainbow}
2\,\arccosh(D/2),
\end{equation}
and the intersection point of $P_{\xi}$ with the line $(\xi,\infty)$ is at hyperbolic distance $2\,\arcsinh(D/2)$ from $P_{\eta}\cap(\eta,\infty)$ by \eqref{eqn:horomu}.
The translation length $\lambda(g)$ is bounded in-between these two values, which are both $2\log D + O(1)$.

\subsection{Relations in a right-angled hyperbolic triangle}

Consider a triangle $ABC$ in~$\HH^2$ with angles $\widehat{A},\widehat{B},\widehat{C}$ and opposite edge lengths $a,b,c$.
Suppose $\widehat{B}=\pi/2$.
Then
\begin{equation}\label{eqn:trigo}
\tan\widehat{A} = \frac{\tanh a}{\sinh c}, \quad \cos\widehat{A} = \frac{\tanh c}{\tanh b}, \quad\text{ and }\ \sin\widehat{A} = \frac{\sinh a}{\sinh b}.
\end{equation}
Indeed, let $(\alpha,\beta,\gamma):=(e^{a/2},e^{b/2},e^{c/2})$ and $(X,Y):=(\cos \frac{\widehat{A}}{2}, \sin \frac{\widehat{A}}{2})$: following the perimeter of the triangle in the order $C,A,B,C$ shows that
$$
T_{-b}\,R_{\widehat{A}}\,T_c\,R_{-\pi/2}\,T_a = \begin{pmatrix} X\frac{\alpha \gamma}{\beta} + Y \frac{\alpha}{\beta\gamma}& -X\frac{\gamma}{\alpha\beta}+Y\frac{1}{\alpha\beta\gamma}\\ X\frac{\alpha\beta}{\gamma}-Y\frac{\alpha\beta\gamma}{1}& X\frac{\beta}{\alpha \gamma} + Y \frac{\beta\gamma}{\alpha} \end{pmatrix}
$$
must be (projectively) a rotation matrix, namely $R_{-\widehat{C}}$.
After multiplying all entries by $\alpha\beta\gamma$, this means
$$\alpha^2(X\gamma^2+Y)=\beta^2(X+Y\gamma^2) ~\text{ and }~ \alpha^2\beta^2 (X-Y\gamma^2)=X\gamma^2-Y.$$
It follows that
\begin{eqnarray*}
\tanh a &=& \frac{\alpha^2-\alpha^{-2}}{\alpha^2+\alpha^{-2}} = \frac{\beta^2 \frac{X+Y\gamma^2}{X\gamma^2+Y}-\beta^2 \frac{X-Y\gamma^2}{X\gamma^2-Y}}{\beta^2 \frac{X+Y\gamma^2}{X\gamma^2+Y}+\beta^2 \frac{X-Y\gamma^2}{X\gamma^2-Y}} \\
&=& \frac{\gamma^2-\gamma^{-2}}{2} \, \frac{2XY}{X^2-Y^2}  = \sinh c \tan\widehat{A} 
\end{eqnarray*}
and
\begin{eqnarray*}
\tanh b &=& \frac{\beta^2-\beta^{-2}}{\beta^2+\beta^{-2}} = \frac{\alpha^2 \frac{X\gamma^2+Y}{X+Y\gamma^2}-\alpha^2 \frac{X-Y\gamma^2}{X\gamma^2-Y}}{\alpha^2 \frac{X\gamma^2+Y}{X+Y\gamma^2}+\alpha^2 \frac{X-Y\gamma^2}{X\gamma^2-Y}} \\
&=& \frac{\gamma^2-\gamma^{-2}}{\gamma^2+\gamma^{-2}} \, \frac{X^2+Y^2}{X^2-Y^2} = \frac{\tanh c}{\cos\widehat{A}}.
\end{eqnarray*}
The last identity in \eqref{eqn:trigo} follows from the first two and from the Pythago\-rean identity $\cosh b=\cosh a\cosh c$, which is just \eqref{eqn:mixedmu} for $(\eta,\theta)=(0,\pi/2)$.

As a consequence of the last identity in \eqref{eqn:trigo}, if $x,y$ are two points on a circle of radius~$r$ in~$\HH^2$, forming an angle~$\theta$ from the center, then
\begin{equation}\label{eqn:circle}
\sin\frac{\theta}{2} = \frac{\sinh(d(x,y)/2)}{\sinh r}.
\end{equation}

\subsection{The closing lemma}

Finally, we recall the following classical statement; see \cite[Th.\,4.5.15]{bbs85} for a proof.

\begin{lemma}\label{lem:closinglemma}
For any $\delta>0$ and $D>0$, there exists $\varepsilon>0$ with the following property: given any broken line $\mathcal{L}=p_0\cdots p_{k+1}$ in~$\HH^n$, if $d(p_i, p_{i+1})\geq D$ for all $1\leq i<k$ and if the angle $\widehat{p_{i-1}p_ip_{i+1}}$ is $\geq\pi-\varepsilon$ for all $1\leq i\leq k$, then $\mathcal{L}$ stays within distance~$\delta$ from the segment $[p_0, p_{k+1}]$, and has total length at most $d(p_0,p_{k+1})+k\delta$.
Moreover, when $\delta$ is fixed, $\varepsilon=\delta$ will do for all large enough~$D$.
\end{lemma}

Taking limits as $k\rightarrow +\infty$, this implies in particular that for a broken line $(p_i)_{i\in\Z}$ invariant under a hyperbolic element $g\in G$ taking each $p_i$ to~$p_{i+m}$, under the same assumptions on length and angle we have
$$\Big|\lambda(g) - \sum_{i=1}^m d(p_i,p_{i+1})\Big| \leq m\delta.$$

\section{Converging fundamental domains}\label{sec:conv-fundamental}

Let $\Gamma_0$ be a discrete group.
It is well known that, in any dimension $n\geq 2$, the set of convex cocompact representations of $\Gamma_0$ into $G=\PO(n,1)=\mathrm{Isom}(\HH^n)$ is open in $\Hom(\Gamma_0,G)$ (see \cite[Prop.\,4.1]{bow98} for instance).
The set of geometrically finite representations is open in the set of representations $\Gamma_0\rightarrow G$ of fixed cusp type if $n\leq 3$ \cite{mar74}, or if all cusps have rank $\geq n-2$ \cite[Prop.~1.8]{bow98}, but not in general for $n\geq 4$ \cite[\S\,5]{bow98}.

In Sections \ref{subsubsec:uppersemicont} and~\ref{subsec:C<1open} of the paper, where we examine the continuity properties of the function $(j,\rho)\mapsto C(j,\rho)$, we need, not only this openness, but also a control on fundamental domains in~$\HH^n$ for converging sequences of geometrically finite representations.
Propositions \ref{prop:rem-ccopen} and~\ref{prop:rem-gfopen} below are certainly well known to experts, but we could not find a proof in the literature.
Note that they easily imply the Hausdorff convergence of the limit sets, but are a priori slightly stronger.

\subsection{The convex cocompact case}

\begin{proposition}\label{prop:rem-ccopen}
Let $\Gamma_0$ be a discrete group and $(j_k)_{k\in\N^{\ast}}$ a sequence of elements of $\Hom(\Gamma_0,G)$ converging to a convex cocompact representation  $j\in\Hom(\Gamma_0,G)$.
Then for any large enough $k\in\N^{\ast}$ the representation $j_k$ is convex cocompact.
Moreover, there exists a compact set $\mathcal{C}\subset\HH^n$ that contains fundamental domains of the convex cores of $j(\Gamma_0)\backslash\HH^n$ and $j_k(\Gamma_0)\backslash\HH^n$ for all large enough $k\in\N^{\ast}$.
If $\Gamma_0$ is torsion-free, then the injectivity radius of $j_k(\Gamma_0)\backslash \HH^n$ is bounded away from~$0$ as $k\rightarrow +\infty$.
\end{proposition}

Proposition \ref{prop:rem-ccopen} for torsion-free $\Gamma_0$ implies the general case, due to the Selberg lemma \cite[Lem.\,8]{sel60}. 
We henceforth assume $\Gamma_0$ to be torsion-free.

\begin{proof}
We build fundamental domains as finite unions of simplices coming from $j(\Gamma_0)$-invariant triangulations: the main step is the following.

\begin{claim}\label{cla:ccopen}
There exists a $j(\Gamma_0)$-invariant geodesic triangulation $\Delta$ of a nonempty convex subset of $\HH^n$ which is finite modulo $j(\Gamma_0)$ and induces dihedral angles $<\pi$ on the boundary.
\end{claim}

Let us prove Claim~\ref{cla:ccopen} (note that the projection of $\Delta$ to $M:=j(\Gamma_0)\backslash\HH^n$ will automatically contain the convex core).
The idea is to use a classical construction, the hyperbolic \emph{Delaunay decomposition} (an analogue of the Euclidean Delaunay decomposition of \cite{del34}), and make sure that it is finite modulo $j(\Gamma_0)$.
Let $N\subset\HH^n$ be the preimage of the convex core of $M=j(\Gamma_0)\backslash\HH^n$ and let $\mathcal{N}$ be the uniform $1$-neighborhood of~$N$. 
For $R\geq 0$, we call $R$-\emph{hyperball} of~$\HH^n$ any convex region of~$\HH^n$ bounded by a connected hypersurface at constant distance~$R$ from a hyperplane.
Since $N$ is the intersection of all half-spaces containing $N$, we see that $\mathcal{N}$ is the intersection of all $1$-hyperballs containing~$\mathcal{N}$.
By the strict convexity of the distance function in~$\HH^n$, there exists $\alpha>0$ such that whenever points $p,q$ of a $1$-hyperball are distance $\geq 1$ apart, the ball of radius~$\alpha$ centered at the midpoint of $[p,q]$ is also contained in the $1$-hyperball.
Therefore, whenever $p,q\in\mathcal{N}$ are distance $\geq 1$ apart, the ball of radius~$\alpha$ centered at the midpoint of $[p,q]$ is also contained in~$\mathcal{N}$.

Let $X$ be a $j(\Gamma_0)$-invariant subset of~$\mathcal{N}$ that is finite modulo $j(\Gamma_0)$ and intersects every ball of radius $\geq \alpha/2$ centered at a point of~$\mathcal{N}$.
We view $X$ as a subset of~$\R^{n+1}$ via the embedding of~$\HH^n$ as the upper hyperboloid~sheet
$$\mathcal{H} := \big\{ x\in\R^{n+1} :\ x_1^2 + \dots + x_n^2 - x_{n+1}^2 = -1 \;\text{and}\; x_{n+1}>0\big\} .$$
Consider the convex hull $\widehat{X}$ of $X$ in~$\R^{n+1}$. 
There is a natural bijection between the following two sets:
\begin{itemize}
  \item the set of supporting hyperplanes of $\widehat{X}$ separating $X$ from $0\in\R^{n+1}$,
  \item the set of open balls, horoballs, or hyperballs of~$\HH^n$ that are disjoint from~$X$ but whose boundary intersects~$X$.
\end{itemize}
Namely, the bijection is given by taking any supporting hyperplane to the set of points of~$\mathcal{H}$ that it separates from~$X$ (see Figure~\ref{fig:X}).
This set is a ball (\resp a horoball, \resp a hyperball) if the intersection of $\mathcal{H}$ with the supporting hyperplane is an ellipsoid (\resp a paraboloid, \resp a hyperboloid).
The degenerate case of a supporting hyperplane tangent to~$\mathcal{H}$ corresponds to an open ball of radius $0$ (the empty set!) centered at a point of $X$; the limit case of a supporting hyperplane containing $0\in\R^{n+1}$ corresponds to a $0$-hyperball, \ie a half-space of~$\HH^n$.

\begin{figure}[h!]
\begin{center}
\labellist
\small\hair 2pt
\pinlabel{$\mathcal{H}$} at 200 180
\pinlabel{$\HH^2$} at 190 60
\endlabellist
\includegraphics[width=12cm]{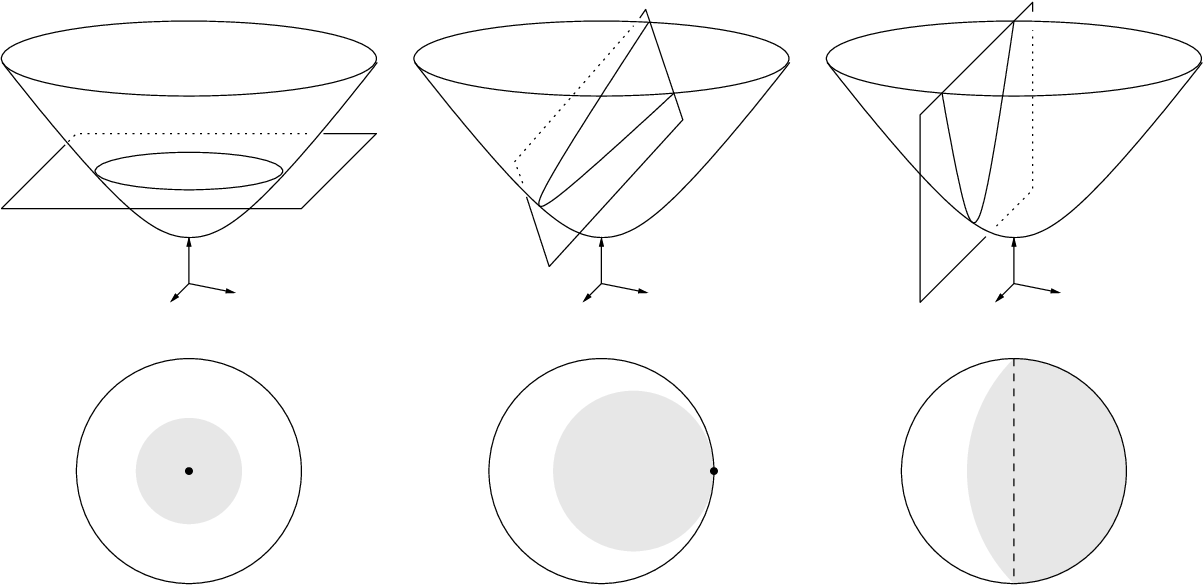}
\caption{Balls, horoballs, and hyperballs of~$\HH^n$ are intersections of the hyperboloid sheet $\mathcal{H}$ with affine half-spaces of~$\R^{n+1}$ containing the origin.}
\label{fig:X}
\end{center}
\end{figure}

For any supporting hyperplane of~$\widehat{X}$, the corresponding open ball, horoball, or hyperball $B\subset\HH^n$ intersects $\mathcal{N}$ in a region of diameter $\leq 1$. 
Indeed, if $p,q\in B\cap\mathcal{N}$ were distance $>1$ apart, then the ball $B'$ of radius~$\alpha$ centered at the midpoint of $[p,q]$ would be contained in~$\mathcal{N}$, by choice of~$\alpha$.
But at least one hemisphere of~$B'$ (the hemisphere closest to the center or to the defining hyperplane of~$B$, depending on whether $B$ is a (horo)ball or a hyperball) would also be contained in~$B$. 
A ball of radius $\alpha/2$ contained in this hemisphere would intersect~$X$ (by assumption on~$X$), while being contained in~$B$: impossible.
Thus $\partial B\cap\mathcal{N}$ has diameter $\leq 1$.
In particular, $\partial B\cap X$ has diameter $\leq 1$.
In particular, $\partial B\cap X$ is finite.

Let $Y\subset \partial \widehat{X}$ be the union of all points that belong to supporting hyperplanes separating $X$ from $0\in\R^{n+1}$.
(In other words, $Y$ is the portion of $\partial\widehat{X}$ that is ``visible from the origin''. 
There can also be an ``invisible'' portion, corresponding to hyperballs whose \emph{complement} is disjoint from~$X$.)
By the previous paragraph, $Y$ has the structure of a \emph{locally finite} polyhedral hypersurface in~$\R^{n+1}$, with vertex set~$X$.
Projecting each polyhedron of~$Y$ to the hyperboloid $\mathcal{H}\simeq \HH^n$ (along the rays through the origin $0\in\R^{n+1}$), we obtain a cellulation $\Delta$ of the convex hull $\Conv{X}$ of $X$ in~$\HH^n$, called the \emph{Delaunay cellulation} of $\Conv{X}$ relative to~$X$.
It is characterized by the fact that any cell of~$\Delta$ is inscribed in a hypersurface of~$\HH^n$ bounding some open ball, horoball, or hyperball disjoint from~$X$.
The cellulation $\Delta$ is $j(\Gamma_0)$-invariant and finite modulo $j(\Gamma_0)$.
Since $j(\Gamma_0)$ is torsion-free, up to taking the points of~$X$ in general position we may assume that $\Delta$ is a triangulation and induces dihedral angles $<\pi$ on the boundary of~$\widehat{X}$.
This completes the proof of Claim~\ref{cla:ccopen}.

Proposition~\ref{prop:rem-ccopen} easily follows from Claim~\ref{cla:ccopen}.
Indeed, let $F\subset\HH^n$ be a finite set such that $X=j(\Gamma_0)\cdot F$.
The vertices of a $d$-dimensional simplex of the triangulation $\Delta$ can be listed in the form $j(\gamma_0)\cdot\nolinebreak p_0,\dots,j(\gamma_d)\cdot\nolinebreak p_d$, where $p_0,\dots,p_d\in F$ and $\gamma_0,\dots,\gamma_d\in\Gamma_0$.
By finiteness of the triangulation, when $j_k$ is close enough to $j$ the points $j_k(\gamma_0)\cdot p_0,\dots,j_k(\gamma_d)\cdot p_d$ still span a simplex and these simplices (obtained by following the combinatorics of~$\Delta$) still triangulate a region of~$\HH^n$ that is locally convex, hence globally convex.
In particular, this region of~$\HH^n$ contains the preimage of the convex core of $j_k(\Gamma_0)\backslash\HH^n$.
Thus $j_k(\Gamma_0)$ is still convex cocompact for large $k\in\nolinebreak\N^{\ast}$.
Any compact neighborhood $\mathcal{C}$ of a union $U$ of representatives of simplex orbits of~$\Delta$ under $j(\Gamma_0)$ contains a fundamental domain of the convex core of $j_k(\Gamma_0)\backslash\HH^n$ for all large enough~$k$. 

To bound the injectivity radius of $j_k(\Gamma_0)\backslash\HH^n$ away from~$0$, we argue as follows.
For any $p\in F$, let $U_p$ be the union of all simplices of $\Delta$ containing~$p$.
Then $p$ is an interior point of~$U_p$.
Provided $X$ is dense enough in~$\mathcal{N}$, each $U_p$ projects injectively to $M=j(\Gamma_0)\backslash\HH^n$. 
For $\varepsilon>0$, let $U_p^{\varepsilon}$ be the complement in~$U_p$ of the $\varepsilon$-neighborhood of $\partial U_p$. 
If $\varepsilon$ is small enough, then any point of~$\Delta$ has a translate belonging to some $U_p^{\varepsilon}$ with $p\in F$, whose $\varepsilon$-neighborhood therefore projects injectively to~$M$. 
This property remains true as $\Delta$ (hence the finitely many sets~$U_p$) are deformed slightly, up to taking a smaller~$\varepsilon$. 
This completes the proof of Proposition~\ref{prop:rem-ccopen}.
\end{proof}

Note that in the above proof, for torsion-free~$\Gamma_0$, the hyperbolic manifolds $j_k(\Gamma_0)\backslash\HH^n$ and $j(\Gamma_0)\backslash\HH^n$ are in fact homeomorphic since their convex cores admit combinatorially identical triangulations.

\subsection{The geometrically finite case when all cusps have rank $\geq n-2$}

Here is an analogue of Proposition~\ref{prop:rem-ccopen} for geometrically finite representations of fixed cusp type with all cusps of rank $\geq n-2$.
Note that cusps always have rank $\geq n-2$ in dimension $n\leq 3$.

\begin{proposition}\label{prop:rem-gfopen}
Let $\Gamma_0$ be a discrete group, $j\in\Hom(\Gamma_0,G)$ a geometrically finite representation with all cusps of rank $\geq n-2$, and $(j_k)_{k\in\N^{\ast}}$ a sequence of elements of $\Hom(\Gamma_0,G)$ converging to~$j$, all of the same cusp type as~$j$ (Definition~\ref{def:typedet}).
Then for any large enough $k\in\N^{\ast}$ the representation $j_k$ is geometrically finite and $(j_k)_{k\in\N}$ converges geometrically to~$j$.
More precisely, if $H_1,\dots,H_c$ are horoballs of~$\HH^n$ whose images in $j(\Gamma_0)\backslash\HH^n$ are disjoint, small enough, and intersect the convex core in standard cusp regions (Definition~\ref{def:standardcusp}), representing all the cusps, then there exist a compact set $\mathcal{C}\subset\HH^n$ and, for any large enough $k\in\N^{\ast}$, horoballs $H_1^k,\dots,H_c^k$ of~$\HH^n$, such that
\begin{itemize}
  \item the images of $H_1^k,\dots,H_c^k$ in $j_k(\Gamma_0)\backslash\HH^n$ are disjoint and intersect the convex core in standard cusp regions;
  \item the stabilizer in~$\Gamma_0$ of $H_i^k$ under~$j_k$ is the stabilizer in~$\Gamma_0$ of $H_i$ under~$j$ for all $1\leq i\leq c$;
  \item the horoballs $H_i^k$ converge to~$H_i$ for all $1\leq i\leq c$, as $k\rightarrow +\infty$;
\item the union of $\mathcal{C}$ and of $H_1\cup\dots\cup H_c$ (resp.\ $H^k_1\cup\dots\cup H^k_c$) contains a fundamental domain of the convex core of $j(\Gamma_0)\backslash\HH^n$ (resp.\ $j_k(\Gamma_0)\backslash\HH^n$);
  \item the union of all geodesic rays from $\mathcal{C}$ to the centers of $H_1,\dots, H_c$ (resp.\ $H^k_1,\dots, H^k_c$) contains a fundamental domain of the convex core of $j(\Gamma_0)\backslash\HH^n$ (resp.\ $j_k(\Gamma_0)\backslash\HH^n$); in particular, the cusp thickness (Definition~\ref{def:thickness}) of $j_k(\Gamma_0)\backslash\HH^n$ at any point of $\bigcup_{1\leq i\leq c} \partial H_i^k$ is uniformly bounded by some constant independent of~$k$.
\end{itemize}
Moreover, if $j(\Gamma_0)$ is torsion-free, then the infimum of injectivity radii of $j_k(\Gamma_0)\backslash\HH^n$ at projections of points of~$\mathcal{C}$ is bounded away from~$0$ as $k\rightarrow +\infty$.
\end{proposition}

Again, we can and will assume that $\Gamma_0$ is torsion-free.
Proposition~\ref{prop:rem-gfopen} fails in dimension $n\geq 4$ when $j$ has a cusp of rank $<n-2$, as can be seen from \cite[\S\,5]{bow98} or by adapting the examples of geometrically finite representations~$j_k$ from Sections \ref{ex:dim4upper} and~\ref{ex:dim4C<1}.

In order to prove Proposition~\ref{prop:rem-gfopen}, we need the following lemma, which is also used in Section~\ref{subsec:C<1open}.

\begin{lemma}\label{lem:convexcusp}
Let $j\in\Hom(\Gamma_0,G)$ be a geometrically finite representation with all cusps of rank $\geq n-2$, and let $\mathcal{N}\subset\HH^n$ be a uniform neighborhood of the preimage $N$ of the convex core of $j(\Gamma_0)\backslash\HH^n$.
For any horoball $H$ of~$\HH^n$ such that $H\cap N$ projects to a standard cusp region and such that the cusp thickness (Definition~\ref{def:thickness}) of $j(\Gamma_0)\backslash\HH^n$ at any point of~$H$ is $\leq 1/3$, the set $\partial H\cap\mathcal{N}$ is convex in $\partial H\simeq\R^{n-1}$, equal to 
\begin{itemize}
\item the full Euclidean space $\partial H$ if the cusp has rank $n-1$;
\item the region contained between two parallel (possibly equal) Euclidean hyperplanes of $\partial H$ if the cusp has rank $n-2$.
\end{itemize}
\end{lemma}

\begin{proof}[Proof of Lemma~\ref{lem:convexcusp}]
The stabilizer $S\subset\Gamma_0$ of $H$ under~$j$ has a finite-index normal subgroup~$S'$ isomorphic to~$\Z^m$, where $m\in\{ n-1,n-2\} $ is the rank of the cusp (see Section~\ref{subsec:geo-finiteness}).
In the upper half-space model $\R^{n-1}\times\R_+^{\ast}$ of~$\HH^n$, in which $\partial_{\infty}\HH^n$ identifies with $\R^{n-1}\cup\{\infty\}$, we may assume that $H$ is centered at infinity, so that $\partial H=\R^{n-1}\times\{ b\} $ for some $b>0$.
Let $\Omega$ be the convex hull of $\Lambda_{j(\Gamma_0)}\smallsetminus\{\infty\}$ in~$\R^{n-1}$, where $\Lambda_{j(\Gamma_0)}$ is the limit set of $j(\Gamma_0)$.
The group $j(S)$ acts on~$\R^{n-1}$ by Euclidean isometries and there is an $m$-dimensional affine subspace $V\subset\Omega$, preserved by $j(S)$, on which $j(S')$ acts as a lattice of translations (see Section~\ref{subsec:geo-finiteness}).
This implies that $\Omega$ is either $\R^{n-1}$ (if $m=n-1$) or the region contained between two parallel hyperplanes of $\R^{n-1}$ (if $m=n-2$).
Let $\delta>0$ be the Euclidean diameter of $j(S)\backslash\Omega$.
Then $\delta$ is the cusp thickness of $j(\Gamma_0)\backslash\HH^n$ at $\R^{n-1}\times\{1\}$, or alternatively the cusp thickness of $j(\Gamma_0)\backslash\HH^n$ is $\leq 1/3$ exactly on $\R^{n-1}\times [3\delta,+\infty)$.
Every $S$-orbit in $\Omega$ is $\delta$-dense in $\Omega$.
Lemma \ref{lem:convexcusp} reduces to:
\begin{equation}\label{eqn:convexcusp}
N \cap \big(\R^{n-1}\times [3\delta,+\infty)\big) = \Omega\times [3\delta,+\infty).
\end{equation}
The left-hand side is always contained in the right-hand side since $N$ is the convex hull in~$\HH^n$ of the limit set $\Lambda_{j(\Gamma_0)}$.
For the converse, suppose a point $p\in\Omega\times\R_+^{\ast}$ lies outside~$N$.
Then $p$ belongs to a closed half-space $D$ of~$\HH^n$ disjoint from~$N$, whose ideal boundary $\partial_\infty D$ is a Euclidean $(n-1)$-dimensional ball disjoint from~$\Lambda_{j(\Gamma_0)}$.
If $\partial_\infty D$ is centered inside~$\Omega$ (\eg if $m=n-1$), then $\partial_\infty D$ has radius $<\delta$ because $\Lambda_{j(\Gamma_0)}$ is $\delta$-dense in~$\Omega$. This yields $p\in \Omega\times (0,\delta)\subset\Omega\times (0,3\delta)$, as desired.

Now suppose $\partial_\infty D$ is centered outside~$\Omega$; this may only happen if $m=n-2$.
Let $P$ be the connected component of $\partial \Omega$ closest to the center of $\partial_\infty D$, and $P'$ the other component: $P$ and $P'$ are parallel hyperplanes of $\R^{n-1}$ both intersecting~$\Lambda_{j(\Gamma_0)}$ (and possibly equal, if $j(\Gamma_0)$ preserves a copy of~$\HH^{n-1}$). 
We claim that $P\cap \Lambda_{j(\Gamma_0)}$ is $3\delta$-dense in $P$: indeed, fix $\xi\in P\cap \Lambda_{j(\Gamma_0)}$. 
If $j(S)\cdot\xi\subset P$ (in particular if $P=P'$) we are done, as $j(S)\cdot\xi$ is already $\delta$-dense in~$\Omega$. 
If not, since $j(S)\cdot\xi\subset P\cup P'$ we can choose $\xi'\in P'\cap j(S)\cdot\xi$. 
Let $S'\subset S$ be the stabilizer of $P$, which has index two in $S$, so that $j(S')\backslash\Omega$ has twice the Euclidean volume of $j(S)\backslash\Omega$. 
Let $U$ and $U'$ be (closed) Dirichlet fundamental domains of $j(S)\backslash\Omega$ for the Euclidean metric, centered at $\xi$ and $\xi'$ respectively. 
Any lifts of $U$ and $U'$ to $j(S')\backslash \Omega$ must overlap, because of their total volume.
Since $U$ and $U'$ are contained in $\delta$-balls centered at $\xi$ and $\xi'$, this shows that $\xi$ and $\xi'$ are at most $2\delta$ apart in the quotient $j(S')\backslash\Omega$. 
Since $j(S)\cdot\xi=j(S')\cdot\{\xi,\xi'\}$ is $\delta$-dense in~$\Omega$, we get that $j(S')\xi\subset P$ is $3\delta$-dense (and contained in $P\cap \Lambda_{j(\Gamma_0)}$).

The intersection $\partial_\infty D \cap P$, being disjoint from $\Lambda_{j(\Gamma_0)}$, is therefore an $(n-2)$-dimensional Euclidean ball of radius $<3\delta$. 
Since the hemisphere $D$ is centered outside $\Omega$, this shows that $D$ does not achieve any height $\geq 3\delta$ inside $\Omega\times \R_+^\ast$.
In particular, $p\in D\cap(\Omega\times\R_+^{\ast})\subset\Omega\times (0,3\delta)$.
This proves~\eqref{eqn:convexcusp}. 
\end{proof}

\begin{proof}[Proof of Proposition~\ref{prop:rem-gfopen}]
We proceed as in the proof of Proposition~\ref{prop:rem-ccopen} and first establish the following analogue of Claim~\ref{cla:ccopen}.

\begin{claim}\label{cla:gfopen}
There exists a $j(\Gamma_0)$-invariant geodesic triangulation $\Delta$ of a non\-empty convex subset of~$\HH^n$ with the following properties:
\begin{itemize}
  \item $\Delta$ is finite modulo $j(\Gamma_0)$, with vertices lying both in $\HH^n$ and in $\partial_{\infty}\HH^n$;
  \item the vertices in $\partial_{\infty}\HH^n$ are exactly the parabolic fixed points of $j(\Gamma_0)$;
  \item no edge of~$\Delta$ connects two such (ideal) vertices;
  \item in a neighborhood of a parabolic fixed point $\xi$ of rank $n-2$, the boundary of~$\Delta$ consists of two totally geodesic hyperplanes of~$\HH^n$ meeting only at~$\xi$.
\end{itemize}
\end{claim}

Let $\mathcal{N}\subset\HH^n$ be the uniform $1$-neighborhood of the preimage $N$ of the convex core of $M=j(\Gamma_0)\backslash\HH^n$. 
Let $X$ be a $j(\Gamma_0)$-invariant subset of~$\mathcal{N}$ that is locally finite modulo $j(\Gamma_0)$ and intersects every ball of diameter $\geq\alpha$ centered at a point of~$\mathcal{N}$, where $\alpha>0$ is chosen as in the proof of Claim~\ref{cla:ccopen}: whenever points $p,q$ of a $1$-hyperball of~$\HH^n$ are distance $\geq 1$ apart, the ball of radius~$\alpha$ centered at the midpoint of $[p,q]$ is also contained in the $1$-hyperball.
By a similar argument to the proof of Claim~\ref{cla:ccopen}, the Delaunay cellulation $\Delta$ of $\Conv{X}$ with respect to~$X$ is locally finite, with all cells equal to compact polyhedra of diameter $\leq 1$.
It remains to make $\Delta$ finite modulo $j(\Gamma_0)$ by modifying it inside each cusp.
For this purpose, we choose $X$ carefully. 

Let $H_1,\dots,H_c$ be open horoballs of~$\HH^n$, centered at points $\xi_1,\dots,\xi_c\in\partial_{\infty}\HH^n$, whose images in $j(\Gamma_0)\backslash\HH^n$ are disjoint and intersect the convex core in standard cusp regions, representing all the cusps.
We take these horoballs at distance $>2$ from each other in $j(\Gamma_0)\backslash\HH^n$, and small enough so that the conclusions of Lemma~\ref{lem:convexcusp} are satisfied.
Choose the $j(\Gamma_0)$-invariant, locally finite set~$X$ in general position subject to the following constraints:
\begin{itemize}
 \item[($\ast$)] $X\smallsetminus\bigcup_{i=1}^c \partial H_i$ stays at distance $\geq \alpha'$ from $\partial\mathcal{N}$ and from each $\partial H_i$, for some $\alpha'\in (0,\alpha)$ to be determined below;
  \item[($\ast\ast$)] for any $1\leq i\leq c$, the set $X\cap\partial H_i$ intersects any ball of~$\HH^n$ of radius $\alpha'/2$ centered at a point of $\partial H_i\cap\mathcal{N}$;
  \item[($\ast\!\ast\!\ast$)] if the stabilizer of $H_i$ has rank $n-2$, then $X$ intersects any Euclidean ball of radius $\alpha'/8$ in the boundary of $\partial H_i\cap\mathcal{N}$ in $\partial H_i\simeq\R^{n-1}$, while all other points of $X$ in $\partial H_i\cap\mathcal{N}$ are at distance $\geq \alpha'/4$ from the boundary of $\partial H_i\cap\mathcal{N}$ (which by Lemma~\ref{lem:convexcusp} consists of two parallel $(n-2)$-dimensional Euclidean hyperplanes of $\partial H_i$).
\end{itemize}
Qualitatively, this implies $X$ is especially concentrated on the horospheres $\partial H_i$ and even more on $\partial \mathcal{N} \cap \partial H_i$, but with a little buffer around these high-concentration regions.

Consider the Delaunay cellulation $\Delta$ of $\Conv{X}$ with respect to such a set~$X$.
Suppose two vertices $x,y$ of a given cell of~$\Delta$ (inscribed in a hypersurface bounding an open ball, horoball, or hyperball $B$ of~$\HH^n$ disjoint from~$X$) lie on opposite sides of one of the horospheres $\partial H_i$.
By ($\ast$), the points $x$ and~$y$ lie at distance $\geq\alpha'$ from $\partial\mathcal{N}$, hence so does the intersection point $\{z\}=[x,y]\cap\partial H_i$.
But at least one half of the ball of radius $\alpha'$ centered at~$z$ is contained in~$B$, hence $B\cap X\neq\emptyset$ by ($\ast\ast$): impossible.
Therefore the given cell of~$\Delta$ either has all its vertices in the closure of $H_i$, or has all its vertices in $\HH^n\smallsetminus H_i$. We can thus partition the cells of~$\Delta$ (of any dimension) into
\begin{itemize}                                                                                                                                                                                                                                                                                                                  \item \emph{interface cells}, with all their vertices in some $j(\gamma)\cdot\partial H_i$;
\item \emph{thin-part cells}, with all their vertices in the closure of some $j(\gamma)\cdot H_i$ (not all in the horosphere $j(\gamma)\cdot\partial H_i$);
\item \emph{thick-part cells}, with all their vertices in $\HH^n\smallsetminus j(\Gamma_0)\cdot\bigcup_{i=1}^c H_i$ (not all in the horospheres $j(\gamma)\cdot\partial H_i$).                                                                                                                                                                                                                                                                                                                   \end{itemize}
Consider the \emph{Euclidean} Delaunay cellulation $\Delta_{\partial H_i}$ of the Euclidean convex hull of $X\cap\partial H_i$ in $\partial H_i$, with respect to $X\cap\partial H_i$, in the classical sense (see~\cite{del34}): by definition, any cell of $\Delta_{\partial H_i}$ is inscribed in some Euclidean sphere bounding an open Euclidean ball of $\partial H_i$ disjoint from $X\cap\partial H_i$.

\begin{claim}
The geodesic straightenings of the  Euclidean triangulations $\Delta_{\partial H_i}$ give exactly the interface cells.
\end{claim}

\begin{proof}
For any interface cell $W$ of~$\Delta$, the projection of $W$ to $\partial H_i$ is a cell of $\Delta_{\partial H_i}$.
Indeed, if $W$ is inscribed in an open ball, horoball, or hyperball $B$ of~$\HH^n$ disjoint from~$X$, then the projection of~$W$ is inscribed in $B\cap\partial H_i$, which is a Euclidean ball (or half-plane) of $\partial H_i$ disjoint from~$X$.

Conversely, for any cell $W_E$ of $\Delta_{\partial H_i}$, the geodesic straightening of~$W_E$ is contained in~$\Delta$ as an interface cell.
Indeed, suppose $W_E$ is inscribed in an open Euclidean ball $B_E$ of $\partial H_i$, disjoint from $X\cap\partial H_i$, and centered in $\partial H_i\cap\mathcal{N}$.
By ($\ast\ast$), the hyperbolic ball $B$ concentric to~$B_E$ such that $B\cap\partial H_i=B_E$ has radius $\leq\alpha'/2$, hence is disjoint from $X$ by ($\ast$), which means that the geodesic straightening of~$W_E$ is contained in~$\Delta$.
Therefore, we just need to see that $W_E$ is always inscribed in such a ball~$B_E$.

If $H_i$ has rank $n-1$, this follows from the fact that $\partial H_i \cap \mathcal{N}=\partial H_i$ by Lemma~\ref{lem:convexcusp}.
If $H_i$ has rank $n-2$, this follows from ($\ast\!\ast\!\ast$): if $W_E$ is inscribed in a Euclidean open ball $B'_E$ of $\partial H_i$, disjoint from $X\cap\partial H_i$, and centered outside $\partial H_i\cap\mathcal{N}$, then $X\cap\partial B'_E$ is contained in a boundary component $P$ of $\mathcal{N}\cap\partial H_i$ (an $(n-2)$-dimensional Euclidean hyperplane by Lemma~\ref{lem:convexcusp}) and $W_E$ is inscribed in another ball $B_E$ of $\partial H_i$, still disjoint from~$X$, but centered at the projection of $p$ to~$P$.
\end{proof}

In fact, ($\ast\!\ast\!\ast$) also implies that the Euclidean Delaunay cellulation $\Delta_P$ of $P$ with respect to $X\cap P$ is contained in $\Delta_{\partial H_i}$.
Up to taking the points of~$X$ in generic position in $P$, in $\partial H_i$, and in~$\HH^n$, we can make sure that all three Delaunay cellulations $\Delta_P\subset\Delta_{\partial H_i}\subset\Delta$ (where the last inclusion holds up to geodesic straigthening) are in fact triangulations.

It follows from the comparison between hyperbolic and Euclidean Delaunay cellulations above that any geodesic ray escaping to the point at infinity $\xi_i\in\partial_{\infty}\HH^n$ of the cusp crosses the interface cells at most once.
Therefore the thin-part cells form a star-shaped domain relative to $\xi_i$.
We now modify $\Delta$ by removing all thin-part simplices and coning the interface simplices of $\Delta_{\partial H_i}$ off to~$\xi_i$.
We repeat for each cusp (these operations do not interfere, since the distance between two horoballs $H_i$ is larger than twice the diameter of any cell), and still denote by $\Delta$ the resulting complex (see Figure~\ref{fig:W}): it is now finite modulo $j(\Gamma_0)$.

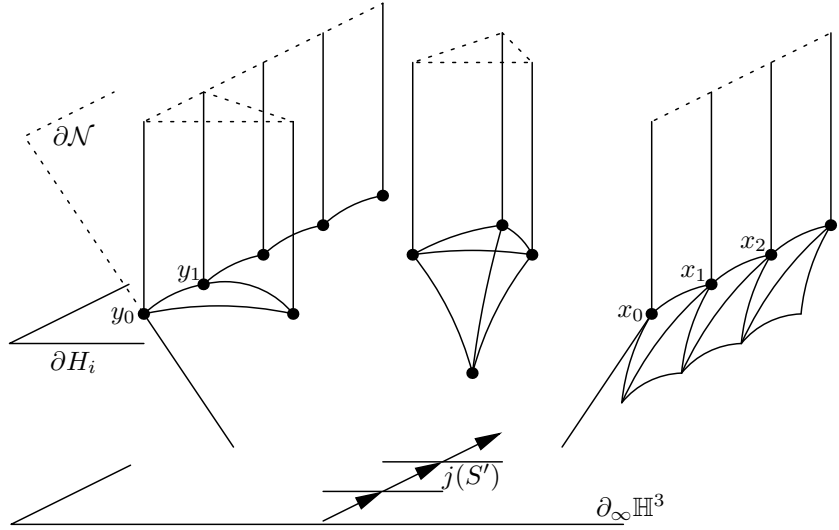
\begin{figure}[h!]
\begin{center}
\labellist
\small\hair 2pt
\pinlabel{$x_0$} at 188 65
\pinlabel{$x_1$} at 207 77
\pinlabel{$x_2$} at 225 86
\pinlabel{$y_0$} at 35 64
\pinlabel{$y_1$} at 55 77
\pinlabel{$\partial H_i$} at 20 50
\pinlabel{$\partial_{\infty}\HH^3$} at 188 6
\pinlabel{$j(S')$} at 140 15
\pinlabel{$\partial\mathcal{N}$} at 20 119
\endlabellist
\includegraphics[width=11cm]{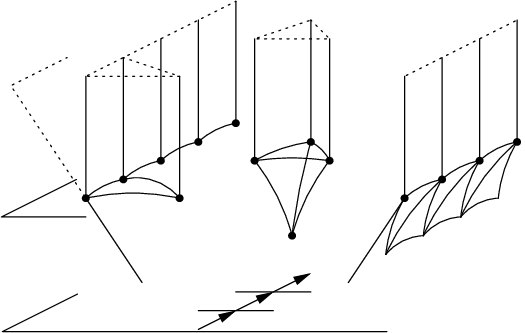}
\caption{The triangulation $\Delta$ in a rank-1 cusp bounded by a horosphere $\partial H_i$ centered at~$\infty$ in the upper half-space model of~$\HH^3$. At great height, the uniform neighborhood $\mathcal{N}$ of the convex core is bounded by just two oblique Euclidean planes. To simplify the picture, we have chosen $X$ to intersect each boundary line $P$ of $\mathcal{N}\cap \partial H_i$ in only one $j(S')$-orbit, $(x_s)_{s\in\Z}$ or $(y_s)_{s\in\Z}$. Since $j(S')$ is unipotent, the triangles $(\infty,x_s,x_{s+1})$ are coplanar. In the center we showed a thick-part tetrahedron and a thin-part tetrahedron (after coning off) which share an interface triangle. }
\label{fig:W}
\end{center}
\end{figure}

To complete the proof of Claim~\ref{cla:gfopen}, we must check that the new complex $\Delta$ is still convex.
This is clear at the cusps of rank $n-1$, since the corresponding $H_i$ satisfy $\mathcal{N}\cap H_i=H_i$ by Lemma~\ref{lem:convexcusp}.
At a cusp of rank $n-2$, above the interface $\Delta_{\partial H_i}$ (which is convex in $\partial H_i$ by the above discussion), the boundary of $\Delta$ consists of two geodesic hyperplanes tangent at infinity (by Lemma~\ref{lem:convexcusp}), and is therefore convex.
At the boundary of $\Delta_{\partial H_i}$, dihedral angles are convex because they already were before removal of the thin simplices.
This completes the proof of Claim~\ref{cla:gfopen}.

We now deduce Proposition~\ref{prop:rem-gfopen} from Claim~\ref{cla:gfopen}.
As above, let $H_1,\dots,H_c$ be horoballs of~$\HH^n$, centered at points $\xi_1,\dots,\xi_c\in\partial_{\infty}\HH^n$, whose images in $j(\Gamma_0)\backslash\HH^n$ are disjoint and intersect the convex core in standard cusp regions, representing all the cusps.
Let $p_1,\dots,p_r\in\HH^n$  be orbit representatives of the vertices of~$\Delta$ lying in~$\HH^n$. 
For $1\leq i\leq c$ and $k\in\N^{\ast}$, let $\xi_i^k\in\nolinebreak\partial_{\infty}\HH^n$ be the fixed point of $j_k(S_i)$, where $S_i$ is the stabilizer in~$\Gamma_0$ of $\xi_i$ under~$j$. 
Since converging parabolic elements have converging fixed points, $(\xi_i^k)_{k\in\N^{\ast}}$ converges to $\xi_i$ for all $1\leq i\leq c$.
We can thus find horoballs $H_i^k$ centered at~$\xi_i^k$ that converge to~$H_i$.
Whenever the corresponding cusp has rank $n-2$, the direction of the $j_k(S_i)$-invariant $(n-2)$-planes in $\partial H_i^k$ converges to the direction of the $j(S_i)$-invariant $(n-2)$-planes in $\partial H_i$.
For $1\leq i\leq r$, we also choose a sequence $(p_i^k)_{k\in\N^{\ast}}$ of points of~$\HH^n$, converging to~$p_i$, such that  if $[j(\gamma)\cdot p_i,j(\gamma')\cdot p_{i'}]$ is a boundary edge of $\Delta_{\partial H_i}$ (such as $[x_0, x_1]$ in Figure~\ref{fig:W}), then $j_k(\gamma)\cdot p_i^k$ and $j_k(\gamma')\cdot p_{i'}^k$ belong to a horocycle of $\partial H_i^k$ contained in some $j_k(S_i)$-stable $(n-2)$-plane of $\partial H_i^k$.
(Inside each boundary component of $j(S_i)\backslash \Delta_{\partial H_i}$, it is enough to enforce this condition over boundary edges that form a spanning tree in the quotient.)

The simplices spanned by the $j_k(\Gamma_0)\cdot p_i^k$ and $j_k(\Gamma_0)\cdot\xi_i^k$, following the combinatorics of~$\Delta$, still locally form a triangulation for large~$k$, because there are only finitely many orbits of simplices to check.
It remains to check that the $j_k(\Gamma_0)$-invariant collection $\Delta_k$ of such simplices triangulates a convex region. 
This can be ensured locally, at every codimension-2 face $W$ contained in the boundary of~$\Delta$.
If $W$ is compact, then the dihedral angle of $\Delta_k$ at $W$ goes to that of~$\Delta$, which is strictly convex. 
If $W$ has an ideal vertex~$\xi_i$, then $\partial\Delta$ is flat at~$W$ by Claim~\ref{cla:gfopen}, and $\partial\Delta_k$ is flat by choice of the~$p_i^k$.
Therefore $\Delta_k$ triangulates a convex region, which necessarily contains the convex core of $j_k(\Gamma_0)\backslash\HH^n$. 
In particular, $j_k(\Gamma_0)$ is still geometrically finite for large~$k$ (in the absence of torsion, the quotient hyperbolic manifolds are in fact homeomorphic since their convex cores admit combinatorially identical triangulations).
For the compact set $\mathcal{C}$ of Proposition~\ref{prop:rem-gfopen}, we can take a neighborhood of a union of orbit representatives of the compact simplices of~$\Delta$.
To bound injectivity radii away from~$0$, we argue as in the convex cocompact case, but in restriction to thick-part simplices only.
\end{proof}

\section{Open questions}\label{sec:questions}

Here we collect a few open questions, organized by themes; some of them were already raised in the core of the paper.

\subsection{General theory of extension of Lipschitz maps in $\HH^2$} \label{sec:questionlip}

Does there exist a function $F : (0,1)\rightarrow (0,1)$ such that for any compact subset $K$ of~$\HH^2$ and any Lipschitz map $\varphi : K\rightarrow\HH^2$ with $\Lip(\varphi)<1$, there is an extension $f : \HH^2\rightarrow\HH^2$ of~$\varphi$ with $\Lip(f)\leq F(\Lip(\varphi))$?
By controlling the sizes of the neigborhoods $\mathcal{U}_p$ in the proofs of Proposition~\ref{prop:KirszbraunC<1} and Lemma~\ref{lem:extend-neighb-p}, it is possible to deal with the case where a bound on the diameter of~$K$ has been fixed a priori.
An encouraging sign for unbounded~$K$ is that in Example~\ref{ex:K=3points}, where $K$ consists of three equidistant points, $C_{K,\varphi}(j,\rho)=\Lip(\varphi)+o(1)$ as the diameter of $K$ goes to infinity with $\Lip(\varphi)\in(0,1)$ fixed.

Fix a compact subset $K$ of~$\HH^2$ and a Lipschitz map $\varphi : K\rightarrow\HH^2$.
Is it possible to find an extension $f : \HH^2\rightarrow\HH^2$ of~$\varphi$ with minimal Lipschitz constant $C_{K,\varphi}(j,\rho)$, which is optimal in the sense of Definition~\ref{def:relstretchlocus} and satisfies $\Lip_p(f)=\Lip_p(\varphi)$ for all points $p\in K$ outside the relative stretch locus $E_{K,\varphi}(j,\rho)$?

Under the same assumptions, if $C:=C_{K,\varphi}(j,\rho)<1$, is it true that for any $p\in E_{K,\varphi}(j,\rho)\smallsetminus K$, there exists a point $q\neq p$ such that $[p,q]$ is $C$-stretched, \ie $d(f(p),f(q))=Cd(p,q)$?
By definition of the relative stretch locus, some segments near $p$ are nearly $C$-stretched, but it is not clear whether we can take $p$ as an endpoint.

\subsection{Geometrically infinite representations $j$ in dimension $n=3$} \label{sec:geom-infinite}

Does Theorem~\ref{thm:adm} hold for finitely generated $\Gamma_0$ but geometrically infinite~$j$?
To prove this in dimension~$3$, using the Ending Lamination Classification \cite{bcm12}, one avenue would be to extend Theorem~\ref{thm:lamin} in a way that somehow allows the stretch locus $E(j,\rho)$ to be an ending lamination. 
One would also need to prove a good quantitative rigidity statement for infinite ends: at least, that if two geometrically infinite manifolds $j(\Gamma_0)\backslash\HH^3$ and $j'(\Gamma_0)\backslash\HH^3$ have a common ending lamination, then $|\mu(j(\gamma_k))-\mu(j'(\gamma_k))|$ is bounded for some appropriate sequence $(\gamma_k)_{k\in\N}$ of elements of~$\Gamma_0$ whose associated loops go deeper and deeper into the common end.
(Here $\mu :\nolinebreak G\rightarrow\nolinebreak\R_+$ is the Cartan projection of \eqref{eqn:defmu}.)

\subsection{Nonreductive representations~$\rho$}

For $(j,\rho)\in\Hom(\Gamma_0,G)^2$ with $j$ geometrically finite and $\rho$ reductive, we know (Lemma~\ref{lem:Fnonempty}) that the infimum $C(j,\rho)$ of Lipschitz constants for $(j,\rho)$-equivariant maps $\HH^n\rightarrow\HH^n$ is always achieved (\ie $\F^{j,\rho}\neq\emptyset$).
Is it still always achieved for nonreductive~$\rho$ when $C(j,\rho)\geq 1$?
When $C(j,\rho)<1$, we know that it may or may not be achieved: see the examples in Sections \ref{ex:nonreductive1} and~\ref{ex:nonreductive2}.

\subsection{Behavior in the cusps for equivariant maps with minimal Lipschitz constant} \label{app:cusps}

Is there always a $(j,\rho)$-equivariant, $C(j,\rho)$-Lipschitz map that is constant in each deteriorating cusp? 
The answer is yes for $C(j,\rho)\geq 1$ (Proposition~\ref{prop:goodincusps}), but for $C(j,\rho)<\nolinebreak 1$ we do not even know if the stretch locus $E(j,\rho)$ has a compact image in $j(\Gamma_0)\backslash\HH^n$.
If it does, then one might ask for a uniform bound: do Proposition~\ref{prop:margulis} and Corollary~\ref{cor:uniform-thick} extend to $C(j,\rho)<1$?

Suppose that $C(j,\rho)=1$ and that $\rho$ is \emph{not} cusp-deteriorating.
If the stretch locus $E(j,\rho)$ is nonempty, does it contain a geodesic lamination whose image in $j(\Gamma_0)\backslash\HH^n$ is compact?

\subsection{Generalizing the Thurston metric}\label{subsec:dTh-questions}

To what extent can the $2$-dimen\-sional theory of the Thurston (asymmetric) metric on Teichm\"uller space be transposed to higher dimension?
In particular, how do the two asymmetric metrics $d_{\mathrm{Th}}$ and $d'_{\mathrm{Th}}$ of Section~\ref{subsec:Thurstonmetric} compare on the deformation space $\T(M)$ of a geometrically finite hyperbolic manifold~$M$?

The topology and geometry of $\T(M)$, or of level sets of the critical exponent $\delta : \T(M)\to (0,n-1]$, seem difficult but interesting to study.
Are any two points connected by a $d_{\mathrm{Th}}$-geodesic?
Is there an analogue of stretch paths (particular geodesics introduced in \cite{thu86})?
Is it possible to relate infinitesimal $d_{\mathrm{Th}}$-balls to the space of projective measured laminations~as~in~\cite{thu86}?

\subsection{Chain recurrence of the stretch locus}\label{subsec:ch-rec}

In dimension $n\geq 3$, when $C(j,\rho)>1$, does the stretch locus $E(j,\rho)$ have a chain-recurrence property as in Proposition~\ref{prop:chainrec}, in the sense that any point in the geodesic lamination $j(\Gamma_0)\backslash E(j,\rho)$ sits on a closed quasi-leaf?
Since there is no classification of geodesic laminations (Fact~\ref{fact:classif-lamin}) available in higher dimension, quasi-leaves can be generalized in at least two ways: either with a bound $\varepsilon\rightarrow 0$ on the \emph{total} size of all jumps from one leaf to the next, or (weaker) on the size of each jump separately.
It is not clear whether the two definitions coincide, even under constraints such as the conclusion of Lemma~\ref{lem:Esimple}.

In dimension $n\geq 2$, does chain recurrence, suitably defined, extend to the convex strata of Lemma~\ref{lem:1StretchedLam} when $C(j,\rho)=1$?

\subsection{Semicontinuity of the stretch locus}

Is the stretch locus map $(j,\rho)\mapsto E(j,\rho)$ upper semicontinuous for the Hausdorff topology when $C(j,\rho)$ is arbitrary, in arbitrary dimension~$n$?
Proposition~\ref{prop:semicontE} answers this question affirmatively in dimension $n=2$ for $C(j,\rho)>1$; the case $C(j,\rho)=1$ might allow for a proof along the same lines, using chain recurrence (suitably generalized as in Section~\ref{subsec:ch-rec} above).

\subsection{Graminations}

If $C(j,\rho)<1$ and $\F^{j,\rho}\neq\emptyset$, is the stretch locus $E(j,\rho)$ generically a trivalent geodesic tree (as in the example of Section~\ref{ex:C'neqCnoncompact})?
Is it, in full generality, what in Conjecture~\ref{conj:gramination} we called a \emph{gramination}, namely the union of a closed discrete set~$F$ and of a lamination in the complement of~$F$ (with leaves possibly terminating on~$F$)?

\vspace{0.5cm}

\end{document}